\documentclass[11pt]{article}
\usepackage{amsmath}

\usepackage{latexsym}
\usepackage{amssymb}
\usepackage{amsmath}
\usepackage{amsfonts}\large
\usepackage{amsthm}
\usepackage{fontenc}
\usepackage{fullpage}
\usepackage{enumerate}
\usepackage{esint}
\usepackage{bm}
\usepackage{hyperref}
\usepackage{titling}

\hypersetup{colorlinks=true}
\usepackage{xcolor}
\usepackage[normalem]{ulem}
\usepackage{marginnote}
\usepackage{enumerate}

\DeclareMathOperator{\dive}{div}
\DeclareMathOperator{\osc}{osc}
\DeclareMathOperator{\e}{\varepsilon}
\DeclareMathOperator{\diam}{diam}
\DeclareMathOperator{\Lip}{Lip}
\DeclareMathOperator{\Dr}{Dr}
\DeclareMathOperator{\loc}{loc}

\newtheorem{thm}{Theorem}[section]
\newtheorem{lemma}[thm]{Lemma}
\newtheorem{prop}[thm]{Proposition}
\newtheorem{cor}[thm]{Corollary}
\newtheorem{dfn}[thm]{Definition}

\begin{document}

\title{Boundary value problems in Lipschitz domains\\for equations with drifts}
\date{}
\author{Georgios Sakellaris\thanks{Department of Mathematics, \textit{University of Chicago}, Chicago, IL, 60637, USA \newline\hspace*{21.5pt}Email: gsakell@math.uchicago.edu}}
	
\thanksmarkseries{arabic}
	
\maketitle
	
\begin{abstract}
In this work we establish solvability and uniqueness for the $D_2$ Dirichlet problem and the $R_2$ Regularity problem for second order elliptic operators $L=-\dive(A\nabla\cdot)+b\nabla\cdot$ in bounded Lipschitz domains, where $b$ is bounded, as well as their adjoint operators $L^t=-\dive(A^t\nabla\cdot)-\dive(b\,\cdot)$. The methods that we use are estimates on harmonic measure, and the method of layer potentials.
		
The nature of our techniques applied to $D_2$ for $L$ and $R_2$ for $L^t$ leads us to impose a specific size condition on $\dive b$ in order to obtain solvability. On the other hand, we show that $R_2$ for $L$ and $D_2$ for $L^t$ are uniquely solvable, assuming only that $A$ is Lipschitz continuous (and not necessarily symmetric) and $b$ is bounded. 
\end{abstract}

\tableofcontents

\section{Introduction}
In this work we will be interested in boundary value problems for the equation
\[
Lu=-\dive(A\nabla u)+b\nabla u=0
\]
as well as the adjoint equation
\[
L^tu=-\dive(A^t\nabla u)-\dive(bu)=0,
\]
in a Lipschitz domain $\Omega\subseteq\mathbb R^d$, which is open and bounded. We will always assume that $d\geq 3$.

The boundary value problems we will be interested in are the \emph{Dirichlet problem}, $D_p$:
\[
\left\{\begin{array}{c l}
Lu=0,&{\rm in}\,\,\Omega\\
u=f,&{\rm on}\,\,\partial\Omega\\
\|u^*\|_{L^p(\partial\Omega)}<\infty,
\end{array}\right.\quad\quad {\rm for}\,\,f\in L^p(\partial\Omega),
\]
and the \emph{Regularity problem}, $R_p$:
\[
\left\{\begin{array}{c l}
Lu=0,&{\rm in}\,\,\Omega\\
u=f,&{\rm on}\,\,\partial\Omega\\
\|(\nabla u)^*\|_{L^p(\partial\Omega)}<\infty,
\end{array}\right. \quad\quad{\rm for}\,\,\,f\in W^{1,p}(\partial\Omega),
\]
and similarly for the equation $L^tu=0$ in $\Omega$, where $u^*$ denotes the nontangential maximal function of $u$ on $\partial\Omega$; that is,
\[
u^*(q)=\sup_{x\in\Gamma(q)}|u(x)|,
\]
where $\Gamma(q)\subseteq\Omega$ is a family of suitably chosen cones, which are based at $q\in\partial\Omega$.

We will show that under specific regularity conditions on the coefficients, those problems are uniquely solvable in $\Omega$, with constants depending only on the relevant norms of the coefficients, the Lipschitz character of $\Omega$ and the diameter of $\Omega$.

The methods we will use will be, first, the properties of harmonic measure for $D_2$ for the operator $L$, and second, the method of layer potentials for the other problems. In order to do this, we first assume that $A$ is symmetric (which is a crucial assumption in applying the Rellich estimates), and we then pass to non-symmetric $A$ with an integration by parts argument, appearing in \cite{PipherDrifts}. This passage will require to increase the assumed regularity on $A$ in the case of $R_2$ for $L^t$, but no such assumption will be needed in the case of the equation $Lu=0$.

The theorems that will be shown are the following (also appearing in theorems \ref{DirichletForNonSymmetric} and \ref{RegularityForNonSymmetric}). The term $\dive b$ will mean the divergence of $b$, in the sense of distributions.

\begin{thm}
Let $\Omega\subseteq\mathbb R^d$ be a bounded Lipschitz domain, and let $A$ be uniformly elliptic and Lipschitz, $b\in L^{\infty}(\Omega)$ and $\dive b\in L^{p_d}(\Omega)$ in the sense of distributions. Then there exists a constant $\e>0$ such that, for any $p\in(2-\e,\infty)$, the Dirichlet problem $D_p$ for the equation
\[
-\dive(A\nabla u)+b\nabla u=0
\]
is uniquely solvable in $\Omega$, with constants depending on $d$, $p$, the ellipticity of $A$, the Lipschitz norm of $A$, $\|b\|_{\infty}$, $\|\dive b\|_{p_d}$, the diameter of $\Omega$, and the Lipschitz character of $\Omega$.

(Here, $p_d=2$ for $d=3$, and $p_d=d/2$ for $d\geq 4$.)
\end{thm}

\begin{thm}
Let $\Omega\subseteq\mathbb R^d$ be a bounded Lipschitz domain, and let $A$ be uniformly elliptic and Lipschitz, $b\in L^{\infty}(\Omega)$. Then, the Regularity problem $R_2$ for the equation
\[
-\dive(A\nabla u)+b\nabla u=0
\]
is uniquely solvable in $\Omega$, with constants depending on $d$, the ellipticity of $A$, the Lipschitz norm of $A$, $\|b\|_{\infty}$, the diameter of $\Omega$, and the Lipschitz character of $\Omega$.
\end{thm}

For the adjoint equation we will show the following theorems (also appearing in theorems \ref{DirichletForNonSymmetricForAdjoint} and \ref{RegularityForNonSymmetricForAdjoint}).

\begin{thm}
Let $\Omega\subseteq\mathbb R^d$ be a bounded Lipschitz domain, let $A$ be uniformly elliptic and Lipschitz, and $b\in L^{\infty}(\Omega)$. Then, the Dirichlet problem $D_2$ for the equation
\[
-\dive(A^t\nabla u)-\dive(bu)=0
\]
is uniquely solvable in $\Omega$, with constants depending on $d$, the ellipticity of $A$, the Lipschitz norm of $A$, $\|b\|_{\infty}$, the diameter of $\Omega$, and the Lipschitz character of $\Omega$.
\end{thm}

\begin{thm}
Let $\Omega\subseteq\mathbb R^d$ be a bounded Lipschitz domain, and let $A$ be uniformly elliptic and $C^{1,\alpha}$, $b\in C^{\alpha}(\mathbb R^d)$ and $\dive b\in L^d(\mathbb R^d)$. Then, the Regularity problem $R_2$ for the equation
\[
-\dive(A^t\nabla u)-\dive(bu)=0
\]
is uniquely solvable in $\Omega$, with constants depending on $d$, the ellipticity of $A$, the $C^{1,\alpha}$ norm of $A$, $\|b\|_{C^{\alpha}},\|\dive b\|_d$, the Lipschitz character of $\Omega$ and the diameter of $\Omega$.

If $A$ is symmetric, it is enough to assume that it is Lipschitz and not $C^{1,\alpha}$.
\end{thm}

The problems we consider are classical, and there has been much work on those for operators not involving drifts; for a comprehensive list of past results, we refer to \cite{HofmannAnalyticity} and \cite{NguyenPaper}, as well as the references therein. A couple of results for operators with drifts include solvability for some large $p>1$, as in \cite{PipherDrifts}, or solvability under specific smallness assumptions on the Lipschitz constant of the domain, as in \cite{DindosPetermichl}. We also refer to \cite{HofmannLewis} for more results on equations with drifts.

We remark that solvability of $D_2$ and $R_2$ is the endpoint result for the best range of exponents for which solvability can be obtained. As we explain, using the theory of weights we can extend the theorems on the Dirichlet problem from $p=2$ to $p\in(2-\e,\infty)$. Also, using the theory of the Hardy space $H^1$, it is expected that a similar extension can be done for the Regularity problem, but in this case the range obtained is $p\in(1,2+\e)$, as in \cite{DahlbergKenig}.  Simple counterexamples involving harmonic functions in cones \cite{KenigCBMS} show that, for Lipschitz domains, those ranges for $p$ are optimal.

It is to be noted that we are assuming H{\"o}lder continuity of $b$ in the case of the Regularity problem $L^t$, but in the other cases there is no continuity assumption on $b$. This continuity assumption is used in the continuity properties of $\nabla u$, if $L^tu=0$: if $b$ is H{\"o}lder continuous, then $\nabla u$ is continuous, but no such assumption is required in order to obtain the analogous result for solutions of $Lu=0$. For the Dirichlet problem for $L^t$ we rely on solvability of $R_2$ for $L$ first, using properties of the single layer potential that have been established without assuming continuity of $b$.

Moreover, we need a size assumption on $\dive b$ for $D_2$ for $L$, as well as $R_2$ for $L^t$. On the other hand, no such assumption is needed for $R_2$ for $L$, as well as $D_2$ for $L^t$. An assumption like this is crucial for the method we will follow in this work (using the Rellich estimates), which shows solvability of the problems discussed above for the optimal exponent $p=2$. Note that this also is consistent with the duality that is discussed in \cite{HofmannDuality}. 

\subsection{Summary}
A short summary of this work is now in order.

In chapter 2 we will start with the various definitions, and we will describe the setting for the problems. We then proceed, in chapter 3, to show various properties of solutions to the equations $Lu=0$ and $L^tu=0$: that is, continuity of $u$, and continuity of $\nabla u$. We will then turn to the Rellich estimates, which will be our main tool in approaching our problems. We show two local estimates, one for $L$ and one for $L^t$; the one for $L^t$ requiring an assumption on $\dive b$.

The next step will be, in chapter 4, the construction of solutions to the inhomogeneous equations $Lu=F$ and $L^tu=F$, for various $F$. Since the operators $L,L^t$ are not necessarily coercive, we will construct the solutions using the Fredholm alternative, following the arguments in \cite{Evans} and the estimates in \cite{Gilbarg}.  It is to be noted that we will use an adjoint operator similar to the one appearing in section 5 of \cite{Littman} in order to construct solutions for measures. However, our solutions are not expected to be continuous up to the boundary for general domains, and we will have to restrict our attention to the space of bounded and continuous functions. The main difficulty that also arises is the exact dependence of the various constants on the coefficients, and how to pass to non Lipschitz coefficients $b$, since this will be one of our assumptions in some constructions.

In chapter 5 we proceed to the construction of Green's function, which will be used later in the formula for harmonic measure, as well as the method of layer potentials. There has been some work on similar constructions, for example in \cite{RiahiGreen}, \cite{MayborodaGreen} and \cite{HofmannLewis}, but our case is not covered by the previous ones, the main difference with \cite{MayborodaGreen} being the absence of coercivity. For this reason, we will construct Green's function for the adjoint equation first, since we obtain the correct dependence on the coefficients, using arguments similar to the ones appearing in \cite{Gruter}. Using a symmetry relation between Green's function for $L$ and the corresponding for $L^t$, we will then construct Green's function for $L$. Furthermore, using Green's representation formula for solutions, we then recover the correct dependence of the constants involved in the constructions of solutions in chapter 4. We also proceed to showing pointwise bounds on the derivative of Green's function, which are crucial in the development of the method of layer potentials, as well as various estimates that will be later used in arguments involving the continuity method. 

Chapter 6 involves construction of harmonic measure, and various estimates on nontangential maximal functions, following \cite{JerisonKenigBoundary} (we also refer to \cite{KenigCBMS} for a summary of those results). We then apply those estimates in chapter 7, which treats solvability and uniqueness for the Dirichlet problem for the equation $Lu=0$, in the case where $A$ is symmetric. A crucial component in this development is the Rellich estimate for the equation $L^tu=0$ (and not the corresponding estimate for $L$), hence a regularity assumption on the divergence of $b$ is imposed. This assumption involves just the divergence of $b$, and not any specific derivatives of $b$, and this is a crucial observation in our passage to non symmetric coefficients $A$ in chapter 11.

In chapter 8 we turn our attention to the Regularity problem for the equation $Lu=0$. The first step involves uniqueness for $R_2$. We then define the single layer potential, which is a singular integral operator on $L^2$ functions defined on the boundary of a Lipschitz domain $\Omega$. We also show convergence properties relying on the similar properties of the analogous potentials for equations that do not involve drifts (as in \cite{KenigShen}, also in \cite{MitreaTaylor} and \cite{MitreaTaylorBook}), using estimates on differences of Green's functions from chapter 5. We then proceed to showing that the single layer potential of any function in $L^2(\partial\Omega)$ is a solution to the equation $Lu=0$, and its maximal function is bounded. Then, using a global analog of the Rellich estimate, we show that boundary values of the single layer potential operator span all of $W^{1,2}(\partial\Omega)$, thus showing solvability of $R_2$ for the equation $Lu=0$, for symmetric matrices $A$. In this development, no assumption on the derivatives of $b$ is required.

In chapter 9 we treat the Dirichlet problem for the equation $L^tu=0$. We first show uniqueness, and, inspired by \cite{HofmannDuality}, we use the adjoint of the single layer potential from chapter 8, to obtain existence of solutions for symmetric matrices $A$. In order to show this, we do not reduce our case to the operators without a drift, as in chapter 8; instead, we rely on boundedness of a maximal truncation operator in order to obtain boundedness of the nontangential maximal functions, and a density argument to show what are the correct boundary values for the adjoint of the single layer potential. As in chapter 8, no assumption on the derivatives of $b$ is required.

Chapter 10 treats the Regularity problem for $L^t$. Here, we assume that $b$ is H{\"o}lder continuous, in order to have a similar formula for the derivative of the single layer potential as in chapter 8. We show the global Rellich estimate for symmetric matrices $A$, which is more complicated to obtain compared to the one in chapter 8, since the Rellich estimate for $L^t$ is more involved than its analog for $L$; for this purpose, we have to show estimates that involve parts of our domain that are close to the boundary.

Finally, in chapter 11, we pass to non-symmetric coefficients, using an integration by parts argument appearing in \cite{PipherDrifts}. Specifically, we transform the matrix $A$ to a matrix with symmetric coefficients, thus reducing to an equation with a new drift, which satisfies the divergence assumptions under which the Dirichlet and the Regularity problem have been shown to be solvable. 
\section{Preliminaries}
In this chapter we will discuss the various definitions and the setting of the problems, and we will perform some preliminary constructions.

\subsection{Definitions}
We say that $A$ is uniformly elliptic with ellipticity $\lambda$, if there exists $\lambda>0$ such that, for almost all $x\in\Omega$, and all $y\in\mathbb R^d$ with $y\neq 0$,
\begin{equation}\label{eq:UniformEllipticity}
\lambda|y|^2\leq\left<A(x)y,y\right>\leq\lambda^{-1}|y|^2.
\end{equation}

In the following, we will write $A\in M_{\lambda}(\Omega)$ to denote that $A$ satisfies \eqref{eq:UniformEllipticity} in $\Omega$. If $A$ is also symmetric, we will write $A\in M_{\lambda}^s(\Omega)$.

The regularity assumption on $A$ will be Lipschitz continuity; that is, for a constant $\mu>0$, for all $x,y\in\Omega$,
\begin{equation}\label{eq:LipschitzContinuity}
|A(x)-A(y)|\leq\mu|x-y|.
\end{equation}

To make the statements of the theorems more succinct, we will write $A\in M_{\lambda,\mu}(\Omega)$, if $A\in M_{\lambda}(\Omega)$ and $A$ also satisfies \eqref{eq:LipschitzContinuity}. $M_{\lambda,\mu}^s(\Omega)$ will denote the matrices in $M_{\lambda,\mu}(\Omega)$ that are symmetric.

We will denote by $C_c^{\infty}(\Omega)$ the space of infinitely differentiable functions in $\Omega$, which are compactly supported in $\Omega$. Moreover, we will be working in the classical Sobolev space $W^{1,p}(\Omega)$, where $p\geq 1$, which consists of functions $u\in L^p(\Omega)$ such that their derivative, in the distributional sense, is an $L^p(\Omega)$ function. Also, $W^{1,p}_{{\rm loc}}(\Omega)$ will denote the space of functions $u$, such that $u\phi\in W^{1,p}(\Omega)$ for any $\phi\in C_c^{\infty}(\Omega)$.

In addition to the above, $W_0^{1,p}(\Omega)$ will denote the closure of $C_c^{\infty}(\Omega)$ in $W^{1,p}(\Omega)$. Finally, the dual space to $W^{1,p}(\Omega)$ will be denoted by $W^{-1,p}(\Omega)$.

Let $\alpha$ be the bilinear form which is defined by
\[
\alpha(u,v)=\int_{\Omega}A\nabla u\nabla v+b\nabla u\cdot v.
\]
For any element $F\in W^{-1,2}(\Omega)$, a function $u\in W^{1,1}_{{\rm loc}}(\Omega)$ is a \emph{weak solution} to the equation $Lu=F$, if $\alpha(u,\phi)=F\phi$ for all $\phi\in C_c^{\infty}(\Omega)$.

Denote the space of bounded and continuous functions in $\Omega$ by $C_b(\Omega)$, and define $\mathcal{B}(\Omega)=C_b(\Omega)^*$. Note that there exists a description of $\mathcal{B}(\Omega)$ via the space of measures on the Stone-\v{C}ech compactification of $\Omega$, but we will not need this description in this work. Then, for any $\mu\in\mathcal{B}(\Omega)$, a function $u\in W^{1,1}_{\rm loc}(\Omega)$ is a \emph{weak solution} to the equation $Lu=\mu$, if $\alpha(u,\phi)=\left<\mu,\phi\right>$ for all $\phi\in C_c^{\infty}(\Omega)$.

In order to treat the adjoint equation $L^tu=0$, we also define the bilinear form
\[
\alpha^t(u,v)=\int_{\Omega}A^t\nabla u\nabla v+b\nabla v\cdot u.
\]
For any $\mu\in\mathcal{B}(\Omega)$, a function $u\in W^{1,1}_{\rm loc}(\Omega)$ is a weak solution to the equation $L^tu=\mu$, if $\alpha^t(u,\phi)=\left<\mu,\phi\right>$ for all $\phi\in C_c^{\infty}(\Omega)$. Note also that $\alpha(u,v)=\alpha^t(v,u)$ for all $v,u\in W^{1,2}(\Omega)$.

A function $u\in W^{1,1}_{\rm loc}(\Omega)$ is a \emph{subsolution} to the equation $Lu=0$, if $\alpha(u,\phi)\leq 0$ for all $\phi\in C_c^{\infty}(\Omega)$ with $\phi\geq 0$ in $\Omega$. Similarly, we define supersolutions.

Finally, for $\alpha\in (0,1]$, we denote by $C^{\alpha}(\Omega)$ the space of H{\"o}lder continuous functions in $\Omega$; that is, functions such that the seminorm
\[
\|b\|_{C^{0,\alpha}(\Omega)}=\sup\left\{\frac{|b(x)-b(y)|}{|x-y|^{\alpha}}\,\Big{|}\,x,y\in\Omega,x\neq y\right\}
\]
is finite. We then define the norm $\|b\|_{C^{\alpha}}=\|b\|_{\infty}+\|b\|_{C^{\alpha}}$. We also define $C^{1,\alpha}$ to be functions $b$ such that $\nabla b$ lies in $C^{\alpha}$, with the norm
\[
\|b\|_{C^{1,\alpha}}=\|b\|_{\infty}+\|\nabla b\|_{C^{\alpha}}.
\]

\subsection{Three extension lemmas}
In the following we will need to extend our coefficients, which are defined in $\Omega$, to either a neighborhood of $\Omega$ or the whole space $\mathbb R^d$, controlling the various norms associated with them. We show how to achieve this in the following lemmas.

\begin{lemma}\label{AExtension}
Let $\Omega\subseteq\mathbb R^d$ be an open domain and $A\in M_{\lambda,\mu}(\Omega)$. Then, there exists $\tilde{A}\in M_{\lambda,C\mu}(\mathbb R^d)$ such that $\tilde{A}=A$ on $\Omega$, where $C$ depends only on $d$.

If $A\in M_{\lambda,\mu}^s(\Omega)$, then there exists an extension $\tilde{A}\in M_{\lambda,C\mu}^s(\mathbb R^d)$.
\end{lemma}
\begin{proof}
Consider the extension operator $\mathcal{E}_0:\Lip(\Omega)\to\Lip(\mathbb R^d)$, appearing in \cite[p.~172]{SteinSingular}, which is given by
\[
\mathcal{E}_0f(x)=\sum_kf(p_k)\phi_k^*(x),
\]
and for each fixed $x$, the sum is in fact finite. Set $\tilde{A}=\mathcal{E}_0A$; that is,
\[
\mathcal{E}_0a_{ij}(x)=\sum_ka_{ij}(p_k)\phi_k^*(x),
\]
for all $i,j=1,\dots d$. Then, theorem 3 in \cite[p.~174]{SteinSingular} shows that $\tilde{A}$ is $C\mu$-Lipschitz continuous on $\mathbb R^d$, with $C$ depending only on $d$. Moreover, for any $x\notin\Omega$ and $y\in\mathbb R^d$ with $y\neq 0$,
\begin{align*}
\left<\tilde{A}(x)y,y\right>&=\sum_{i,j}\tilde{a}_{ij}(x)y_iy_j=\sum_{i,j}\mathcal{E}_0a_{ij}(x)y_iy_j=\sum_{i,j}\sum_ka_{ij}(p_k)\phi_k^*(x)y_iy_j\\
&=\sum_k\left(\sum_{i,j}a_{ij}(p_k)y_iy_j\right)\phi_k^*(x)=\sum_k\left<A(p_k)y,y\right>\phi_k^*(x)\\
&\geq\sum_k\lambda|y|^2\phi_k^*(x)=\lambda|y|^2,
\end{align*}
since $(\phi_k^*)$ is a partition of unity, and all the functions $\phi_k^*$ are positive. Since $\tilde{A}$ is an extension of $A$, the same inequality holds for $x\in\Omega$, which shows that $\tilde{A}$ is $\lambda$-uniformly elliptic in $\mathbb{R}^d$.

If $A$ is symmetric, the definition of $\tilde{A}$ shows that $\tilde{A}$ is also symmetric, which completes the proof.
\end{proof}

The next lemma shows how we can extend $A$ to also be periodic; thi will be a useful property in order to consider fundamental solutions in all of $\mathbb R^d$ for the equation $-\dive(A\nabla u)$. We will follow remark 6.2 in \cite{KenigShen} for this purpose.

\begin{lemma}\label{PeriodicAExtension}
Let $\Omega\subseteq\mathbb R^d$ be an open domain, with $\diam(\Omega)<1/4$, and $0\in\Omega$. Suppose also that $A\in M_{\lambda,\mu}(\Omega)$. Then, there exists an extension $A_p$ of $A$, which is $1$- periodic, and also $A_p\in M_{\lambda,\mu_0}(\mathbb{R}^d)$, where $\mu_0$ depends on $d,\lambda$ and $\mu$.
\end{lemma}
\begin{proof}
Since $0\in\Omega$ and $\diam(\Omega)<1/4$, we obtain that $\Omega\subseteq Q_1\subseteq Q_2\subseteq Q$, where
\[
Q_1=\left[-\frac{3}{8},\frac{3}{8}\right]^d,\quad Q_2=\left[-\frac{7}{16},\frac{7}{16}\right]^d, \quad Q=\left[-\frac{1}{2},\frac{1}{2}\right]^d.
\]
Fix now a smooth cutoff $\phi$ which is supported in $Q_2$, with $\phi\equiv 1$ in $Q_1$, and define
\[
A_p=\phi\tilde{A}+(1-\phi)\lambda I,
\]
where $I$ is the identity matrix, and $\tilde{A}$ is the extension that appears in lemma \ref{AExtension}. Since $\tilde{A}$ is $C\mu$-Lipschitz, we obtain that $A_p$ is $\mu_0$ Lipschitz, where $\mu_0$ depends on $d,\lambda$ and $\mu$, since the gradient of $\phi$ will be involved in the estimate. Moreover, since $\tilde{A}$ and $\lambda I$ are $\lambda$-uniformly elliptic, the same will be true for $A_p$.

Note now that, since $\phi\equiv 0$ in $Q_2$, $A_p=\lambda I$ in $Q_2$. Therefore, if we extend $A_p$ by $\lambda I$ in $Q\setminus Q_2$, we can then extend periodically to the rest of $\mathbb R^d$ by translations of $Q$, which completes the proof.
\end{proof}

Finally, we turn to Lipschitz extensions of drifts.

\begin{lemma}\label{bExtension}
Let $\Omega\subseteq\mathbb R^d$ be an open domain, and $b\in\Lip(\Omega)$. Then there exists an extension $\tilde{b}$ of $b$ in $\mathbb R^d$, such that $\tilde{b}\in\Lip(\mathbb R^d)$, and $\|\tilde{b}\|_{L^{\infty}(\mathbb R^d)}=\|b\|_{L^{\infty}(\Omega)}$.	
\end{lemma}
\begin{proof}
We will use the extension operator from lemma \ref{AExtension}, and set
\[
\tilde{b}(x)=\mathcal{E}_0b(x)=\sum_kb(p_k)\phi_k^*(x).
\]
From theorem 3 in \cite[p.~174]{SteinSingular}, $\tilde{b}$ is Lipschitz in $\mathbb R^d$, and $\tilde{b}$ extends $b$, hence $\|\tilde{b}\|_{L^{\infty}(\mathbb R^d)}\geq\|b\|_{L^{\infty}(\Omega)}$. Since now $(\phi_k^*)$ is a partition of unity and all the $\phi_k^*$ are positive, we obtain that
\[
|\tilde{b}(x)|\leq \sum_k|b(p_k)|\phi_k^*(x)\leq \sum_k\|b\|_{L^{\infty}(\Omega)}\phi_k^*(x)=\|b\|_{L^{\infty}(\Omega)},
\]
hence $\|\tilde{b}\|_{L^{\infty}(\mathbb R^d)}\leq\|b\|_{L^{\infty}(\Omega)}$.
\end{proof}

\subsection{Lipschitz domains}
For the next definitions, we will follow \cite[p.~575-577]{Verchota}, and \cite{KenigShen}.

Let $\Omega\subseteq\mathbb R^d$ be bounded. We say that $\Omega$ is a \emph{Lipschitz domain}, if for each $q\in\partial\Omega$ there exists a neighborhood $U\subseteq\mathbb R^d$ containing $q$ and a Lipschitz function $\phi_U:\mathbb R^{d-1}\to\mathbb R$ such that, after translation and rotation,
\[
U\cap\Omega=\left\{(x',t)\big{|}t>\phi_U(x')\right\}\cap\Omega.
\]
We also define a \emph{coordinate cylinder} $Z=Z(q,r)$, for $q\in\partial\Omega$, and $r>0$, to be a cylinder with radius equal to $r$, that also has the following properties. 
\begin{enumerate}[i)]
\item The bases of $Z$ are some positive distance from $\partial\Omega$.
\item There is a rectangular coordinate system for $\mathbb R^d$, $(x',t)$, with $t$- axis containing the axis of $Z_i$.
\item There is a Lipschitz function $\phi=\phi_Z:\mathbb R^{d-1}\to\mathbb R$.
\item $Z\cap\Omega=Z\cap\left\{(x',t)\big{|}t>\phi_Z(x')\right\}$
\item $q=(0,\phi_Z(0))$.
\end{enumerate}
We will call the pair $(Z,\phi)$ a \emph{coordinate pair}. 

By compactness, it is possible to cover $\partial\Omega$ by coordinate cylinders $Z_i=Z_i(q_i,r_{\Omega})$, for $q_i\in\partial\Omega$, $i=1,\dots N$ such that, for any $i$ there exists a coordinate pair $(Z_i^*,\phi_i)$ with $Z_i^*=c_{\Omega}Z_i(q,r_{\Omega})$ (the dilation with respect to $q$), and $c_{\Omega}=10\left(1+\|\nabla\phi_j\|_{\infty}\right)^{1/2}$.

Given a Lipschitz domain there exists $M>0$ such that, for any covering of coordinate cylinders, $\|\nabla\phi_j\|_{\infty}\leq M$. The smallest such number is called the Lipschitz constant for $\Omega$.

In order to quantify the results that will follow, given a Lipschitz domain with the fore mentioned properties, we will say that $\Omega\in\Pi(M,N)$.

Note now that, given any $q\in\partial\Omega$, $q$ belongs to one of the coordinate cylinders $Z_i^*=c_{\Omega}Z_i(q_i,r_{\Omega})$. Therefore, there exists a coordinate cylinder $Z(q,10r_{\Omega})$ that contains $q$, it is a subset of $Z_i^*$, with axis parallel to the axis of $Z_i^*$, and height comparable to $r_{\Omega}$.

\begin{dfn}\label{CylinderPortions}
For $q\in\partial\Omega$ and $r\in(0,10r_{\Omega})$, we define
\[
\Delta_r(q)=Z(q,r)\cap\partial\Omega,\quad T_r(q)=Z(q,r)\cap\Omega,
\]
where $Z(q,r)$ is a dilation of $Z(q,10r_{\Omega})$ as above, with respect to $q$.
\end{dfn}

A constant $C$ will be said to \emph{depend on the Lipschitz character} of $\Omega$, if $\Omega\in\Pi(M,N)$ for some $M,N$, and the constant can be made uniform for any $\Omega\in\Pi(M,N)$.

\begin{lemma}\label{r_Omega}
Let $\Omega$ be a Lipschitz domain. Then $r_{\Omega}$ is bounded above and below by constants that only depend on $d$, $\diam(\Omega)$ and the Lipschitz character of $\Omega$.
\end{lemma}
\begin{proof}
Since the coordinate cylinders $Z_i^*$ cover $\partial\Omega$, we have that
\[
\sigma(\partial\Omega)\leq\sigma\left(\bigcup_{i=1}^NZ_i^*\cap\partial\Omega\right)\leq\sum_{i=1}^N\sigma(Z(q_i,c_{\Omega}r_{\Omega})\cap\partial\Omega)=\sum_{i=1}^N\sigma(\Delta_{c_{\Omega}r_{\Omega}}(q_i))\leq NC_Mr_{\Omega}^{d-1}.
\]
From the isoperimetric inequality, $\sigma(\partial\Omega)$ is bounded below by a constant that depends on $\diam(\Omega)$ and $d$; therefore, $r_{\Omega}$ is bounded below by a constant depending only on $d$, $\diam(\Omega)$, and the Lipschitz character of $\Omega$. Since the coordinate cylinder $Z_i^*$ cannot contain all of $\Omega$, we also obtain that $r_{\Omega}$ is bounded above by a constant that depends on $\diam(\Omega)$ and $d$, which completes the proof.
\end{proof}

Note also that the definition of a Lipschitz domain shows the next lemma.

\begin{lemma}\label{InnerRadius}
For any Lipschitz domain $\Omega$, there exists a ball $B\subseteq\Omega$ which is compactly supported in $\Omega$, such that, for any $q\in\partial\Omega$,
\[
Z(q,s_{\Omega})\cap B=\varnothing,
\]
where the number $s_{\Omega}>0$ is bounded above and below by constants that depend on the diameter of $\Omega$ and the Lipschitz character of $\Omega$.
\end{lemma}

\begin{dfn}\label{GoodConstant}
Given a bounded Lipschitz domain $\Omega\subseteq\mathbb R^d$, $A\in M_{\lambda,\mu}(\Omega)$ and $b\in L^{\infty}(\Omega)$, we call a constant $C$ a \emph{good constant}, if it depends on $d,\lambda,\mu,\|b\|_{\infty},\diam(\Omega)$, and the Lipschitz character of $\Omega$.
\end{dfn}

Note that the diameter of $\Omega$ is allowed in the definition above, since we are assuming that $b\in L^{\infty}(\Omega)$, which is not a scale invariant space for $b$ for the equation considered; therefore, it is to be expected that the constants will depend on the size of the domain.

Given $q\in\partial\Omega$, $\Gamma(q)$ will denote an open, circular, doubly truncated cone with two nonempty, convex components, with vertex at $q$, and one component in $\Omega$ and the other in $\mathbb R^d\setminus\Omega$. The component interior to $\Omega$ will be denoted by $\Gamma_i(q)$ and the component exterior to $\Omega$ will be denoted by $\Gamma_e(q)$. Assigning a cone $\Gamma(q)$ to each $q\in\partial\Omega$, we call the family $\{\Gamma(q)\}$ regular if there is a finite covering of $\partial\Omega$ by coordinate cylinders, as described above, such that for each $(Z(p,r),\phi)$ there are three cones, $\alpha,\beta$ and $\gamma$, each with vertex at the origin and axis along the axis of $Z$ such that
\[
\alpha\subseteq\overline{\beta}\setminus\{0\}\subseteq\gamma,
\]
and for all $(x,\phi(x))=q\in\frac{4}{5}Z^*\cap\partial\Omega$,
\begin{gather*}
\alpha+q\subseteq\Gamma(q)\subseteq\overline{\Gamma(q)}\setminus\{q\}\subseteq\beta+q,\\
(\gamma+q)_i\subseteq\Omega\cap Z^*,\quad(\gamma+q)_e\subseteq Z^*\setminus\overline{\Omega}.
\end{gather*}

We now turn to the definition of the nontangential maximal function.

\begin{dfn}\label{NonTangentialMaximal}
Let $\Omega$ be a Lipschitz domain, and let $u:\Omega\to\mathbb R$. We then define the \emph{non-tangential maximal function} of $u$, for $q\in\partial\Omega$, by
\[
u^*(q)=\sup_{x\in\Gamma_i(q)}|u(x)|.
\]
Similarly, we define the nontangential maximal function for functions defined outside $\Omega$.
\end{dfn}

Finally, we state the next theorem on approximating Lipschitz domains by domains that are smooth, which is theorem 1.12 in \cite{Verchota}.

\begin{thm}\label{ApproximationScheme}
Let $\Omega$ be a Lipschitz domain. Then,
\begin{enumerate}[i)]
\item There is a sequence of $C^{\infty}$ domains, $\Omega_j\subseteq\Omega$, and homeomorphisms $\Lambda_j:\partial\Omega\to\partial\Omega_j$  such that $\sup_{q\in\partial\Omega}|q-\Lambda_j(q)|\to 0$ as $j\to\infty$, and $\Lambda_j(q)\in\Gamma_i(q)$ for all $j\in\mathbb N$ and $q\in\partial\Omega$
\item There is a covering of $\partial\Omega$ by coordinate cylinders $Z$ so that, given a coordinate pair, $(Z,\phi)$, then $Z^*\cap\partial\Omega_j$ is given for each $j$ as the graph of a $C^{\infty}$ function $\phi_j$ such that $\phi_j\to\phi$ uniformly, $\|\nabla\phi_j\|_{\infty}\leq\|\nabla\phi\|_{\infty}$ and $\nabla\phi_j\to\nabla\phi$ pointwise almost everywhere and in every $L^q(Z^*\cap\mathbb R^{d-1})$, $1\leq q<\infty$
\item There are positive functions $\tau_j:\partial\Omega\to\mathbb R_+$, bounded away from zero and infinity uniformly in $j$, such that for any measurable set $E\subseteq\partial\Omega$, $\int_E\tau_j\,d\sigma=\int_{\Lambda_j(E)}\,d\sigma_j$, and so that $\tau_j\to 1$ pointwise almost everywhere and in every $L^q(\partial\Omega)$, $1\leq q<\infty$
\item The normal vectors to $\Omega_j$, $\nu(\Lambda_j(q))$ converge pointwise almost everywhere and in every $L^q(\partial\Omega)$, $1\leq q<\infty$, to $\nu(q)$. An analogous statement holds for locally defined tangent vectors.
\item There exists a $C^{\infty}$ vector field, $h$, in $\mathbb R^d$ such that $\left<h(\Lambda_j(q)),\nu(\Lambda_j(q))\right>\geq C>0$ for all $j$ and $q\in\partial\Omega$, for some $C$ depending only on $h$ and the Lipschitz constant for $\Omega$.
\end{enumerate}
\end{thm}
This approximation scheme will be denoted by $\Omega_j\uparrow\Omega$.

Consider now a Lipschitz domain $\Omega$, and a large ball $B$ containing $\Omega$. A similar construction can be carried out for the Lipschitz domain $2B\setminus\Omega$, where $2B$ is the double ball of $B$, and we obtain a sequence $U_j\uparrow (2B\setminus\Omega)$. Eventually, for $j$ large, the sets $U_j$ will contain the boundary of $B$. Then, for those $j$, set $\Omega_j'=B\setminus U_j$, and note that $\Omega_j'$ is an approximation scheme similar to the above, but then the sequence $\Omega_j'$ decreases to $\Omega$ (as in definition 1.13 in \cite{Verchota}).

\subsection{Function spaces}
The definition that follows will be the setting for space of drifts for which solvability of the Dirichlet problem for the equation $Lu=0$ will be shown.

\begin{dfn}\label{DistributionalDivergence}
Let $\Omega\subseteq\mathbb R^d$ be a Lipschitz domain, and $p>1$. We define $\Dr_p(\Omega)$ to be the space of bounded vector functions $b$ on $\Omega$, such that their distributional divergence belongs to $L^p(\Omega)$; that is, there exists $C>0$ such that, for all $\phi\in C_c^{\infty}(\Omega)$,
\[
\left|\int_{\Omega}b\nabla\phi\right|\leq C\|\phi\|_{L^{\frac{p}{p-1}}(\Omega)}.
\]
We also define the $\Dr_p$-norm of $b\in\Dr(\Omega)$ to be
\[
\|b\|_{\Dr_p}=\|b\|_{L^{\infty}(\Omega)}+\|\dive b\|_{L^p(\Omega)},
\]
where $\|\dive b\|_{L^p(\Omega)}$ is the infimum of the $C$ in the inequality above, for $\phi\in C_c^{\infty}(\Omega)$.
\end{dfn}

In some cases, we will need to assume some further regularity on $b$, together with the fact that $b\in\Dr_p$. For this purpose, we give the next definition.

\begin{dfn}
Let $\Omega\subseteq\mathbb R^d$ be a Lipschitz domain, and $p>1$, $\alpha\in(0,1]$. We then define
\[
\Dr_{p,\alpha}(\Omega)=\Dr_p(\Omega)\cap C^{\alpha}(\Omega),
\]
with the norm
\[
\|b\|_{\Dr_{p,\alpha}(\Omega)}=\|b\|_{\Dr_p(\Omega)}+\|b\|_{C^{\alpha}(\Omega)}.
\]
\end{dfn}

We now turn to Lorentz spaces.

\begin{dfn}
For $p\in[1,\infty)$, the Lorentz space $L^p_*(\Omega)$ is the space of measurable functions $f:\Omega\to\mathbb R$ such that 
\[
\|f\|_{L^p_*(\Omega)}=\sup_{t>0}\left(t\lambda_f^{1/p}(t)\right)<\infty,
\]
where $\lambda_f$ is the distribution function of $f$; that is, $\lambda_f(t)=\left|\left\{x\in\Omega\big{|} |f(x)|>t\right\}\right|$.
\end{dfn}
Note that, if $f\in L^p_*(\Omega)$, Chebyshev's inequality shows that
\[
\|f\|_{L^p_*(\Omega)}\leq\|f\|_{L^p(\Omega)}.
\]
The Lorentz norm also bounds the $p$ norms of lower order: from estimate 1.12 in \cite{Gruter}, if $\Omega\subseteq\mathbb R^d$ is bounded, $p\in[1,\infty)$ and $\delta\in(0,p-1]$, then
\begin{equation}\label{eq:lorbound}
\|f\|_{L^{p-\delta}(\Omega)}\leq\left(\frac{p}{\delta}\right)^{\frac{1}{p-\delta}}|\Omega|^{\frac{\delta}{p(p-\delta)}}\|f\|_{L^p_*(\Omega)}.
\end{equation}

Next, we turn to class of Gehring weights, which are functions that satisfy the reverse H{\"o}lder inequality, the definition of which can also be found in \cite{Gehring}.

\begin{dfn}\label{Bp}
Let $\Omega$ be a Lipschitz domain and $p\geq 1$. We say that $f\in L^p(\partial\Omega)$ belongs to the class $B_p(\partial\Omega)$, if there exists a constant $C>0$ such that, for any surface ball $\Delta_r(q)\subseteq\partial\Omega$,
\[
\left(\fint_{\Delta_r(q)}|f|^p\,d\sigma\right)^{1/p}\leq C\fint_{\Delta_r(q)}|f|\,d\sigma.
\]
\end{dfn}

The main property of the $B_p$ weights we will use is their ability to self-improve. Specifically, similarly to lemma 3 in \cite{Gehring}, we obtain the next proposition.

\begin{prop}\label{BpProperty}
Let $\Omega$ be a Lipschitz domain, $p>1$, and $f\in B_p(\partial\Omega)$. Then, there exists $\e>0$, which depends only on $d,p$ and the $B_p$ constant of $f$, such that $f\in B_{p+\e}(\partial\Omega)$.
\end{prop}

Finally, we define weak derivatives on the boundary of a Lipschitz domain $\Omega$, and the space $W^{1,p}(\partial\Omega)$.

\begin{dfn}\label{W1pPartialOmega}
Let $\Omega$ be a Lipschitz domain. Then, we say that $f\in W^{1,p}(\partial\Omega)$ if $f\in L^p(\partial\Omega)$ and if for each coordinate pair $(Z,\phi)$ there exist $L^p(Z\cap\partial\Omega)$ functions $g_1,\dots g_{d-1}$, so that, for every $h\in C_c^{\infty}(Z\cap\mathbb R^{d-1})$,
\[
\int_{\mathbb R^{d-1}}h(x)g_i(x,\phi(x))\,dx=-\int_{\mathbb R^{d-1}}\partial_ih(x)f(x,\phi(x))\,dx.
\]
\end{dfn}
In local coordinates, we then define (as in \cite[pg.~580]{Verchota})
\[
-\nabla_Tf(p)=(g_i^1(p'),\dots g_i^{d-1}(p'),0)-\left<(g_i^1(p'),\dots g_i^{d-1}(p'),0),\nu(p)\right>\cdot\nu(p).
\]
Then $\nabla_Tf(p)$ is normal to $\nu(p)$ almost everywhere on $\partial\Omega$, and it is independent of the choice of coordinates. We also define the norm
\[
\|f\|_{W^{1,p}(\partial\Omega)}=\|f\|_{L^p(\partial\Omega)}+\|\nabla_Tf\|_{L^p(\partial\Omega)}.
\]
In the special case $p=2$, $W^{1,2}(\partial\Omega)$ becomes a Hilbert space, with the inner product
\[
\left<f,g\right>_{W^{1,2}(\partial\Omega)}=\int_{\partial\Omega}\left(f\cdot g+\nabla_Tf\cdot\nabla_Tg\right)\,d\sigma.
\]

For $p\in(1,\infty)$, we denote the dual of $W^{1,p}(\partial\Omega)$ by $W^{-1,p}(\partial\Omega)$. The fact that the dual to $L^p(\partial\Omega)$ is $L^{p'}(\partial\Omega)$,  where $p'$ is the conjugate exponent to $p$, and reflexivity of $L^p(\partial\Omega)$ for $p\in(1,\infty)$ imply the next lemma.
\begin{lemma}\label{Riesz}
$W^{-1,p}(\partial\Omega)$ is reflexive, and for every $F\in W^{-1,p}(\partial\Omega)$, there exists a unique $f\in W^{1,p'}(\partial\Omega)$ such that
\[
Fg=\int_{\partial\Omega}f\cdot g+\nabla_Tf\cdot\nabla_Tg
\]
for all $g\in W^{1,p}(\partial\Omega)$. We will then write $F=R_pf$.
\end{lemma}

For any $p\in(1,\infty)$, consider the canonical embedding operator
\[
E_p:L^{p'}(\partial\Omega)\to W^{-1,p}(\partial\Omega),\quad (E_pf)g=\int_{\partial\Omega}fg\,d\sigma
\]
for all $g\in W^{1,p}(\partial\Omega)$. Under this embedding of $L^{p'}(\partial\Omega)$ in $W^{-1,p}(\partial\Omega)$, we will show that the image $E_p(L^{p'}(\partial\Omega))$ is dense in $W^{-1,p}(\partial\Omega)$; we first show the local analog in the setting of $\mathbb R^{d-1}$.

\begin{lemma}\label{L2DensityInRd}
Let $B$ be a ball in $\mathbb R^{d-1}$, and $p\in (1,\infty)$. Then, $E_p(L^{p'}(B))$ is dense in $W^{-1,p}(B)$, where $E_p:L^{p'}(B)\to W^{-1,p}(B)$ is the canonical embedding, and $W^{-1,p}(B)=\left(W^{1,p}(B)\right)^*$.
\end{lemma}
\begin{proof}
Let $g\in (W^{-1,p}(B))^*$, which is such that
\[
\left<g,Pf\right>=0
\]
for all $f\in L^{p'}(B)$. From reflexivity of $W^{-1,p}(B)$, there exists $\tilde{g}\in W^{1,p}(B)$, such that
\[
\left<Pf,\tilde{g}\right>=0
\]
for all $f\in L^{p'}(B)$. This implies that, for any $f\in L^{p'}(B)$, $\int_Bf\tilde{g}=0$, therefore $\tilde{g}=0$, hence $g=0$. This shows that every element of $(W^{-1,p}(\partial\Omega))^*$ that vanishes on $E_p(L^{p'}(B))$, has to be identically zero, hence the Hahn-Banach theorem shows that $E_p(L^{p'}(B))$ is dense in $W^{-1,p}(B)$.
\end{proof}

\begin{lemma}\label{DensityOfL2}
The image of $L^{p'}(\partial\Omega)$ in $W^{-1,p}(\partial\Omega)$ under $E_p$ is dense in $W^{-1,p}(\partial\Omega)$.
\end{lemma}
\begin{proof}
Let $F\in W^{-1,p}(\partial\Omega)$ and $\varepsilon>0$. Then, from lemma \ref{Riesz}, $F=R_pf$ for some $f\in W^{1,p'}(\partial\Omega)$.

Consider the coordinate cylinders $(Z_j,\phi_j)$, $j=1,\dots N$ that cover $\partial\Omega$, and let $(\psi_j)$ be a partition of unity subordinate to the $Z_j$. Let also $g\in W^{1,p}(\partial\Omega)$, $B_j\subseteq\mathbb R^{d-1}$ be the basis of the cylinder $Z_j$, and define
\[
\tilde{f}_j(x)=f(x,\phi_j(x)),\quad \tilde{g}_j(x)=g(x,\phi_j(x)),\quad \tilde{\psi}_j(x)=\psi_j(x,\phi_j(x)),
\]
for $x\in B_j$. Then, since $f\in W^{1,p'}(\partial\Omega)$, we obtain that $\tilde{f}_j\tilde{\psi}_j\in W_0^{1,p'}(B_j)$.

Denote the $d-1$ dimensional gradient by $\tilde{\nabla}$. Using the definition of the tangential gradient, we compute in local coordinates in $\partial\Omega\cap Z_j$,
\[
\nabla_Tf\cdot\nabla_Tg=\tilde{\nabla}\tilde{f}_j\cdot\tilde{\nabla}\tilde{g}_j-\left<\tilde{\nabla}\tilde{f}_j,\nu_j\right>\left<\tilde{\nabla}\tilde{g}_j,\nu_j\right>,
\]
where $\nu_j$ is the unit normal in $\partial\Omega\cap Z_j$. Hence, setting $\theta_j=\sqrt{1+|\nabla\phi_j|^2}$, we compute
\begin{align*}
Fg&=\int_{\partial\Omega}\left(fg+\nabla_Tf\nabla_Tg\right)\,d\sigma=\sum_{j=1}^N\int_{\partial\Omega\cap Z_j}\left(fg+\nabla_Tf\nabla_Tg\right)\psi_j\,d\sigma\\
&=\sum_{j=1}^N\int_{B_j}\left(\tilde{f}_j\tilde{g}_j+\tilde{\nabla}\tilde{f}_j\cdot\tilde{\nabla}\tilde{g}_j-\left<\tilde{\nabla}\tilde{f}_j,\nu_j\right>\left<\tilde{\nabla}\tilde{g}_j,\nu_j\right>\right)\tilde{\psi}_j\theta_j.
\end{align*}
Consider now, for $j=1,\dots N$, the operator
\[
\tilde{F}_j\tilde{g}=\int_{B_j}\left(\tilde{f}_j\tilde{g}+\tilde{\nabla}\tilde{f}_j\cdot\tilde{\nabla}\tilde{g}-\left<\tilde{\nabla}\tilde{f}_j,\nu_j\right>\left<\tilde{\nabla}\tilde{g},\nu_j\right>\right)\tilde{\psi}_j\theta_j.
\]
Since $\tilde{f}_j\in W^{1,p'}(B_j)$ and $\tilde{\psi}_j,\theta_j$ are bounded, $\tilde{F}_j\in W^{-1,p}(B_j)$.

Let now $\e>0$. Then, from lemma \ref{L2DensityInRd}, for every $j=1,\dots N$, there exists $\tilde{h}_j\in L^{p'}$ such that
\[
\|P_j\tilde{h}_j-\tilde{F}_j\|_{W^{-1,p}(B_j)}<\e.
\]
where $\tilde{P}_j:L^{p'}(B_j)\to W^{-1,p}(B_j)$ is the canonical embedding. Also, for $q\in\partial\Omega\cap Z_j$, $q=(x,\phi_j(x))$ for some $x\in B_j$; we then define $h_j(q)=\tilde{h}_j(x)$, and we extend $h_j$ by zero on $\partial\Omega\setminus Z_j$. We then obtain that $h_j\in L^{p'}(\partial\Omega\cap Z_j)$. Define also
\[
h=\sum_{j=1}^Nh_j.
\]
Then, we compute, for $g\in W^{1,p}(\partial\Omega)$,
\[
P_ph(g)=\int_{\partial\Omega}hg\,d\sigma=\sum_{j=1}^N\int_{\partial\Omega\cap Z_j}h_jg\,d\sigma=\sum_{j=1}^N\int_{B_j}\tilde{h}_j\tilde{g}_j\theta_j,
\]
therefore
\begin{align*}
|Fg-P_ph(g)|&\leq \sum_{j=1}^N\left|\tilde{F}_j\tilde{g}_j-\tilde{P}_j\tilde{h}_j(\tilde{g}_j)\right|\leq\sum_{j=1}^N\|\tilde{P}_jh_j-\tilde{F}_j\|_{W^{-1,p}(B_j)}\|\tilde{g_j}\|_{W^{1,p}(B_j)}\\
&\leq \e\sum_{j=1}^N\|\tilde{g_j}\|_{W^{1,p}(B_j)}\leq CN\e\|g\|_{W^{1,p}(\partial\Omega)}.
\end{align*}
This shows that $\|F-P_ph\|_{W^{-1,p}(\partial\Omega)}\leq CN\e$, which completes the proof.
\end{proof}

\section{A priori estimates}
In this chapter we will discuss some a priori estimates related to solutions of the equations $Lu=0$ and $L^tu=0$.

\subsection{The Cacciopoli estimate}
We begin with the Cacciopoli estimate, which is a reverse Poincare inequality for solutions of the equation $Lu=0$.

\begin{lemma}\label{Cacciopoli}
Let $\Omega$ be a Lipschitz domain, and let $A\in M_{\lambda}(\Omega)$, $b\in L^{\infty}(\Omega)$.
\begin{enumerate}[i)]
\item Let $u\in W^{1,2}_{{\rm loc}}(\Omega)$ be a solution to $Lu=0$ in $\Omega$. Then, for all balls $B_r\subseteq\Omega$ such that $B_{2r}$ is compactly supported in $\Omega$,
\[
\int_{B_r}|\nabla u|^2\leq C\left(1+\frac{1}{r^2}\right)\int_{B_{2r}}u^2,
\]
where $C=C(d,\lambda,\|b\|_{\infty})$.
\item If $u\in W^{1,2}(\Omega)$ is a nonnegative solution in $\Omega$, that vanishes on $\Delta_{2r}(q)$ for some $q\in\partial\Omega$, then the same inequality holds in $\Delta_{2r}(q)$.
\end{enumerate}
\end{lemma}
\begin{proof} 
Let $\phi$ be a smooth cutoff which is supported in $B_{2r}$, $\phi\equiv 1$ in $B_r$, and $|\nabla\phi|\leq C/r$. We then use $u\phi^2$ as a test function, to obtain that
\[
\int_{\Omega}A\nabla u\nabla (u\phi^2)+b\nabla u\cdot u\phi^2=0,
\]
which shows that
\[
\lambda\int_{\Omega}|\phi\nabla u|^2\leq\int_{\Omega}A\nabla u\nabla u\cdot\phi^2\leq-2\int_{\Omega}A\nabla u\nabla\phi\cdot u\phi-\int_{\Omega}b\nabla u\cdot u\phi^2.
\]
To bound the last two terms, we use the Cauchy inequality with $\delta$, to obtain that
\begin{align*}
-2\int_{\Omega}A\nabla u\nabla\phi\cdot u\phi-\int_{\Omega}b\nabla u\cdot u\phi^2&\leq C\|\phi\nabla u\|_2\|u\nabla \phi\|_2+\|bu\phi\|_2\|\phi\nabla u\|_2\\
&=(C\|u\nabla \phi\|_2+\|bu\phi\|_2)\|\phi\nabla u\|_2\\
&\leq\frac{1}{4\delta}(C\|u\nabla \phi\|_2+\|bu\phi\|_2)^2+\delta\|\phi\nabla u\|_2^2.
\end{align*}
Choosing $\delta=\lambda/2$, we obtain 
\[
\int_{\Omega}|\phi\nabla u|^2\leq C\|u\nabla\phi\|_2^2+C\|u\phi\|_2^2=C\int_{\Omega}(|\phi|^2+|\nabla\phi|^2)u^2,
\]
where $C$ depends on $d$, $\lambda$ and $\|b\|_{\infty}$. This shows that
\[
\int_{B_r}|\nabla u|^2\leq C\int_{\Omega}(|\phi|^2+|\nabla\phi|^2)u^2\leq C\left(1+\frac{1}{r^2}\right)\int_{B_{2r}}u^2,
\]
which is the desired estimate.

For the second estimate, we let $\phi$ be a smooth cutoff which is supported in $T_{2r}(q)$, $\phi\equiv 1$ in $T_r(q)$, and $|\nabla\phi|\leq C/r$. We then apply the same argument as above in $T_{2r}(q)$, noting that $u\phi^2\in W_0^{1,2}(T_{2r}(q))$.
\end{proof}

We now show that the same inequality holds for solutions to the equation $L^tu=\dive f$.

\begin{lemma}\label{CacciopoliForAdjoint}
Let $\Omega$ be a Lipschitz domain, and let $A\in M_{\lambda}(\Omega)$, $b\in L^{\infty}(\Omega)$, and $f\in L^2(\Omega)$.
\begin{enumerate}[i)]
\item Let $u\in W_{{\rm loc}}^{1,2}(\Omega)$ be a solution to $L^tu=\dive f$ in $\Omega$. Then, for all balls $B_r\subseteq\Omega$ such that $B_{2r}$ is compactly supported in $\Omega$,
\[
\int_{B_r}|\nabla u|^2\leq C\left(1+\frac{1}{r^2}\right)\int_{B_{2r}}u^2+\int_{B_{2r}}|f|^2,
\]
where $C=C(d,\lambda,\|b\|_{\infty})$.
\item If $u\in W^{1,2}(\Omega)$ is a nonnegative solution of $L^tu=0$ in $\Omega$ that vanishes on $\Delta_{2r}(q)$ for some $q\in\partial\Omega$, then the same inequality holds in $\Delta_{2r}(q)$.
\end{enumerate}
\end{lemma}
\begin{proof} 
The proof is similar to the proof of proposition \ref{Cacciopoli}. Let $\phi$ be a smooth cutoff which is supported in $B_{2r}$, $\phi\equiv 1$ in $B_r$, and $|\nabla\phi|\leq C/r$. We then use $u\phi^2$ as a test function, to obtain that
\[
\int_{\Omega}A\nabla u\nabla (u\phi^2)+b\nabla(u\phi^2)\cdot u=\int_{\Omega}f\nabla(u\phi^2),
\]
which shows that
\begin{align*}
\int_{\Omega}A\nabla u\nabla u\cdot\phi^2&\leq-2\int_{\Omega}A\nabla u\nabla\phi\cdot u\phi-\int_{\Omega}b\nabla u\cdot u\phi^2-\int_{\Omega}2b\nabla\phi\cdot u\phi+\int_{\Omega}f\nabla(u\phi^2)\\
&\leq(2\|u\nabla \phi\|_2+\|bu\phi\|_2+\|f\phi\|_2)\|\phi\nabla u\|_2+(2\|b\phi\|_2+\|f\phi\|_2)\|u\nabla\phi\|_2,
\end{align*}
and the uniform ellipticity of $A$ shows that
\[
\lambda\|\phi\nabla u\|_2^2\leq(2\|u\nabla \phi\|_2+\|bu\phi\|_2+\|f\phi\|_2)\|\phi\nabla u\|_2+(2\|b\phi\|_2+\|f\phi\|_2)\|u\nabla\phi\|_2.
\]
Hence, using the Cauchy inequality with $\delta$, we obtain that
\[
\int_{\Omega}|\phi\nabla u|^2\leq C\|u\nabla\phi\|_2^2+C\|u\phi\|_2^2+C\|f\|_2^2=C\int_{\Omega}(|\phi|^2+|\nabla\phi|^2)u^2+\int_{\Omega}|f\phi|^2,
\]
where $C$ depends on $d$, $\lambda$ and  $\|b\|_{\infty}$. The estimate now follows.

For the second estimate, we apply the same argument in $T_{2r}(q)$ for $\phi$ vanishing outside $T_{2r}(q)$, being equal to $1$ in $T_r(q)$, and $|\nabla\phi|\leq C/r$, noting that $u\phi^2\in W_0^{1,2}(T_{2r}(q))$.
\end{proof}

\subsection{Low regularity estimates}
In this section we will use the ellipticity of the equation to show how we can gain $L^p$ regularity. We first show a local weak reverse H{\"o}lder inequality for solutions to the equation $Lu=0$. 

\begin{lemma}\label{LowRegularityEstimate}
Let $\Omega$ be a bounded domain, $A\in M_{\lambda,\mu}(\Omega)$, and $b\in L^{\infty}(\Omega)$. Suppose that a ball $B_{2r}$ is compactly supported in $\Omega$, and, for some $p\in(1,d)$, $u\in W^{1,p}(B_{2r})$ is a solution to the equation $Lu=0$ in $\Omega$. Let also $B_{\alpha r}\subseteq B_{\beta r}$ be concentric balls, with $1\leq \alpha<\beta\leq 2$. Then there exists a constant $c$ depending only on $d,p,\mu,\lambda$ such that, if $r<c$, then
\[
\|\nabla u\|_{L^{p^*}(B_{\alpha r})}\leq\frac{C}{r}\|\nabla u\|_{L^p(B_{\beta r})},
\]
where $C$ is a constant that depends on $d,p,\mu,\lambda,\|b\|_{\infty}$, and $\beta-\alpha$.
\end{lemma}
\begin{proof}
Let $\gamma=\frac{\alpha+\beta}{2}$, and set $B_1=B_{\alpha r}$, $B_2=B_{\gamma r}$, $B_3=B_{\beta r}$. Note that, since constants are solutions to the equation, we can assume that the average of $u$ over $B_3$ is $0$. Let $\phi$ be a a smooth cutoff which is supported in $B_2$, it is equal to $1$ in $B_1$, and $|\nabla\phi|\leq\frac{C}{(\beta-\alpha)r}$, where $C=C(d)$. Let also $\psi_n$ be a sequence of mollifiers, with $|\nabla\psi_n|\leq Cn$ for all $n\in\mathbb N$, and set
\[
v=u\phi,\,\,v_n=v*\psi_n.
\]
Note then that $v_n\in C^2(\overline{B_3})$, and $v_n\equiv 0$ on $\partial B_3$ for $n$ sufficiently large. In addition, we compute
\begin{align*}
\dive(A\nabla v)&=\dive(A\nabla u\cdot\phi+A\nabla\phi\cdot u)=\dive(A\nabla u)\cdot\phi+A\nabla u\nabla\phi+A\nabla\phi\nabla u+\dive(A\nabla\phi)\cdot u\\
&=b\nabla u\cdot\phi+(A+A^t)\nabla u\nabla\phi+\dive(A\nabla\phi)\cdot u=f,
\end{align*}
where we used that $u$ solves the equation $Lu=0$ in $\Omega$. In addition,
\[
\|f\|_{L^p(B_3)}\leq\frac{C}{(\beta-\alpha)r}\|\nabla u\|_{L^p(B_3)}+\frac{C}{(\beta-\alpha)^2r^2}\|u\|_{L^p(B_3)},
\]
and since $u$ has average $0$ over $B_3$, we can apply Poincare's inequality to obtain that
\begin{equation}\label{eq:pNormOfF}
\|f\|_{L^p(B_3)}\leq\frac{C}{r}\|\nabla u\|_{L^p(B_3)},
\end{equation}
where $C$ now also depends on the difference $\beta-\alpha$ and $p$.

Let now $x\in B_3$. Since $B_3$ is compactly supported in $\Omega$, the function $y\mapsto\psi_n(x-y)$ is compactly supported in $\Omega$ for $n$ sufficiently large. Therefore, since $v$ solves the equation $\dive(A\nabla v)=f$, we obtain that
\[
\int_{\Omega}A(y)\nabla v(y)\nabla\psi_n(x-y)\,dy=-\int_{\Omega}A(y)\nabla v(y)\nabla_y(\psi_n(x-y))\,dy=\int_{\Omega}f(y)\psi_n(x-y)\,dy,
\]
which shows that
\begin{equation}\label{eq:bnablau}
(a_{ij}\partial_iv)*\partial_j\psi_n(x)=f*\psi_n(x).
\end{equation}

Consider now the constant coefficient operator $L_0=a_{ij}(x_0)\partial_{ij}$. Note that, from estimate $9.37$ in \cite{Gilbarg}, if $n$ is sufficiently large,
\begin{equation}\label{eq:SecondDerBound}
\|\nabla^2v_n\|_{L^p(B_3)}\leq C_{d,p,\lambda}\|L_0v_n\|_{L^p(B_3)}.
\end{equation}
But, we compute, for $x\in B_3$,
\begin{align*}
|L_0v_n(x)|&\leq|a_{ij}(x_0)-a_{ij}(x)||\partial_{ij}v_n(x)|+|a_{ij}(x)\partial_{ij}v_n(x)|\\
&\leq\mu\beta r|\partial_{ij}v_n(x)|+|a_{ij}(x)\partial_{ij}v_n(x)|\leq 2\mu r|\partial_{ij}v_n(x)|+|a_{ij}(x)\partial_{ij}v_n(x)|.
\end{align*}
This shows that
\[
\|\nabla^2v_n\|_{L^p(B_3)}\leq C_{d,p,\lambda}\mu r\|\partial_{ij}v_n\|_{L^p(B_3)}+\|a_{ij}\partial_{ij}v_n\|_{L_p(B_3)},
\]
and if we choose $r<2\delta(x_0)$ such that $C_{d,p,\lambda}\mu r<1/2$, we obtain that
\[
\|\nabla^2v_n\|_{L^p(B_3)}\leq C\|a_{ij}\partial_{ij}v_n\|_{L_p(B_3)}.
\]
We now compute, in $B_3$,
\begin{align*}
a_{ij}\partial_{ij}v_n&=a_{ij}\partial_i(v*\partial_j\psi_n)-(a_{ij}\partial_iv)*\partial_j\psi_n+f*\psi_n\\
&=a_{ij}(\partial_iv*\partial_j\psi_n)-(a_{ij}\partial_iv)*\partial_j\psi_n+f*\psi_n,
\end{align*}
where we also used \eqref{eq:bnablau}. But, if we set $B_n=B_{1/n}(0)$, we compute
\[
a_{ij}(x)(\partial_iv*\partial_j\psi_n)(x)-(a_{ij}\partial_iv)*\partial_j\psi_n(x)=\int_{B_n}\left(a_{ij}(x)-a_{ij}(x-y)\right)\partial_iv(x-y)\partial_j\psi_n(y)\,dy,
\]
and the last integral is bounded by
\begin{align*}
\int_{B_n}\mu|y||\partial_iv(x-y)||\partial_j\psi_n(y)|\,dy&\leq\frac{C}{n}\int_{B_{1/n}(x)}|\partial_iv(y)||\partial_j\psi_n(x-y)|\,dy\\
&\leq Cn^d\int_{B_{1/n}(x)}|\partial_iv(y)|\,dy\\
&\leq C\left(n^d\int_{B_{1/n}(x)}|\partial_iv(y)|^p\,dy\right)^{1/p}.
\end{align*}
This shows that 
\[
|a_{ij}\partial_{ij}v_n|\leq C\left(n^d\int_{B_{1/n}(x)}|\partial_iv(y)|^p\,dy\right)^{1/p}+|f*\psi_n|.
\]
If we consider the $L^p$ norm, we obtain that, for large $n$,
\begin{align*}
\int_{B_2}\left(Cn^d\int_{B_{1/n}(x)}|\partial_iv(y)|^p\,dy\right)\,dx&\leq Cn^d\int_{B_3}\int_{B_{1/n}(y)}|\partial_iv(y)|^p\,dxdy\\
&\leq C\int_{B_3}|\partial_iv(y)|^p\,dy,
\end{align*}
therefore
\[
\|a_{ij}\partial_{ij}v_n\|_{L^p(B_3)}\leq C\|\nabla u\|_{L^p(B_3)}+C\|f\|_{L^p(B_3)}.
\]
Plugging this back to \eqref{eq:SecondDerBound} and letting $n\to\infty$, we finally obtain that
\begin{equation}\label{eq:SecondByFirst}
\limsup_{n\to\infty}\|\nabla^2v_n\|_{L^p(B_3)}\leq\|f\|_{L^p(B_3)}\leq\frac{C}{r}\|\nabla u\|_{L^p(B_3)},
\end{equation}
for all $n$ sufficiently large, where we also used \eqref{eq:pNormOfF}.

(Note also that, up this point, we have only used that $p\in(1,\infty)$).

Since now $v_n$ is a mollification of $v$ in $\Omega$ and $u=v$ in $B_1$, we obtain that $\nabla v_n\to\nabla u$ in $L^p(B_1)$. Therefore, there exists a subsequence of $(v_n)$ such that $\nabla v_{k_n}\to \nabla u$ almost everywhere in $B_1$. Since now $v_n$ vanishes close to $\partial B_3$ for large $n$, Fatou's lemma and Sobolev's inequality show that
\[
\|\nabla u\|_{L^{p^*}(B_1)}\leq\limsup_{n\to\infty}\|\nabla v_{k_n}\|_{L^{p^*}(B_3)}\leq C\limsup_{n\to\infty}\|\nabla^2v_{k_n}\|_{L^p(B_3)}\leq\frac{C}{r}\|\nabla u\|_{L^p(B_3)},
\]
where we used \eqref{eq:SecondByFirst} in the last step. This completes the proof.
\end{proof}

We also turn to the analog for solutions to the adjoint equation.

\begin{lemma}\label{LowRegularityEstimateForAdjoint}
Let $\Omega$ be a bounded domain, $A\in M_{\lambda,\mu}(\Omega)$ and $b\in\Lip(\Omega)$. Suppose that, for
some $p\in(1, d)$, $u\in W^{1,p}_{{\rm loc}}(\Omega)$ is a solution to the equation $L^tu = 0$ in $\Omega$. Then, $u\in W^{1,p^*}_{{\rm loc}}(\Omega)$.
\end{lemma}
\begin{proof}
We mimic the proof of lemma \ref{LowRegularityEstimate}: let $U$ be compactly supported in $\Omega$, and consider a set $U_0$ with $U\subseteq U_0\subseteq V\subseteq\Omega$, where all inclusions are compact. Let $\phi$ be a a smooth cutoff which is supported in $U_0$ and it is equal to $1$ in $U$. Let also $\psi_n$ be a sequence of mollifiers, with $|\nabla\psi_n|\leq Cn$ for all $n\in\mathbb N$, and set
\[
v=u\phi,\,\,v_n=v*\psi_n.
\]
Note then that $v_n\in C^2(\overline{B_3})$, and $v_n\equiv 0$ on $\partial B_3$ for $n$ sufficiently large. In addition, we compute
\begin{align*}
\dive(A\nabla v)&=\dive(A\nabla u\cdot\phi+A\nabla\phi\cdot u)=\dive(A\nabla u)\cdot\phi+A\nabla u\nabla\phi+A\nabla\phi\nabla u+\dive(A\nabla\phi)\cdot u\\
&=-\dive b\cdot u\phi-b\nabla u\cdot\phi+(A+A^t)\nabla u\nabla\phi+\dive(A\nabla\phi)\cdot u=f,
\end{align*}
where we used that $u$ solves the equation $L^tu=0$ in $\Omega$. Since now $u$ is bounded and $\nabla u\in W^{1,p}(V)$, we obtain that $f\in L^p(V)$.

Let now $x\in V$. Since $V$ is compactly supported in $\Omega$, the function $y\mapsto\psi_n(x-y)$ is compactly supported in $\Omega$ for $n$ sufficiently large. Therefore, since $v$ solves the equation $\dive(A\nabla v)=f$, we obtain that
\[
\int_{\Omega}A(y)\nabla v(y)\nabla\psi_n(x-y)\,dy=-\int_{\Omega}A(y)\nabla v(y)\nabla_y(\psi_n(x-y))\,dy=\int_{\Omega}f(y)\psi_n(x-y)\,dy,
\]
which shows that
\begin{equation}\label{eq:bnablau2}
(a_{ij}\partial_iv)*\partial_j\psi_n(x)=f*\psi_n(x).
\end{equation}

Consider now the constant coefficient operator $L_0=a_{ij}(x_0)\partial_{ij}$. Note that, from estimate $9.37$ in \cite{Gilbarg}, if $n$ is sufficiently large,
\begin{equation}\label{eq:SecondDerBound2}
\|\nabla^2v_n\|_{L^p(V)}\leq C_{d,p,\lambda}\|L_0v_n\|_{L^p(V)}.
\end{equation}
But, we compute, for $x\in V$,
\begin{align*}
|L_0v_n(x)|&\leq|a_{ij}(x_0)-a_{ij}(x)||\partial_{ij}v_n(x)|+|a_{ij}(x)\partial_{ij}v_n(x)|\\
&\leq\mu\cdot\diam(V)|\partial_{ij}v_n(x)|+|a_{ij}(x)\partial_{ij}v_n(x)|\leq 2\mu r|\partial_{ij}v_n(x)|+|a_{ij}(x)\partial_{ij}v_n(x)|.
\end{align*}
This shows that
\[
\|\nabla^2v_n\|_{L^p(V)}\leq C_{d,p,\lambda}\mu\cdot\diam(V)\|\partial_{ij}v_n\|_{L^p(V)}+\|a_{ij}\partial_{ij}v_n\|_{L_p(V)},
\]
and if the diameter of $V$ is small enough, we obtain that
\[
\|\nabla^2v_n\|_{L^p(B_3)}\leq C\|a_{ij}\partial_{ij}v_n\|_{L_p(B_3)}.
\]
We now compute, in $B_3$,
\begin{align*}
a_{ij}\partial_{ij}v_n&=a_{ij}\partial_i(v*\partial_j\psi_n)-(a_{ij}\partial_iv)*\partial_j\psi_n+f*\psi_n\\
&=a_{ij}(\partial_iv*\partial_j\psi_n)-(a_{ij}\partial_iv)*\partial_j\psi_n+f*\psi_n,
\end{align*}
where we also used \eqref{eq:bnablau}. But, if we set $B_n=B_{1/n}(0)$, we compute
\[
a_{ij}(x)(\partial_iv*\partial_j\psi_n)(x)-(a_{ij}\partial_iv)*\partial_j\psi_n(x)=\int_{B_n}\left(a_{ij}(x)-a_{ij}(x-y)\right)\partial_iv(x-y)\partial_j\psi_n(y)\,dy,
\]
and the last integral is bounded by
\begin{align*}
\int_{B_n}\mu|y||\partial_iv(x-y)||\partial_j\psi_n(y)|\,dy&\leq\frac{C}{n}\int_{B_{1/n}(x)}|\partial_iv(y)||\partial_j\psi_n(x-y)|\,dy\\
&\leq Cn^d\int_{B_{1/n}(x)}|\partial_iv(y)|\,dy\\
&\leq C\left(n^d\int_{B_{1/n}(x)}|\partial_iv(y)|^p\,dy\right)^{1/p}.
\end{align*}
This shows that 
\[
|a_{ij}\partial_{ij}v_n|\leq C\left(n^d\int_{B_{1/n}(x)}|\partial_iv(y)|^p\,dy\right)^{1/p}+|f*\psi_n|.
\]
If we consider the $L^p$ norm, we obtain that, for large $n$,
\[
\int_{U_0}\left(Cn^d\int_{B_{1/n}(x)}|\partial_iv(y)|^p\,dy\right)\,dx\leq Cn^d\int_{V}\int_{B_{1/n}(y)}|\partial_iv(y)|^p\,dxdy\leq C\int_{V}|\partial_iv(y)|^p\,dy,
\]
therefore
\[
\|a_{ij}\partial_{ij}v_n\|_{L^p(U_0)}\leq C\|\nabla u\|_{L^p(U_0)}+C\|f\|_{L^p(U_0)}.
\]
Plugging this back to \eqref{eq:SecondDerBound2} and letting $n\to\infty$, we finally obtain that $\|\nabla^2v_n\|_{L^p(V)}$ is bounded.

The last estimate shows that $(\nabla v_n)$ is bounded in $W^{1,p}(V)$. From the Rellich-Kondrachov compactness theorem and almost everywhere convergence, there exists a subsequence $(\nabla v_{k_n})$ which converges to a function $w$, weakly in $W^{1,p}(V)$, and almost everywhere in $V$. Then $w\in W^{1,p}(U)$, hence $w\in L^{p^*}(U)$. But, $\nabla v_n$ converges to $\nabla u$ almost everywhere in $U$, hence $u=w\in L^{p^*}(U)$. Since also the sequence $(v_n)$ is bounded in $W^{1,p}(U)$, we obtain that $u\in L^{p^*}(U)$, hence $u\in W^{1,p^*}(U)$, which completes the proof.
\end{proof}

By iterating the previous lemmas over smaller domains, we obtain the next propositions.

\begin{prop}\label{LowRegularity}
Let $\Omega$ be a bounded domain, $A\in M_{\lambda,\mu}(\Omega)$, and $b\in L^{\infty}(\Omega)$. Suppose that, for some $p\in(1,2)$, $u\in W_{{\rm loc}}^{1,p}(\Omega)$ is a solution to the equation $Lu=0$ in $\Omega$. Then $u\in W_{{\rm loc}}^{1,2}(\Omega)$, with
\[
\|u\|_{W^{1,2}(U)}\leq C\|u\|_{W^{1,p}(U_1)}
\]
for any $U\subseteq U_1\subseteq\Omega$, where all inclusions are compact.
\end{prop}

\begin{prop}\label{LowRegularityForAdjoint}
Let $\Omega$ be a bounded domain, $A\in M_{\lambda,\mu}(\Omega)$, and $b\in \Lip(\Omega)$. Suppose that, for some $p\in(1,2)$, $u\in W_{{\rm loc}}^{1,p}(\Omega)$ is a solution to the equation $L^tu=0$ in $\Omega$. Then $u\in W_{{\rm loc}}^{1,2}(\Omega)$, with
\[
\|u\|_{W^{1,2}(U)}\leq C\|u\|_{W^{1,p}(U_1)}
\]
for any $U\subseteq U_1\subseteq\Omega$, where all inclusions are compact.
\end{prop}

\subsection{Local estimates on the gradient}
The assumption that $b\in L^{\infty}(\Omega)$ together with the fact that $A$ is Lipschitz and elliptic guarantee that solutions to the equation $Lu=0$ in $\Omega$ have gradients that are locally H{\"o}lder continuous. This will be shown in the next proposition.

\begin{prop}\label{DerivativeRegularity}
Let $\Omega$ be a bounded domain, and let $A\in M_{\lambda,\mu}(\Omega)$ and $b\in L^{\infty}(\Omega)$. Let also $B_r$ be a ball in $\Omega$, such that its double $B_{2r}$ is compactly supported in $\Omega$. Then, there exists $\alpha\in(0,1)$ such that, for any solution $u\in W^{1,2}_{{\rm loc}}(\Omega)$ of the equation $-\dive(A\nabla u)+b\nabla u=0$ in $\Omega$,
\[
|\nabla u(x)-\nabla u(y)|\leq\frac{C}{r}\left(\frac{|x-y|}{r}\right)^{\alpha}\left(\fint_{B_{2r}}|u|^2\right)^{1/2},
\]
for all $x,y\in B_r$, where $C$ depends on $d,\lambda,\mu,\|b\|_{\infty}$ and ${\rm diam}(\Omega)$; that is, $\nabla u$ is localy H{\"o}lder continuous.
\end{prop}
\begin{proof}
Suppose first that $r<c$, where $c$ is the constant that appears in the proof of lemma \ref{LowRegularityEstimate}. Fix an irrational number $s\in(1,3/2)$. Then, H{\"o}lder's inequality shows that
\begin{equation}\label{eq:HolderForS}
\|\nabla u\|_{L^s(B_{31r/16})}\leq r^{d/s-d/2}\|\nabla u\|_{L^2(B_{31r/16})}.
\end{equation}
We now set $k=\left[\frac{d}{s}\right]$, and we note that $k$ is the first integer such that
\[
\frac{1}{s}-\frac{k}{d}<\frac{1}{d},\,\,\frac{1}{s}-\frac{k-1}{d}>\frac{1}{d};
\]
the fact that $s<3/2$ guarantees that such a $k\geq 2$ exists, with the inequalities being strict since $s$ is irrational. We also consider a sequence
\[
\frac{15}{8}=c_1>c_2>\dots>c_k>c_{k+1}=\frac{3}{2},
\]
and we set $p_1=s$, and $p_{m+1}=p_m^*$ for $m=1,\dots k$; then, for $m=1,\dots k+1$, 
\[
\frac{1}{p_m}=\frac{1}{s}-\frac{m-1}{d}.
\]
We then apply lemma \ref{LowRegularityEstimate} for $\alpha=c_{m+1}$ and $\beta=c_m$, $m=1,\dots k$, to obtain that
\[
\|\nabla u\|_{L^{p_{m+1}}(B_{c_{m+1}r})}\leq\frac{C}{r}\|\nabla u\|_{L^{p_m}(B_{c_mr})},
\]
where $C$ also depends on the difference $c_m-c_{m+1}$. This will show that
\[
\|\nabla u\|_{L^{p_{k+1}}(B_{c_{k+1}}r)}\leq\frac{C}{r^k}\|\nabla u\|_{L^s(B_{31r/16})}\leq Cr^{d/s-d/2-k}\|\nabla u\|_{L^2(B_{31r/16})},
\]
where $C$ is also depends on the sequence $(c_m)$, and where we used \eqref{eq:HolderForS} in the last step.

Recall now the definition of $v_n$ from the proof of lemma \ref{LowRegularityEstimate}, where the construction takes place for $\alpha=3/2$ and $\beta=c_k$. Since $v_n$ is a mollification of $v$ and $u=v$ in $B_{3r/2}$, there exists a subsequence $v_{k_n}$ such that $\nabla v_{k_n}\to \nabla u$ almost everywhere in $B_{3r/2}$. But, estimate \eqref{eq:SecondByFirst} in the proof of lemma \ref{LowRegularityEstimate} shows that
\[
\limsup_{n\to\infty}\|\nabla^2v_n\|_{{L^{p_{k+1}}(B_{3r/2})}}\leq\frac{C}{r}\|\nabla u\|_{L^{p_{k+1}}(B_{c_kr})}.
\]
Since $p_{k+1}>d$, Morrey's inequality shows that, for almost every $x,y\in B_r$,
\begin{align*}
|\nabla u(x)-\nabla u(y)|&\leq\limsup_{n\to\infty}|\nabla v_n(x)-\nabla v_n(y)|\\
&\leq\limsup_{n\to\infty}C|x-y|^{1-d/p_{k+1}}\|\nabla^2v_n\|_{{L^{p_{k+1}}(B_{3r/2})}}\\
&\leq C|x-y|^{1-d/p_{k+1}}r^{d/s-d/2-k-1}\|\nabla u\|_{L^2(B_{31r/16})}\\
&\leq C\left(\frac{|x-y|}{r}\right)^{\alpha}r^{d/s-d/2+\alpha-k-1}\|\nabla u\|_{L^2(B_{31r/16})}\\
&\leq C\left(\frac{|x-y|}{r}\right)^{\alpha}r^{d/s+\alpha-k-1}\left(\fint_{B_{31r/16}}|\nabla u|^2\right)^{1/2},
\end{align*}
where $\alpha=1-d/p_{k+1}\in(0,1)$. But, we compute
\[
\frac{d}{s}+\alpha-k=\frac{d}{s}+1-\frac{d}{p_{k+1}}-k-1=\frac{d}{s}+1-d\left(\frac{1}{s}-\frac{k}{d}\right)-k-1=0,
\]
which shows that
\[
|\nabla u(x)-\nabla u(y)|\leq C\left(\frac{|x-y|}{r}\right)^{\alpha}\left(\fint_{B_{31r/16}}|\nabla u|^2\right)^{1/2},
\]
But, from Cacciopoli's inequality,
\[
\left(\fint_{B_{31r/16}}|\nabla u|^2\right)^{1/2}\leq\frac{C}{r}\left(\fint_{B_{2r}}|u|^2\right)^{1/2},
\]
and this completes the proof.
\end{proof}

The last proposition leads to the following corollary.

\begin{cor}\label{LocalBoundOnGradientOfSolutions}
Let $\Omega$ be a bounded domain, and let $A\in M_{\lambda,\mu}(\Omega)$ and $b\in L^{\infty}(\Omega)$. Let also $B_r$ be a ball in $\Omega$, such that its double $B_{2r}$ is compactly supported in $\Omega$. Then, for any solution $u$ of the equation $-\dive(A\nabla u)+b\nabla u=0$ in $\Omega$,
\[
\|\nabla u\|_{L^{\infty}(B_r)}\leq\frac{C}{r}\left(\fint_{B_{2r}}|u|^2\right)^{1/2},
\]
where $C$ depends on $d,\lambda,\mu,\|b\|_{\infty}$ and $\diam(\Omega)$. Hence, $u$ is locally Lipschitz in $\Omega$, with
\[
|u(x)-u(y)|\leq C\frac{|x-y|}{r}\left(\fint_{B_{2r}}|u|^2\right)^{1/2}
\]
for any $x,y\in B_r$.
\end{cor}
\begin{proof}
Fix $x\in B_r$. Then, for any $y\in B_r$, proposition \ref{DerivativeRegularity} shows that
\[
|\nabla u(x)|\leq |\nabla u(y)|+|\nabla u(x)-\nabla u(y)|\leq|\nabla u(y)|+\frac{C}{r}\left(\fint_{B_{2r}}|u|^2\right)^{1/2}.
\]
We now integrate for $y\in B_r$ and we apply the Cauchy-Schwartz inequality, to obtain that
\[
|\nabla u(x)|\leq\fint_{B_r}|\nabla u|+\frac{C}{r}\left(\fint_{B_{2r}}|u|^2\right)^{1/2}\leq\left(\fint_{B_r}|\nabla u|^2\right)^{1/2}+\frac{C}{r}\left(\fint_{B_{2r}}|u|^2\right)^{1/2}.
\]
We now apply Cacciopoli's inequality, and this completes the proof.
\end{proof}

We also obtain the next qualitative corollary, by combining the last estimate with the low regularity estimates from the previous section.

\begin{cor}\label{GreenContinuity}
Let $\Omega$ be a bounded domain, $A\in M_{\lambda,\mu}(\Omega)$ and $b\in L^{\infty}(\Omega)$. Suppose that $U\subseteq\Omega$ is compactly supported, and $u\in W_{\loc}^{1,p}(\Omega)$ is a solution to $Lu=0$ in $\Omega$, for some $p>1$. Then $u$ is continuously differentiable in $U$.
\end{cor}
\begin{proof}
Let $U\subseteq U_0\subseteq \Omega$, where all inclusions are compact. From proposition \ref{LowRegularity}, we then obtain that $u\in W^{1,2}(U_0)$. We then cover $\overline{U_0}$ by balls such that their doubles are contained in $\Omega$ and we apply corollary proposition \ref{DerivativeRegularity} to each one of them to obtain that $u$ is continuously differentiable in $U$, which finishes the proof.
\end{proof}

\subsection{Local estimates on the gradient for $L^t$}
We now turn to showing that the gradient of a solution to the equation $L^tu=0$ is locally H{\"o}lder continuous, provided that $b$ is H{\"o}lder continuous. We begin with a lemma.

\begin{prop}\label{DerivativeRegularityForAdjoint}
Let $\Omega$ be a bounded domain, $A\in M_{\lambda,\mu}(\Omega)$ and $b\in C^{\alpha}(\Omega)$ for some $\alpha\in(0,1]$. Consider a ball $B$ with radius $r$, such that $16B\subseteq\Omega$, and let $g\in C^{\alpha}(16B)$. Let also $u$ be an $W^{1,2}(\Omega)$ solution of the equation $L^tu=\dive g$ in $\Omega$. Then,
\[
\|\nabla u\|_{C^{0,\alpha}(B)}\leq\frac{C}{r^{1+\alpha}}\left(\fint_{4B}|u|^2\right)^{1/2}+Cr^{-\alpha}\|g\|_{L^{\infty}(2B)}+C\|g\|_{C^{0,\alpha}(2B)},
\]
where $C$ depends on $d,\lambda,\mu,\|b\|_{C^{0,\alpha}}$ and $\diam(\Omega)$.
\end{prop}
\begin{proof}
We can assume that $r=1$; the general case can be then recovered after a dilation.

First, note that theorem 8.22 in \cite{Gilbarg} shows that $u\in C^{\beta}(8B)$ for some $\beta\in(0,1)$. Hence, if $\gamma=\min\{\alpha,\beta\}$, the function $f=bu$ belongs to $C^{\gamma}(8B)$. Note also that $u$ solves the equation
\[
-\dive(A\nabla u)=\dive(bu)+\dive g=\dive(f+g),
\]
and $f+g\in C^{\gamma}(8B)$. Hence we can apply estimate 2.2 in \cite{KenigShen} to obtain that
\[
\|\nabla u\|_{C^{0,\gamma}(4B)}\leq C\left(\fint_{8B}|u|^2\right)^{1/2}+C\|f+g\|_{C^{0,\gamma}(8B)},
\]
therefore $\nabla u\in C^{\gamma}(4B)$. In particular, $\nabla u$ is bounded in $4B$, therefore $u$ is Lipschitz in $4B$, hence $f=bu\in C^{\alpha}(4B)$. Then, again from estimate 2.2 in \cite{KenigShen},
\[
\|\nabla u\|_{C^{0,\alpha}(2B)}\leq C\left(\fint_{4B}|u|^2\right)^{1/2}+C\|f+g\|_{C^{0,\alpha}(4B)},
\]
therefore $u\in C^{\alpha}(2B)$.

We can now apply theorem $8.32$ in \cite{Gilbarg}, for the domains $B\subseteq 2B$, to obtain that
\begin{align*}
\|\nabla u\|_{C^{0,\alpha}(B)}&\leq C\|u\|_{L^{\infty}(2B)}+\|g\|_{L^{\infty}(2B)}+\|g\|_{C^{0,\alpha}(2B)}\\
&\leq C\left(\fint_{4B}|u|^2\right)^{1/2}+\|g\|_{L^{\infty}(2B)}+\|g\|_{C^{0,\alpha}(2B)},
\end{align*}
where we also used theorem 8.17 in \cite{Gilbarg}. This completes the proof.
\end{proof}

We also obtain local Lipschitz continuity of solutions to the adjoint equation, as the next corollary shows.

\begin{cor}\label{BoundedDerivativeForAdjoint}
Let $\Omega$ be a bounded domain, $A\in M_{\lambda,\mu}(\Omega)$ and $b\in C^{\alpha}(\Omega)$. Consider a ball $B$ with radius $r$, such that $16B\subseteq\Omega$, and let $g\in C^{\alpha}(16B)$. Let also $u$ be an $W^{1,2}(\Omega)$ solution of the equation $L^tu=\dive g$ in $\Omega$. Then,
\[
\|\nabla u\|_{L^{\infty}(B_r)}\leq\frac{C}{r}\left(\fint_{B_{4r}}|u|^2\right)^{1/2}+C\|g\|_{L^{\infty}(B_{2r})}+Cr^{\alpha}\|g\|_{C^{0,\alpha}(B_{2r})},
\]
where $C$ depends on $d,\lambda,\mu,\|b\|_{C^{0,\alpha}}$ and $\diam(\Omega)$.
\end{cor}
\begin{proof}
Let $x\in B$. Then, for any $y\in B$,
\begin{align*}
|\nabla u(x)|&\leq|\nabla u(x)-\nabla u(y)|+|\nabla u(y)|\leq \|\nabla u\|_{C^{0,\alpha}(B_r)}|x-y|^{\alpha}+|\nabla u(y)|\\
&\leq \|\nabla u\|_{C^{0,\alpha}(B_r)}r^{\alpha}+|\nabla u(y)|,
\end{align*}
and after integrating for $y\in B$, we obtain that
\begin{align*}
|\nabla u(x)|&\leq\|\nabla u\|_{C^{0,\alpha}(B_r)}r^{\alpha}+\fint_{B_r}|\nabla u|\leq \|\nabla u\|_{C^{0,\alpha}(B_r)}r^{\alpha}+\left(\fint_{B_r}|\nabla u|^2\right)^{1/2}\\
&\leq \|\nabla u\|_{C^{0,\alpha}(B_r)}r^{\alpha}+\frac{C}{r}\left(\fint_{B_{2r}}|u|^2\right)^{1/2},
\end{align*}
where we also used Cacciopoli's inequality (lemma \ref{CacciopoliForAdjoint}). We then use proposition \ref{DerivativeRegularityForAdjoint} to bound $\|\nabla u\|_{C^{0,\alpha}(B_r)}$ and we consider the supremum for $x\in B_r$ to conclude the proof.
\end{proof}

\subsection{Global estimates}
In this section we will show two global results. The first will be an $L^2$ control of the gradient of a solution from the solution itself, and the second is the maximum principle.

\begin{lemma}\label{L2BoundOnGradient}
Let $\Omega$ be a bounded domain, and suppose that $A\in M_{\lambda}(\Omega)$, $b\in L^{\infty}(\Omega)$. Let also $u\in W_0^{1,2}(\Omega)$ be a solution to the equation $Lu=f$, for $f\in L^2(\Omega)$. Then
\[
\|\nabla u\|_2^2\leq C\|u\|_2^2+C\|f\|_2^2,
\]
where $C$ depends on $\lambda$ and $\|b\|_{\infty}$.
\end{lemma}
\begin{proof}
We use the definition of solution with $u$ as a test function, to obtain that
\[
\int_{\Omega}A\nabla u\nabla u+b\nabla u\cdot u=\int_{\Omega}fu,
\]
which shows that, for any $\delta>0$,
\begin{align*}
\lambda\|\nabla u\|^2&\leq\int_{\Omega}A\nabla u\nabla u=-\int_{\Omega}b\nabla u\cdot u+\int_{\Omega}fu\leq\|b\|_{\infty}\|\nabla u\|_2\|u\|_2+\|f\|_2\|u\|_2\\
&\leq\delta\|b\|_{\infty}\|\nabla u\|_2^2+\frac{1}{4\delta}\|b\|_{\infty}\|u\|_2^2+\|f\|_2\|u\|_2
\end{align*}
Choosing $\delta=\frac{\lambda}{2\|b\|_{\infty}+1}$, we then obtain the desired inequality.
\end{proof}

We now turn to the maximum principle for subsolutions of the equation $Lu=0$ in $\Omega$. We will need a notion of inequality on the boundary of $\Omega$ for Sobolev functions; for this purpose, we will use the supremum in the $W^{1,2}$ sense (definition in section $8.1$ in \cite{Gilbarg}).

To show the next proposition, we will follow the proof of theorem $8.1$ in \cite{Gilbarg}.

\begin{prop}\label{MaxForSubsolutions}
Let $A\in M_{\lambda}(\Omega)$, $b\in L^{\infty}(\Omega)$, and let $u\in W^{1,2}(\Omega)$ be a subsolution of $Lu=0$. Then,
\[
\sup_{\Omega}u\leq\sup_{\partial\Omega}u,
\]
and the $\sup_{\partial\Omega}$ is considered in the $W^{1,2}$-sense.
\end{prop}
\begin{proof}
Note first that, since $u$ is a subsolution, for any $v\in W_0^{1,2}(\Omega)$ with $v\geq 0$,
\[
\int_{\Omega}A\nabla u\nabla v+b\nabla u\cdot v\leq 0\Rightarrow\int_{\Omega}A\nabla u\nabla v\leq-\int_{\Omega}b\nabla u\cdot v.
\]
If $\sup_{\partial\Omega}u=\infty$, then the inequality is valid. Suppose now that $k_0=\sup_{\partial\Omega}u<\infty$, and suppose that $k_0<\sup_{\Omega}u=k_1$. Let $k\in[k_0,k_1)$, and define $v_k=(u-k)^+\in W_0^{1,2}(\Omega)$. We then have that $\partial_iv_k=\partial_iu\cdot\chi_{[u>k]}$, therefore 
\[
\int_{\Omega}A\nabla v_k\cdot\nabla v_k\leq-\int_{\Omega}b\nabla v_k\cdot v_k,
\]
and the ellipticity of $A$ and the Sobolev inequality show that
\begin{equation}\label{eq:k_0}
\lambda\int_{\Omega}|\nabla v_k|^2\leq \|b\|_{\infty}\|v_k\|_{L^{2^*}}\|\nabla v_k\|_{L^2}|{\rm supp}(v_k)|^{1/d}\leq C_d\|b\|_{\infty}\|\nabla v_k\|_{L^2}^2|{\rm supp}(v_k)|^{1/d},
\end{equation}
where ${\rm supp}(v_k)$ denotes the support of $v_k$.

If $\|\nabla v_k\|_{L^2(\Omega)}=0$, then $v_k$ is a constant. Since $v_k\in W_0^{1,2}(\Omega)$, this constant has to be zero, therefore $(u-k)^+=0$ in $\Omega$. Therefore, $u(x)\leq k$ for all $x\in \Omega$, so
\[
k_1=\sup_{\Omega}u\leq k,
\]
which is a contradiction. Therefore $\|\nabla v_k\|_{L^2(\Omega)}>0$, which shows that $\|b\|_{\infty}\neq 0$. Hence, \eqref{eq:k_0} shows that
\[
|{\rm supp}(v_k)|\geq\lambda^dC_d^{-d}\|b\|_{\infty}^{-d}.
\]
Since this inequality does not depend on $k$, it should be true as $k\to\sup_{\Omega}u$; therefore
\[
|{\rm supp}(v_{k_1})|\geq C_d^{-1/n},
\]
which shows that $u$ attains its supremum at a set of positive measure. Since also $v$ is integrable, this supremum should be finite, therefore $k_1<\infty$.

Let now $v=(u-k_0)^+$. Let also
\[
l=\sup_{\Omega}v=\sup_{\Omega}(u-k_0)^+=k_1-k_0>0,
\]
which is also finite, since $k_1<\infty$. Then, for any $\e>0$, if we use $\frac{v}{l-v+\e}$ as a test function, we obtain that
\[
\int_{\Omega}\frac{A\nabla v\cdot\nabla v}{(l-v+\e)^2}\leq-\int_{\Omega}\frac{b\nabla u\cdot v}{l-v+\e}=-\int_{\Omega}\frac{b\nabla v\cdot v}{l-v+\e}.
\]
Set now
\[
w_{\e}=\log(l+\e)-\log(l-v+\e)=\log\frac{l+\e}{l-v+\e}.
\]
Since $l+\e>l-v+\e$ and $l-v+\e>0$, we obtain that $w_{\e}\in W_0^{1,2}(\Omega)$, therefore
\begin{align*}
\lambda\int_{\Omega}|\nabla w_{\e}|^2&\leq\int_{\Omega}A\nabla w_{\e}\nabla w_{\e}=\int_{\Omega}\frac{A\nabla v\cdot\nabla v}{(l-v+\e)^2}=-\int_{\Omega}\frac{b\nabla v\cdot v}{l-v+\e}\\
&=-\int_{\Omega}b\nabla w_{\e}\cdot v\leq\|b\|_{\infty}\|\nabla w_{\e}\|_2\|v\|_2\leq l\|b\|_{\infty}\|\nabla w_{\e}\|_2|\Omega|^{1/2},
\end{align*}
since $0\leq v\leq l$. This shows that, for any $\e>0$,
\[
\|\nabla w_{\e}\|_2\leq l\lambda^{-1}\|b\|_{\infty}|\Omega|^{1/2},
\]
and Sobolev's inequality shows that
\[
\|w_{\e}\|_{2^*}\leq C_d\|\nabla w_{\e}\|_2\leq l\lambda^{-1}\|b\|_{\infty}|\Omega|^{1/2}.
\]
Letting $\e\to 0$, we obtain that $w_0=\log l-\log(l-v)$ is integrable, therefore $v=l$ only on a set of measure zero. But, if $x\in\Omega$ with $u(x)=k_1$,
\[
u(x)-k_0=k_1-k_0=l>0\Rightarrow v(x)=(u-k_0)^+(x)=l,
\]
so $u$ achieves its supremum only on a set of measure $0$. But this is a contradiction, which completes the proof.
\end{proof}

We also obtain the next analog for supersolutions.

\begin{prop}\label{MinForSupersolutions}
Let $A\in M_{\lambda}(\Omega)$, $b\in L^{\infty}(\Omega)$, and let $u\in W^{1,2}(\Omega)$ be a supersolution of $Lu=0$. Then,
\[
\inf_{\Omega}u\geq\inf_{\partial\Omega}u.
\]
\end{prop}
\begin{proof}
We apply proposition \ref{MaxForSubsolutions} to $-u$, which is an $W^{1,2}(\Omega)$ subsolution of $Lu=0$.
\end{proof}

We are thus led to the maximum principle for solutions.

\begin{thm}\label{MaximumPrinciple}
Let $A\in M_{\lambda}(\Omega)$, $b\in L^{\infty}(\Omega)$, and let $u\in W^{1,2}(\Omega)$ be a solution of $Lu=0$. Then, for almost all $x\in\Omega$,
\[
\inf_{\partial\Omega}u\leq u(x)\leq\sup_{\partial\Omega}u.
\]
\end{thm}

We will also need a version of the maximum principle that will hold for the inhomogeneous equation $Lu=f$.

\begin{thm}\label{SupremumBound}
Let $\Omega$ be a bounded domain, and let $A\in M_{\lambda}(\Omega)$, $b\in L^{\infty}(\Omega)$. Let $p\in\left[1,\frac{d}{d-1}\right)$, and $F\in W^{-1,p}(\Omega)$. If $u\in W_0^{1,2}(\Omega)$ is a solution to the equation $Lu=F$ in $\Omega$, then
\[
\|u\|_{L^{\infty}(\Omega)}\leq C\|F\|_{W^{-1,p}(\Omega)},
\]
where $C$ depends on $d,p,\lambda,\|b\|_{\infty}$ and ${\rm diam}(\Omega)$.
\end{thm}
\begin{proof}
The proof can be found in \cite{Gilbarg}, theorem 8.16. Note that the suprema on the boundary in this proof are equal to $0$, since we are assuming that $u\in W_0^{1,2}(\Omega)$.
\end{proof}

\subsection{Local estimates}
For matrices $A$ that are just uniformly elliptic (and not Lipschitz continuous), we will need the local regularity estimates which appear in \cite{Gilbarg}. We begin with theorem $8.20$, which is Harnack's inequality.

\begin{prop}\label{Harnack}
Let $\Omega$ be a bounded domain, $A\in M_{\lambda}(\Omega)$ and $b\in L^{\infty}(\Omega)$. Suppose that $u\in W^{1,2}(\Omega)$ is a nonnegative solution to the equation $Lu=0$, or $L^tu=0$ in $\Omega$. Let also $B_r$ be a ball, such that $B_{4r}\subseteq\Omega$. Then,
\[
\sup_{B_r}u\leq C\inf_{B_r}u,
\]
where $C$ depends on $d,\lambda,\mu,\|b\|_{\infty}$ and $\diam(\Omega)$.
\end{prop}

We also refer to the next continuity results, which are theorems 8.22 and 8.27 in \cite{Gilbarg}, respectively.

\begin{prop}\label{HolderInside}
Let $\Omega$ be a bounded domain, $A\in M_{\lambda}(\Omega)$ and $b\in L^{\infty}(\Omega)$. Suppose that $u\in W^{1,2}(\Omega)$ is a solution to the equation $Lu=F$, or $L^tu=F$ in $\Omega$, where $F\in W^{-1,p}(\Omega)$, $p\in\left[1,\frac{d}{d-1}\right)$. Let also $B_R(x)\subseteq\Omega$ be a ball. Then, for all $r\in(0,R)$,
\[
\osc_{B_r(x)}u\leq Cr^{\alpha}\left(R^{-\alpha}\sup_{B_R(x)}|u|+\|F\|_{W^{-1,p}(\Omega)}\right),
\]
where $C$ depends on $d,\lambda,\mu,\|b\|_{\infty},\diam(\Omega)$ and $p$.
\end{prop}

For the next proposition to hold, note that we need some regularity on the boundary; in our case, we will assume that our domain is Lipschitz.

\begin{prop}\label{HolderOnTheBoundary}
Let $\Omega$ be a Lipschitz domain, $A\in M_{\lambda}(\Omega)$ and $b\in L^{\infty}(\Omega)$. Let also $q\in\partial\Omega$, and $R>0$. Suppose that $u\in W^{1,2}(\Omega)$ is a solution to the equation $Lu=F$, or $L^tu=F$ in $\Omega$, which vanishes on $\Delta_{2R}(q))$, and where $F\in W^{-1,p}(\Omega)$, $p\in\left[1,\frac{d}{d-1}\right)$. Then, for all $r\in(0,R)$,
\[
\osc_{T_r(x)}u\leq Cr^{\alpha}\left(R^{-\alpha}\sup_{T_R(x)}|u|+\|F\|_{W^{-1,p}(\Omega)}\right),
\]
where $C$ is a good constant that also depends on $p$.
\end{prop}

Finally, we will need the next equicontinuity result, which follows from theorem $8.24$.

\begin{prop}\label{Equicontinuity}
Let $\Omega$ be a bounded domain, $A_n\in M_{\lambda}(\Omega)$ and $b_n\in L^{\infty}(\Omega)$, with $\|b_n\|_{\infty}\leq M$ for some $M>0$. Let $K\subseteq\Omega$ be compact, and suppose that $u_n\in W^{1,2}(\Omega)$ are solutions to the equations $Lu_n=0$, or $L^tu_n=0$ in $\Omega$ which are uniformly bounded in $K$. Then, $(u_n)$ is equicontinuous in $K$.
\end{prop}

\subsection{The Rellich estimate for $L$}

We now turn our attention to the Rellich estimate. This is the main estimate that relates tangential with normal derivatives of solutions on the boundary of a domain, and it will be the basis of our approach to the Dirichlet and Regularity problems.

Let $\Omega$ be a smooth domain with Lipschitz constant $M$. Let also $A\in M_{\lambda,\mu}^s(\Omega)$ (the symmetry assumption will be crucial here) and $b\in L^{\infty}(\Omega)$, and let $q\in\partial\Omega$. Suppose also that $r>0$ is given, such that $r<r_{\Omega}$, where $r_{\Omega}$ is as in section $1.1$.
 
Suppose now that $u:\overline{\Omega}\to\mathbb R$ is a $C^1(\overline{\Omega})\cap W^{2,2}_{\loc}(\Omega)$ solution of $Lu=0$ in $\Omega$. Consider a smooth cutoff $\theta_0$ in $\mathbb R^d$, with $\theta_0\equiv 1$ in $T_r(q)$ (from definition \ref{CylinderPortions}), $\theta_0$ supported in $T_{2r}(q)$, $0\leq\theta_0\leq 1$ and $|\nabla\theta|\leq C/r$. Then, if $e_d$ is the unit vector field in the direction of the axis of $T_r(q)$ and $\partial_du$ denotes $\left<\nabla u,e_d\right>$, we compute
\begin{align*}
\dive
(\left<A\nabla u,\nabla u\right>e_d)-2\dive(\partial_duA\nabla u)&=\partial_d(\left<A\nabla u,\nabla u\right>)-2\partial_du\dive(A\nabla u)-2\left<\nabla\partial_d u,A\nabla u\right>\\
&=\left<\partial_dA\cdot\nabla u,\nabla u\right>-2\partial_d u\cdot b\nabla u,
\end{align*}
since $A$ is symmetric. Therefore, after multiplying with $\theta_0$, we obtain that
\begin{multline*}
\dive(\theta_0\left<A\nabla u,\nabla u\right>e_d)-2\dive(\theta_0\partial_du\cdot A\nabla u)=\\
\theta_0\left<\partial_dA\cdot\nabla u,\nabla u\right>-2\theta_0\partial_du\cdot b\nabla u+\partial_d\theta_0\left<A\nabla u,\nabla u\right>-2\partial_du\left<A\nabla u,\nabla\theta_0\right>.
\end{multline*}

Note now that the domain $T_{2r}(q)$ is Lipschitz. Therefore, using the divergence theorem in a domain slightly smaller than $T_{2r}(q)$ (such that the solution is twice differentiable there), an approximation argument and the support properties of $\theta_0$, we obtain the identity
\begin{multline}\label{eq:FirstRellichIdentity}
\int_{\Delta_{2r}(q)}\theta_0\left(\left<A\nabla u,\nabla u\right>\left<e_d,\nu\right>-2\left<\nabla u,e_d\right>\left<A\nabla u,\nu\right>\right)\,d\sigma=\\
\int_{T_{2r}(q)}\left(\theta_0\left<\partial_dA\cdot\nabla u,\nabla u\right>-2\theta_0\partial_du\cdot b\nabla u+\partial_d\theta_0\left<A\nabla u,\nabla u\right>-2\partial_du\left<A\nabla u,\nabla\theta_0\right>\right)\,dx.
\end{multline}
We now treat the left hand side, and we compute
\begin{align*}
\left<A\nabla u,\nabla u\right>\left<e_d,\nu\right>&=\left<A\nabla u,\nabla_T u\right>\left<e_d,\nu\right>+\partial_{\nu}^0u\left<A\nabla u,\nu\right>\left<e_d,\nu\right>\\
&=\left(\left<A\nabla u,\nabla_T u\right>-\partial_{\nu}^0u\left<A\nabla u,\nu\right>\right)\left<e_d,\nu\right>+2\left<A\nabla u,\nu\right>\left<e_d,\partial_{\nu}^0u\cdot\nu\right>,
\end{align*}
where $\partial_{\nu}^0u$ is the normal derivative of $u$ on $\partial\Omega$. Therefore, since
\begin{align*}
\left<A\nabla u,\nu\right>\left<e_d,\partial_{\nu}^0u\cdot\nu\right>&=\left<A\nabla u,\nu\right>\left<e_d,\nabla u-\nabla_Tu\right>\\
&=
\left<A\nabla u,\nu\right>\left<e_d,\nabla u\right>-\left<A\nabla u,\nu\right>\left<e_d,\nabla_Tu\right>,
\end{align*}
we obtain that
\begin{multline*}
\left<A\nabla u,\nabla u\right>\left<e_d,\nu\right>-2\left<\nabla u,e_d\right>\left<A\nabla u,\nu\right>=\\
\left(\left<A\nabla u,\nabla_T u\right>-\partial_{\nu}^0u\left<A\nabla u,\nu\right>\right)\left<e_d,\nu\right>-2\left<A\nabla u,\nu\right>\left<\nabla_Tu,e_d\right>.
\end{multline*}
For the term in the parenthesis, we write
\begin{align*}
\left<A\nabla u,\nabla_T u\right>-\partial_{\nu}^0u\left<A\nabla u,\nu\right>&=\left<A\nabla u,\nabla_T u-\partial_{\nu}^0u\cdot\nu\right>\\
&=\left<A(\nabla_T u+\partial_{\nu}^0u\cdot\nu),\nabla_T u-\partial_{\nu}^0u\cdot\nu\right>\\
&=\left<A\nabla_T u,\nabla_Tu\right>-\left<A\nu,\nu\right>|\partial_{\nu}^0u|^2,
\end{align*}
where we used that $A$ is symmetric on the last step. Plugging in \eqref{eq:FirstRellichIdentity}, we obtain the \textit{Rellich identity}
\begin{multline}\label{SecondRellichIdentity}
\int_{\Delta_{2r}(q)}\theta_0\left(\left<A\nabla_T u,\nabla_Tu\right>\left<e_d,\nu\right>-\left<A\nu,\nu\right>|\partial_{\nu}^0u|^2\left<e_d,\nu\right>-2\left<A\nabla u,\nu\right>\left<\nabla_Tu,e_d\right>\right)\,d\sigma=\\
\int_{T_{2r}(q)}\left(\theta_0\left<\partial_dA\cdot\nabla u,\nabla u\right>-2\theta_0\partial_dub\nabla u+\partial_d\theta_0\left<A\nabla u,\nabla u\right>-2\partial_du\left<A\nabla u,\nabla\theta_0\right>\right)\,dx.
\end{multline}

We arrange the terms so that the $|\partial_{\nu}^0u|^2$ stays on the left hand side. Since $0\leq\theta_0\leq 1$ and $|e_d|=1$, we obtain that
\begin{multline*}
\int_{\Delta_{2r}(q)}\theta_0\left<e_d,\nu\right>\left<A\nu,\nu\right>|\partial_{\nu}^0u|^2\,d\sigma\leq\\
\int_{\Delta_{2r}(q)}\theta_0|\nabla_Tu||\left<A\nabla u,\nu\right>|\left<e_d,\nu\right>\,d\sigma+\int_{\Delta_{2r}(q)}\theta_0|\left<A\nabla_Tu,\nabla_Tu\right>|^2\left<e_d,\nu\right>\,d\sigma+\\
\int_{T_{2r}(q)}\left(\left|\left<\partial_dA\cdot\nabla u,\nabla u\right>\right|+2\left|\partial_dub\nabla u\right|+\left|\partial_d\theta_0\left<A\nabla u,\nabla u\right>\right|+2\left|\partial_du\left<A\nabla u,\nabla\theta_0\right>\right|\right)\,dx.
\end{multline*}

Now, $\left<e_d,\nu\right>\leq 1$, and also
\[
\left<e_d,\nu\right>=\left<e_d,\frac{(-\nabla\phi,1)}{|(-\nabla\phi,1)|}\right>=\frac{1}{\sqrt{|\nabla\phi|^2+1}}\geq\frac{1}{\sqrt{dM^2+1}}.
\]
In addition, $|\nabla\theta_0|\leq C/r$ and $\left<A\nu,\nu\right>\geq\lambda|\nu|^2=\lambda$, from the ellipticity of $A$, therefore
\begin{multline*}
\int_{\Delta_{2r}(q)}\theta_0|\partial_{\nu}^0u|^2\,d\sigma\leq C\int_{\Delta_{2r}(q)}\theta_0|\nabla_Tu||\left<A\nabla u,\nu\right>|\,d\sigma+
C\int_{\Delta_{2r}(q)}\theta_0|\nabla_Tu|^2\,d\sigma\\
+C\int_{T_{2r}(q)}\left(|\nabla A|+\|b\|_{\infty}\right)|\nabla u|^2+\frac{C}{r}\int_{T_{2r}(q)}|\nabla u|^2\,dx,
\end{multline*}
where $C$ depends on $d,\lambda$ and $M$.

We now add the term $\int_{\Delta_{2r}(q)}\theta_0|\nabla_Tu|^2$ to both sides, to obtain that
\begin{multline*}
\int_{\Delta_{2r}(q)}\theta_0|\nabla u|^2\,d\sigma\leq C\int_{\Delta_{2r}(q)}\theta_0|\nabla_Tu||\partial_{\nu}u|\,d\sigma+
C\int_{\Delta_{2r}(q)}\theta_0|\nabla_Tu|^2\,d\sigma\\
+C\int_{T_{2r}(q)}\left(\mu+\|b\|_{\infty}\right)|\nabla u|^2+\frac{C}{r}\int_{T_{2r}(q)}|\nabla u|^2\,dx,
\end{multline*}
and since
\[
\int_{\Delta_{2r}(q)}\theta_0|\partial_{\nu}u|^2\,d\sigma=\int_{\Delta_{2r}(q)}\theta_0|\left<A\nabla u,\nu\right>|^2\,d\sigma\leq C\int_{\Delta_{2r}(q)}\theta_0|\nabla u|^2\,d\sigma,
\]
we obtain 
\begin{multline*}
\int_{\Delta_{2r}(q)}\theta_0|\partial_{\nu}u|^2\,d\sigma\leq C\int_{\Delta_{2r}(q)}\theta_0|\nabla_Tu||\partial_{\nu}u|\\
+C\int_{\Delta_{2r}(q)}\theta_0|\nabla_Tu|^2\,d\sigma+C\left(1+r^{-1}\right)\int_{T_{2r}(q)}|\nabla u|^2\,dx.
\end{multline*}
Finally, we apply the Cauchy-Schwartz inequality on the first term on the right hand side, and the Cauchy inequality with $\delta$, to obtain that
\[
\int_{\Delta_r(q)}\left|\left<A\nabla u,\nu\right>\right|^2\,d\sigma\leq C\int_{\Delta_{2r}(q)}|\nabla_Tu|^2\,d\sigma+\frac{C}{r}\int_{T_{2r}(q)}|\nabla u|^2\,dx,
\]
where $C$ depends on $d,\lambda,\mu,\|b\|_{\infty},M$ and $r_{\Omega}$, since $1<r_{\Omega}/r$. Therefore, we are led to the next proposition.

\begin{prop}\label{LocalRellich}[First Local Rellich estimate]
Let $\Omega$ be a smooth domain with Lipschitz constant $M$. Let also $A\in M_{\lambda,\mu}^s(\Omega)$, $b\in L^{\infty}(\Omega)$, and suppose that $u:\overline{\Omega}\to\mathbb R$ is a $C^1(\overline{\Omega})\cap W^{2,2}_{\loc}(\Omega)$ solution of $Lu=0$ in $\Omega$. Then, for every $q\in\partial\Omega$ and $r<r_{\Omega}$,
\[
\int_{\Delta_r(q)}\left|\partial_{\nu}u\right|^2\,d\sigma\leq C\int_{\Delta_{2r}(q)}|\nabla_Tu|^2\,d\sigma+
\frac{C}{r}\int_{T_{2r}(q)}|\nabla u|^2,
\]
where $\partial_{\nu}u=\left<A\nabla u,\nu\right>$, $C$ depends on $d,\lambda,\mu,\|b\|_{\infty},M$ and $r_{\Omega}$, and where $r_{\Omega}$ is defined in section $1.1$.
\end{prop}

\subsection{The Rellich estimate for $L^t$}
We now show the local Rellich estimate for the equation $L^tu=0$; that is, $u$ solves the equation
\[
-\dive(A^t\nabla u)-\dive(bu)=0.
\]

Let $\Omega$ be a smooth domain with Lipschitz constant $M$. Let also $A\in M_{\lambda,\mu}^s(\Omega)$ (as in the proof of the Rellich estimate for $L$, the symmetry assumption will be crucial here), $b\in\Lip(\Omega)$, and let $q\in\partial\Omega$. Suppose also that $r>0$ is given, such that $r<r_{\Omega}$, where $r_{\Omega}$ is as in section $1.1$.
 
Suppose now that $u:\overline{\Omega}\to\mathbb R$ is a $C^1(\Omega)\cap W_{\loc}^{2,2}(\Omega)$ solution of $Lu=0$ in $\Omega$. Consider a cutoff $\theta_0$, with $\theta_0\equiv 1$ in $T_r(q)$ (from definition \ref{CylinderPortions}), $\theta_0$ supported in $T_{2r}(q)$, $0\leq\theta_0\leq 1$ and $|\nabla\theta|\leq C/r$. Then, if $e_d$ is the unit vector field in the direction of the $t$ axis and $\partial_du$ denotes $\left<\nabla u,e_d\right>$, we compute
\begin{align*}
\dive
(\left<A\nabla u,\nabla u\right>e_d)-2\dive(\partial_duA\nabla u)&=\partial_d(\left<A\nabla u,\nabla u\right>)-2\partial_du\dive(A\nabla u)-2\left<\nabla\partial_d u,A\nabla u\right>\\
&=\left<\partial_dA\cdot\nabla u,\nabla u\right>+2\partial_d u\dive(bu)\\
&=\left<\partial_dA\cdot\nabla u,\nabla u\right>+2b\nabla u\cdot\partial_du+2\dive b\cdot u\partial_du,
\end{align*}
since $A$ is symmetric. Therefore, after multiplying with $\theta_0$, we obtain that
\begin{multline*}
\dive(\theta_0\left<A\nabla u,\nabla u\right>e_d)-2\dive(\theta_0\partial_duA\nabla u)=\\
\theta_0\left<\partial_dA\nabla u,\nabla u\right>+2\theta_0\partial_dub\nabla u+2\theta_0\dive b\cdot u\partial_du+\partial_d\theta_0\left<A\nabla u,\nabla u\right>-2\partial_du\left<A\nabla u,\nabla\theta_0\right>.
\end{multline*}
So, as in the proof of the Rellich estimate for $L$, the divergence theorem and the support properties of $\theta_0$ show that
\begin{multline*}
\int_{\Delta_{2r}(q)}\theta_0\left(\left<A\nabla u,\nabla u\right>\left<e_d,\nu\right>-2\left<\nabla u,e_d\right>\left<A\nabla u,\nu\right>\right)\,d\sigma\leq\\
\int_{T_{2r}(q)}\mu|\nabla u|^2+|b||\nabla u|^2+2|\dive b||u\nabla u|+\frac{C}{r}|\nabla u|^2.
\end{multline*}
We now treat the left hand side exactly as in the proof of the Rellich estimate for the equation $Lu=0$, to finally obtain that
\begin{multline*}
\int_{\Delta_{2r}(q)}\theta_0|\partial_{\nu}u|^2\,d\sigma\leq C\int_{\Delta_{2r}(q)}|\nabla_Tu|^2\,d\sigma\\
+C
\int_{T_{2r}(q)}\mu|\nabla u|^2+|b||\nabla u|^2+2|\dive b||u\nabla u|+\frac{C}{r}|\nabla u|^2.
\end{multline*}
We are then led to the following estimate for solutions to the adjoint equation.

\begin{prop}\label{LocalRellichForAdjoint}[Local Rellich estimate for the adjoint]
Let $\Omega$ be a smooth domain with Lipschitz constant $M$. Let also $A\in M_{\lambda,\mu}^s(\Omega)$, $b\in\Lip(\Omega)$, and suppose that $u:\overline{\Omega}\to\mathbb R$ is a $C^1(\overline{\Omega})\cap W_{\loc}^{2,2}(\Omega)$ solution of $L^tu=0$ in $\Omega$. Then, for every $q\in\partial\Omega$ and $r<r_{\Omega}$,
\[
\int_{\Delta_r(q)}\left|\partial_{\nu}u\right|^2\,d\sigma\leq C\int_{T_{2r}(q)}|\nabla_Tu|^2\,d\sigma+C\int_{T_{2r}(q)}|\dive b||u\nabla u|+ \frac{C}{r}\int_{T_{2r}(q)}|\nabla u|^2,
\]
where $C$ depends on $d,\lambda,\mu,\|b\|_{\infty},M$ and $r_{\Omega}$, and where $r_{\Omega}$ is defined in section $1.1$.
\end{prop}
\section{Solvability in various spaces}
\subsection{Solvability in $W^{-1,2}(\Omega)$}
The goal of this section is to treat solvability of the equations $Lu=F$ and $L^tu=F$ when $F\in W^{-1,2}(\Omega)$. For this purpose, we will follow the arguments in \cite{Evans}: we will change the equation so that the bilinear form that it defines is coercive, and we will pass to solvability for the original equation using the Fredholm alternative.

In the following, for $\gamma\in\mathbb R$, we will need to consider the bilinear form
\[
\alpha_{\gamma}(u,v)=\int_{\Omega}A\nabla u\nabla v+b\nabla u\cdot v+\gamma uv,
\]
as well as its adjoint form
\[
\alpha_{\gamma}^t(u,v)=\alpha_{\gamma}(v,u).
\]
We will then say that $u\in W^{1,1}_{\loc}(\Omega)$ is a solution to the equation $Lu+\gamma u=0$, if $\alpha_{\gamma}(u,\phi)=0$ for all $\phi\in C_c^{\infty}(\Omega)$. Similarly, we will say that $u\in W^{1,1}_{\loc}(\Omega)$ is a solution to the equation $L^tu+\gamma u=0$, if $\alpha^t_{\gamma}(u,\phi)$=0 for all $\phi\in C_c^{\infty}(\Omega)$.

\begin{prop}\label{GammaSolvability}
Let $\Omega\subseteq\mathbb R^d$ be a bounded domain, and $A\in M_{\lambda}(\Omega)$, with $b\in L^{\infty}(\Omega)$. Then there exists a constant $\gamma>0$ depending only on $\lambda$ and $\|b\|_{\infty}$ such that, for any $F\in W^{-1,2}(\Omega)$, the equation
\[
-\dive(A\nabla u)+b\nabla u+\gamma u=F
\]
has a unique weak solution $u\in W_0^{1,2}(\Omega)$. If we denote $u=g_{\gamma}F$, then the operator $g_{\gamma}:W^{-1,2}(\Omega)\to W_0^{1,2}(\Omega)$ is bounded and onto, and also
\[
\|g_{\gamma}F\|_{W_0^{1,2}(\Omega)}\leq C\|F\|_{W^{-1,2}(\Omega)},
\]
where $C$ depends on $d$, $\lambda$ and $\|b\|_{\infty}$. 
\end{prop}
\begin{proof}
Note first that the bilinear form $\alpha_{\gamma}$ is continuous on $W_0^{1,2}(\Omega)$, since, for $u,v\in W_0^{1,2}(\Omega)$,
\begin{align*}
|\alpha_{\gamma}(u,v)|&\leq\int_{\Omega}|A\nabla u||\nabla v|+|b\nabla uv|+\gamma |uv|\\
&\leq C\|\nabla u\|_2\|\nabla v\|_2+\|b\|_{\infty}\|\nabla u\|_2\|v\|_{2}+\gamma\|u\|_2\|v\|_2\leq C\|u\|_{W_0^{1,2}(\Omega)}\|v\|_{W_0^{1,2}(\Omega)},
\end{align*}
where we used the Cauchy-Schwartz inequality.

For coercivity of $\alpha_{\gamma}$, note that since $A$ is uniformly elliptic, then for every $u\in W_0^{1,2}(\Omega)$,
\begin{align*}
\alpha_{\gamma}(u,u)&=\int_{\Omega}A\nabla u\nabla u+\int_{\Omega}b\nabla u\cdot u+\int_{\Omega}\gamma u^2\\
&\geq\lambda\|\nabla u\|_2^2-\|b\|_{\infty}\|\nabla u\|_2\|u\|_2+\gamma\|u\|_2^2\\
&\geq\lambda\|\nabla u\|_2^2-\|b\|_{\infty}\delta\|\nabla u\|_2^2-\frac{\|b\|_{\infty}}{4\delta}\|u\|_2^2+\gamma\|u\|_2^2,
\end{align*}
for any $\delta>0$, from Cauchy's inequality. Let now
\[
\delta=\frac{\lambda}{2(\|b\|_{\infty}+1)},\quad\gamma=\frac{\|b\|_{\infty}+1}{4\delta}>0.
\]
Note that $\gamma$ depends only on $\lambda$ and $\|b\|_{\infty}$. Moreover, $\lambda-\|b\|_{\infty}\delta\geq\frac{\lambda}{2}$, therefore
\begin{equation}\label{eq:GammaCoercivity}
\alpha_{\gamma}(u,u)\geq\frac{\lambda}{2}\|\nabla u\|_2^2-\frac{\|b\|_{\infty}}{4\delta}\|u\|_2^2+\gamma\|u\|_2^2 \geq\frac{\lambda}{2}\|\nabla u\|_2^2\geq C\|u\|_{W_0^{1,2}(\Omega)},
\end{equation}
where we also used the Sobolev inequality. Therefore, for this $\gamma$, $\alpha_{\gamma}$ is continuous and coercive, where the coercivity constant and the continuity constant depend only on $d$ (from the Sobolev inequality), $\lambda$, and $\|b\|_{\infty}$. Therefore, the Lax-Milgram theorem shows that, for any $F\in W^{-1,2}(\Omega)$, there exists a unique $u\in W_0^{1,2}(\Omega)$ such that
\[
\int_{\Omega}A\nabla u\nabla\phi+b\nabla u\cdot \phi+\gamma u\phi=F\phi,
\]
for any $\phi\in C_c^{\infty}(\Omega)$. We now write $u=g_{\gamma}F$, and we apply the definition of solution for $u$ as a test function; then, coercivity of $\alpha_{\gamma}$ shows that
\[
C\|u\|_{W_0^{1,2}(\Omega)}^2\leq\alpha_{\gamma}(u,u)=Fu\leq\|F\|_{W^{-1,2}(\Omega)}\|u\|_{W_0^{1,2}(\Omega)},
\]
therefore the operator $g_{\gamma}:W^{-1,2}(\Omega)\to W_0^{1,2}(\Omega)$ is bounded, with the bound depending only on $d$, $\lambda$ and $\|b\|_{\infty}$.

Finally, to show that $g_{\gamma}$ is onto, consider $u\in W_0^{1,2}(\Omega)$, and set
\[
Fv=\int_{\Omega}A\nabla u \nabla v+b\nabla u\cdot v,
\]
for $v\in W_0^{1,2}(\Omega)$. We then have that $F\in W^{-1,2}(\Omega)$, and also $g_{\gamma}F=u$. This completes the proof.
\end{proof}

To pass to the original equation $Lu=0$, we use the Fredholm alternative, as the next proposition shows.

\begin{prop}\label{InhomogeneousSolvability}
Let $\Omega$ be a bounded domain, $A\in M_{\lambda}(\Omega)$, and $b\in L^{\infty}(\Omega)$. Then, for any $F\in W^{-1,2}(\Omega)$, the equation
\[
-\dive(A\nabla u)+b\nabla u=F
\]
has a unique weak solution $u\in W_0^{1,2}(\Omega)$.
\end{prop}
\begin{proof}
Consider first the number $\gamma$ that appears in proposition \ref{GammaSolvability}, and also the operator $g_{\gamma}:W^{-1,2}(\Omega)\to W_0^{1,2}(\Omega)$. Consider also the operator $T:L^2(\Omega)\to W^{-1,2}(\Omega)$, with
\[
Tf(v)=\int_{\Omega}fv
\]
for all $v\in W_0^{1,2}(\Omega)$. Since the embedding $W_0^{1,2}(\Omega)\hookrightarrow L^2(\Omega)$ is compact, the operator
\[
K=T\circ \gamma g_{\gamma}:W^{-1,2}(\Omega)\to W^{-1,2}(\Omega)
\]
is compact.

Suppose now that $F,G\in W^{-1,2}(\Omega)$, and $G=KG+F$. Then, for any $v\in W_0^{1,2}(\Omega)$,
\begin{align*}
\alpha(g_{\gamma}G,v)&=\alpha_{\gamma}(g_{\gamma}G,v)-\int_{\Omega}\gamma g_{\gamma}G\cdot v=\left<G,v\right>-\int_{\Omega}\gamma g_{\gamma}G\cdot v\\
&=\left<KG+F,v\right>-\int_{\Omega}\gamma g_{\gamma}G\cdot v=\left<F,v\right>+\left<KG,v\right>-\int_{\Omega}\gamma g_{\gamma}G\cdot v\\
&=\left<F,v\right>,
\end{align*}
therefore $u=g_{\gamma}G\in W_0^{1,2}(\Omega)$ solves the equation $Lu=F$. Hence, if $G-KG=0$, then $u=g_{\gamma}G$ solves the equation $Lu=0$, therefore the maximum principle (theorem \ref{MaximumPrinciple}) shows that $g_{\gamma}G=0$. Since $g_{\gamma}$ is injective, this implies that $G=0$; therefore, the operator $I-K$ is injective.

We now use compactness of $K$, and we apply the Fredholm alternative to obtain that $I-K:W^{-1,2}(\Omega)\to W^{-1,2}(\Omega)$ is bijective. The open mapping theorem shows then that $I-K$ is invertible. Therefore the operator
\[
g=g_{\gamma}\circ(I-K)^{-1}:W^{-1,2}(\Omega)\to W_0^{1,2}(\Omega)
\]
is bounded. So, if $F\in W^{-1,2}(\Omega)$ and $u=gF$, then 
\[
g_{\gamma}^{-1}u-Kg_{\gamma}^{-1}u=(I-K)(g_{\gamma}^{-1}u)=F,
\]
therefore $u=g_{\gamma}(g_{\gamma}^{-1}u)\in W_0^{1,2}(\Omega)$ solves the equation $Lu=F$. Uniqueness now follows from the maximum principle. 
\end{proof}

By considering the adjoint operator $g^t$, we can show solvability of the equation $L^tu=F$.

\begin{prop}\label{InhomogeneousSolvabilityForAdjoint}
Under the same conditions as in proposition \ref{InhomogeneousSolvability}, for every $F\in W^{-1,2}(\Omega)$, the equation
\[
-\dive(A^t\nabla u)-\dive(bu)=F
\]
has a unique solution $u\in W_0^{1,2}(\Omega)$.
\end{prop}
\begin{proof}
Consider the $\gamma$ that appears in proposition \ref{GammaSolvability}, and the operator
\[
g=g_{\gamma}\circ(I-K)^{-1}:W^{-1,2}(\Omega)\to W_0^{1,2}(\Omega)
\]
that appears in the proof of proposition \ref{InhomogeneousSolvability}. Note first that $g$ is a composition of two bijective operators, therefore it is bijective.

Suppose now that $F\in W^{-1,2}(\Omega)$, and set $u=g^tF$. By identifying $W^{-1,2}(\Omega)^*$ with $W_0^{1,2}(\Omega)$, we can consider that $g^t:W^{-1,2}(\Omega)\to W_0^{1,2}(\Omega)$. Then, if $\phi\in W_0^{1,2}(\Omega)$, there exists $G\in W^{-1,2}(\Omega)$ such that $\phi=gG$, and then
\[
\alpha^t(u,\phi)=\alpha(\phi,u)=\alpha(gG,u)=\left<G,u\right>=\left<G,g^tF\right>=\left<F,gG\right>=\left<F,u\right>,
\]
therefore $u=g^tF$ solves the equation $L^tu=F$.

For uniqueness, suppose that $u\in W_0^{1,2}(\Omega)$ solves the equation $L^tu=0$. Then, for any $G\in W^{-1,2}(\Omega)$,
\[
\left<G,u\right>=\alpha(gG,u)=\alpha^t(u,gG)=0,
\]
therefore $u=0$, and the proof is complete.
\end{proof}

\subsection{Solvability for measures}
In chapter 2 we defined the spaces $C_b(\Omega)$ and $\mathcal{B}(\Omega)$. In this section we will show solvability of the equations $Lu=\mu$ and $L^tu=\mu$, where $\mu\in\mathcal{B}(\Omega)$. We will later apply the theorems in this section for Dirac masses, in order to construct Green's function. 

In order to construct those solutions, we will consider the operators from the previous section and we will restrict their domains so that their images are contained in $C_b(\Omega)$; this procedure will require pointwise estimates on the solutions. We will then consider their adjoint operators, which will be defined on $\mathcal{B}(\Omega)$.

We first show solvability for the adjoint equation.

\begin{prop}\label{MeasureSolvabilityForAdjoint}
Let $\Omega$ be a bounded domain, $A\in M_{\lambda}(\Omega)$, and $b\in L^{\infty}(\Omega)$. For every $\mu\in\mathcal{B}(\Omega)$, the equation
\[
-\dive(A^t\nabla u)-\dive(bu)=\mu
\]
has a unique weak solution, which belongs to $W_0^{1,p}(\Omega)$ for every $p\in\left(1,\frac{d}{d-1}\right)$, and also satisfies the inequality
\[
\|u\|_{W_0^{1,p}(\Omega)}\leq C_p\|\mu\|_{\mathcal{B}(\Omega)},
\]
where $C_p$ is a constant that depends on $d,\lambda,\|b\|_{\infty},\diam(\Omega)$, and $p$.
\end{prop}
\begin{proof}
Consider the operator $g:W^{-1,2}(\Omega)\to W_0^{1,2}(\Omega)$ that appears in proposition \ref{InhomogeneousSolvability}, which maps $F\in W^{-1,2}(\Omega)$ to the solution $u\in W_0^{1,2}(\Omega)$ of the equation
\[
-\dive(A\nabla u)+b\nabla u=F.
\]
Using proposition \ref{HolderInside} we conclude that $u\in C(\Omega)$. Moreover, from theorem \ref{SupremumBound}, this operator maps $W^{-1,p}(\Omega)$ to $C_b(\Omega)$, and its norm is a constant $C_p$ which depends on $d,\lambda,\|b\|_{\infty}$, $\diam(\Omega)$, and $p$. This shows that $g^t$ maps $\mathcal{B}(\Omega)=C_b(\Omega)^*$ to $W^{-1,p}(\Omega)^*$, which we identify with $W_0^{1,p}(\Omega)$; hence, for any $\mu\in\mathcal{B}(\Omega)$, we obtain that
\[
\|g^t\mu\|_{W_0^{1,p}(\Omega)}\leq C_p\|\mu\|_{\mathcal{B}(\Omega)}.
\]

Given $\phi\in C_c^{\infty}(\Omega)$, and for $v\in W_0^{1,p}(\Omega)$, set
\[
F_{\phi}v=\int_{\Omega}A\nabla\phi\nabla v+b\nabla\phi\cdot v.
\]
Then $F_{\phi}\in W^{-1,p}(\Omega)$ for all $\phi\in C_c^{\infty}(\Omega)$. Since $\Omega$ is bounded, $W^{-1,p}(\Omega)\subseteq W^{-1,2}(\Omega)$, therefore $F_{\phi}\in W^{-1,2}(\Omega)$. Also, $g$ is injective on $W^{-1,2}(\Omega)$, so we obtain that $gF_{\phi}=\phi$, hence $C_c^{\infty}(\Omega)$ is contained in the image of $g:W^{-1,p}(\Omega)\to C_b(\Omega)$.

We will now show uniqueness: set $X$ to be the subspace of $W^{-1,p}(\Omega)$ that contains all the $F_{\phi}$. Note that $X$ is dense in $W^{-1,2}(\Omega)$, since $g:W^{-1,2}(\Omega)\to W_0^{1,2}(\Omega)$ is an isomorphism, and $X=g^{-1}(C_c^{\infty}(\Omega))$. Since the inclusion $W^{-1,p}(\Omega)\hookrightarrow W^{-1,2}(\Omega)$ is continuous, this will mean that $X$ is dense in $W^{-1,p}(\Omega)$. Therefore, if $u\in W_0^{1,p}(\Omega)$ solves the equation $-\dive(A^t\nabla u)-\dive(bu)=0$, then for every $\phi\in C_c^{\infty}(\Omega)$,
\[
0=\alpha^t(u,\phi)=\alpha(\phi,u)=\alpha(gF_{\phi},u)=\left<F_{\phi},u\right>,
\]
therefore $Gu=0$ for all $G\in X$. But $X$ is dense in $W^{-1,p}(\Omega)$, therefore $u=0$.

Let now $\mu\in\mathcal{B}(\Omega)$, and set $u=g^t\mu\in W_0^{1,p}(\Omega)$. Let also $\phi\in C_c^{\infty}(\Omega)$. Then, we compute
\[
\alpha^t(u,\phi)=\alpha(\phi,u)=\alpha(gF_{\phi},u)=\left<F,u\right>=\left<F,g^t\mu\right>=\left<\mu,gF_{\phi}\right>.
\]
Therefore $u=g^t\mu$ solves the equation $L^tu=\mu$, and also
\[
\|u\|_{W_0^{1,p}(\Omega)}=\|g^t\mu\|_{W_0^{1,p}(\Omega)}\leq C_p\|\mu\|_{\mathcal{B}(\Omega)}.
\]
\end{proof}

We now proceed to show the analogous result for the equation $Lu=\mu$. To do this, we would need pointwise bounds for solutions to the equation $L^tu=0$, which might not hold. For this reason, we will transform our operator to a coercive operator, and then use the Fredholm alternative. We first show the next lemma.

\begin{lemma}\label{MeasureSolvabilityWithGamma}
Suppose that $b\in\Lip(\Omega)$, and $\gamma>0$ is sufficiently large. Then, for every $\mu\in\mathcal{B}(\Omega)$, the equation
\[
-\dive(A\nabla u)+b\nabla u+\gamma u=\mu
\]
has a weak solution $u\in W_0^{1,p}(\Omega)$, for any $p\in\left(1,\frac{d}{d-1}\right)$.
\end{lemma}
\begin{proof}
First, fix $p\in\left(1,\frac{d}{d-1}\right)$. Consider also a constant $\gamma>0$ which is larger than the one appearing in proposition \ref{GammaSolvability}, and also
\[
\gamma\geq\dive b.
\]
Consider now the operator $g_{\gamma}^t:W^{1,2}(\Omega)\to W_0^{1,2}(\Omega)$, which is the adjoint to $g_{\gamma}:W^{1,2}(\Omega)\to W_0^{1,2}(\Omega)$, and where the last operator appears in proposition \ref{GammaSolvability}. Let $F\in W^{-1,2}(\Omega)$ and $v\in W_0^{1,2}(\Omega)$. Since $g_{\gamma}$ is onto $W_0^{1,2}(\Omega)$, there exists $G\in W^{-1,2}(\Omega)$ such that $gG=v$. Therefore,
\[
\alpha_{\gamma}^t(g^tF,v)=\alpha_{\gamma}(\phi, g^tF)=\alpha_{\gamma}(g_{\gamma}G,g^tF)=\left<G,g^tF\right>=\left<F,gG\right>=\left<F,v\right>.
\]
Hence, $g_{\gamma}$ maps $F\in W^{-1,2}(\Omega)$ to $u\in W_0^{1,2}(\Omega)$, such that
\[
-\dive(A\nabla u)-\dive(bu)+\gamma u=0,
\]
and, since $b$ is almost everywhere differentiable, this is equivalent to
\[
-\dive(A\nabla u)-b\nabla u+(\gamma-\dive b)u=0.
\]
We now use proposition 8.22 in \cite{Gilbarg} to conclude that $g_{\gamma}^t$ maps $W^{-1,p}(\Omega)$ to $C(\Omega)$. Our choice of $\gamma$ shows now that proposition 8.16 in \cite{Gilbarg}, is applicable for the equation $L^tu+\gamma u=0$, hence there exists a constant $C>0$, such that
\[
\|g_{\gamma}^tF\|_{L^{\infty}(\Omega)}\leq C\|F\|_{W^{-1,p}(\Omega)},
\]
so $g_{\gamma}^t:W^{-1,p}(\Omega)\to C_b(\Omega)$ is bounded.

Denote by $g_{\gamma}^{tT}$ the adjoint of $g_{\gamma}^t:W^{-1,p}(\Omega)\to C_b(\Omega)$; then $g_{\gamma}^{tT}:\mathcal{B}(\Omega)\to W_0^{1,p}(\Omega)$ is bounded. Let $\mu\in\mathcal{B}(\Omega)$ and set $u=g_{\gamma}^{tT}\mu\in W_0^{1,p}(\Omega)$. As in the proof of proposition \ref{MeasureSolvabilityForAdjoint}, if $\phi\in C_c^{\infty}(\Omega)$, there exists $F\in W^{-1,p}(\Omega)$ such that $g_{\gamma}^tF=\phi$. Then, we compute
\[
a_{\gamma}(u,\phi)=\alpha_{\gamma}^t(\phi,u)=\alpha_{\gamma}^t(g_{\gamma}^tF,u)=\left<F,u\right>=\left<F,g_{\gamma}^{tT}\mu\right>=\left<\mu,g_{\gamma}^tF\right>,
\]
therefore $g_{\gamma}^{tT}\mu=u\in W_0^{1,p}(\Omega)$ is a solution to the equation
\[
-\dive(A\nabla u)+b\nabla u+\gamma u=\mu.
\]
This completes the proof.
\end{proof}

To pass to the equation $Lu=\mu$, we use the Fredholm alternative, as the next proposition shows. 

\begin{prop}\label{MeasureSolvability}
Suppose that $b\in\Lip(\Omega)$, and let $\mu\in\mathcal{B}(\Omega)$. Then, there exists a unique weak solution $u$ of the equation
\[
-\dive(A\nabla u)+b\nabla u=\mu
\]
in $\Omega$, which also belongs to $W_0^{1,p}(\Omega)$ for every $p\in\left(1,\frac{d}{d-1}\right)$.
\end{prop}
\begin{proof}
The proof is similar to the proof of proposition \ref{InhomogeneousSolvability}; we begin by fixing a $p\in\left(1,\frac{d}{d-1}\right)$. Consider the $\gamma$ that appears in lemma \ref{MeasureSolvabilityWithGamma}, and the operators
\[
g_{\gamma}^t:W^{-1,p}(\Omega)\to C_b(\Omega),\quad g_0=g_{\gamma}^{tT}:\mathcal{B}(\Omega)\to W_0^{1,p}(\Omega).
\]
Consider also the operator $T_0:L^p(\Omega)\to\mathcal{B}(\Omega)$, with
\[
\left<T_0f,\phi\right>=\int_{\Omega}f\phi
\]
for every $\phi\in C_b(\Omega)$. Since now the embedding $i:W_0^{1,p}(\Omega)\hookrightarrow L^p(\Omega)$ is compact, the operator
\[
K_0=T_0\circ i\circ\gamma g_0:\mathcal{B}(\Omega)\to\mathcal{B}(\Omega)
\]
is compact.

Suppose now that $\mu,\nu\in\mathcal{B}(\Omega)$, and $\nu=K_0\nu+\mu$. If $\phi\in C_c^{\infty}(\Omega)$, note that there exists $F\in W^{-1,p}(\Omega)$ such that $g_{\gamma}^tF=\phi$. We then compute
\begin{align*}
\alpha(g_0\nu,\phi)&=\alpha^t(\phi,g_0\nu)=\alpha_{\gamma}^t(g_{\gamma}^tF,g_0\nu)-\int_{\Omega}\gamma\phi\cdot g_0\nu=\left<F,g_0\nu\right>-\int_{\Omega}\gamma\phi\cdot g_0\nu\\
&=\left<F,g_{\gamma}^{tT}\nu\right>-\int_{\Omega}\gamma\phi\cdot g_0\nu=\left<g_{\gamma}^tF,\nu\right>-\int_{\Omega}\gamma\phi\cdot g_0\nu\\
&=\left<\phi,K_0\nu+\mu\right>-\int_{\Omega}\gamma\phi\cdot g_0\nu=\left<\phi,\mu\right>+\left<\phi,K_0\nu\right>-\int_{\Omega}\gamma\phi\cdot g_0\nu\\
&=\left<\phi,\mu\right>+\left<\phi,T_0(i(\gamma g_0\nu))\right>-\int_{\Omega}\gamma\phi\cdot g_0\nu\\
&=\left<\phi,\mu\right>,
\end{align*}
therefore $u=g_0\nu$ solves the equation $Lu=\mu$.

Let now $q>d$ be the conjugate exponent to $p$. Since we have assumed that $p>1$, we have that $q<\infty$. We now consider the operator
\[
\tilde{T}:C_b(\Omega)\to L^q(\Omega),
\]
such that $\tilde{T}f=f$ for all $f\in C_b(\Omega)$. Note then that, if $f\in C_b(\Omega)$ and $g\in L^p(\Omega)$,
\[
\left<f,\tilde{T}^tg\right>=\left<\tilde{T}f,g\right>=\int_{\Omega}fg=\left<T_0g,f\right>.
\]
After identifying $L^q(\Omega)^*$ with $L^p(\Omega)$ (which is possible, since $p\in(d,\infty)$), the last equality shows that $T_0$ is equal to $\tilde{T}^t$, the adjoint of $\tilde{T}:C_b(\Omega)\to L^q(\Omega)$. Therefore, if we consider the adjoint $i^t:L^q(\Omega)\to W^{-1,p}(\Omega)$, and we set
\[
\tilde{K}=\gamma g_{\gamma}^t\circ i^t\circ\tilde{T}:C_b(\Omega)\to C_b(\Omega),
\]
we obtain that
\[
\tilde{K}^t=\tilde{T}^t\circ (i^t)^t\circ (\gamma g_{\gamma}^t)^t=T_0\circ i\circ\gamma g_0=K_0.
\]
Since $K_0$ is compact, we obtain that $\tilde{K}$ is also compact.

We now show that $I-K_0$ is injective. To do this, it is enough to show that $I-\tilde{K}$ is injective, from compactness of $\tilde{K}$. For this purpose, suppose that $f\in C_b(\Omega)$ is such that $f=\tilde{K}f$. Set $F=i^t(\tilde{T}f)\in W^{-1,p}(\Omega)$, then $F$ belongs also to $W^{-1,2}(\Omega)$, since $\Omega$ is bounded. Then,
\[
\gamma g_{\gamma}^tF=\gamma(g_{\gamma}^t\circ i^t\circ \tilde{T})f=\tilde{K}f=f,
\]
and $g_{\gamma}^t$ maps $W^{-1,2}(\Omega)$ to $W_0^{1,2}(\Omega)$, therefore $f\in W_0^{1,2}(\Omega)$. Hence, for every $\phi\in W_0^{1,2}(\Omega)$,
\[
\alpha^t(f,\phi)=\alpha_{\gamma}^t(\gamma g_{\gamma}^tf,\phi)-\int_{\Omega}\gamma f\phi=\left<\gamma f,\phi\right>-\int_{\Omega}\gamma f\phi=0.
\]
But then, proposition \ref{InhomogeneousSolvabilityForAdjoint} shows that $f=0$, hence $I-\tilde{K}$ is injective, therefore $I-K_0$ is injective as well.

We now apply the Fredholm alternative to obtain that $I-K_0:\mathcal{B}(\Omega)\to\mathcal{B}(\Omega)$ is bijective, therefore its inverse is also bijective, hence the operator
\[
\tilde{g_0}=g_0\circ(I-K_0)^{-1}:\mathcal{B}(\Omega)\to W_0^{1,p}(\Omega)
\]
is bounded. Hence, if $\mu\in\mathcal{B}(\Omega)$ and we set $u=\tilde{g_0}\mu$, we obtain that
\[
g_0^{-1}u-K_0g_0^{-1}u=(I-K_0)g_0^{-1}u=\mu,
\]
therefore $u$ solves the equation $Lu=\mu$. Uniqueness follows from the fact that $I-K_0$ is injective, and this concludes the proof.
\end{proof}
In this last proof, we could show injectivity of $I-K_0$ without passing through injectivity of $I-\tilde{K}$ by using a uniqueness result for $W_0^{1,p}(\Omega)$ solutions to the equation $Lu=0$, where $p$ here is strictly smaller than $2$. We avoid using a result like this by showing first that $I-\tilde{K}$ is injective, since this is reduced to uniqueness for $W_0^{1,2}(\Omega)$ solutions to the equation $L^tu=0$.

\section{Green's function}
\subsection{Preliminary constructions}
In this section we will apply the results of the previous chapter to specific measures, in order to construct Green's functions for $L$ and $L^t$ in the case where $b\in\Lip(\Omega)$.

Suppose that $\Omega$ is a bounded domain, and set $\delta_x$ to be the Dirac mass at $x\in\Omega$. Then, we have the following lemma.

\begin{lemma}\label{WhatIsInB}
Let $\Omega$ be a bounded domain, and $x\in\Omega$. Then, $\delta_x\in\mathcal{B}(\Omega)$, and $\|\delta_x\|_{\mathcal{B}(\Omega)}=1$. Moreover, if $g\in L^1(\Omega)$, then the functional $T_g$ which maps $f$ to $\int_{\Omega}fg$ belongs to $\mathcal{B}(\Omega)$, with $\|T_g\|_{\mathcal{B}(\Omega)}=\|g\|_{L^1(\Omega)}$.
\end{lemma}
\begin{proof}
For any $f\in C_b(\Omega)$, we compute
\[
\left|\left<\delta_x,f\right>\right|=|f(x)|\leq \|f\|_{L^{\infty}(\Omega)}=\|f\|_{C_b(\Omega)},
\]
hence $\delta_x\in\mathcal{B}(\Omega)$, and $\|\delta_x\|_{\mathcal{B}(\Omega)}\leq 1$. To show the reverse inequality, we test against the function $h\equiv 1$, and we compute \[
1=|g(x)|=\left|\left<\delta_x,f\right>\right|\leq\|\delta_x\|_{\mathcal{B}(\Omega)}\|g\|_{C_b(\Omega)}=\|\delta_x\|_{\mathcal{B}(\Omega)}.
\]
The proof for the second claim is similar. Indeed, if $f\in C_b(\Omega)$, we compute
\[
|T_gf|=\left|\int_{\Omega}fg\right|\leq\|f\|_{L^{\infty}(\Omega)}\|g\|_{L^1(\Omega)}\leq \|g\|_{L^1(\Omega)}\|f\|_{C_b(\Omega)},
\]
therefore $T_g\in\mathcal{B}(\Omega)$, and $\|T_g\|_{\mathcal{B}(\Omega)}\leq\|g\|_{L^1(\Omega)}$. The reverse inequality follows from applying Lusin's theorem to the function ${\rm sgn}(g)$, which completes the proof.
\end{proof}

Given $A$ uniformly elliptic and bounded, $b\in\Lip(\Omega)$ and $y\in\Omega$, we apply proposition \ref{MeasureSolvability} and the previous lemma to obtain that there exists a solution $G_y$ to the equation $LG_y=\delta_y$, where $\delta_y$ is the Dirac delta at $y\in\Omega$. The same proposition also shows that $G_y\in W_0^{1,p}(\Omega)$ for all $p\in\left[1,\frac{d}{d-1}\right)$. We then set $G(z,y)=G_y(z)$ for $z\in\Omega$, and we note that
\[
\int_{\Omega}\left(A(z)\nabla_zG(z,y)\nabla\phi(z)+b(z)\nabla_zG(z,y)\cdot\phi(z)\right)\,dz=\phi(y),
\]
for all $\phi\in C_c^{\infty}(\Omega)$. We call $G$ \emph{Green's function} for the equation $Lu=0$ in $\Omega$.

For the adjoint equation, consider $A$ to be uniformly elliptic and bounded, $b\in L^{\infty}(\Omega)$ and $x\in\Omega$. We then apply proposition \ref{MeasureSolvabilityForAdjoint} to obtain that there exists a solution $G_x^t$ to the equation $L^tG_x^t=\delta_x$. Since $\|\delta_x\|_{\mathcal{B}(\Omega)}=1$, the same proposition also shows that, for any $p\in\left[1,\frac{d}{d-1}\right)$,
\[
\|G_x^t\|_{W_0^{1,p}(\Omega)}\leq C_p,
\]
where $C_p$ is a constant that depends on $d,\lambda,\|b\|_{\infty},\diam(\Omega)$, and $p$. We then set $G^t(z,x)=G_x^t(z)$ for $z\in\Omega$, and we note that
\[
\int_{\Omega}\left(A^t(z)\nabla_zG^t(z,x)\nabla\phi(z)+b(z)\nabla\phi(z)\cdot G^t(z,x)\right)\,dz=\phi(x),
\]
for all $\phi\in C_c^{\infty}(\Omega)$. We call $G^t$ \emph{Green's function} for the equation $L^tu=0$ in $\Omega$.

Note that we have established existence of Green's function for the equation $Lu=0$ only for $b$ that are Lipschitz. Moreover, it is not clear at this point how the $W_0^{1,p}(\Omega)$ norms of $G$ relate to the given quantities. In the following, we will establish existence of Green's function for $b\in L^{\infty}(\Omega)$, and show good pointwise and $L^p$ estimates on Green's function and its derivative. 

To show those properties, we will follow the arguments appearing in the Gr{\"u}ter-Widman paper on Green's function \cite{Gruter}, but first we need to show the symmetry relation between $G$ and $G^t$, in proposition \ref{SymmetryWithAdjoint}. For the proof of this proposition (as well as other arguments that will appear later, for example in the proof of proposition \ref{GreenFunctionEstimatesForAdjoint}) we need to construct specific families of approximations to $G,G^t$ and show that they satisfy various boundedness and continuity properties. This is done in the next lemma.

\begin{lemma}\label{LimitGreenConstruction}
Let $\Omega$ be a bounded domain, $A\in M_{\lambda}(\Omega)$ and $b\in\Lip(\Omega)$. Fix $x,y\in\Omega$. Then, for $n,m\in\mathbb N$ large enough, there exist $G_n,G_m^t\in W_0^{1,2}(\Omega)$, such that
\[
\int_{\Omega}A\nabla G_n\nabla v+b\nabla G_n\cdot v=\frac{1}{|B_{1/n}(y)|}\int_{B_{1/n}(y)}w,\quad\forall w\in W_0^{1,2}(\Omega),
\]
and also
\[
\int_{\Omega}A^t\nabla G_m^t\nabla w+b\nabla w\cdot G_m^t=\frac{1}{|B_{1/m}(x)|}\int_{B_{1/m}(x)}w,\quad\forall w\in W_0^{1,2}(\Omega).
\]
In addition,  the following properties hold.
\begin{enumerate}[i)]
\item For any $n\in\mathbb N$ and $z\in\Omega$, $G_n(z)\geq 0$
\item For any $p\in\left[1,\frac{d}{d-1}\right)$, $\|G_m^t\|_{W_0^{1,p}(\Omega)}\leq C_p$, uniformly in $y$ and $m$, and $G_n\in W_0^{1,p}(\Omega)$
\item There exist subsequences $(G_{k_n})$, $(G^t_{l_m})$, such that $G_{k_n}\to G(\cdot,y)$ and $G_{l_m}^t\to G^t(\cdot,x)$ weakly in every $W_0^{1,p}(\Omega)$, strongly in $L^p(\Omega)$, and almost everywhere in $\Omega$, where $p\in\left[1,\frac{d}{d-1}\right)$
\item $\|G_m^t\|_{L^{\frac{d}{d-2}}_*}\leq C$ uniformly in $y$ and $m$, and $G_n\in L^{\frac{d}{d-2}}_*(\Omega)$
\item For any compact $K\subseteq \Omega$ with $y\notin K$, $(G_n)$ is uniformly bounded and equicontinuous in $K$.
\end{enumerate}
In the above, $C$ is a constant that depends on $d,\lambda,\|b\|_{\infty}$ and $\diam(\Omega)$, and $C_p$ also depends on $p$.
\end{lemma}
\begin{proof}
Let $N,M\in\mathbb N$ such that $B_{1/N}(y),B_{1/M}(x)\subseteq\Omega$, and for $n\geq N,m\geq M$ we define $\e_n=1/n,\delta_m=1/m$, and
\[
f_n=|B_{\e_n}(y)|^{-1}\chi_{B_{\e_n}(y)},\quad g_m=|B_{\delta_m}(x)|^{-1}\chi_{B_{\delta_m}(x)}.
\]
Note that $f_n,g_m\in L^2(\Omega)$; hence, propositions \ref{InhomogeneousSolvability} and \ref{InhomogeneousSolvabilityForAdjoint} show that there exist unique solutions $G_n,G_m^t\in W_0^{1,2}(\Omega)$ to the equations $Lu=f_n,L^tG_m^t=g_m$, therefore
\[
\int_{\Omega}A\nabla G_n\nabla v+b\nabla G_n\cdot v=\frac{1}{|B_{\e_n}(y)|}\int_{B_{\e_n}(y)}w,\quad\forall w\in W_0^{1,2}(\Omega),
\]
and also
\[
\int_{\Omega}A^t\nabla G_m^t\nabla w+b\nabla w\cdot G_m^t=\frac{1}{|B_{\delta_m}(x)|}\int_{B_{\delta_m}(x)}w,\,\,\,\,\forall w\in W_0^{1,2}(\Omega).
\]
For positivity of $G_n$, note that $G_n\in W_0^{1,2}(\Omega)$ is a supersolution to the equation $Lu=0$, therefore the maximum principle (proposition \ref{MinForSupersolutions}) shows that $G_n\geq 0$ in $\Omega$.

Consider now the measures $d\mu_n=f_n\,dx$, $d\nu_m=g_m\,dx$, then lemma \ref{WhatIsInB} shows that $\mu_n,\nu_m\in\mathcal{B}(\Omega)$, with norm $1$. Moreover, we see that $LG_n=\mu_n$ and $LG_m^t=\nu_m$; therefore, proposition \ref{MeasureSolvability} shows that, for any $p\in\left[1,\frac{d}{d-1}\right)$, $G_n\in W_0^{1,p}(\Omega)$. In addition, from proposition \ref{MeasureSolvabilityForAdjoint}, there exists a constant $C_p$ depending only on $d,\lambda,\|b\|_{\infty}$ and $\diam(\Omega)$, such that $\|G_m^t\|_{W_0^{1,p}(\Omega)}\leq C_p$. Therefore (ii) is proved.

For (iii) consider any $p\in\left(1,\frac{d}{d-1}\right)$. From part (ii), $(G_n)$ and $(G_m^t)$ are bounded in $W_0^{1,p}(\Omega)$; therefore, there exist subsequences $G_{k_n}, G_{l_m}^t$, such that
\[
G_{k_n}\to g_y,\quad G_{l_m}^t\to g^t_x,
\]
weakly in $W_0^{1,p}(\Omega)$, for some $g_y,g_x^t\in W_0^{1,p}(\Omega)$. Then, for every $\phi\in C_c^{\infty}(\Omega)$,
\begin{align*}
\int_{\Omega}A\nabla g_y\nabla\phi+b\nabla g_y\cdot\phi&=\lim_{n\to\infty}\int_{\Omega}A\nabla G_{k_n}\nabla\phi+b\nabla G_{k_n}\cdot\phi\\
&=\lim_{n\to\infty}\frac{1}{|B_{1/k_n}(y)|}\int_{B_{k_n}(y)}\phi\\
&=\phi(y),
\end{align*}
therefore $g_y\in W_0^{1,p}(\Omega)$ solves the equation $Lg_y=\delta_y$. Uniqueness of solutions in proposition  \ref{MeasureSolvability} shows that $g_y=G_y$ (Green's function for $L$ at $y$), and similarly, uniqueness in proposition \ref{MeasureSolvabilityForAdjoint} shows that $g_x^t=G_x^t$ (Green's function for $L^t$ at $x$, therefore 
\[
G_{k_n}\to G_y,\quad G_{l_m}^t\to G^t_x,
\]
weakly in $W_0^{1,p}(\Omega)$. Therefore, from the Rellich-Kondrachov compactness theorem, there exist further subsequences, still denoted by $(G_{k_n})$, $(G_{l_m}^t)$, which converge to $G_y$ and $G_x^t$ in $L^{p_0}(\Omega)$ and almost everywhere in $\Omega$.

We now come to the $L^{\frac{d}{d-2}}_*$ bounds for $G_m^t$. For this purpose fix $s>0$, and set
\[
\Omega_s^m=\{z\in\Omega|G_m^t(z)>s\}.
\]
Consider also the positive part
\[
w(z)=\left(\frac{1}{s}-\frac{1}{G_m^t(z)}\right)^+\in W_0^{1,2}(\Omega).
\]
Then, $w\equiv 0$ outside $\Omega_s^m$, and $0\leq w\leq 1/s$, therefore, using $w$ as a test function we obtain
\[
\int_{\Omega_s^m}A^t\nabla G_m^t\nabla G_m^t\cdot(G_m^t)^{-2}+\int_{\Omega_s^m}b\nabla G_m^t\cdot (G_m^t)^{-1}=\frac{1}{|B_{\delta_m}(x)|}\int_{B_{\delta_m}(x)}v\leq\frac{1}{s},
\]
which implies that
\begin{align*}
\int_{\Omega_s^m}A^t\nabla G_m^t\nabla G_m^t\cdot(G_m^t)^{-2}&\leq\frac{1}{s}-\int_{\Omega_s^m}b\nabla G_m^t\cdot(G_m^t)^{-1}\\
&\leq\frac{1}{s}+\|b\|_{\infty}\int_{\Omega_s^m}|\nabla G_m^t|(G_m^t)^{-1}\\
&\leq\left(1+\|b\|_{\infty}\int_{\Omega_s^m}|\nabla G_m^t|\right)s^{-1}\leq Cs^{-1},
\end{align*}
where $C$ depends on $d,\lambda,\|b\|_{\infty}$ and $\diam(\Omega)$, since $(G_m^t)^{-1}\leq 1/s$ in $\Omega_s^m$, and where we also used (i) for $p=1$. Therefore, if we set $w_0(x)=(\log G_m^t-\log s)^+$, the last estimate shows that
\[
\int_{\Omega_s^m}|\nabla w_0|^2\leq\lambda^{-1}\int_{\Omega_s^m}A\nabla w_0\nabla w_0\leq\frac{C}{s},
\]
and Sobolev's inequality shows that 
\[
\left(\int_{\Omega_{2s}^m}\left(\log\frac{G_m^t}{s}\right)^{\frac{2d}{d-2}}\right)^{\frac{d-2}{d}}\leq C(d)\left(\int_{\Omega_s^m}w_0^{\frac{2d}{d-2}}\right)^{\frac{d-2}{d}}\leq C(d)\lambda^{-1}\int_{\Omega_s^m}A\nabla w_0\nabla w_0\leq\frac{C}{s}.
\]
But, $G_m^s\geq 2s$ in $\Omega_{2s}^m$, so we obtain that $s|\Omega_s^m|^{\frac{d-2}{d}}\leq C$ for all $s>0$, which shows that $\|G_m^t\|_{L^{\frac{d}{d-2}}_*}\leq C$.

To show the $L_*^{\frac{d}{d-2}}(\Omega)$ bound on $G_n$, we follow the same procedure: fix $r>0$, and set
\[
\Omega_r^n=\{z\in\Omega|G_n(z)>r\}.
\]
Consider also the positive part
\[
v(z)=\left(\frac{1}{r}-\frac{1}{G_n(z)}\right)^+\in W_0^{1,2}(\Omega).
\]
Then, $v\equiv 0$ outside $\Omega_t^n$, and $0\leq v\leq 1/r$, therefore, using $v$ as a test function we obtain
\[
\int_{\Omega_r^n}A\nabla G_n\nabla G_n\cdot(G_n)^{-2}+\int_{\Omega_r^n}b\nabla G_n\cdot v=\frac{1}{|B_{\e_n}(y)|}\int_{B_{\e_n}(y)}v\leq\frac{1}{r},
\]
which implies that
\begin{align*}
\int_{\Omega_r^n}A\nabla G_n\nabla G_n\cdot(G_n)^{-2}&\leq\frac{1}{r}-\int_{\Omega_r^n}b\nabla G_n\cdot v\\
&\leq\frac{1}{r}+\|b\|_{\infty}\int_{\Omega_r^n}|\nabla G_n|v\\
&\leq\left(1+\|b\|_{\infty}\int_{\Omega_r^n}|\nabla G_n|\right)r^{-1}\leq \tilde{C}r^{-1},
\end{align*}
since $G_n\in W^{1,1}(\Omega)$, from part (i); the only difference here being that $\tilde{C}$ might depend on the derivatives of $b$. Therefore, if we set $v_0(x)=(\log G_n-\log r)^+$, the last estimate shows that
\[
\int_{\Omega_r^n}|\nabla v_0|^2\leq\lambda^{-1}\int_{\Omega_r^n}A\nabla v_0\nabla v_0\leq\frac{\tilde{C}}{r},
\]
and Sobolev's inequality shows that 
\[
\left(\int_{\Omega_{2r}^n}\left(\log\frac{G_n}{r}\right)^{\frac{2d}{d-2}}\right)^{\frac{d-2}{d}}\leq \tilde{C}\left(\int_{\Omega_r^n}v_0^{\frac{2d}{d-2}}\right)^{\frac{d-2}{d}}\leq \tilde{C}\int_{\Omega_r^n}A\nabla v_0\nabla v_0\leq\frac{C}{r}.
\]
But, $G_n\geq 2r$ in $\Omega_{2r}^n$, so we obtain that $r|\Omega_r^n|^{\frac{d-2}{d}}\leq\tilde{C}$ for all $s>0$, which shows that $G_n\in L^{\frac{d}{d-2}}_*$.

For equicontinuity of $(G_n)$, consider $K\subseteq U\subseteq \Omega\setminus\{y\}$, where all inclusions are compact, and consider a covering of $U$ by balls $B_i$, such that their doubles $4B_i$ are compactly supported in $\Omega\setminus\{y\}$. Let $r_i$ be the radius of $B_i$, set $\e_i=\delta(4B_i,\partial\Omega_x)$, and let $\e$ be the minimum of the $\e_i$, and $r$ be the minimum of the $r_i$. Then, for $n>\e^{-1}$, $G_n\in W^{1,2}(4B_i)$ and it is a positive solution of the equation $Lu=0$ in $4B_i$. Consequently, from the Cacciopoli inequality (lemma \ref{Cacciopoli}), we obtain that
\begin{equation}\label{eq:TowardsNablaGnBound}
\fint_{B_i}|\nabla G_n|^2\leq\frac{C}{r_i^2}\fint_{2B_i}G_n^2\leq\frac{C}{r_i^2}\sup\{G_n(z)^2\big{|}z\in 2B_i\}\leq\frac{C}{r_i^2}\left(\fint_{2B_i}G_n\right)^2,
\end{equation}
where $C$ depends on $d,\lambda$ and $\|b\|_{\infty}$, and where the last inequality follows from Harnack's inequality (proposition \ref{Harnack}) to $G_n$ in $4B_i$.

Now, we use estimate \eqref{eq:lorbound} with $p=\frac{d}{d-2}$ and $\delta=p-1$; since $\frac{\delta}{p}=\frac{2}{d}$, we obtain
\begin{align*}
\fint_{B_{2B_i}}G_n&=Cr_i^{-d}\int_{2B_i}G_n\leq Cr_i^{-d}|2B_i|^{\frac{\delta}{p}}\|G_n\|_{L^p_*(2B_i)}\\
&= Cr_i^{-d}|2B_i|^{\frac{2}{d}}\|G_n\|_{L^p_*(2B_i)}\leq \tilde{C}r_i^{(2-d)},
\end{align*}
where we used part (iii) in the last step. Harnack's inequality now shows that $(G_n)$ is uniformly bounded in $2B_i$. In addition, combining with \eqref{eq:TowardsNablaGnBound}, we obtain that 
\[
\fint_{B_i}|\nabla G_n|^2\leq\frac{C}{r_i^2}\left(\tilde{C}r_i^{2-d}\right)^2=\tilde{C}r_i^{2-2d}\leq\tilde{C}r^{2-2d}.
\]
This shows that $(G_n)$, with respect to $n$, is bounded in $W^{1,2}(2B_i)$; hence $(G_n)$ is bounded in $W^{1,2}(U)$. We then apply proposition \ref{Equicontinuity} to obtain that $(G_n)$ is equicontinuous in $K$, which finishes the proof.
\end{proof}

We can now use the previous approximations as test functions, to show the symmetry relation $G(x,y)=G^t(y,x)$ for all $x,y\in\Omega$ with $x\neq y$.

\begin{prop}\label{SymmetryWithAdjoint}
Let $\Omega$ be a bounded domain, $A\in M_{\lambda}(\Omega)$, and $b\in\Lip(\Omega)$. If $G$, $G^t$ are Green's functions for the equations $Lu=0$, $L^tu=0$ in $\Omega$ respectively, then 
\[
G(x,y)=G^t(y,x)
\]
for every $x\in\Omega$ and almost every $y\in\Omega$, with $x\neq y$.
\end{prop}
\begin{proof}
Fix $x,y\in\Omega$ with $x\neq y$, and consider the construction that appears in lemma \ref{LimitGreenConstruction}. Since $G_n,G_m^t\in W_0^{1,2}(\Omega)$, we can use them as test functions: set $v=G_m^t$ and $w=G_n$, to obtain that
\[
\int_{\Omega}A\nabla G_n\nabla G_m^t+b\nabla G_n\cdot G_m^t=\int_{B_{1/n}(y)}G_m^t,
\]
and also
\[
\int_{\Omega}A^t\nabla G_m^t\nabla G_n+b\nabla G_n\cdot G_m^t=\frac{1}{|B_{1/m}(x)|}\int_{B_{1/m}(x)}G_n.
\]
Since the integrals on the left hand sides of the two equations above coincide, we obtain that for all $n,m\in\mathbb N$ that are large enough (in the notation of lemma \ref{LimitGreenConstruction}),
\begin{equation}\label{eq:GreenWithAdjointGreen}
\frac{1}{|B_{1/k_n}(y)|}\int_{B_{1/k_n}(y)}G_{l_m}^t=\frac{1}{|B_{1/l_m}(x)|}\int_{B_{1/l_m}(x)}G_{k_n}.
\end{equation}

Consider now a small closed ball $B$ centered at $x$, which is far from $\partial\Omega$ and $y$, and fix $n$. From lemma \ref{LimitGreenConstruction}, every $G_n$ is continuous in $B$, hence letting $m\to\infty$ in \eqref{eq:GreenWithAdjointGreen} and using that $G_{l_m}^t$ converges to $G_x^t$ in $L^1(B_{1/k_n}(y))$, we obtain that 
\[
\frac{1}{|B_{k_n}(y)|}\int_{B_{k_n}(y)}G_x^t=G_{k_n}(x,y).
\]
Moreover, from lemma \ref{LimitGreenConstruction}, $(G_n)$ is uniformly bounded and equicontinuous in $K$, therefore there exists a subsequence $(G_{k_n})$ that converges uniformly to a continuous function in $K$. Since $G_{k_n}$ converges to $G_y$ almost everywhere, we obtain that $G_{k_n}\to G_y$ uniformly in $B$; therefore, from Lebesgue's differentiation theorem, for almost every $y\in\Omega$ with $y\notin B$,
\[
G^t(y,x)=G_x^t(y)=\lim_{n\to\infty}\frac{1}{|B_{\e_{k_n}}(y)|}\int_{B_{\e_{k_n}}(y)}G_x^t=\lim_{n\to\infty}G_{k_n}(x,y)=G(x,y).
\]
By considering smaller balls $B$, we obtain the equality for all $x\in\Omega$ and almost every $y\in\Omega$, whenever $y\neq x$; this completes the proof.
\end{proof}

\subsection{The pointwise estimates}

In this section we will drop the assumption on differentiability of $b$ and show pointwise estimates on $G$ and $G^t$. We first show the size bounds and the pointwise estimates for Green's function for the adjoint equation $L^tu=0$.

\begin{prop}\label{GreenFunctionEstimatesForAdjoint}
Let $\Omega\subseteq\mathbb R^d$ be a bounded domain, and let $A\in M_{\lambda}(\Omega)$, $b\in\Lip(\Omega)$. There exists a function $G^t:(\Omega\times\Omega)\setminus\{(x,x)|x\in\Omega\}\to\mathbb[0,\infty)$ that satisfies the following properties.
\begin{enumerate}[i)]
\item For any $x\in\Omega$, and any $p\in\left[1,\frac{d}{d-1}\right)$, $\|G^t(-,x)\|_{W_0^{1,p}(\Omega)}\leq C_p$, uniformly in $x$.
\item For all $\phi\in C_c^{\infty}(\Omega)$, $\alpha^t(G^t(-,x),\phi(-))=\phi(x)$: that is,
\[
\int_{\Omega}\left(A^t(z)\nabla_zG^t(z,x)\nabla\phi(z)+b(z)\nabla\phi(z)\cdot G^t(z,x)\right)\,dz=\phi(x).
\]
\item For all $x,y\in\Omega$, $G^t(y,x)\leq C|x-y|^{2-d}$.
\item $\|G^t(-,x)\|_{L^{\frac{d}{d-2}}_*(\Omega)}\leq C$, uniformly in $x$.
\end{enumerate}
All the constants $C$ depend on $d,\lambda$, $\|b\|_{\infty}$ and $\diam(\Omega)$, and $C_p$ also depends on $p$. In particular, the constants do not depend on the derivatives of $b$.
\end{prop}
\begin{proof}
First, note that positivity of $G_x^t$ follows from combining proposition \ref{SymmetryWithAdjoint} with lemma \ref{LimitGreenConstruction}.
	
Fix $x\in\Omega$, let $\varepsilon>0$, so that $B_{\e}=B_{\e}(x)\subseteq\Omega$. Consider the construction in lemma \ref{LimitGreenConstruction}; it is shown there that, for any $p\in\left[1,\frac{d}{d-1}\right)$, there exists a subsequence such that $G^t_{l_m}\to G(-,x)$ weakly in $W_0^{1,p}(\Omega)$, strongly in $L^{p}(\Omega)$, and pointwise in $\Omega$. Moreover, it is also shown that
\[
\|G_m^t\|_{L^{\frac{d}{d-2}}_*}\leq C,
\]
where $C$ depends on $d,\lambda,\|b\|_{\infty}$ and $\diam(\Omega)$. Set now
\[
\Omega_s^m=\{z\in\Omega|G_{l_m}^t(z)>s\},\quad\Omega_s=\{z\in\Omega|G_x^t(z)>s\}.
\]
Since $G_m^t\to G_x^t$ almost everywhere, we obtain that for any fixed $s$, $\chi_{\Omega_s^m}\to\chi_{\Omega_s}$ almost everywhere, as $m\to\infty$. Therefore, the dominated convergence theorem shows that
\[
s|\Omega_s|^{\frac{d-2}{d}}\leq\lim_{m\to\infty}s|\Omega_s^m|^{\frac{d-2}{d}}\leq C,
\]
which shows the uniform $L^{\frac{d}{d-2}}_*$ bound on $G_x^t$; that is,
\begin{equation}\label{eq:*boundOnG}
\|G_x^t\|_{L^{\frac{d}{d-2}}_*}= \sup_{s>0}\left(s|\Omega_s|^{\frac{d-2}{d}}\right)\leq C,
\end{equation}
where $C$ depends on $d,\lambda,\|b\|_{\infty}$ and $\diam(\Omega)$.

We now turn to the pointwise upper bound for $G_x^t$. We fix $z\neq x\in\Omega$, and set $r=|z-x|$. We will consider the following cases: $B_{r/2}(x)\subseteq\Omega$, and $B_{r/2}(x)\not\subseteq\Omega$.

In the first case, $B_{r/4}(z)\subseteq\Omega\setminus B_{3r/4}(x)$, and $G_x^t$ is a positive solution to the equation $L^tu=0$ in $B_{r/4}(z)$. We now apply Harnack's inequality (proposition \ref{Harnack}), to obtain that
\[
G_x^t(z)\leq\sup\{G_x^t(y)|y\in B_{r/4}(z)\}\leq C\fint_{B_{r/4}(z)}G_x^t,
\]
where $C$ depends on $d,\lambda,\|b\|_{\infty}$ and $\diam(\Omega)$. Now, we use estimate \eqref{eq:lorbound} with $p=\frac{d}{d-2}$ and $\delta=p-1$; since $\frac{\delta}{p}=\frac{2}{d}$, we obtain
\begin{align*}
\fint_{B_{r/4}(z)}G_x^t&=Cr^{-d}\int_{B_{r/4}(z)}G_x^t\\
&\leq Cr^{-d}|B_{r/4}(z)|^{\frac{\delta}{p}}\|G_x^t\|_{L^p_*(B_{r/4}(x))}\\
&\leq Cr^{-d}|B_{r/4}(z)|^{\frac{2}{d}}\leq Cr^{(2-d)},
\end{align*}
where we also used \eqref{eq:*boundOnG}. Therefore, in this case, $G_x^t(z)\leq Cr^{2-d}=C|x-z|^{2-d}$.

If, now, $B_{r/2}(x)\not\subseteq\Omega$: in this case, consider a larger domain $\tilde{\Omega}$ such that $B_{r/2}(x)\subseteq\tilde{\Omega}$, and extend the operator $L$ to $\tilde{L}$ on $\tilde{\Omega}$. In $\tilde{\Omega}$ we consider Green's function $\tilde{G}^t$, for the operator $\tilde{L}^t$. Then, the estimate above shows that $\tilde{G}^t(z,x)\leq Cr^{2-d}$. If $\tilde{G}$ is Green's function for the operator $\tilde{L}$ in $\tilde{\Omega}$, lemma \ref{LimitGreenConstruction} shows that
\[
\tilde{G}(x,z)=\tilde{G}^t(z,x)\leq Cr^{2-d}.
\]
Note now that, if $\tilde{G}_n(x,z)$ is the approximation of $\tilde{G}(x,z)$ constructed in proposition \ref{SymmetryWithAdjoint}, then, for fixed $z$ and $n\in\mathbb N$,
\[
G_{k_n}(-,z)-\tilde{G}_{k_n}(-,z)\in W^{1,2}(\Omega)
\]
is a solution to $Lu=0$ in $\Omega$, which is nonnegative on $\partial\Omega$. Then, the maximum principle (proposition \ref{MaxForSubsolutions}) shows that, for all $z\in\Omega$,
\[
G_{k_n}(x,z)-\tilde{G}_{k_n}(x,z)\leq 0.
\]
Therefore
\[
G_x^t(z)=G(x,z)=\lim_{n\to\infty}G_{k_n}(x,z)\leq\lim_{n\to\infty}\tilde{G}_{k_n}(x,z)=\tilde{G}(x,z)\leq Cr^{2-d},
\]
which implies that, in all cases, we have that $G^t(z,x)\leq C|x-y|^{2-d}$, where $C$ depends on $d,\lambda, \|b\|_{\infty}$ and $\diam(\Omega)$.
\end{proof}

Note that in the previous proposition, none of the constants depend on the derivatives of $b$. This fact leads us to the next theorem, in which the differentiability assumption on $b$ is dropped.

\begin{thm}\label{GoodGreenFunctionEstimatesForAdjoint}
Let $\Omega\subseteq\mathbb R^d$ be a bounded domain, and let $A\in M_{\lambda}(\Omega)$, $b\in L^{\infty}(\Omega)$. There exists a function $G^t:(\Omega\times\Omega)\setminus\{(x,x)|x\in\Omega\}\to\mathbb[0,\infty)$ that satisfies the following properties.
\begin{enumerate}[i)]
\item For any $x\in\Omega$, and any $p\in\left[1,\frac{d}{d-1}\right)$, $\|G^t(-,x)\|_{W_0^{1,p}(\Omega)}\leq C_p$, uniformly in $x$.
\item For all $\phi\in C_c^{\infty}(\Omega)$, $\alpha^t(G^t(-,x),\phi(-))=\phi(x)$: that is,
\[
\int_{\Omega}\left(A^t(z)\nabla_zG^t(z,x)\nabla\phi(z)+b(z)\nabla\phi(z)\cdot G^t(z,x)\right)\,dz=\phi(x).
\]
\item For all $x,y\in\Omega$, $G^t(y,x)\leq C|x-y|^{2-d}$.
\item $\|G^t(-,x)\|_{L^{\frac{d}{d-2}}_*(\Omega)}\leq C$, uniformly in $x$.
\end{enumerate}
All the constants $C$ depend on $d,\lambda$, $\|b\|_{\infty}$ and $\diam(\Omega)$, and $C_p$ also depends on $p$. 
\end{thm}
\begin{proof}
Let $p\in\left[1,\frac{d}{d-1}\right)$. Consider a mollification of $b$: that is, $b_n\in\Lip(\Omega)$, $\|b_n\|_{\infty}\leq \|b\|_{\infty}$ for all $n\in\mathbb N$, and $b_n\to b$ in $L^d(\Omega)$. From proposition \ref{GreenFunctionEstimatesForAdjoint}, we can construct Green's function $G_x^{t,n}$, for every point $x\in\Omega$, such that
\[
\int_{\Omega}A^t\nabla G_x^{t,n}\nabla\phi+b\nabla\phi\cdot G_x^{t,n}=\phi(x)\quad\forall\phi\in C_c^{\infty}(\Omega),
\]
where $G_x^{t,n}$ also satisfies the estimates
\[
|G_x^{t,n}(y)|\leq C|x-y|^{2-d},\quad\|G_x^{t,n}\|_{W_0^{1,p}(\Omega)}\leq C_p,
\]
where $C$ is a constant that depends $d,\lambda$, $\|b\|_{\infty},\diam(\Omega)$, and $C_p$ also depends on $p$. Hence, there exists a subsequence $G_x^{t,k_n}$ which converges to a function $G_x^t\in W_0^{1,p}$ weakly in $W_0^{1,p}$, strongly in $L^p(\Omega)$, and almost everywhere in $\Omega$. In particular, this function does not depend on $p$. Then, for any $\phi\in C_c^{\infty}(\Omega)$, we compute
\[
\int_{\Omega}A^t\nabla G_x^t\nabla\phi+b\nabla\phi\cdot G_x^t=\lim_{n\to\infty}\int_{\Omega}A^t\nabla G_x^{t,k_n}\nabla\phi+b\nabla\phi\cdot G_x^{t,k_n}=\phi(x).
\]
From almost everywhere convergence of $G_x^{t,k_n}$ to $G_x^t$, we obtain the pointwise bound (iii) and the Lorentz bound (iv), which completes the proof.
\end{proof}

The final step involves the construction of Green's function without any differentiability assumption on $b$. Since we now have a pointwise bound on Green's function, we will bound its derivative using an analog of Cacciopoli's inequality.

\begin{lemma}\label{BoundOnGreenGradient}
Let $\Omega$ be a bounded domain, and let $A\in M_{\lambda}(\Omega)$, $b\in\Lip(\Omega)$. If $G_y$ is Green's function for $L$ at $y$, then for every $p\in\left[1,\frac{d}{d-1}\right)$ there exists a constant $C_p$, depending on $d,\lambda,\|b\|_{\infty},p$ and $\diam\Omega$, such that $\|G_y\|_{W_0^{1,p}(\Omega)}\leq C_p$, uniformly in $y$.
\end{lemma}
\begin{proof} 
Without loss of generality, assume that $\diam(\Omega)<1$.

Let $r>0$ and $y\in\Omega$, and consider a smooth cutoff $\phi$ which is equal to $1$ in $B_{2r}(y)\setminus B_r(y)$, it is equal to $0$ in $B_{r/2}(y)$ and outside $B_{3r}(y)$, and $|\nabla\phi|\leq C/r$. Consider also the functions $G_n\in W_0^{1,2}(\Omega)$ that appear in the proof of lemma \ref{LimitGreenConstruction}. Using $G_n\phi^2\in W_0^{1,2}(\Omega)$ as a test function, we obtain that
\[
\int_{\Omega}A\nabla G_n\cdot\nabla (G_n\phi^2)+b\nabla G_n\cdot G_n\phi^2=0,
\]
which implies that
\begin{align*}
\lambda\int_{\Omega}|\nabla G_n|^2\phi^2&\leq\int_{\Omega}A\nabla G_n\nabla G_n\cdot\phi^2=
-\int_{\Omega}2A\nabla G_n\nabla\phi\cdot G_n\phi+b\nabla G_n\cdot G_n\phi^2\\
&\leq C\left(\int_{\Omega}|\nabla G_n|^2\phi^2\right)^{1/2}\left(\left(\int_{\Omega}|\nabla\phi|^2|G_n|^2\right)^{1/2}+\|b\|_{\infty}\left(\int_{\Omega}G_n^2\phi^2\right)^{1/2}\right),
\end{align*}
and the last estimate shows that
\[
\int_{\Omega}|\nabla G_n|^2\phi^2\leq C\int_{\Omega}|\nabla\phi|^2|G_n|^2+C\int_{\Omega}G_n^2\phi^2,
\]
where $C$ depends on $\lambda$ and $\|b\|_{\infty}$. Considering the support properties of $\phi$, using the previous estimate we obtain
\begin{align*}
\int_{B_{2r}(y)\setminus B_r(y)}|\nabla G_n|^2&=\int_{B_{2r}(y)\setminus B_r(y)}|\nabla G_n|^2\phi^2\leq \int_{\Omega}|\nabla G_n|^2\phi^2\\
&\leq C\int_{\Omega}|\nabla\phi|^2|G_n|^2+C\int_{\Omega}G_n^2\phi^2\\
&\leq\frac{C}{r^2}\int_{B_{3r}\setminus B_{r/2}}|G_n|^2+C\int_{B_{3r}\setminus B_{r/2}(y)}|G_n|^2\\
&=C\left(\frac{1}{r^2}+1\right)\int_{B_{3r}\setminus B_{r/2}(y)}|G_n|^2.
\end{align*}
If, now, $p\in\left[1,\frac{d}{d-1}\right)$, then, from H\"{o}lder's inequality,
\begin{align*}
\int_{B_{2r}(y)\setminus B_r(y)}|\nabla G_n|^p&\leq\left(\int_{B_{2r}(y)\setminus B_r(y)}|\nabla G_n|^2\right)^{p/2}|B_{2r}(y)|^{1-p/2}\\
&\leq C_p\left(\frac{1}{r^2}+1\right)^{p/2}\left(\int_{B_{3r}\setminus B_{r/2}(y)}|G_n|^2\right)^{p/2}r^{d-pd/2}.
\end{align*}
Hence, if $r<1$, we obtain that $1<\diam(\Omega)/r$, so
\[
\int_{B_{2r}(y)\setminus B_r(y)}|\nabla G_n|^p\leq C_p\left(\int_{B_{3r}\setminus B_{r/2}(y)}|G_n|^2\right)^{p/2}r^{d-pd/2-p}.
\]
If $r\geq 1$ the same inequality holds, since then $(B_{2r}(y)\setminus B_r(y))\cap\Omega=\emptyset$, because we have assumed that $\diam(\Omega)<1$.

We now extend the operator $L$ to an operator $\tilde{L}$ on $B_2(y)$, and let $\tilde{G}_y$ be Green's function for $\tilde{L}$ at $y$, and $\tilde{G}_n$ be the sequence constructed in lemma \ref{LimitGreenConstruction}. Note that $G_n$ and $\tilde{G}_n$ are in $W^{1,2}(\Omega)$, $G_n-\tilde{G}_n$ is a solution to $\tilde{L}u=0$ in $\Omega$, and, from proposition \ref{GreenFunctionEstimatesForAdjoint}, $\tilde{G}_n\geq 0$ on $\partial\Omega$. Therefore, from the maximum principle, $G_n\leq\tilde{G}_n$ in $\Omega$, hence 
\[
\int_{B_{2r}(y)\setminus B_r(y)}|\nabla G_n|^p\leq C_p\left(\int_{B_{3r}\setminus B_{r/2}(y)}|\tilde{G}_n|^2\right)^{p/2}r^{d-pd/2-p}.
\]
Let now $\tilde{K}$ be a compactly supported subset of $B_2(y)\setminus\{y\}$. Then, lemma \ref{LimitGreenConstruction} shows that $\tilde{G}_n$ is uniformly bounded and continuous in $K$. Hence, for some subsequence, the same lemma shows that $\tilde{G}_{k_n}\to \tilde{G}_y$ uniformly. Hence, for every $r\in(0,1)$, the dominated convergence theorem shows that
\begin{align*}
\limsup_{n\to\infty}\int_{B_{2r}(y)\setminus B_r(y)}|\nabla G_{k_n}|^p&\leq C_p\left(\limsup_{n\to\infty}\int_{B_{3r}\setminus B_{r/2}(y)}|\tilde{G}_{k_n}|^2\right)^{p/2}r^{d-pd/2-p}\\
&\leq C_p\left(\int_{B_{3r}\setminus B_{r/2}(y)}\limsup_{n\to\infty}|\tilde{G}_{k_n}|^2\right)^{p/2}r^{d-pd/2-p}\\
&\leq C_p\left(\int_{B_{3r}\setminus B_{r/2}(y)}|\tilde{G}_y(x)|^2\right)^{p/2}r^{d-pd/2-p}.
\end{align*}
We now use the pointwise bound on $\tilde{G}^t$ from proposition \ref{GreenFunctionEstimatesForAdjoint}, and the symmetry relation from proposition \ref{SymmetryWithAdjoint}, to conclude that
\begin{align*}
\limsup_{n\to\infty}\int_{B_{2r}(y)\setminus B_r(y)}|\nabla G_{k_n}|^p&\leq C_p\left(\int_{B_{3r}\setminus B_{r/2}(y)}|y-x|^{2(2-d)}\right)^{p/2}r^{d-pd/2-p}\\
&\leq C_p\left(\int_{B_{3r}\setminus B_{r/2}(y)}(r/2)^{2(2-d)}\right)^{p/2}r^{d-pd/2-p}\\
&=C_pr^{d+p-dp}=C_pr^{\alpha},
\end{align*}
where $\alpha=p+d-dp$; since $p<\frac{d}{d-1}$, we obtain then that $\alpha>0$. Therefore, since a subsequence of $G_{k_n}$ converges to $G_y$ in $W_0^{1,p}$, we obtain that
\[
\int_{B_{2r}(y)\setminus B_r(y)}|\nabla G_y|^p\leq C_pr^{\alpha},
\]
where $C_p$ is a constant that depends on $d,\lambda,\|b\|_{\infty}$ and $p$.

Finally, we apply this inequality for $r=2^{-j}$, $j\in\mathbb N$ and we add the resulting terms, to finally conclude
\[
\int_{\Omega}|\nabla G_y|^p=\sum_{j=0}^{\infty}\int_{B_{2^{1-j}(y)}\setminus B_{2^{-j}}(y)}|\nabla G_y|^p\leq \sum_{j=0}^{\infty}C_p2^{-j\alpha}\leq C_p,
\]
since $\alpha>0$, which implies that the series converges, and where $C_p$ depends on $d,\lambda,\|b\|_{\infty}$ and $p$. To bound the $L^p$ norm of $G_y$ we use the pointwise bound from proposition \ref{GreenFunctionEstimatesForAdjoint}, and the symmetry relation from proposition \ref{SymmetryWithAdjoint}. This completes the proof.
\end{proof}

We are now in position to construct Green's function for bounded drifts $b$.

\begin{thm}\label{GoodGreenFunctionEstimates}
Let $\Omega\subseteq\mathbb R^d$ be a bounded domain, and suppose that $A\in M_{\lambda}(\Omega)$, $b\in L^{\infty}(\Omega)$. There exists a function $G:(\Omega\times\Omega)\setminus\{(x,x)|x\in\Omega\}\to\mathbb[0,\infty)$ that satisfies the following properties.
\begin{enumerate}[i)]
\item For any $y\in\Omega$, and any $p\in\left[1,\frac{d}{d-1}\right)$, $\|G(-,y)\|_{W_0^{1,p}(\Omega)}\leq C_p$, uniformly in $p$
\item For all $\phi\in C_c^{\infty}(\Omega)$, $\alpha(G(-,y),\phi(-))=\phi(y)$: that is,
\[
\int_{\Omega}\left(A(z)\nabla_zG(z,y)\nabla\phi(z)+b(z)\nabla_zG(z,y)\cdot\phi(z)\right)\,dz=\phi(y)
\]
\item For all $x,y\in\Omega$, $G(x,y)\leq C|x-y|^{2-d}$.
\end{enumerate}
In the above, $C$ depends on $d,\lambda,\|b\|_{\infty}$ and $\diam(\Omega)$, and $C_p$ also depends on $p$.
\end{thm}
\begin{proof}
The proof is identical to the proof of theorem \ref{GoodGreenFunctionEstimatesForAdjoint}. For the good bound on the norm $\|G_y\|_{W_0^{1,p}(\Omega)}$ we apply lemma \ref{BoundOnGreenGradient}. In addition, for the pointwise bound we use the pointwise bound from proposition \ref{GreenFunctionEstimatesForAdjoint}, and the symmetry relation from proposition \ref{SymmetryWithAdjoint}. This completes the proof.
\end{proof}

We conclude this section with the following estimates on $W_0^{1,2}(\Omega)$ solutions to the equations $Lu=F$ and $L^tu=F$ in $\Omega$.

\begin{prop}\label{GoodBoundOnSolutions}
Let $\Omega$ be a bounded domain, $A\in M_{\lambda}(\Omega)$, and $b\in L^{\infty}(\Omega)$. Then, for every $F\in W^{-1,2}(\Omega)$, there exists a unique solution $u\in W_0^{1,2}(\Omega)$ of the equation $-\dive(A\nabla u)+b\nabla u=F$ in $\Omega$, and also
\[
\|u\|_{W_0^{1,2}(\Omega)}\leq C\|F\|_{W^{-1,2}(\Omega)},
\]
where $C$ depends on $d,\lambda,\|b\|_{\infty}$ and $\diam(\Omega)$.
\end{prop}
\begin{proof}
Existence and uniqueness follows from proposition \ref{InhomogeneousSolvability}. To show the estimate, consider the $\gamma$ that appears in proposition \ref{GammaSolvability}, and let $u_0\in W_0^{1,2}(\Omega)$ be the solution to
\[
-\dive(A\nabla u_0)+b\nabla u_0+\gamma u_0=F
\]
in $\Omega$, which exists from proposition \ref{GammaSolvability}. Then, the same proposition shows that
\[
\|u_0\|_{W_0^{1,2}(\Omega)}\leq C\|F\|_{W^{-1,2}(\Omega)},
\]
where $C$ is a good constant. Let now
\[
v(x)=\gamma\int_{\Omega}G(x,y)u_0(y)\,dy,
\]
then we compute that $v\in W_0^{1,2}(\Omega)$ is the solution to $-\dive(A\nabla v)+b\nabla v=\gamma u_0$ in $\Omega$. From the pointwise estimates on Green's function in theorem \ref{GoodGreenFunctionEstimates} and lemma \ref{L2BoundOnGradient}, we obtain
\[
\|v\|_{W_0^{1,2}(\Omega)}\leq C\|v\|_{L^2(\Omega)}+C\|\gamma u_0\|_{L^2(\Omega)}\leq C\|u_0\|_{L^2(\Omega)}\leq C\|F\|_{W^{-1,2}(\Omega)},
\]
where $C$ is a good constant, since $\gamma$ is a good constant. If we now set $w=u_0-v$, we compute
\[
-\dive(A\nabla w)+b\nabla w=-\dive(A\nabla u_0)+b\nabla u_0-\dive(A\nabla v)+b\nabla v=F,
\]
therefore $w=u$. Hence
\[
\|u\|_{W_0^{1,2}(\Omega)}\leq \|u_0\|_{W_0^{1,2}(\Omega)}+\|v\|_{W_0^{1,2}(\Omega)}\leq C\|F\|_{W^{-1,2}(\Omega)},
\]
where $C$ is a good constant, which completes the proof.
\end{proof}

We also show the analog of the last proposition for the adjoint equation $L^tu=0$.

\begin{prop}\label{GoodBoundOnSolutionsForAdjoint}
Let $\Omega$ be a bounded domain, $A\in M_{\lambda}(\Omega)$, and $b\in L^{\infty}(\Omega)$. Then, for every $F\in W^{-1,2}(\Omega)$, there exists a unique solution $u\in W_0^{1,2}(\Omega)$ of the equation $-\dive(A\nabla u)-\dive(bu)=0$ in $\Omega$, and also
\[
\|u\|_{W_0^{1,2}(\Omega)}\leq C\|F\|_{W^{-1,2}(\Omega)},
\]
where $C$ depends on $d,\lambda,\|b\|_{\infty}$ and ${\rm diam}(\Omega)$.
\end{prop}
\begin{proof}
Recall the definition of the operator $g:W^{-1,2}(\Omega)\to W_0^{1,2}(\Omega)$ from proposition \ref{InhomogeneousSolvability}, that sends $F\in W^{-1,2}(\Omega)$ to $u=gF\in W_0^{1,2}(\Omega)$, which is the unique solution of $Lu=F$ in $\Omega$. Then, proposition \ref{GoodBoundOnSolutions} shows that
\[
\|g\|_{W^{-1,2}\to W_0^{1,2}}\leq C
\]
for some $C$ that depends only on $d,\lambda,\|b\|_{\infty}$ and $\diam(\Omega)$. This shows that $\|g^t\|_{W^{-1,2}\to W_0^{1,2}}\leq C$. But, from proposition \ref{InhomogeneousSolvabilityForAdjoint}, $g^t:W^{-1,2}(\Omega)\to W_0^{1,2}(\Omega)$ sends $F\in W^{-1,2}(\Omega)$ to the unique $W_0^{1,2}(\Omega)$ solution of $L^tu=F$ in $\Omega$, which completes the proof.
\end{proof}

\subsection{Estimates on the gradients of $G,G^t$}
In this section we will assume that $A$ is Lipschitz continuous, to obtain pointwise bounds and H\"{o}lder continuity on the derivative of Green's function and its adjoint.

\begin{prop}\label{GreenDerivativeBounds}
Let $B$ be a ball of radius $2\rho$, and $A\in M_{\lambda,\mu}(B)$, $b\in L^{\infty}(B)$. Then, for any $x,y\in B_{\rho}$,
\[
|\nabla_xG(x,y)|\leq C|x-y|^{1-d},
\]
where $C$ depends on $d,\lambda,\mu,\|b\|_{\infty}$ and $\rho$.
\end{prop}
\begin{proof}
Let $r=|x-y|/16$, then $B_{2r}(x)\subseteq B_{\rho}$, and also $y\notin B_{2r}(x)$. Set now $u(x)=G(x,y)$, then $u$ is a solution of the equation $L^tu=0$ in $B_{2r}(x)$. From theorem \ref{GoodGreenFunctionEstimates}, $u\in W_0^{1,p}(B_{2r}(y))$ for $p=\frac{d}{2(d-1)}$, therefore proposition \ref{LowRegularity} shows that $u\in W^{1,2}(B_r(y))$. Hence, corollary \ref{LocalBoundOnGradientOfSolutions} shows that
\[
\|\nabla u\|_{L^{\infty}(B_{r/2}(x))}\leq\frac{C}{r}\left(\fint_{B_r(x)}u^2\right)^{1/2}\leq\frac{C}{r}\left(\fint_{B_r(x)}|z-y|^{2-d}\right)^{1/2},
\]
where we used the pointwise estimates on $G$, from theorem \ref{GoodGreenFunctionEstimates}. But, for $z\in B_r(x)$,
\[
|z-y|\geq|x-y|-|x-z|>|x-y|-r=15r,
\]
therefore
\[
|\nabla_xG(x,y)|\leq\|\nabla u\|_{L^{\infty}(B_{r/2}(x))}\leq\frac{C}{r}\left(\fint_{B_r(x)}r^{2-d}\right)^{1/2}=Cr^{1-d}=C|x-y|^{1-d},
\]
which completes the proof.
\end{proof}

We also obtain local H{\"o}lder continuity of the gradient of Green's function. 

\begin{prop}\label{HolderContinuityOfGreen}
Let $B$ be a ball of radius $2\rho$, and suppose that $A\in M_{\lambda,\mu}(B)$, $b\in L^{\infty}(B)$. Let also $G$ be Green's function for $Lu=0$ in $B$. Then there exists $\alpha\in(0,1)$ such that, for all $x_1,x_2\in B_{\rho}$, $y\in B_{2\rho}$,
\[
|\nabla_xG(x_1,y)-\nabla_xG(x_2,y)|\leq C\left(|x_1-y|^{1-d-\alpha}+|x_2-y|^{1-d-\alpha}\right)|x_1-x_2|^{\alpha},
\]
where $C$ depends on $d,\lambda,\mu,\|b\|_{\infty}$ and $\rho$.
\end{prop}
\begin{proof}
For simplicity, assume that $B$ is centered at $0$. Without loss of generality, assume that $|y-x_1|\leq|y-x_2|$. Set $G_y(x)=G(x,y)$, and define
\[
R=\frac{1}{5}\min\{|y-x_2|,\rho\}.
\]
First, suppose that $|x_1-x_2|\geq R$. Consider two cases: if $|y-x_2|\leq\rho$, the estimate in proposition \ref{GreenDerivativeBounds} shows that
\begin{align*}
|\nabla G_y(x_1)-\nabla G_y(x_2)|&\leq C|x_1-y|^{1-d}+C|x_2-y|^{1-d}\\
&=C|x_1-y|^{\alpha}|x_1-y|^{1-d-\alpha}+C|x_2-y|^{\alpha}|x_2-y|^{1-d-\alpha}\\
&\leq C|x_2-y|^{\alpha}\left(|x_1-y|^{1-d-\alpha}+|x_2-y|^{1-d-\alpha}\right)\\
&\leq C(5R)^{\alpha}\left(|x_1-y|^{1-d-\alpha}+|x_2-y|^{1-d-\alpha}\right)\\
&\leq C5^{\alpha}|x_1-x_1|^{\alpha}\left(|x_1-y|^{1-d-\alpha}+|x_2-y|^{1-d-\alpha}\right).
\end{align*}
If, now, $|y-x_2|>\rho$, we have that
\[
|y-x_2|^{\alpha}\leq\left(|y|+|x_2|\right)^{\alpha}\leq(2\rho+\rho)^{\alpha}=3^{\alpha}\rho^{\alpha}=3^{\alpha}5^{\alpha}R^{\alpha},
\]
since $x_2\in B_{\rho}$. Therefore, as above, we obtain
\begin{align*}
|\nabla G_y(x_1)-\nabla G_y(x_2)|&\leq C|x_1-y|^{1-d}+C|x_2-y|^{1-d}\\
&\leq C|x_2-y|^{\alpha}\left(|x_1-y|^{1-d-\alpha}+|x_2-y|^{1-d-\alpha}\right)\\
&\leq C\cdot15^{\alpha}R^{\alpha}\left(|x_1-y|^{1-d-\alpha}+|x_2-y|^{1-d-\alpha}\right)\\
&\leq C\cdot15^{\alpha}|x_1-x_1|^{\alpha}\left(|x_1-y|^{1-d-\alpha}+|x_2-y|^{1-d-\alpha}\right),
\end{align*}
which shows the estimate in all cases when $|x_1-x_2|\geq R$.

Suppose now that $|x_1-x_2|<R$. Then $x_1\in B_R(x_2)$, and, if $x\in B_{2R}(x_2)$, we obtain that
\[
|x-y|\geq|x_2-y|-|x-x_2|\geq|x_2-y|-2R\geq|x_2-y|-\frac{2}{5}|x_2-y|=\frac{3}{5}|x_2-y|,
\]
and also
\[
|x|\leq|x-x_2|+|x_2|\leq 2\rho,
\]
therefore $B_{2R}(x_2)\subseteq B_{2\rho}$. Therefore, $x\mapsto G(x,y)$ is a solution of $Lu=0$ in $B_{2R}(x_2)$, hence proposition \ref{DerivativeRegularity} shows that
\begin{align*}
|\nabla_xG(x_1,y)-\nabla_xG(x_2,y)|&\leq\frac{C}{R}\left(\frac{|x_1-x_2|}{R}\right)^{\alpha}\left(\fint_{B_{2R}(x_2)}|G(x,y)|^2\right)^{1/2}\\
&\leq C|x_1-x_2|^{\alpha}R^{-1-\alpha}|x_2-y|^{2-d}\\
&=C|x_1-x_2|^{\alpha}|x_2-y|^{1-d-\alpha}\left(\frac{|x_2-y|}{R}\right)^{1+\alpha}.
\end{align*}
If, now, $R=\frac{1}{5}|x_2-y|$, we obtain the required bound. On the other hand, if $R=\frac{1}{5}\rho$, then
\[
\left(\frac{|x_2-y|}{R}\right)^{1+\alpha}\leq \left(\frac{2\rho}{\rho/5}\right)^{1+\alpha}=10^{1+\alpha}
\]
which shows that the bound also holds in this case, and completes the proof.
\end{proof}

The same argument as above, using corollary \ref{LocalBoundOnGradientOfSolutions} instead of proposition \ref{DerivativeRegularity}, shows the next estimate.

\begin{prop}\label{HolderContinuityOfGreenAsIs}
Let $B$ be a ball of radius $2\rho$, and suppose that $A\in M_{\lambda,\mu}(B)$, $b\in L^{\infty}(B)$. Let also $G$ be Green's function for $Lu=0$ in $B$. Then for all $x_1,x_2\in B_{\rho}$, $y\in B_{2\rho}$,
\[
|G(x_1,y)-G(x_2,y)|\leq C\left(|x_1-y|^{1-d}+|x_2-y|^{1-d}\right)|x_1-x_2|,
\]
where $C$ depends on $d,\lambda,\mu,\|b\|_{\infty}$ and $\rho$.
\end{prop}

We now turn to the analogous estimates for the gradient of $G^t$. We first show the pointwise estimate.

\begin{prop}\label{GreenDerivativeBoundsForAdjoint}
Let $B$ be a ball of radius $2\rho$, and $A\in M_{\lambda,\mu}(B)$, $b\in C^{\alpha}(B)$, for some $\alpha\in(0,1]$. Then, for any $x,y\in B_{\rho}$,
\[
|\nabla_yG^t(y,x)|\leq C|x-y|^{1-d},
\]
where $C$ depends on $d,\lambda,\mu,\|b\|_{C^{\alpha}}$ and $\rho$.
\end{prop}
\begin{proof}
Let $(b_n)$ be a mollification of $b$, where all the $b_n\in\Lip(\Omega)$, and $b_n\to b$ in $L^d$. Consider also the operator $L_n^t=-\dive(A\nabla u)-\dive(b_nu)$, and set $u_n=(y)=G_n^t(y,x)$ to be Green's function for $L_t^n$ in $B$, centered at $x$.

Let $r=|x-y|/16$, then $B_{2r}(y)\subseteq B_{\rho}$, and also $x\notin B_{2r}(y)$. Then $u$ is a solution of the equation $L_n^tu=0$ in $B_{2r}(y)$. From proposition \ref{GreenFunctionEstimatesForAdjoint}, $u_n\in W_0^{1,p}(B_{2r}(y))$ for $p=\frac{d}{2(d-1)}$, therefore proposition \ref{LowRegularityForAdjoint} shows that $u_n\in W^{1,2}(B_r(y))$. Hence, corollary
\ref{BoundedDerivativeForAdjoint} shows that
\[
\|\nabla u_n\|_{L^{\infty}(B_{r/2}(y))}\leq\frac{C}{r}\left(\fint_{B_r(y)}u_n^2\right)^{1/2}\leq\frac{C}{r}\left(\fint_{B_r(y)}|z-x|^{2-d}\right)^{1/2},
\]
where $C$ depends on $d,\lambda,\mu$ and $\|b\|_{C^{\alpha}}$, and where we used the pointwise estimates on $G^t_n$, from proposition \ref{GreenFunctionEstimatesForAdjoint}. But, for $z\in B_r(y)$,
\[
|x-z|\geq|x-y|-|y-z|>|x-y|-r=15r,
\]
therefore
\[
\|\nabla u_n\|_{L^{\infty}(B_{r/2}(y))}\leq\left(\fint_{B_r(y)}r^{2-d}\right)^{1/2}\leq Cr^{1-d},
\]
where $C$ depends on $d,\lambda,\mu$ and $\|b\|_{C^{\alpha}}$.

Note now that, from proposition \ref{GreenFunctionEstimatesForAdjoint}, $(u_n)$ is uniformly bounded in $B_r(y)$; hence, proposition \ref{DerivativeRegularityForAdjoint} shows that $(\nabla u_n)$ is equicontinuous in $B_r(y)$. Hence, there exists a subsequence $(u_{k_n})$ which converges to some $u$ in $C^1(\overline{B_r(y)})$. But, as in theorem \ref{GoodGreenFunctionEstimatesForAdjoint}, a subsequence of $(u_{k_n})$ converges weakly to $G^t(\cdot,x)$ almost everywhere in $B_r(y)$. This shows that $G^t(\cdot,x)$ is continuously differentiable in $B_r(y)$, and also
\begin{align*}
|\nabla_yG^t(y,x)&|\leq \|\nabla_yG^t(\cdot,x)\|_{L^{\infty}(B_r(y))}\leq\limsup_{n\to\infty}\|\nabla u_n\|_{L^{\infty}(B_{r/2}(y))}\leq Cr^{1-d},
\end{align*}
which completes the proof.
\end{proof}

We also show the H{\"o}lder estimate on the gradient of $\nabla G^t$.

\begin{prop}\label{HolderContinuityOfGreenForAdjoint}
Let $B$ be a ball of radius $2\rho$, and suppose that $A\in M_{\lambda,\mu}(B)$, $b\in C^{\alpha}(B)$, for some $\alpha\in(0,1]$. Let also $G^t$ be Green's function for the equation $L^tu=0$ in $B$. Then, for all $y_1,y_2\in B_{\rho}$, $x\in B_{2\rho}$,
\[
|\nabla_yG^t(y_1,x)-\nabla_yG^t(y_2,x)|\leq C\left(|y_1-x|^{1-d-\alpha}+|y_2-x|^{1-d-\alpha}\right)|y_1-y_2|^{\alpha},
\]
where $C$ depends on $d,\lambda,\mu, \|b\|_{C^{\alpha}}$ and $\rho$.
\end{prop}
\begin{proof}
For simplicity, assume that $B$ is centered at $0$. After applying a mollification argument similar to the proof of proposition \ref{GreenDerivativeBoundsForAdjoint}, it is enough to assume that $b\in\Lip(B)$.

Without loss of generality, assume that $|x-y_1|\leq|x-y_2|$. Set $G_x^t(y)=G^t(y,x)$, and define
\[
R=\frac{1}{5}\min\{|x-y_2|,\rho\}.
\]
First, suppose that $|y_1-y_2|\geq R$. Consider two cases: if $|x-y_2|\leq\rho$, then proposition \ref{GreenDerivativeBoundsForAdjoint} shows that
\begin{align*}
|\nabla G_x^t(y_1)-\nabla G_x^t(y_2)|&\leq C|y_1-x|^{1-d}+C|y_2-x|^{1-d}\\
&=C|y_1-x|^{\alpha}|y_1-x|^{1-d-\alpha}+C|y_2-x|^{\alpha}|y_2-x|^{1-d-\alpha}\\
&\leq C|y_2-x|^{\alpha}\left(|y_1-x|^{1-d-\alpha}+|y_2-x|^{1-d-\alpha}\right)\\
&\leq C(5R)^{\alpha}\left(|y_1-x|^{1-d-\alpha}+|y_2-x|^{1-d-\alpha}\right)\\
&\leq C5^{\alpha}|y_1-y_2|^{\alpha}\left(|y_1-x|^{1-d-\alpha}+|y_2-x|^{1-d-\alpha}\right).
\end{align*}
If, now, $|x-y_2|>\rho$, we have that
\[
|x-y_2|^{\alpha}\leq\left(|x|+|y_2|\right)^{\alpha}\leq(2\rho+\rho)^{\alpha}=3^{\alpha}\rho^{\alpha}=3^{\alpha}5^{\alpha}R^{\alpha},
\]
since $y_2\in B_{\rho}$. Therefore, as above, we obtain
\begin{align*}
|\nabla G_x^t(y_1)-\nabla G_x^t(y_2)|&\leq C|y_1-x|^{1-d}+C|y_2-x|^{1-d}\\
&\leq C|y_2-x|^{\alpha}\left(|y_1-x|^{1-d-\alpha}+|y_2-x|^{1-d-\alpha}\right)\\
&\leq C\cdot15^{\alpha}R^{\alpha}\left(|y_1-x|^{1-d-\alpha}+|y_2-x|^{1-d-\alpha}\right)\\
&\leq C\cdot15^{\alpha}|y_1-y_2|^{\alpha}\left(|y_1-x|^{1-d-\alpha}+|y_2-x|^{1-d-\alpha}\right),
\end{align*}
which shows the estimate in all cases when $|y_1-y_2|\geq R$.

Suppose now that $|y_1-y_2|<R$. Then $y_1\in B_R(y_2)$, and, if $z\in B_{2R}(y_2)$, we obtain that
\[
|z-x|\geq|y_2-x|-|z-y_2|\geq|y_2-x|-2R\geq|y_2-x|-\frac{2}{5}|y_2-x|=\frac{3}{5}|y_2-x|,
\]
and also
\[
|z|\leq|z-y_2|+|y_2|\leq 2\rho,
\]
therefore $B_{2R}(y_2)\subseteq B_{2\rho}$. Therefore, $z\mapsto G^t(z,x)$ is a solution of $L^tu=0$ in $B_{2R}(y_2)$, hence proposition \ref{DerivativeRegularityForAdjoint} shows that
\begin{align*}
|\nabla_yG^t(y_1,x)-\nabla_yG^t(y_2,x)|&\leq\frac{C}{R}\left(\frac{|y_1-y_2|}{R}\right)^{\alpha}\left(\fint_{B_{2R}(y_2)}|G^t(z,x)|^2\right)^{1/2}\\
&\leq C|y_1-y_2|^{\alpha}R^{-1-\alpha}|y_2-x|^{2-d}\\
&=C|y_1-y_2|^{\alpha}|y_2-x|^{1-d-\alpha}\left(\frac{|y_2-x|}{R}\right)^{1+\alpha}.
\end{align*}
If, now, $R=\frac{1}{5}|y_2-x|$, we obtain the required bound. On the other hand, if $R=\frac{1}{5}\rho$, then
\[
\left(\frac{|y_2-x|}{R}\right)^{1+\alpha}\leq \left(\frac{2\rho}{\rho/5}\right)^{1+\alpha}=10^{1+\alpha}
\]
which shows that the bound also holds in this case, and completes the proof.
\end{proof}

The same argument as above, using corollary \ref{BoundedDerivativeForAdjoint} instead of proposition \ref{DerivativeRegularityForAdjoint}, shows the next estimate.

\begin{prop}\label{HolderContinuityOfGreenAsIsForAdjoint}
Let $B$ be a ball of radius $2\rho$, and suppose that $A\in M_{\lambda,\mu}(B)$, $b\in C^{\alpha}(B)$. Let also $G^t$ be Green's function for $L^tu=0$ in $B$. Then for all $y_1,y_2\in B_{\rho}$, $x\in B_{2\rho}$,
\[
|G(y_1,x)-G(y_2,x)|\leq C\left(|y_1-x|^{1-d}+|y_2-x|^{1-d}\right)|y_1-y_2|,
\]
where $C$ depends on $d,\lambda,\mu,\|b\|_{C^{\alpha}(B)}$ and $\rho$.
\end{prop}

\subsection{Mixed derivatives}
We now turn our attention to properties of the function $\nabla_xG(x,y)$, as a function of $y$. We first show the next lemma.

\begin{lemma}\label{MixedGreenIsASolution}
Let $B$ be a ball of radius $\rho$, and suppose that $A\in M_{\lambda,\mu}(B)$, $b\in L^{\infty}(B)$. Fix also $x\in B$. Then, for any $i\in\{1,\dots d\}$, the function
\[
u(y)=\partial_i^xG(x,y)
\]
is a $W^{1,2}_{{\rm loc}}(B\setminus\{x\})$ solution to the equation $L^tu=0$ in $B\setminus\{x\}$, where $\partial_i^x$ denotes the $i$-th partial derivative with respect to $x$. 
\end{lemma}
\begin{proof}
Assume first that $b\in\Lip(\Omega)$.
	
Let $U\subseteq B\setminus\{x\}$ be compactly supported, and consider a set $V$ with $U\subseteq V\subseteq B\setminus\{x\}$, where all inclusions are compact. Then there exists $\e_0>0$ such that $B_{2\e_0}(x)\cap V=\emptyset$. Let also $|h|<\e_0$, fix $i\in\{1,\dots d\}$, and consider the function
\[
g_h(y)=\frac{G(x+he_i,y)-G(x,y)}{h}.
\]
Note first that, from proposition \ref{SymmetryWithAdjoint} and theorem \ref{GoodGreenFunctionEstimatesForAdjoint}, $g_h$ is a $W_0^{1,\frac{d}{2(d-1)}}(B)$ solution of $L^tu=0$ in $V$; hence, since $b$ is Lipschitz, proposition \ref{LowRegularityForAdjoint} shows that $g_h\in W^{1,2}(V)$. In addition, for $y\in V$, we have that $y\notin B_{2\e_0}(x)$, hence $G(\cdot,y)$ is a solution of $Lu=0$ in $B_{2\e_0}(x)$, therefore it is continuously differentiable in $B_{\e_0}(x)$, from proposition \ref{DerivativeRegularity}. Hence, the mean value theorem shows that
\[
|g_h(y)|=\left|\frac{G(x+he_i,y)-G(x,y)}{h}\right|\leq|\nabla_zG(z,y)|\leq C|z-y|^{1-d},
\]
for some $z$ lying on the segment $[x,x+he_i]$, where we also used proposition \ref{GreenDerivativeBounds}. But then $z\in B_{\e_0}(x)$, and since $y\notin B_{2\e_0}(x)$, we obtain that
\[
|g_h(y)|\leq C|z-y|^{1-d}\leq C\e_0^{1-d}=C_0.
\]
This shows that, for $|h|<\e_0$, $g_h$ is a uniformly bounded solution of $L^tu=0$ in $V$, with respect to $h$.

Consider now a covering of $U$ by a finite number of balls $B_i=B_{r_i}(x_i)$, $i=1,\dots N$, such that $4B_i\subseteq V$. Then, Cacciopoli's inequality shows that
\begin{equation}\label{eq:g_hUniform}
\int_{B_i}|\nabla g_h|^2\leq\frac{C}{r_i}^2\int_{2B_i}|g_h|^2\leq \frac{CC_0^2|2B_i|}{r_i^2},
\end{equation}
hence $\nabla g_h\in L^2(B_i)$, with a uniform bound on its norm, for $|h|<\e_0$, where this bound depends on $d,\lambda,\|b\|_{\infty}$ and $\diam(\Omega)$. Therefore $\nabla g_h\in L^2(U)$ uniformly, hence $(g_h)$ is uniformly bounded in $W^{1,2}(U)$, with respect to $h$.

From weak compactness, we obtain the existence of a function $g_0\in W^{1,2}(U)$ such that, for a sequence $h_n\to 0$,
\[
g_{h_n}\xrightarrow[n\to\infty]{}g_0,\quad{\rm weakly}\,\,{\rm in}\,\,W^{1,2}(U).
\]
From the definition of weak solution, we have that $g_0$ is a weak solution of $L^tu=0$ in $U$. In addition, the Rellich compactness theorem and almost everywhere convergence show that there exists a subsequence $g_{t_n}$, with $t_n\to 0$, such that
\[
g_{t_n}\xrightarrow[n\to\infty]{}g_0,\quad{\rm almost}\,\,{\rm everywhere}\,\,{\rm in}\,\,U.
\]
Pick a $y\in U$ such that this convergence holds. Then, we obtain that
\[
g_0(y)=\lim_{n\to\infty}g_{t_n}(y)=\lim_{n\to\infty}\frac{G(x+t_ne_i,y)-G(x,y)}{t_n}=\partial_i^xG(x,y),
\]
since $G(\cdot,y)$ is continuously differentiable in $B_{\e_0}(x)$, and $y\notin B_{2\e_0}(x)$. But, $g_0\in W^{1,2}(U)$ is a solution to $L^tu=0$ in $U$, therefore $\partial_i^xG(x,y)$ is a solution to $L^tu=0$ in $U$, at least when $b\in\Lip(B)$.

In order to pass to non differentiable drifts, let $(b_n)$ be a mollification of $b$, consider the operator $Lu=-\dive(A\nabla u)+b_n\nabla u$, let $G_n$ be Green's function for this operator in $B$, and set
\[
g_h^n(y)=\frac{G_n(x+he_i,y)-G_n(x,y)}{h}
\]
as above. Since $\partial_i^xG_n(x,y)$ is the weak $W^{1,2}(V)$ limit of a subsequence $(g_{h_m}^n)$, as $n\to\infty$, \eqref{eq:g_hUniform} shows that $(\partial_i^xG_n(x,\cdot))$ is bounded in $W^{1,2}(V)$. Therefore, the sequence $(G_n(x,\cdot))$ is bounded in $W^{2,2}(V)$. In addition, a subsequence $(\partial_i^xG_{k_n}(x,\cdot))$ converges weakly in $W^{1,2}(V)$ to a solution $v$ of $L^tu=0$. Moreover, a subsequence of $G_{k_n}(x,\cdot)$ converges almost everywhere to $G(x,\cdot)$ in $V$, and the derivatives with respect to $x$ of this subsequence converge almost everywhere in $V$; this shows that $\partial_i^xG(x,\cdot)=v$ is a $W^{1,2}(B)$ solution to $L^tu=0$ in $V$, which completes the proof.
\end{proof}

We can now show estimates for the adjoint variable of the derivative of Green's function.

\begin{prop}\label{HolderContinuityOfGreenForAdjointOnAdjoint}
Let $B$ be a ball of radius $2\rho$, and suppose that $A\in M_{\lambda,\mu}(B)$, $b\in L^{\infty}(B)$. Let also $G$ be Green's function for the equation $Lu=0$ in $B$. Then, for all $y_1,y_2\in B_{\rho}$, $x\in B_{2\rho}$,
\[
|\nabla_xG(x,y_1)-\nabla_xG(x,y_2)|\leq C|y_1-y_2|^{\alpha}\left(|y_1-x|^{1-d-\alpha}+|y_2-x|^{1-d-\alpha}\right),
\]
where $\alpha\in(0,1)$ and $C$ depend on $d,\lambda,\mu,\|b\|_{\infty}$ and $\rho$.
\end{prop}
\begin{proof}
For simplicity, assume that $B$ is centered at $0$. Without loss of generality, assume that $|y_1-x|\leq|y_2-x|$. Set $G_y(x)=G(x,y)$, and define
\[
R=\frac{1}{5}\min\{|x-y_2|,\rho\}.
\]
First, suppose that $|y_1-y_2|<R$. Fix $x\in B_{2\rho}$, $i\in\{1,\dots d\}$, and set $g_i(y)=\partial_i^xG(x,y)$. Then $y_1\in B_R(y_2)$, and if $y\in B_{2R}(y_2)$, then
\[
|y-x|\geq |y_2-x|-|y_2-y|\geq |y_2-x|-2R\geq |y_2-x|-\frac{2}{5}|y_2-x|=\frac{3}{5}|y_2-x|,
\]
hence the pointwise bounds on $\nabla G_y$ (theorem \ref{GoodGreenFunctionEstimates}) show that $g_i$ is bounded in $B_{2R}(y_2)$, with
\[
|g_i(y)|\leq C|y_2-x|^{1-d}.
\]
In addition, lemma \ref{MixedGreenIsASolution} shows that $g_i\in W^{1,2}(B_{2R}(y_2))$ is a solution to the equation $L^tu=0$, therefore theorems 8.20 and 8.22 in \cite{Gilbarg} show that
\begin{align*}
|g_i(y_1)-g_i(y_2)|&\leq C\left(\frac{|y_1-y_2|}{R}\right)^{\alpha}\left(\fint_{B_{2R}(y_2)}|g_i|^2\right)^{1/2}\leq C|y_1-y_2|^{\alpha}R^{-\alpha}|y_2-x|^{1-d}\\
&=C|y_1-y_2|^{\alpha}|y_2-x|^{1-d-\alpha}\left(\frac{|y_2-x|}{R}\right)^{\alpha},
\end{align*}
where $\alpha$ is a good constant. If now $R=|y_2-x|/5$ we obtain the estimate. On the other hand, if $R=\rho/5$, then
\[
\left(\frac{|y_2-x|}{R}\right)^{\alpha}\leq\left(\frac{\rho}{\rho/5}\right)^{\alpha}=5^{\alpha},
\]
which shows the bound in this case as well.

For the case $|y_1-y_2|\geq R$, we follow the first part of the proof of proposition \ref{HolderContinuityOfGreen}, which only uses the pointwise bounds on the gradient of $G$, to obtain the inequality.
\end{proof}

\subsection{Continuity arguments estimates}
Fix a uniformly elliptic matrix $A$, and denote the operator $-\dive(A\nabla u)+b\nabla u$ by $L_b$, and also denote Green's function for the equation $L_bu=0$ by $G_b(x,y)$. In what follows, we will need to estimate the difference between $G_{b_1}$ and $G_{b_2}$, as well as $\nabla G_{b_1}$ and $\nabla G_{b_2}$.

To accomplish this, we first show a lemma.

\begin{lemma}\label{ForContinuity}
Let $r\geq 1$, and consider two numbers $p_1,p_2$ with
\[
p_1,p_2>\frac{r-1}{r}d,\,\,r(p_1+p_2-2d)+d<0.
\]
Then, for every $x,y\in\mathbb R^d$,
\[
\int_{\mathbb R^d}\frac{dz}{|x-z|^{r(d-p_1)}|y-z|^{r(d-p_2)}}\leq C_d|x-y|^{r(p_1+p_2-2d)+d}.
\]
\end{lemma}
\begin{proof}
Let $x-y=a$, set $z=w+x$, and write
\[
\int_{\mathbb R^d}\frac{dz}{|x-z|^{r(d-p_1)}|y-z|^{r(d-p_2)}}=\int_{\mathbb R^d}\frac{dw}{|w|^{r(d-p_1)}|a-w|^{r(d-p_2)}}=I_1+I_2.
\]
where for $I_1$ we integrate over $U_1=\{|w|<|a-w|\}$, and for $I_2$ we integrate over $U_2=\{|w|>|a-w|\}$.

We will bound $I_1$, the estimate for $I_2$ being similar. We split $I_1$ as
\[
\int_{U_1\cap\left\{|w|>\frac{|a|}{2}\right\}}\frac{dw}{|w|^{r(d-p_1)}|a-w|^{r(d-p_2)}}+\int_{U_1\cap\left\{|w|>\frac{|a|}{2}\right\}}\frac{dw}{|w|^{r(d-p_1)}|a-w|^{r(d-p_2)}}=I_3+I_4.
\]
We then estimate
\[
I_3\leq\int_{|w|>\frac{|a|}{2}}\frac{dw}{|w|^{r(2d-p_1-p_2)}}=C_d\int_{|a|/2}^{\infty}\rho^{d-1}\rho^{r(p_1+p_2-2d)}\,d\rho=C_d|a|^{r(p_1+p_2-2d)+d},
\]
and, for $I_4$, since $|w|<|a|/2$, we have that $|w-a|\geq|a|/2$, so
\begin{align*}
I_4&\leq C_d|a|^{r(p_2-d)}\int_{|w|<\frac{|a|}{2}}\frac{dw}{|w|^{r(d-p_1)}}\\
&\leq C_d|a|^{r(p_2-d)}\int_0^{|a|/2}\rho^{d-1}\rho^{r(p_1-d)}\,d\rho\\
&=C_d|a|^{r(p_1+p_2-2d)+d},
\end{align*}
since the hypotheses imply that $r(p_1-d)+d>0$. This shows the bound for $I_1$, and the proof is complete.
\end{proof}

We then have the following estimates.

\begin{prop}\label{ContinuityArgumentForB}
Let $B$ be a ball with radius $\rho$, and $A\in M_{\lambda,\mu}(\Omega)$. Suppose also that $b_1,b_2\in L^{\infty}(B)$. Then, there exists $C=C(d,p,\lambda,\mu,\|b_1\|_{\infty},\|b_2\|_{\infty},\rho)$ such that, for every $x,y\in B$ with $x\neq y$,
\[
|G_{b_1}(x,y)-G_{b_2}(x,y)|\leq C\|b_1-b_2\|_{L^{2d}(B)}|x-y|^{5/2-d}.
\]
and also
\[
|\nabla_xG_{b_1}(x,y)-\nabla_xG_{b_2}(x,y)|\leq C\|b_1-b_2\|_{L^{2d}(B)}|x-y|^{3/2-d}.
\]
\end{prop}
\begin{proof}
Without loss of generality, assume that $b\in\Lip(B)$; we can then recover the case $b\in L^{\infty}(B)$ using a mollification argument.

Let $x\neq y$ in $B$, and let $G_{b_2}^t$ be Green's function for the adjoint equation $L_{b_2}^tu=0$. We then obtain that
\begin{equation}\label{eq:FirstGreen}
\int_B\left(A(z)\nabla_zG_{b_1}(z,y)\nabla_zG_{b_2}^t(z,x)+b_1(z)\nabla_zG_{b_1}(z,y)G_{b_2}^t(z,x)\right)\,dz=G_{b_2}^t(y,x),
\end{equation}
since the poles of $G_{b_1}(\cdot,x),\nabla G_{b_1}(\cdot,x)$, and $G_{b_2}^t(\cdot,y),\nabla G_{b_1}(\cdot,x)$ occur at different points, from the pointwise bounds. In addition,
\[
\int_B\left(A^t(z)\nabla_zG_{b_2}^t(z,x)\nabla_zG_{b_1}(z,y)+b_2(z)\nabla_zG_{b_1}(z,y) G_{b_2}^t(z,x)\right)\,dz=G_{b_1}(x,y).
\]
We now subtract the identities above, to obtain
\begin{equation}\label{eq:Diff}
\int_B(b_1(z)-b_2(z))\nabla_zG_{b_1}(z,y)G_{b_2}^t(z,x)\,dz=G_{b_1}(x,y)-G_{b_2}^t(y,x)=G_{b_1}(x,y)-G_{b_2}(x,y),
\end{equation}
where we also used proposition \ref{SymmetryWithAdjoint}. Therefore, from the pointwise bounds on Green's function and its derivative, if $r=\frac{2d}{2d-1}$ is the conjugate exponent to $2d$, then
\begin{align*}
|G_{b_1}(x,y)-G_{b_2}(x,y)|&\leq\|b_1-b_2\|_{L^{2d}(B)}\left(\int_B\left|\nabla_zG_{b_1}(z,y)G_{b_2}^t(z,x)\right|^r\,dz\right)^{1/r}\\
&\leq \|b_1-b_2\|_{L^{2d}(B)}\left(\int_{\mathbb R^d}\frac{dz}{|z-y|^{r(d-1)}|x-z|^{r(d-2)}}\right)^{1/r}
\end{align*}
Set now $p_1=1$, $p_2=2$ in the previous lemma. Since $r<\frac{d}{d-1}$, we obtain that $\frac{r-1}{r}d<1=p_1<2=p_2$, and also $r>1$, therefore
\[
r(p_1+p_2-2d)+d<p_1+p_2-d=3-d\leq 0,
\]
therefore the hypotheses of lemma \ref{ForContinuity} are satisfied. Hence, we obtain that
\begin{align*}
|G_{b_1}(x,y)-G_{b_2}(x,y)|&\leq\|b_1-b_2\|_{L^{2d}(B)}\left(C|x-y|^{r(3-2d)+d}\right)^{1/r}\\
&=C\|b_1-b_2\|_{L^{2d}(B)}|x-y|^{3-d(2-1/r)},
\end{align*}
which shows the first estimate.

Fix now $y\in B$, and set $v(x)$ to be the left hand side of \eqref{eq:Diff}. We can then compute that $v$ is weakly differentiable in $B$, and
\[
\nabla v(x)=\int_B(b_1(z)-b_2(z))\nabla_zG_{b_1}(z,y)\nabla_xG_{b_2}^t(z,x)\,dz,
\]
therefore we obtain that
\[
\nabla_xG_{b_1}(x,y)-\nabla_xG_{b_2}(x,y)=\int_B(b_1(z)-b_2(z))\nabla_zG_{b_1}(z,y)\nabla_xG_{b_2}(x,z)\,dz.
\]
This shows that, as before,
\[
|\nabla_xG_{b_1}(x,y)-\nabla_xG_{b_2}(x,y)|\leq\|b_1-b_2\|_{L^{2d}(B)}\left(\int_B\frac{dz}{|z-y|^{r(d-1)}|x-z|^{r(d-1)}}\right)^{1/r},
\]
and lemma \ref{ForContinuity} shows the second estimate.
\end{proof}

We also treat the derivative of Green's function with respect to the adjoint variable.

\begin{prop}\label{ContinuityArgumentForBAdjoint}
Let $B$ be a ball with radius $\rho$, and $A\in M_{\lambda,\mu}(\Omega)$. Suppose also that $b_1,b_2\in C^{\alpha}(B)$. Then, there exists $C=C(d,p,\lambda,\mu,\|b_1\|_{C^{\alpha}},\|b_2\|_{C^{\alpha}},\rho)$ such that
\[
|\nabla_yG_{b_1}^t(y,x)-\nabla_yG_{b_2}^t(y,x)|\leq C\|b_1-b_2\|_{C^{\alpha}(B_{2r})}|x-y|^{3/2-d},
\]
for every $x,y\in B$ with $x\neq y$.
\end{prop}
\begin{proof}
Let $r=|x-y|/32$, and set $u(z)=G_{b_1}^t(z,x)-G_{b_2}^t(z,x)$, for $z\in B_{16r}=B_{16r}(y)$. We then compute, in $B_{2r}$,
\begin{align*}
-\dive(A\nabla u)-\dive(b_1u)&=-\dive(A\nabla G_{b_1}^t)-\dive(b_1G_{b_1}^t)+\dive(A\nabla G_{b_2}^t)+\dive(b_1G_{b_2}^t)\\
&=\dive((b_1-b_2)G_{b_2}^t)=\dive g.
\end{align*}
We now apply corollary \ref{BoundedDerivativeForAdjoint}, to obtain that
\[
\|\nabla u\|_{L^{\infty}(B_r)}\leq\frac{C}{r}\left(\fint_{B_{2r}}|u|^2\right)^{1/2}+C\|g\|_{L^{\infty}(2B)}+Cr^{\alpha}\|g\|_{C^{0,\alpha}(B_{2r})}.
\]
For the last term, we compute
\begin{align*}
\|g\|_{C^{0,\alpha}(B_{2r})}&=\|(b_1-b_2)G_{b_2}^t\|_{C^{0,\alpha}(B_{2r})}\\
&\leq\|b_1-b_2\|_{L^{\infty}(B_{2r})}\|G_{b_2}^t\|_{C^{0,\alpha}(B_{2r})}+\|b_1-b_2\|_{C^{0,\alpha}(B_{2r})}\|G_{b_2}^t\|_{L^{\infty}(B_{2r})}\\
&\leq \|b_1-b_2\|_{C^{\alpha}(B_{2r})}\left(\|\nabla G_{b_2}^t\|_{L^{\infty}(B_{2r})}r^{1-\alpha}+\|G_{b_2}^t\|_{L^{\infty}(B_{2r})}\right).
\end{align*}
Using proposition \ref{GreenDerivativeBoundsForAdjoint} we then obtain that
\begin{align*}
\|g\|_{C^{0,\alpha}(B_{2r})}&\leq\|b_1-b_2\|_{C^{\alpha}(B_{2r})}\left(Cr^{1-d}r^{1-\alpha}+r^{2-d}\right)\\
&\leq C\|b_1-b_2\|_{C^{\alpha}(B_{2r})}r^{2-d-\alpha}.
\end{align*}
In addition,
\[
\|g\|_{L^{\infty}(2B)}\leq\|b_1-b_2\|_{L^{\infty}(B_{2r})}\|G_{b_2}^t\|_{L^{\infty}(B_{2r})}\leq C\|b_1-b_2\|_{L^{\infty}(B_{2r})}r^{2-d}.
\]
Moreover, using propositions \ref{SymmetryWithAdjoint} and \ref{ContinuityArgumentForB},
\begin{align*}
\frac{C}{r}\left(\fint_{B_{2r}}|u|^2\right)^{1/2}&=\frac{C}{r}\left(\fint_{B_{2r}}|G_{b_1}^t(z,x)-G_{b_2}^t(z,x)|^2\,dz\right)^{1/2}\\
&=\frac{C}{r}\left(\fint_{B_{2r}}|G_{b_1}(x,z)-G_{b_2}^t(x,z)|^2\,dz\right)^{1/2}\\
&\leq\frac{C}{r}\|b_1-b_2\|_{L^{2d}}\left(\fint_{B_{2r}}|x-z|^{5-2d}\,dz\right)^{1/2}\\
&\leq\frac{C}{r}\|b_1-b_2\|_{L^{2d}}r^{5/2-d}=C\|b_1-b_2\|_{L^{2d}}r^{3/2-d},
\end{align*}
therefore
\begin{align*}
\|\nabla u\|_{L^{\infty}(B_r)}&\leq C\|b_1-b_2\|_{L^{2d}}r^{3/2-d}+C\|b_1-b_2\|_{L^{\infty}(B_{2r})}r^{2-d}+C\|b_1-b_2\|_{C^{\alpha}(B_{2r})}r^{2-d}\\
&\leq C\|b_1-b_2\|_{C^{\alpha}(B_{2r})}r^{3/2-d}=C\|b_1-b_2\|_{C^{\alpha}(B_{2r})}|x-y|^{3/2-d}.
\end{align*}
The last estimate completes the proof.
\end{proof}
\section{Harmonic measure}

In this chapter we will be concerned with the classical Dirichlet problem, and we will define harmonic measure for the equation $Lu=0$ in a Lipschitz domain $\Omega$. We will then show how the harmonic measure relates to Green's function, and we will show estimates analogous to the ones appearing in \cite{KenigCBMS}; for a more comprehensive treatment, we also refer to \cite{JerisonKenigBoundary}.

\subsection{The classical Dirichlet problem}
We turn our attention to the Dirichlet problem for the equation $Lu=0$ with boundary data $f\in C(\partial\Omega)$. We first give the following definition.

\begin{dfn}\label{WeakSolutionForDirichlet}
Let $\Omega$ be a Lipschitz domain, $A\in M_{\lambda}(\Omega)$, and $b\in L^{\infty}(\Omega)$. Given $f\in C(\partial\Omega)$, we say that $u\in W^{1,2}_{{\rm loc}}(\Omega)\cap C(\overline{\Omega})$ is a weak solution to the Dirichlet problem with data $f$,
\[
\left\{\begin{array}{c l}
Lu=0,&{\rm in}\,\,\Omega\\
u=f,&{\rm on}\,\,\partial\Omega,
\end{array}\right.
\]
if $u$ is a weak solution of $Lu=0$ in $\Omega$, and $u=f$ on $\partial\Omega$.
\end{dfn}

It is the case that we can always solve the Dirichlet problem with boundary values in $C(\partial\Omega)$, only assuming that $A$ is bounded and uniformly elliptic and $b$ is bounded. In order to show this, we first treat the case of $f$ being Lipschitz.

\begin{prop}\label{WeakSolvabilityForLipschitz}
Let $\Omega$ be a Lipschitz domain, $A\in M_{\lambda}(\Omega)$, and $b\in L^{\infty}(\Omega)$. Then there exists $\alpha\in(0,1)$ such that, for every $f\in{\rm Lip}(\partial\Omega)$, the Dirichlet problem
\[
\left\{
\begin{array}{c l}
-\dive(A\nabla u)+b\nabla u=0,&{\rm in}\,\,\Omega\\
u=f,&{\rm on}\,\,\partial\Omega
\end{array}
\right.
\]
has a unique weak solution in $W^{1,2}(\Omega)\cap C^{\alpha}(\overline{\Omega})$.
\end{prop}
\begin{proof}
Uniqueness follows from the maximum principle (theorem \ref{MaximumPrinciple}). For existence, let $f\in{\rm Lip}(\partial\Omega)$, and extend $f$ to a Lipschitz function $f\in{\rm Lip}(\mathbb R^d)$. For $v\in W_0^{1,2}(\Omega)$, define
\[
Fv=\alpha(f,v)=\int_{\Omega}A\nabla f\nabla v+b\nabla f\cdot v.
\]
Since $f$ is Lipschitz, we obtain that $F\in W^{-1,2}(\Omega)$,
therefore proposition \ref{InhomogeneousSolvability} shows that there exists $u_0\in W_0^{1,2}(\Omega)$ such that, for all $v\in W_0^{1,2}(\Omega)$, $\alpha(u_0,v)=Fv$. Note that, from propositions \ref{HolderInside} and \ref{HolderOnTheBoundary}, $u_0\in C^{\alpha}(\overline{\Omega})$.

Set $u=f-u_0\in W^{1,2}(\Omega)$. Since $f$ is Lipschitz, we obtain that $u\in C^{\alpha}(\overline{\Omega})$. In addition, for every $v\in W_0^{1,2}(\Omega)$, we compute
\[
\alpha(u,v)=\alpha(f,v)-\alpha(u_0,v)=0,
\]
which shows that $u$ is a solution of $Lu=0$ in $\Omega$. Since $u_0$ has trace $0$ on $\partial\Omega$, this shows that $u$ has trace $f$ on $\partial\Omega$, and continuity of $u$ shows that $u=f$ on $\partial\Omega$.
\end{proof}

By a density argument, we can show solvability of the Dirichlet problem for all $f\in W^{1,2}_{\loc}(\Omega)\cap C(\partial\Omega)$.

\begin{thm}\label{WeakSolvability}
Under the same assumptions as in proposition \ref{WeakSolvabilityForLipschitz}, for any $f\in C(\partial\Omega)$, the Dirichlet problem
\[
\left\{\begin{array}{c l}
-\dive(A\nabla u)+b\nabla u=0,&{\rm in}\,\,\Omega\\
u=f,&{\rm on}\,\,\partial\Omega,
\end{array}\right.
\]
has a unique weak solution $u\in W^{1,2}_{\loc}\cap C(\overline{\Omega})$.
\end{thm}
\begin{proof}
For uniqueness, consider two solutions $u,v$ of the Dirichlet problem with data $f$. Then, $u-v\in W_0^{1,2}(\Omega)$ solves $Lu=0$, therefore the maximum principle (theorem \ref{MaximumPrinciple}) shows that $u-v=0$ in $\Omega$, hence $u=v$.

For existence, consider a sequence $(f_n)$ of Lipschitz functions which converge to $f$ in $(C(\partial\Omega),\|\cdot\|_{\infty})$. Let also $u_n\in W^{1,2}(\Omega)\cap C(\overline{\Omega})$ be the weak solutions with trace $f_n$ to the equation, whose existence is guaranteed by proposition \ref{WeakSolvabilityForLipschitz}. Then, the maximum principle (theorem \ref{MaximumPrinciple}) shows that
\[
\|u_n-u_m\|_{L^{\infty}(\Omega)}\leq\|f_n-f_m\|_{L^{\infty}(\partial\Omega)}.
\]
The supremum is the usual in this case, since the $u_n$ are continuous, and the $f_n$ are Lipschitz functions. This shows that the sequence $(u_n)$ is Cauchy in $(C(\Omega),\|\cdot\|_{\infty})$, therefore a convergent subsequence of $(u_n)$, still denoted by $(u_n)$, converges uniformly to some $u\in C(\overline{\Omega})$. Then, $u=f$ on $\partial\Omega$.

To show that $u\in W_{\loc}^{1,2}(\Omega)$, let $U\subseteq\Omega$ be compactly supported. Cover $U$ by $N$ balls $B_1,\cdots B_N$ of radius $\delta$, where $2\delta$ is the distance from $K$ to $\partial\Omega$. Then, Cacciopoli's inequality shows that $(u_n)$ is uniformly bounded in $W^{1,2}(B_k)$, so it has a weakly convergent subsequence in $W^{1,2}(B_k)$. By repeating this process for all of the $B_k$, we have that a subsequence of $(u_n)$, still denoted by $(u_n)$, converges to some $v\in W^{1,2}(U)$ weakly in $W^{1,2}(U)$ and strongly in $L^2(U)$. Hence, a further subsequence converges to $v$ almost everywhere in $U$. Since a further subsequence converges uniformly to $u$, this shows that $u\in W^{1,2}(U)$, therefore $u\in W_{\loc}^{1,2}(\Omega)$.

We now show that $u$ is a solution in $\Omega$: let $\phi\in C_c^{\infty}(\Omega)$ and $K={\rm supp}\phi$. Consider also a subsequence $(u_n)$ as above, which converges to $u$ weakly in $W^{1,2}(K)$. Then,
\[
\int_{\Omega}A\nabla u\nabla\phi+b\nabla u\cdot\phi=\lim_{n\to\infty}\int_{\Omega}A\nabla u_n\nabla\phi+b\nabla u_n\cdot\phi=0,
\]
since the $u_n$ are solutions to the equation. Hence $u$ is a solution in $\Omega$.
\end{proof}

\subsection{Construction of harmonic measure}
Consider a Lipschitz domain $\Omega$, let $A\in M_{\lambda}(\Omega)$, and $b\in L^{\infty}(\Omega)$. For a fixed $x_0\in\Omega$, we consider the functional 
\[
T_{x_0}:C(\partial\Omega)\to\mathbb R,\quad T_{x_0}f=u(x_0),
\]
where $u$ is the continuous weak solution of $Lu=0$ in $\Omega$ with boundary data $f$, which exists from  theorem \ref{WeakSolvability}. The maximum principle shows that $T_{x_0}$ is a positive linear functional on $C(\partial\Omega)$, therefore there exists a positive Borel measure $\omega^{x_0}$ such that
\[
u(x_0)=\int_{\partial\Omega}f(q)\,d\omega^{x_0}(q).
\]
The choice $f\equiv 1$ also shows that $\omega^{x_0}$ is a probability measure on $\partial\Omega$. $\omega^{x_0}$ will be called the \emph{harmonic measure} for the equation $Lu=0$ in $\Omega$, centered at $x_0$.

Two basic properties of harmonic measure are the following.

\begin{prop}\label{HarmonicMeasureProperties}
\begin{enumerate}[i)]
\item If $x_1, x_2\in\Omega$, then $\omega^{x_1}<<\omega^{x_2}$.
\item If $E\subseteq\partial\Omega$ is a Borel set, then $u(x)=\omega^x(E)$ is a solution in $\Omega$, with boundary values $\chi_E$ on $\partial\Omega$, in the sense of $W^{1,2}(\Omega)$ (as in the definition in section $8.1$ in \cite{Gilbarg}).
\end{enumerate}
\end{prop}
\begin{proof}
For the first part, suppose that $E\subseteq\partial\Omega$ with $\omega^{x_2}(E)=0$. From regularity of the harmonic measure, there exists a sequence of open sets $(E_n)$ in $\partial\Omega$ that contain $E$, such that $\omega^{x_2}(E_n)\leq 1/n$.

Let now $f\in C(\partial\Omega)$ be a nonnegative function which is supported in $E_n$ with $f\equiv 1$ on $E$, and let $u$ be the solution of $Lu=0$ with data $f$. Then $u\geq 0$ from the maximum principle, so $u+\frac{1}{n}>0$ in $\Omega$. Therefore, Harnack's inequality (proposition \ref{Harnack}) shows that, for some $C=C_{x_1,x_2}$,
\begin{align*}
\omega^{x_1}(E)&=\int_Ef\,d\omega^{x_1}\leq\int_{\partial\Omega}f\,d\omega^{x_1}=u(x_1)=u(x_1)+\frac{1}{n}-\frac{1}{n}\\
&\leq C\left(u(x_2)+\frac{1}{n}\right)-\frac{1}{n}=C\int_{\partial\Omega}f\,d\omega^{x_2}+\frac{C-1}{n}\\
&\leq C\int_{E_n}f\,d\omega^{x_2}+\frac{C-1}{n}=C\omega^{x_2}(E_n)+\frac{C-1}{n}\\
&\leq\frac{C}{n}+\frac{C-1}{n}.
\end{align*}
Letting $n\to\infty$ shows then that $\omega^{x_1}(E)=0$.

For the second part, write $E=\bigcap_{j\in\mathbb N}U_j\cup N$, where $\omega^x(N)=0$ and $(U_j)$ is a decreasing sequence of open subsets of $\partial\Omega$; it is then enough to show that $\omega^x(U)$ is a solution, for all $U\subseteq\partial\Omega$ which are open. For this purpose, let $K_i\subseteq\partial\Omega$ be an increasing sequence of compact sets, with $\bigcup_{i\in\mathbb N}K_i=U$. Let also $g_i$ be a continuous function which satisfies $\chi_{K_i}\leq g_i\leq\chi_U$, and let $u_i$ be the solution to Dirichlet's problem, with data $g_i$, which exists from theorem \ref{WeakSolvability}. From the maximum principle (theorem \ref{MaximumPrinciple}), $0\leq u_i\leq 1$ throughout $\Omega$. Therefore, from compactness of solutions (proposition \ref{Equicontinuity}), there exists a subsequence $u_{i_k}$ which converges to a solution $u_0$ in $\Omega$, uniformly in compact subsets of $\Omega$. But, for $x\in\Omega$,
\[
u_{i_k}(x)=\int_{\partial\Omega}g_{i_k}(q)d\omega^x(q)\xrightarrow[k\to\infty]{}\int_{\partial\Omega}\chi_U(q)d\omega^x(q)=\omega^x(U),
\]
from the dominated convergence theorem. Hence $\omega^x(U)$ is a solution in $\Omega$. The previous convergence, as well as the convergence of the boundary values $(g_{i_k})$ to $\chi_U$ also shows that $\omega^x(U)$ is equal to $\chi_U$ in the sense of $W^{1,2}(\Omega)$ on $\partial\Omega$, which completes the proof.
\end{proof}

Of particular importance is the following representation formula for the harmonic measure, which holds in smooth domains.

\begin{prop}\label{HarmonicRepresentation}
Suppose that $\Omega$ is smooth, and $A\in M_{\lambda,\mu}(\Omega)$, $b\in\Lip(\Omega)$. Then
\[
d\omega^{x_0}(q)=-\partial_{\nu_{A^t}}^qG(x_0,q)\,d\sigma(q)\,\,\forall x_0\in\Omega,
\]
where $\sigma$ is the surface measure on $\partial\Omega$, $G$ is Green's function for the equation
\[
Lu=-\dive(A\nabla u)+b\nabla u=0
\]
in $\Omega$, and $\nu_{A^t}$ denotes the conormal derivative associated with $L^t$.
\end{prop}
\begin{proof}
Set $G_{x_0}^t(x)=G^t(x,x_0)$, and note that then $G_{x_0}^t$ is Green's function for the adjoint equation $L^tu=0$ in $\Omega$, with pole at $x_0$. Let $f\in{\rm Lip}(\partial\Omega)$, and consider $F\in{\rm Lip}(\mathbb R^d)$ which is a Lipschitz extension of $f$, with $F\equiv 0$ and $\nabla F\equiv 0$ in a neighborhood $B_0=B_{\e_0}(x_0)$ of $x_0$. Then $u-F\in W_0^{1,2}(\Omega)$, and since $\nabla u$ is bounded close $x_0$, the defining property of Green's function shows that
\[
\int_{\Omega}A^t\nabla G_{x_0}^t\nabla(u-F)+b(x)\nabla(u-F)\cdot G_{x_0}^t\,dx=u(x_0)-F(x_0)=u(x_0),
\]
which implies that
\[
\int_{\Omega}A\nabla(u-F)\nabla G_{x_0}^t+b(x)\nabla(u-F)\cdot G_{x_0}^t\,dx=u(x_0).
\]
Since now $u$ is a solution of $Lu=0$ in $\Omega$, $\nabla u$ is bounded close to $x_0$ and $G_{x_0}^t$ is bounded away from $x_0$, and after approximating $G_{x_0}^t$ with $C_c^{\infty}(\Omega)$ functions, the last identity shows that
\[
\int_{\Omega}A\nabla F\nabla G_{x_0}^t+b\nabla F\cdot G_{x_0}^t\,dx=-u(x_0).
\]
But, $A,b$ and $\Omega$ are smooth, hence theorem $8.12$ in \cite{Gilbarg} shows that $G_{x_0}^t$ is smooth away from $x_0$. Hence, $L^tG_{x_0}^t=0$ pointwise, away from $x_0$. In addition, the support properties of $F$ show that
\begin{align*}
-u(x_0)&=\int_{\Omega\setminus B_0}A^t\nabla G_{x_0}^t\nabla F+b\nabla F\cdot G_{x_0}^t\\
&=\int_{\Omega\setminus B_0}\dive(F\cdot A^t\nabla G_{x_0}^t)-F\dive(A^t\nabla G_{x_0}^t)+\dive(bG_{x_0}^tF)-\dive(bG_{x_0}^t)F\\
&=\int_{\partial(\Omega\setminus B_0)}F\left<A^t(q)\nabla G_{x_0}^t(q),\nu(q)\right>\,d\sigma+\int_{\Omega\setminus B_0}F(-\dive(A^t\nabla G_{x_0}^t)-\dive(bG_{x_0}^t))\\
&=\int_{\partial\Omega}f(q)\cdot\partial_{\nu_{A^t}}^qG^t(q,x_0)\,d\sigma(q),
\end{align*}
from the divergence theorem, the fact that $L^tG_{x_0}^t=0$ pointwise away from $x_0$, and the support properties of $F$. This concludes the proof.
\end{proof}

\subsection{Estimates on harmonic measure}
We now turn to the basic estimates on harmonic measure. Throughout this section we will assume that $\Omega$ is a Lipschitz domain, $A\in M_{\lambda}(\Omega)$, and $b\in L^{\infty}(\Omega)$. In order to show our estimates, we will follow the method that is outlined in \cite{KenigCBMS}.

Consider the number $r_{\Omega}$ that appears in the definition of the Lipschitz character; then, given any point $q\in\partial\Omega$ and $0<r<r_{\Omega}$, the ball $B_{10r}(q)$ lies in a coordinate cylinder. Also, for any point $q\in\partial\Omega$ and $0<r<r_{\Omega}$, there exists a point $A_r(q)\in\Omega$ such that
\[
r\leq|A_r(q)-q|,\quad\delta(A_r(q))\leq c_0r,
\]
with the constant $c_0$ only depending on the Lipschitz constant of $\Omega$. The points $A_r(q)$ will be the analogs of the similar points in the definition of an NTA domain.

The next lemma is a consequence of the Harnack inequality, after applying a Harnack chain argument \cite{KenigCBMS}.

\begin{lemma}\label{ChainInequality}
Let $\Omega$ be a Lipschitz domain, $A\in M_{\lambda}(\Omega)$ and $b\in L^{\infty}(\Omega)$. Let also $x_1,x_2\in\Omega$, with $\delta(x_1),\delta(x_2)\geq\e$ and $|x_1-x_2|<2^k\e$. Then, for every positive solution $u$ to $Lu=0$ or $L^tu=0$ in $\Omega$, we have that
\[
C^{-k}u(x_2)\leq u(x_1)\leq C^k u(x_2),
\]
where $C$ depends on $d,\lambda,\|b\|_{\infty}$ and $\diam(\Omega)$.
\end{lemma}

The main estimates that we will show connect $L$ harmonic measure to Green's function for the equation $Lu=0.$ We will first need the following lemma, originally due to Carleson (lemma 4.4 in \cite{JerisonKenigBoundary}).

\begin{lemma}\label{CarlesonEstimate}
Let $\Omega\subseteq\mathbb R^d$ be a Lipschitz domain, and $A\in M_{\lambda}(\Omega)$, $b\in L^{\infty}(\Omega)$. For any nonnegative solution $u$ to $Lu=0$ or $L^tu=0$ in $\Omega$ which vanishes continuously on $\Delta_{2r}(q)$, and any $r<r_{\Omega}$, we have that
\[
\forall x\in T_r(q),\quad u(x)\leq C u(A_r(q)),
\]
where $C$ depends on $d,\lambda,\|b\|_{\infty},\diam(\Omega)$ and the Lipschitz constant of $\Omega$.
\end{lemma}
\begin{proof}
After normalizing, we can suppose that $u(A_r(q))=1$. From theorem 8.27 in \cite{Gilbarg}, there exists a constant $c_1$, only depending on $d,\lambda,\|b\|_{\infty}$, the Lipschitz constant of $\Omega$ and $\diam(\Omega)$, such that for all $p\in\partial\Omega$ and $s<r_{\Omega}$,
\begin{equation}\label{eq:Sup}
\sup\left\{u(x)|x\in T_{c_1s}(p)\right\}\leq\frac{1}{2}\sup\left\{u(x)|x\in T_s(p)\right\}.
\end{equation}
Now, by lemma \ref{ChainInequality}, there exists a constant $c_2$ such that, if $u(y)>c_2^h$ and $y\in T_r(q)$, then $\delta(y)\leq c_1^{-h}r$. Set $c_3=c_2^h$, where $h=N+3$, and where $N$ is such that $2^N>c_2$.

Suppose now that, for some $y_0\in T_r(q)$, $u(y_0)>c_3=c_2^h$. Then, $\delta(y_0)\leq c_1^{-h}r$. So, if the distance $\delta(y_0)$ is achieved at $p\in\partial\Omega$, 
\[
|q-p|\leq|q-y|+|y-p|<r+c_1^{-h}r<\frac{3}{2}r.
\]
Therefore, from \eqref{eq:Sup}, we obtain that
\[
\sup\{u(x)|x\in T_{c_1^{-h+N}r}(p)\}\geq 2^N\sup\{u(x)|x\in T_{c_1^{-h}r}(p)\}\geq 2^Nu(y_0)\geq c_2^{h+1}.
\]
Therefore, there exists $y_1\in T_{c_1^{-h+N}r}(p)$ such that $u(y_1)\geq c_2^{h+1}$. Let the distance $\delta(y_1)$ be achieved at $q_1\in\partial\Omega$. Inductively, we construct two sequences $(y_n),(q_n)$ such that
\[
u(y_k)\geq c_2^{h+k},\,\,\delta(y_k)=|y_k-q_k|<c_2^{-h-k}r,\,\,y_k\in T_{c_1^{-h-k-N}r}(q_{k-1}).
\]
But,
\[
|y_k-q|\leq |y_k-q_k|+|q_k-y_{k-1}|+|y_{k-1}-q|\leq(c_1^{-h-k}+c_1^{-h-k-n})r+|y_{k-1}-q|,
\]
and, since $|y_0-q|<r$, we obtain that
\[
|y_k-q|\leq r+\sum_{i=1}^k(c_1^{-h-i}+c_1^{-h-i-N})r<2r,
\]
from the choice of $N$. Hence, $y_k\in T_{2r}(q)$, and $y_k$ contains a subsequence that converges to $\Delta_{2r}(q)$, while $u(y_k)$ does not converge to $0$, and this is a contradiction. Hence, for all $y_0\in T_r(q)$, $u(y_0)\leq c_3$.
\end{proof}

We are now in position to prove the estimates that connect harmonic measure with Green's function.

\begin{lemma}\label{BelowForHarmonic}
Let $r<r_{\Omega}$, and $q\in\partial\Omega$. Then there exists a constant $C$ depending on $d,\lambda,\|b\|_{\infty}$ and the Lipschitz constant of $\Omega$, such that
\[
\forall x\in B_{r/2}(A_r(q)),\quad\omega^x(\Delta_r(q))\geq C.
\]
\end{lemma}
\begin{proof}
Let $\phi$ be a cutoff function which is equal to $1$ in $T_{r/2}(q)$, it is supported in $T_r(q)$, and also satisfies that $0\leq\phi\leq 1$. Let $u$ be the solution to the classical Dirichlet problem for $L$, with data $\phi$. Since $\phi\leq\chi_{\Delta_r(q)}$, and $\omega^x(\Delta_r(q))$ is a solution with data $\chi_{\Delta_r(q)}$ on $\partial\Omega$ from proposition \ref{HarmonicMeasureProperties}, the maximum principle shows that
\[
x\in\Omega\Rightarrow\omega^x(\Delta_r(q))\geq u(x).
\]
Now, set $v=1-u$, then $v$ is a solution in $\Omega$ that vanishes on $\Delta_r(q)$. From the maximum principle, $0\leq v\leq 1$ in $\Omega$, therefore, from theorem 8.27 in \cite{Gilbarg}, we obtain that
\[
v(x)\leq C\left(\frac{|x-q|}{r}\right)^{\beta}\,\,\forall x\in T_{r/2}(q).
\]
Now, since $|A_{r/2C_0}(q)-q|\leq C_0\frac{r}{2C_0}=r/2$, we can apply this inequality to $A_{r/2C_0}(q)$, to obtain that $v(A_{r/2C_0}(q))\leq C2^{-\beta}$, therefore
\[
\omega^{A_{r/2C_0}(q)}(\Delta_r(q))\geq 1-C2^{-\beta}.
\]
Since now $\omega^x(\Delta_r(q))$ is a positive solution in $\Omega$, lemma \ref{ChainInequality} completes the proof.
\end{proof}

\begin{lemma}\label{GFromAbove}
Let $r<r_{\Omega}$. Then, for all $q\in\partial\Omega$,
\[
\forall x\in\Omega\setminus B_{r/2}(A_r(q)),\quad r^{d-2}G(x,A_r(q))\leq C\omega^x(\Delta_r(q)),
\]
where $C$ depends on $d,\lambda, \|b\|_{\infty}$ and the Lipschitz constant of $\Omega$.
\end{lemma}
\begin{proof}
If $x$ is on the boundary of $B_{r/2}(A_r(q))$, the pointwise estimates on Green's function (in theorem \ref{GoodGreenFunctionEstimates}) we obtain that
\[
r^{d-2}G(x,A_r(q))\leq Cr^{d-2}|x-A_r(q)|^{2-d}\leq C,\]
and, from lemma \ref{BelowForHarmonic},
\[
\omega^x(\Delta_r(q))\geq C,
\]
which implies that $r^{d-2}G(x,A_r(q))\leq C\omega^x(\Delta_r(q))$ for $x$ on the boundary of $B_{c_0r}(A_r(q))$. On the other hand, $G$ vanishes on $\partial\Omega$, so the same estimate holds for $x\in\partial\Omega$. Since $r^{d-2}G(x,A_r(q))-C\omega^x(\Delta_r(q))$ is a solution of the equation $Lu=0$ in $\Omega\setminus B_{r/2}(A_r(q))$, the estimate follows from the maximum principle.
\end{proof}
For the reverse inequality, we show the next lemma.

\begin{lemma}\label{GFromBelow}
If $\Omega$ is a Lipschitz domain, then there exists a constant $C$ which depends on $d,\lambda,\|b\|_{\infty},\diam(\Omega)$ and the Lipschitz constant of $\Omega$, such that, if $q\in\partial\Omega$ and $r<r_{\Omega}$, then, for all $x\notin B_{2r}(q)$,
\[
\omega^x(\Delta_r(q))\leq C r^{d-2}G(x,A_r(q)).
\]
\end{lemma}
\begin{proof}
Fix $1<\alpha<\beta<\gamma<2$. Let $\psi$ be a smooth cutoff function, which is supported in $T_{\beta r}(q)$, is equal to $1$ in $T_{\alpha r}(q)$, and satisfies the bounds $0\leq\psi\leq 1$, and $|\nabla\psi|\leq C/r$. Consider now the classical solution $u$ to $Lu=0$, with boundary data $\psi$. Since then $\chi_{\Delta_r(q)}\leq u$ on $\partial\Omega$, from the maximum principle we obtain
\[
\omega^y(\Delta_r(q))\leq u(y).
\]
Now, $u-\psi\in W_0^{1,2}(\Omega)$, and, from the fact that $u$ is a solution we get that
\[
u(y)=-\int_{\Omega}A(x)\nabla\psi(x)\nabla_xG^t(x,y)+b(x)\nabla\psi(x)G^t(x,y)\,dx,\,\,\,\,{\rm a.e.}\,\,y\in\Omega\setminus B_{2r}(q).
\]
Therefore, for any $y$ satisfying this equality, we obtain that
\begin{align*}
\omega^y(\Delta_r(q))&\leq u(y)\leq\frac{C}{r}\int_{T_{\beta r}(q)}|\nabla_xG^t(x,y)|\,dx+\int_{T_{\beta r}(q)}|b(x)\nabla\psi(x)G^t(x,y)|\,dx\\
&\leq \frac{C}{r}r^{d/2}\left(\int_{T_{\beta r}(q)}|\nabla_xG^t(x,y)|\,dx\right)^{1/2}+Cr^{d-1}\|b\|_{\infty}G^t(A_r(q),y)\\
&\leq Cr^{d/2-1}\left(\frac{C}{r^2}\int_{T_{\gamma r}(q)}|G^t(x,y)|^2\,dx\right)^{1/2}+Cr^{d-1}\|b\|_{\infty}G^t(A_r(q),y)\\
&\leq Cr^{d-2}G^t(A_r(q),y)
\end{align*}
for some constant $C$ depending on $d,\lambda,\|b\|_{\infty},\diam(\Omega)$ and the Lipschitz constant of $\Omega$, where we also used Cacciopoli's inequality and the estimate in lemma \ref{CarlesonEstimate}. Then, proposition \ref{SymmetryWithAdjoint} shows that
\[
\omega^y(\Delta_r(q))\leq Cr^{d-2}G(y,A_r(q))
\]
for almost all $y\in\Omega\setminus B_{2r}(q)$. Since the functions involved are continuous, we obtain the inequality for all $y\in\Omega\setminus B_{2r}(q)$.
\end{proof}

As a corollary of lemmas \ref{GFromAbove} and \ref{GFromBelow}, we obtain the following comparison.

\begin{prop}\label{GreenRelation}
If $\Omega$ is a Lipschitz domain, then for all $0<r<r_{\Omega}$ and $q\in\partial\Omega$,
\[
\omega^x(\Delta_r(q))\simeq r^{d-2}G(x,A_r(q)),\,\,\forall x\in\Omega\setminus B_{2r}(q),
\]
with $C$ being a good constant. In particular, $\omega^x$ is a doubling measure on $\partial\Omega$; that is, for every $x\in\Omega$, there exists $C=C_x>0$ which also depends on $d,\lambda,\|b\|_{\infty},\diam(\Omega)$ and the Lipschitz constant of $\Omega$ such that, for every $q\in\partial\Omega$ and $r>0$,
\[
\omega^x(\Delta_{2r}(q))\leq C\omega^x(\Delta_r(q)).
\]
\end{prop}

To obtain the fact that $\omega^x$ is doubling, we use the Harnack inequality for $r>0$ sufficiently small, such that $x\in\Omega\setminus B_{2r}(q)$ for every $q\in\partial\Omega$.

The previous connection between the harmonic measure and Green's function leads to the next lemma.
\begin{lemma}[Comparison Principle]\label{ComparisonPrinciple}
Let $u,v$ be positive solutions of $Lu=0$ in $\Omega$, which vanish continuously on $\Delta_{2r}(q)$, and $r<r_{\Omega}$. Then,
\[
\forall x\in T_r(q),\quad c_1^{-1}\frac{u(A_r(q))}{v(A_r(q))}\leq\frac{u(x)}{v(x)}\leq c_1\frac{u(A_r(q))}{v(A_r(q))}.
\]
\end{lemma}
\begin{proof}
For $r<r_{\Omega}$, consider the Lipschitz domain $T=T_r(q)$. Let $x\in T$, and consider the following partition of $\partial T\setminus\partial\Omega$: set
\begin{align*}
L_1=\{x\in\partial T\setminus\partial\Omega\bm{|}\delta(x)>cr\}\\
L_2=\{x\in\partial T\setminus\partial\Omega\bm{|}\delta(x)\leq cr\},
\end{align*}
where $\delta(x)$ denotes the distance from $x$ to $\partial\Omega$. Denote the harmonic measure for $T$ with respect to $x$ by $\omega^x_T$. First, $L_2$ contains a surface ball of radius comparable to $r$, therefore, from the doubling property of $\omega^x_T$ (proposition \ref{GreenRelation}), we obtain that
\[
\omega_T^x(\partial\Omega\setminus\partial T)=\omega^x_T(L_1\cup L_2)\leq C\omega^x_T(L_2)
\]
where $C$ is a good constant. Now, by lemma \ref{CarlesonEstimate}, since $u$ vanishes on $\Delta_{cr}(q)$, we obtain that $u(x)\leq Cu(A_r(q))$. Therefore, the function
\[
\frac{u(x)}{u(A_r(q))}-M\omega_T^x(\partial\Omega\setminus\partial T)
\]
is a solution in $T$, which is $0$ on $\partial\Omega$ (since $u$ vanishes there), and is nonpositive on $\partial\Omega\setminus\partial T$ (since the harmonic measure is equal to $1$ there). Therefore, we obtain that
\[
u(x)\leq C\omega_T^x(\partial\Omega\setminus\partial T)u(A_r(q))\,\,\,\forall x\in T.
\]
Also, if $x\in L_2$, since $L_2$ is about $r$-far from $\partial\Omega$, we obtain that $v(x)\geq Cv(A_r(q))$. Similarly, we obtain that
\[
v(x)\geq C\omega_T^x(L_2)v(A_r(q))\,\,\forall x\in T,
\]
which shows the second statement. By interchanging the roles of $u$ and $v$, we also obtain the first statement.
\end{proof}

\subsection{Maximal functions}
Fix $x_0\in\Omega$, and let $\omega=\omega^{x_0}$. For a Borel measure $\nu$ on $\partial\Omega$, we define the Hardy-Littlewood maximal function of $\nu$ with respect to $\omega$,
\[
M_{\omega}\nu(q)=\sup_{\Delta\ni q}\frac{1}{\omega(\Delta)}\int_{\Delta}d|\nu|=\sup_{\Delta\ni q}\frac{|\nu|(\Delta)}{\omega(\Delta)}.
\]
In particular, for a function $f\in L^1(\partial\Omega)$, we define
\[
M_{\omega}f(q)=\sup_{\Delta\ni q}\frac{1}{\omega(\Delta)}\int_{\Delta}|f|d\omega.
\]
From the doubling property of $\omega$, the usual estimates for the Hardy-Littlewood maximal function hold; that is, $M_{\omega}$ is weakly $(1,1)$ bounded, and strongly $(p,p)$ bounded, for $p\in(1,\infty]$ (see for example \cite{SteinHarmonic}).

Our goal is to express any solution $u$ of the equation $Lu=0$ in $\Omega$ as an integral of its boundary values, integrated with respect to the measure $\omega$. For this purpose, we define the kernel function $K(x,q)$, which is the Radon-Nikodym derivative of $\omega^x$ with repect to $\omega$ at the point $q\in\partial\Omega$; or, equivalently,
\begin{equation}\label{eq:KFormula}
K(x,q)=\lim_{r\to 0}\frac{\omega^x(\Delta_r(q))}{\omega(\Delta_r(q))}.
\end{equation}
This kernel function exists, and also $K(x,\cdot)\in L^1(\partial\Omega,d\omega)$, since $\omega^x$ is absolutely continuous with respect to $\omega$, which follows from proposition \ref{HarmonicMeasureProperties}. Then, for all $f\in L^1(\partial\Omega,d\omega^x)$,
\[
u(x)=\int_{\partial\Omega}f(q)d\omega^x(q)=\int_{\partial\Omega}f(q)K(x,q)d\omega(q).
\]

To obtain the relation between the maximal functions defined above and the nontangential maximal function, we will need estimates on the kernel $K$. For this purpose, we first prove the Carleson-Hunt-Wheeden lemma (lemma 4.11 in \cite{JerisonKenigBoundary}), which will follow from the comparison principle.

\begin{lemma}\label{OmegaComparison}
Suppose that $r<r_{\Omega}$, and let $\Delta=\Delta(q_0,r)$, $\Delta'=\Delta(q,s)\subseteq\Delta(q_0,r/2)$. Then, there exists $C=C(M)$ such that, for all $x\in\Omega\setminus B(q_0,2r)$, we have that
\[
C\omega^{A_r(q_0)}(\Delta')\leq\frac{\omega^x(\Delta')}{\omega^x(\Delta)}\leq C^{-1}\omega^{A_r(q_0)}(\Delta').
\]
\end{lemma}
\begin{proof}
Note that, from proposition \ref{GreenRelation}, we have to show that (since $x$ is far from $q_0$ and $q$),
\[
G(A_r(q_0),A_s(q))\simeq r^{2-d}\frac{G(x,A_s(q))}{G(x,A_r(q_0))},
\]
or, equivalently,
\[
\frac{G(A_r(q_0),A_s(q))}{G(x,A_s(q))}\simeq\frac{r^{2-d}}{G(x,A_r(q_0))}.
\]
For this purpose, set $u(y)=G(A_r(q_0),y)$ and $v(y)=G(x,y)$. Since $|A_r(q_0)-q_0|>c_0r$ and $|x-q_0|>2r$, this implies that $u$ and $v$ are solutions to $Lu=0$ in $T_{c_0r}(q_0)$, which also vanish on $\Delta_{c_0r}(q_0)$. Therefore, from the comparison principle, we obtain that
\[
\frac{G(A_r(q_0),A_s(q))}{G(x,A_s(q))}=\frac{u(A_s(q))}{v(A_s(q)}\simeq\frac{u(A_r(q_0))}{v(A_r(q_0))}=\frac{G(A_r(q_0),A_{c_0r}(q_0))}{G(x,A_{c_0r}(q))}.
\]
First, the distance from $A_r(q_0)$ to $A_{c_0r}(q_0)$ is comparable to $r$, and is also comparable to the distances of $A_r(q_0)$ and $A_{c_0r}(q_0)$ to $\partial\Omega$, therefore, from the pointwise bounds on Green's function, we obtain that
\[
G(A_r(q_0),A_{c_0r}(q_0))\simeq |A_r(q_0)-A_{c_0r}(q_0)|^{2-d}\simeq r^{2-d},
\]
therefore it only suffices to show that
\[
G(x,A_{c_0r}(q_0))\simeq G(x,A_r(q_0))\Leftrightarrow v(A_{c_0r}(q_0))\simeq v(A_r(q_0)).
\]
But, this is a corollary of lemma \ref{ChainInequality}, since $v$ is also a solution in $\Omega\setminus B_r(x)$.
\end{proof}
The previous lemma leads to the pointwise bound of the kernel $K$.
\begin{lemma}\label{BoundsOnK}
Let $q_0\in\partial\Omega$, and $A=A_r(q_0)$. Define also $\Delta_j=\Delta(q_0,2^jr)$, and $R_j=\Delta_j\setminus\Delta_{j-1}$. Then,
\[
\sup\{K(A,q)|q\in R_j\}\leq\frac{C2^{-\alpha j}}{\omega(\Delta_j)},
\]
where $C=C(\Omega)$, and $\alpha=\alpha(\Omega)$.
\end{lemma}
\begin{proof}
Suppose first that $j$ is such that $2^{j+1}r<\min\{2r_0,|x-q_0|\}=M_0$.

Consider $q\in R_j$, and let $s>0$. Then, if $s$ is sufficiently small, we have that $\Delta_s(q)\subseteq R_j\subseteq\Delta_j$. Therefore, from lemma \ref{OmegaComparison}, if $A_j=A_{2^jr}(q_0)$, we obtain that
\[
\frac{\omega^x(\Delta_s(q))}{\omega^x(\Delta_j)}\simeq\omega^{A_j}(\Delta_s(q)).
\]
Now, since $\Delta_s(q)\cap\Delta_j=\emptyset$, the function $u(x)=\omega^x(\Delta_s(q))$ is a positive solution in $T_{2^{j-1}r}(q_0)$ which vanishes continuously on $\Delta_j$, hence, from proposition \ref{HolderOnTheBoundary} and lemma \ref{CarlesonEstimate},
\[
\omega^A(\Delta_s(q))=u(A)\leq Cu(A_{j-1})\left(\frac{|A_{j-1}-A|}{2^{j-1}r}\right)^{\beta}\leq C\frac{\omega^x(\Delta_s(q))}{\omega^x(\Delta_j)}2^{-\beta j},
\]
from the similarity estimate above. This shows that, for sufficiently small $s$,
\[
\frac{\omega^A(\Delta_s(q))}{\omega^x(\Delta_s(q))}\leq\frac{C2^{-\beta j}}{\omega(\Delta_j)},
\]
and the result follows by taking the limit as $s\to 0$.

For the rest of the $j$, since we have finitely many $j$ such that $R_j\neq\emptyset$ and $2^{j+1}r\geq M_0$, it is enough to show that
\[
\sup\{K(A,q)|q\in\partial\Omega\setminus\Delta_{M_0/2}(q_0)\}\leq C.
\]
For this purpose, consider a ball $\Delta_s(q)$, with $\Delta_s(q)\subseteq\Omega\setminus\Delta_{M_0/2}(q_0)$. Then, similarly to the above,
\[
\omega^A(\Delta_s(q))=u(A)\leq Cu(A_{M_0/4}(q_0))\leq C u(x)=\omega^x(\Delta_s(q)),
\]
where we also used Harnack's inequality. The result now follows again by taking the limit as $s\to 0$.
\end{proof}

The last lemma leads to the following theorem, which relates the maximal functions $M_{\omega}f$ and $u^*$.

\begin{thm}\label{MaximalBounding}
Let $\nu$ be a finite Borel measure on $\partial\Omega$ and set
\[
u(x)=\int_{\partial\Omega}K(x,q)d\nu(q).
\]
Then, $u^*\leq CM_{\omega}\nu$. In addition, if $\nu$ is positive, we have that $M_{\omega}\nu\leq Cu^*$.
\end{thm}
\begin{proof}
Let $p\in\partial\Omega$, and $x\in\Gamma_{\alpha}(q)$. Set $r=|x-p|$, $\Delta_j=\Delta_{2^jr}(p)$, and $R_j=\Delta_j\setminus\Delta_{j-1}$. We then write
\[
u(x)=\int_{\partial\Omega}K(x,q)d\nu(q)=\int_{\Delta_{r/2}(p)}K(x,q)d\nu(q)+\sum_{j=0}^{\infty}\int_{R_j}K(x,q)d\nu(q).
\]
For the summands, note that, for fixed $q\in R_j$, from lemma \ref{BoundsOnK}, we obtain that
\[
\left|\int_{\Delta_j\setminus\Delta_{j-1}}K(x,q)d\nu(q)\right|\leq\frac{C2^{-\beta j}}{\omega(\Delta_j)}\int_{R_j}|d\nu|\leq C\frac{|\nu|(\Delta_j)}{\omega(\Delta_j)}\leq C 2^{-\beta j} M_{\omega}(p),
\]
therefore the sum is dominated by $CM_{\omega}(p)$. For the first integral, note that for $q\in\Delta_r(p)$, lemma \ref{OmegaComparison} and \eqref{eq:KFormula} show that
\[
K(x,q)\simeq\frac{1}{\Delta_r(p)},
\]
which concludes the first claim.

For the second claim, note that if $\nu$ is a positive measure, then for any $r>0$
\[
u(x)\geq\int_{\Delta_r(p)}K(x,q)\,d\nu(q)\geq\frac{C}{\omega(\Delta_r(p))}\int_{\Delta_r(p)}\,d\nu(q),
\]
which completes the proof.
\end{proof}
\section{The Dirichlet problem for $L$}
In this chapter we turn our attention to the Dirichlet problem for the equation $Lu=0$, with boundary data in $L^p(\partial\Omega)$. We will follow the method outlined in \cite{KenigCBMS}, and we will use the estimates in the previous chapter in order to obtain solvability.

\subsection{Formulation, and the weight property}
 
\begin{dfn}
Let $\Omega$ be a Lipschitz domain, and $p\in(1,\infty)$. We say that $D_p$ for $Lu=0$ is solvable in $\Omega$, if there exists $C>0$ such that, for every $f\in C(\partial\Omega)$, the solution $u:\Omega\to\mathbb R$ of $Lu=0$ in $\Omega$ with boundary data $f$ satisfies the estimate $\|u^*\|_{L^p(\partial\Omega)}\leq C\|f\|_{L^p(\partial\Omega)}$.
\end{dfn}

Alternatively, we could have defined solvability for $D_p$ such that the nontangential boundary values of the solution lie in $L^p(\partial\Omega)$, and the bound on the nontangential maximal function holds. However, our definition above will imply this property after a density argument, as we will show later (proposition \ref{DirichletSolvability}).

In the next theorem (which is analogous to theorem 1.7.3 in \cite{KenigCBMS}) we show that the Dirichlet problem is solvable if and only if the harmonic measure kernel satisfies a weight property. For this purpose, recall the space $B_p$, defined in definition \ref{Bp}.

\begin{thm}\label{DirichletToWeightsRelation}
Let $\Omega\subseteq\mathbb R^d$ be a Lipschitz domain, $A\in M_{\lambda,\mu}(\Omega)$, and $b\in L^{\infty}(\Omega)$. Suppose also that the classical Dirichlet problem is solvable in $\Omega$, and that $\omega<<\sigma$. Then $D_p$ is solvable in $\Omega$ if and only if the kernel $k=\frac{d\omega}{d\sigma}$ is in $B_{p'}(\partial\Omega)$. In this case, the constant in $D_p$ is comparable to the $B_{p'}$ norm of $k$, with the comparability constants depending only on $d,p,\lambda,\mu,\|b\|_{\infty}$, the Lipschitz constant of $\Omega$, and $\diam(\Omega)$.
\end{thm}
\begin{proof}
Suppose that $k\in B_{p'}(\partial\Omega)$, and let $f\in C(\partial\Omega)$, and $u$ be the solution of $Lu=0$ with boundary data $f$. From the self improving property of $B_{p'}$ functions (proposition \ref{BpProperty}), there exists a constant $\e>0$, depending only on $d,p'$ and the $B_p'$ constant of $k$, such that $k\in B_{p'+\e}$. Therefore, if $\Delta$ is a surface ball and $s<p$ is the conjugate exponent to $p'+\e$, 
\begin{align*}
\frac{1}{\omega(\Delta)}\int_{\Delta}|f|\,d\omega&=\frac{1}{\omega(\Delta)}\int_{\Delta}|f|k\,d\sigma\leq\frac{1}{\omega(\Delta)}\left(\int_{\Delta}|f|^s\,d\sigma\right)^{1/s}\left(\int_{\Delta}k^{p'+\e}\,d\sigma\right)^{\frac{1}{p'+\e}}\\
&\leq C\left(\frac{1}{\sigma(\Delta)}\int_{\Delta}|f|^s\,d\sigma\right)^{1/s}\leq C(M_{\sigma}(|f|^s))^{1/s}.
\end{align*}
Therefore, $M_{\omega}f\leq C(M_{\sigma}(|f|^s))^{1/s}$. So, from the Hardy-Littlewood maximal theorem, since $p/s>1$, we obtain
\begin{align*}
\int_{\partial\Omega}|M_{\omega}f|^p\,d\sigma & \leq C\int_{\partial\Omega}\left(M_{\sigma}(|f|^s)\right)^{p/s}\,d\sigma\\
&\leq C\int_{\partial\Omega}\left(|f|^s\right)^{p/s}\,d\sigma=C\int_{\partial\Omega}|f|^p\,d\sigma.
\end{align*}
But, from theorem \ref{MaximalBounding}, applied to $\nu=f\omega$, we have that $u^*$ is bounded pointwise by $CM_{\omega}f$, therefore
\[
\int_{\Omega}|u^*|^p\,d\sigma\leq C\int_{\partial\Omega}|f|^p\,d\sigma,
\]
where $C$ depends on $s$ and the constant appearing in the $B_{p'}$ property of $k$. Therefore, $D_p$ is solvable in $\Omega$.

Conversely, suppose that $D_p$ is solvable in $\Omega$. Let $\Delta\subseteq\partial\Omega$ be a surface ball, and consider a positive and continuous function $f:\partial\Omega\to\mathbb R$, which is supported on $\Delta$, with $\|f\|_{L^p(\Delta)}\leq 1$. Then, for any $q\in\Delta$,
\[
\frac{1}{\omega(\Delta)}\int_{\Delta}kf\,d\sigma=\frac{1}{\omega(\Delta)}\int_{\Delta}f\,d\omega\leq M_{\omega} f(q).
\]
Raise this relation to the $p$ power, and integrate it for $q\in\Delta$: then, if $u$ is the solution to $Lu=0$ in $\Omega$ with boundary data $f$,
\begin{align*}
\left(\frac{1}{\omega(\Delta)}\int_{\Delta}kf\,d\sigma\right)^p&\leq\frac{1}{|\Delta|}\int_{\Delta}|M_{\omega} f(q)|^p\,d\sigma(q)\\
&\leq\frac{C}{|\Delta|}\int_{\partial\Omega}|u^*|^p\,d\sigma\leq\frac{C}{|\Delta|}\int_{\partial\Omega}|f|^p\,d\sigma\leq\frac{1}{|\Delta|},
\end{align*}
from theorem \ref{MaximalBounding} (since $f\geq 0$, so the measure $\nu=f\omega$ is positive) and the fact that $D_p$ is solvable in $\Omega$. Therefore,
\[
\int_{\Delta}kf\,d\sigma\leq\frac{C\omega(\Delta)}{|\Delta|^{1/p}},
\]
hence, by duality,
\[
\left(\int_{\Delta}k^{p'}\,d\sigma\right)^{1/p'}\leq\frac{C\omega(\Delta)}{|\Delta|^{1/p'}}\Rightarrow\left(\frac{1}{|\Delta|}\int_{\Delta}k^{p'}\,d\sigma\right)^{1/p'}\leq\frac{C\omega(\Delta)}{|\Delta|}=\frac{C}{|\Delta|}\int_{\Delta}k\,d\sigma,
\]
therefore $k\in B_{p'}$, with constant $C$ only depending on the constant appearing in $D_p$.
\end{proof}

The monotonicity property of $B_{p'}$ weights leads to the following corollary.

\begin{cor}\label{Range}
If $D_p$ is solvable in $\Omega$ for some $p\in(1,\infty)$, then there exists $\e>0$ such that $D_q$ is solvable in $\Omega$ for all $q\in(p-\e,\infty)$.
\end{cor}

\subsection{Solvability of the Dirichlet problem}
In this section, we turn to solvability of the Dirichlet problem, for symmetric matrices $A$. The results above show that, in order to show solvability for the range $(1,2+\e)$, it is enough to show that the harmonic measure kernel is in $B_2(\partial\Omega)$. This will be done first in smooth domains, and we will pass to Lipschitz domains using an approximation argument.

For the next lemma, we will assume that the ball $B$ that appears in lemma \ref{InnerRadius} is centered at $0$.

\begin{lemma}\label{B2Lemma}
Let $\Omega\subseteq\mathbb R^d$ be a bounded $C^{\infty}$ domain, and suppose that $A$ is smooth, $A\in M_{\lambda,\mu}^s(\Omega)$, and $b\in \Lip(\Omega)$. Then $k(q)=\frac{d\omega^0}{d\sigma}\in L^2(\partial\Omega)$, and, for all $r<r_{\Omega},s_{\Omega}$ and $q\in\partial\Omega$,
\[
\left(\frac{1}{\sigma(\Delta_r(q))}\int_{\Delta_r(q)}k^2\,d\sigma\right)^{\frac{1}{2}}\leq\frac{C}{\sigma(\Delta_r(q))}\int_{\Delta_r(q)}k\,d\sigma,
\]
where $C$ is a good constant that also depends on the $\Dr_{p_d}$-norm of $b$, and $s_{\Omega}$ appears in lemma \ref{InnerRadius}. (Here, $p_d=2$ for $d=3$, and $p_d=d/2$ for $d\geq 4$.)
\end{lemma}
\begin{proof}
Set $G^t(x)=G^t(x,0)$ to be Green's function for the equation $L^tu=0$, with pole at $0$. Then, from proposition \ref{HarmonicRepresentation}, we obtain that for all $q\in\partial\Omega$,
\[
d\omega(q)=-\partial_{\nu}G^t(q)\,d\sigma(q)\Rightarrow k(q)=-\partial_{\nu}G^t(q).
\]
Consider now $r>0$, with $r<r_{\Omega},s_{\Omega}$. Consider also the ball $B$ that appears in lemma \ref{InnerRadius}, which we assume it is centered st $0$. Then, $T_{2r}(q)\cap B=\emptyset$, therefore $G^t$ is a solution of $L^tG^t=0$ in $T_{2r}(q)$. From theorem 8.12 in \cite{Gilbarg}, $G^t\in W^{2,2}(\Omega\setminus B)$, hence the Rellich estimate for the adjoint equation (proposition \ref{LocalRellichForAdjoint}), is applicable. Note that $G^t$ vanishes on $\partial\Omega$, therefore $\nabla_TG^t\equiv 0$; hence, using Carleson's estimate (lemma \ref{CarlesonEstimate}) we obtain that
\begin{align*}
\int_{\Delta_r(q)}|\partial_{\nu}G^t|^2\,d\sigma&\leq 
\frac{C}{r}\int_{T_{2r}(q)}|\nabla G^t|^2+C\int_{T_{2r}(q)}|\dive b||G^t\cdot\nabla G^t|\\
&\leq\frac{C}{r}\int_{T_{2r}(q)}|\nabla G^t|^2+CG^t(A_r(q))\int_{T_{2r}(q)}|\dive b||\nabla G^t|\\
&\leq\frac{C}{r}\int_{T_{2r}(q)}|\nabla G^t|^2+CG^t(A_r(q))\left(\int_{T_{2r}(q)}|\dive b|^2\right)^{\frac{1}{2}}\left(\int_{T_{2r}(q)}|\nabla G^t|^2\right)^{\frac{1}{2}}
\end{align*}
since $G^t$ vanishes on $\partial\Omega$.  Now, from the boundary Cacciopoli inequality for the adjoint equation (lemma \ref{CacciopoliForAdjoint}) and lemma \ref{CarlesonEstimate}, we obtain that
\begin{align*}
\int_{\Delta_r(q)}\left|\partial_{\nu}G^t\right|^2\,d\sigma&\leq\frac{C}{r^3}\int_{T_{4r}(q)}|G^t|^2+CG^t(A_r(q))\left(\int_{T_{2r}(q)}|\dive b|^2\right)^{\frac{1}{2}}\left(\frac{C}{r^2}\int_{T_{4r}(q)}|G^t|^2\right)^{\frac{1}{2}}\\
&\leq\left(Cr^{d-3}+Cr^{d/2-1}\|\dive b\|_{L^2(T_{2r}(q))}\right)\left(G^t(A_r(q))\right)^2.
\end{align*}
We now consider the two cases $d=3$ and $d\geq 4$ separately.

If $d=3$, then we have shown that
\begin{align*}
\int_{\Delta_r(q)}\left|\partial_{\nu}G^t\right|^2\,d\sigma&\leq C\left(1+r^{1/2}\|\dive b\|_{L^2(T_{2r}(q))}\right)\left(G^t(A_r(q))\right)^2\leq C\left(G^t(A_r(q))\right)^2\\
&=Cr^{d-3}\left(G^t(A_r(q))\right)^2,
\end{align*}
where $C$ depends on the Lipschitz character of $\Omega$ (since $r<r_{\Omega}, s_{\Omega}$) and the $L^2$ norm of $\dive b$.

If, now, $d\geq 4$, then we apply H{\"o}lder's inequality for the exponent $p=d/4$, to obtain that
\[
\int_{T_{2r}(q)}|\dive b|^2\leq\left(\int_{T_{2r(q)}}|\dive b|^{d/2}\right)^{4/d}|T_{2r}(q)|^{1-4/d}\leq C\|\dive b\|_{d/2}^2r^{d-4},
\]
which implies that
\[
\int_{\Delta_r(q)}|\partial_{\nu}G^t|^2\,d\sigma\leq Cr^{d-3}\left(G^t(A_r(q))\right)^2=Cr^{d-3}\left(G(0,A_r(q))\right)^2,
\]
where $C$ also depends on the $d/2$-norm of $\dive b$. Therefore, in all cases, we have shown that
\[
\int_{\Delta_r(q)}|\partial_{\nu}G^t|^2\,d\sigma\leq Cr^{d-3}\left(G(0,A_r(q))\right)^2.
\]
But, from lemma \ref{GFromAbove}, the last quantity is bounded by
\[
Cr^{d-3}\left(G(0,A_r(q))\right)^2\leq Cr^{d-3}\left(r^{2-d}\omega(\Delta_r(q))\right)^2=Cr^{1-d}\omega(\Delta_r(q)).
\]
This finally shows that
\[
\int_{\Delta_r(q)}k^2\leq Cr^{1-d}(\omega(\Delta_r(q)))^2,
\]
which concludes the proof.
\end{proof}

For arbitrary Lipschitz domains and $b\in\Dr_{p_d}(\Omega)$, we will approximate with smooth domains to finally obtain that $k\in B_2(\partial\Omega)$, with the constants being good constants.

Denote by $L_{A,b}$ the operator $-\dive(A\nabla u)+b\nabla u$; then, the main approximation argument to pass to Lipschitz domains is contained in the next lemma.

\begin{lemma}\label{WeakMeasureConvergence}
Let $\Omega\subseteq\mathbb R^d$ be a Lipschitz domain, and consider the approximation scheme $\Omega_j\uparrow\Omega$ from theorem \ref{ApproximationScheme}, with $0\in\Omega_j$ for all $j$. Suppose that the following hold.
\begin{enumerate}[i)]
\item $A\in M_{\lambda,\mu}(\Omega)$, $A_j\in M_{\lambda,\mu}(\Omega_j)$, and $\chi_{\Omega_j}(A_j-A)\to 0$, almost everywhere in $\Omega$.
\item $\|b\|_{L^{\infty}(\Omega)},\|b_j\|_{L^{\infty}(\Omega_j)}\leq M$, and $\chi_{\Omega_j}(b_j-b)\to 0$, almost everywhere in $\Omega$.
\end{enumerate}
Let also $k$, $k_j$ denote the harmonic measure with respect to $0$ for $L=L_{A,b}$ in $\Omega$ and $L_j=L_{A_j,b_j}$ in $\Omega_j$, respectively. Then, for all $f\in C(\mathbb R^d)$,
\[
\int_{\partial\Omega_j}k_jf\,d\sigma_j\xrightarrow[j\to\infty]{}\int_{\partial\Omega}kf\,d\sigma.
\]
\end{lemma}
\begin{proof}
Let $u\in W^{1,2}_{\loc}(\Omega)\cap C(\overline{\Omega})$ be the solution of $Lu=0$ in $\Omega$ with boundary values $f$ on $\partial\Omega$, which exists from theorem \ref{WeakSolvability}. Then, corollary \ref{LocalBoundOnGradientOfSolutions} shows that $u$ is Lipschitz in $\overline{\Omega_j}$. Let now $u_j\in W^{1,2}(\Omega_j)\cap C^{\alpha}(\overline{\Omega_j})$ be the solution of $L_ju=0$ in $\Omega_j$, with boundary values $u$ on $\partial\Omega_j$, which exists from proposition \ref{WeakSolvabilityForLipschitz}. We now define
\[
v_j=u_j-u\in W_0^{1,2}(\Omega_j)\cap C(\overline{\Omega_j}),\quad f_j=(A_j-A)\nabla u\in L^2(\Omega_j),\quad g_j=(b_j-b)\nabla u\in L^2(\Omega_j),
\]
where $f_j,g_j\in L^2(\Omega_j)$ since $u\in W^{1,2}_{\loc}(\Omega)$, hence $\nabla u\in L^2(\Omega_j)$. Then, almost everywhere convergence in (i), (ii) and the dominated convergence theorem show that
\[
\|f_j\|_{L^2(\Omega_j)}^2=\int_{\Omega}\chi_{\Omega_j}|A_j-A||\nabla u|^2\xrightarrow[j\to\infty]{}0,\quad\|g_j\|_{L^2(\Omega_j)}^2\xrightarrow[j\to\infty]{}0.
\]
We also compute
\[
L_jv_j=-L_ju=-\dive(A_j\nabla u)+b_j\nabla u=-\dive(f_j)+g_j,
\]
therefore $v_j$ solves the equation $L_jv_j=F\in W^{-1,2}(\Omega_j)$, with
\[
F_j\phi=\int_{\Omega_j}f_j\nabla\phi+g_j\cdot\phi
\]
for all $\phi\in W_0^{1,2}(\Omega_j)$. Then, for all such $\phi$,
\[
|F_j\phi|\leq\int_{\Omega_j}|f_j||\nabla\phi|+|g_j||\phi|\leq\left(\|f_j\|_{L^2(\Omega_j)}+\|g_j\|_{L^2(\Omega_j)}\right)\|\phi\|_{W_0^{1,2}(\Omega_j)},
\]
therefore $\|F_j\|_{W^{-1,2}(\Omega_j)}\leq\|f_j\|_{L^2(\Omega_j)}+C\|g_j\|_{L^2(\Omega_j)}$. Hence, proposition \ref{GoodBoundOnSolutions} shows that, for a good constant $C$,
\begin{equation}\label{eq:V_j}
\|v_j\|_{W_0^{1,2}(\Omega_j)}\leq C\|F_j\|_{W^{-1,2}(\Omega_j)}\leq C\|f_j\|_{L^2(\Omega_j)}+C\|g_j\|_{L^2(\Omega_j)}\xrightarrow[j\to\infty]{}0.
\end{equation}
Set now $\delta=\delta(x,\partial\Omega)/2$, then for large $j\in\mathbb N$ we obtain that $B_{\delta}=B_{\delta}(0)\subseteq\Omega_j$. Consider now a subsequence $(v_{n_j})$; then, \eqref{eq:V_j} shows that a subsequence of $(v_{m_{n_j}})$ converges to $0$ almost everywhere in $B_{\delta}$. Also, for $x\in B_{\delta}$, the maximum principle shows that
\[
|v_j(x)|\leq|u(x)|+|u_j(x)|\leq 2\|f\|_{L^{\infty}(\partial\Omega)},
\]
hence the $(v_j)$ are uniformly bounded in $B_{\delta}$. Therefore, equicontinuity of solutions (proposition \ref{Equicontinuity}) shows that there exists a further subsequence $(v_{l_{{m_{n_j}}}})$ that converges uniformly to a continuous function in $B_{\delta}$. Since $(v_{m_{n_j}})$ converges to $0$ almost everywhere in $B_{\delta}$, we obtain that $(v_{l_{{m_{n_j}}}}(0))$ converges to $0$ as $j\to\infty$. Therefore, any subsequence of $(v_j(0))$ has a subsequence that converges to $0$; this shows that $v_j(0)\to 0$. 

The definition of harmonic measure now shows that
\begin{equation}\label{eq:FirstHarmonicConvergence}
\int_{\partial\Omega_j}k_ju\,d\sigma_j=u_j(0)=u(0)+v_j(0)\xrightarrow[j\to\infty]{}u(0)=\int_{\partial\Omega}kf\,d\sigma_j.
\end{equation}
Since now $u\in C(\overline{\Omega})$ and $f\in C(\mathbb R^d)$, $u-f$ is uniformly continuous in $\Omega$. Since also $u-f$ is equal to $0$ on $\partial\Omega$, given $\e>0$, there exists $\delta>0$ such that, if $x\in\overline{\Omega}$ with $\delta(x)<\delta$, then $|u(x)-f(x)|<\e$. Also, from theorem \ref{ApproximationScheme},
\[
\sup_{q\in\partial\Omega}\left|\Lambda_j(q)-q\right|\to 0,
\]
therefore there exists $j_0\in\mathbb N$ such that, for all $j\geq j_0$, ${\rm dist}(\partial\Omega_j,\partial\Omega)<\delta$. Hence, for $j\geq j_0$, we compute
\[
\left|\int_{\partial\Omega_j}k_ju\,d\sigma_j-\int_{\partial\Omega_j}k_jf\,d\sigma_j\right|\leq\int_{\partial\Omega_j}k_j|u-f|\,d\sigma_j\leq\sup_{x\in\partial\Omega_j}|u(x)-f(x)|\leq\e,
\]
since $k_j$ is a probability measure on $\partial\Omega_j$. Combining with \eqref{eq:FirstHarmonicConvergence}, we obtain that
\[
\lim_{j\to\infty}\int_{\partial\Omega_j}k_jf\,d\sigma_j=\int_{\partial\Omega}kf\,d\sigma,
\]
which completes the proof.
\end{proof}

Using the previous approximation lemma, we will show that the harmonic measure kernel belongs to $B_2$ for any Lipschitz domain. In order to obtain Lipschitz approximations to drifts $b\in\Dr_p(\Omega)$, we will have to ensure that the divergence of the approximations belongs to $L^p$; for this purpose, we will carefully mollify $b$, so that the mollification ``matches" the rate in which $\Omega_j$ approximates $\Omega$.

\begin{prop}\label{GeneralB2Lemma}
Let $\Omega\subseteq\mathbb R^d$ be a bounded Lipschitz domain with $0\in\Omega$, and $B_{\beta}(0)\subseteq\Omega$. Suppose also that $A\in M_{\lambda,\mu}^s(\Omega)$, and $b\in \Dr_{p_d}(\Omega)$. Then $k(q)=\frac{d\omega^0}{d\sigma}(q)\in L^2(\partial\Omega)$, and 
\[
\left(\frac{1}{|\Delta_r(q)|}\int_{\Delta_r(q)}k^2\,d\sigma\right)^{\frac{1}{2}}\leq\frac{C}{|\Delta_r(q)|}\int_{\Delta_r(q)}k\,d\sigma
\]
for all surface balls $\Delta_r(q)$, where $C$ is a good constant that also depends on the $\Dr_{p_d}$-norm of $b$, and $p_d$ appears in lemma \ref{B2Lemma}.
\end{prop}
\begin{proof}
First, extend $b$ by $0$ outside $\Omega$. We will construct a mollification of $b$: consider a smooth function $\psi$ which is positive, supported in $B(0,1)$, and has integral $1$. Let $\delta_j>0$ be the distance from $\partial\Omega_j$ to $\partial\Omega$, and consider $m_j\in\mathbb N$ such that $1/m_j<\delta_j$. Set $\phi_m(x)=m^d\phi(mx)$ for $m\in\mathbb N$, and define
\[
b_j(x)=b*\psi_{m_j}(x)=\int_{\mathbb R^d}b(x-y)\psi_{m_j}(y)\,dy.
\]
First, every $b_j$ is in $C^{\infty}(\mathbb R^d)$. In addition, for all $\phi\in C_c^{\infty}(\Omega_j)$,
\begin{align*}
\int_{\Omega_j}b_j\nabla\phi=&\int_{\Omega_j}\int_{\Omega_j}b(x-y)\psi_{m_j}(y)\nabla\phi(x)\,dydx=\int_{\Omega_j}\left(\int_{\Omega_j}b(x-y)\nabla\phi(x)\,dy\right)\psi_{m_j}(y)dx\\
=&\int_{\Omega_j}\left(\int_{\Omega_j-y}b(z)\nabla\phi(z+y)\,dz\right)\psi_{m_j}(y)dy.
\end{align*}
Now, since $\psi_{m_j}$ is supported in $B_{1/m_j}(0)$, we obtain that $\Omega_j-y\subseteq\Omega$ and $\phi(x+y)\in C_c^{\infty}(\Omega)$. Therefore, since $b\in\Dr_{p_d}(\Omega)$, the inner integral above can be written as
\[
\left|\int_{\Omega_j-y}b(z)\nabla\phi(z+y)\,dx\right|=\left|\int_{\Omega}b(z)\nabla\phi(z+y)\,dx\right|\leq\|\dive b\|_{L^{p_d}(\Omega)}\|\phi\|_{L^{\frac{p_d}{p_d-1}}(\Omega)},
\]
therefore
\begin{align*}
\left|\int_{\Omega_j}b_j\nabla\phi\right|&=\left|\int_{\Omega_j}\left(\int_{\Omega_j-y}b(x)\nabla\phi(x+y)\,dx\right)\psi_{m_j}(y)dy\right|\\
&\leq\int_{\Omega_j}\left|\int_{\Omega_j-y}b(x)\nabla\phi(x+y)\,dx\right|\psi_{m_j}(y)dy\\
&\leq\|\dive b\|_{L^{p_d}(\Omega)}\|\phi\|_{L^{\frac{p_d}{p_d-1}}(\Omega)} \int_{\Omega_j}\psi_{m_j}(y)dy\\
&=\|\dive b\|_{L^{p_d}(\Omega)}\|\phi\|_{L^{\frac{p_d}{p_d-1}}(\Omega)},
\end{align*}
since the integral of $\phi_j$ in $B_{1/m_j}(0)$ is equal to $1$. Since the previous estimate holds for all $\phi\in C_c^{\infty}(\Omega_j)$, we obtain that $\|\dive b_j\|_{L^{p_d}(\Omega_j)}\leq \|\dive b\|_{L^{p_d}(\Omega)}.$ Note also that $\|b_j\|_{\infty}\leq\|b\|_{\infty}$, therefore
\[
\|b_j\|_{\Dr_{p_d}(\Omega_j)}\leq\|\dive b\|_{\Dr_{p_d}(\Omega)}.
\]
Consider now the same mollification for $A$; that is, set $A_j(x)=A*\psi_{m_j}(x)$. Then the coefficients $A_j,b_j$ satisfy the hypotheses of lemma \ref{WeakMeasureConvergence}, therefore
\[
\int_{\partial\Omega_j}k_jf\,d\sigma_j\xrightarrow[j\to\infty]{}\int_{\partial\Omega}kf\,d\sigma,
\]
in the notation of the same lemma.

We will first show the inequality for $r<r_{\Omega},s_{\Omega}$. For this purpose, let $q\in\partial\Omega$, and consider the cylinder $Z(q,r)$. Consider also a positive function $f_0\in {\rm Lip}(\Delta_r(q))$, and extend it to a function $f\in{\rm Lip}(\mathbb R^d)$, which is supported in $Z(q,3r/2)$ and also satisfies the inequality
\[
\int_{\partial\Omega}|f|^2\leq 2\int_{\Delta_r(q)}|f_0|^2.
\]
For $q\in\Delta_r(q)$, let $q_j\in\partial\Omega_j$ be the point on $\partial\Omega_j$ that lies above $q$, in the direction of the axis of $Z(q,r)$. Then, if we set $\Delta_r^j(q_j)=Z(q,r)\cap\partial\Omega_j$, the support properties of $f$, the Cauchy-Schwartz inequality and lemma \ref{B2Lemma} (which is applicable, since $A_j\in M^s_{\lambda,\mu}(\Omega_j)$ and $b_j\in\Dr_{p_d}(\Omega_j)$) show that
\begin{align*}
\int_{\partial\Omega_j}k_jfd\sigma_j&=\int_{\Delta_{2r}^j(q_j)}k_jfd\sigma_j\leq\left(\int_{\Delta_{2r}^j(q_j)}k_j^2\,d\sigma_j\right)^{1/2}\left(\int_{\Delta_{2r}^j(q_j)}f^2d\sigma_j\right)^{1/2}\\
&\leq\frac{C}{\sigma_j^{1/2}(\Delta_{2r}^j(q_j))}\int_{\Delta_{2r}^j(q_j)}k_j\,d\sigma_j\cdot\left(\int_{\Delta_{2r}^j(q_j)}f^2d\sigma_j\right)^{1/2}\\
&\leq\frac{C}{\sigma^{1/2}(\Delta_r(q))}\int_{\Delta_{2r}^j(q_j)}k_j\,d\sigma_j\cdot\left(\int_{\Delta_{2r}^j(q_j)}f^2d\sigma_j\right)^{1/2},
\end{align*}
where we used that $\Delta_{2r}^j(q_j)$ is about equal to $r^{d-1}$, where $C$ is a good constant, which also depends on the $\Dr_p(\Omega)$ norm of $b$. Hence, letting $j\to\infty$ and applying lemma \ref{WeakMeasureConvergence}, we obtain that
\begin{align*}
\int_{\partial\Omega}kfd\sigma&\leq\frac{C}{\sigma^{1/2}(\Delta_r(q))}\limsup_{j\to\infty}\int_{\Delta_{2r}^j(q_j)}k_j\,d\sigma_j\left(\int_{\Delta_{2r}^j(q_j)}f^2d\sigma_j\right)^{1/2}\\
&\leq\frac{C}{\sigma^{1/2}(\Delta_r(q))}\limsup_{j\to\infty}\int_{\Delta_{2r}^j(q_j)}k_j\,d\sigma_j\cdot\left(\int_{\partial\Omega}f^2d\sigma\right)^{1/2},
\end{align*}
since $f$ is continuous on $\mathbb R^d$. From our choice of $f$, we obtain that
\[
\int_{\Delta_r(q)}kf_0\,d\sigma\leq \frac{C}{\sigma^{1/2}(\Delta_r(q))}\limsup_{j\to\infty}\int_{\Delta_{2r}^j(q_j)}k_j\,d\sigma_j\cdot\left(\int_{\Delta_r(q)}f_0^2d\sigma\right)^{1/2},
\]
and since this inequality holds for all $f_0\in{\rm Lip}(\Delta_r(q))$, we obtain that
\begin{equation}\label{eq:ForLimsup}
\left(\int_{\Delta_r(q)}k^2\,d\sigma\right)^{1/2}\leq\frac{C}{\sigma^{1/2}(\Delta_r(q))}\limsup_{j\to\infty}\int_{\Delta_{2r}^j(q_j)}k_j\,d\sigma_j.
\end{equation}
To treat the last term, let $g$ be a continuous function which is supported in $Z(q,3r)$ and it is equal to $1$ in $Z(q,2r)$. We then apply lemma \ref{WeakMeasureConvergence}, to obtain that
\begin{align*}
\limsup_{j\to\infty}\int_{\Delta_{2r}^j(q_j)}k_j\,d\sigma_j&\leq\limsup_{j\to\infty}\int_{\partial\Omega_j}k_jg\,d\sigma_j=\int_{\partial\Omega}kg\,d\sigma\\
&\leq\int_{\Delta_{3r}(q)}k\,d\sigma=\omega(\Delta_{3r}(q))\leq C\omega(\Delta_r(q)),
\end{align*}
from the doubling property of $\omega$. Plugging the last inequality in \eqref{eq:ForLimsup}, we obtain that, for $r<r_{\Omega},s_{\Omega}$,
\[
\left(\int_{\Delta_r(q)}k^2\,d\sigma\right)^{1/2}\leq\frac{C\omega(\Delta_r(q))}{\sigma^{1/2}(\Delta_r(q))}=C\sigma^{1/2}(\Delta_r(q))\fint_{\Delta_r(q)}k\,d\sigma.
\]
The last inequality shows that $k\in L^2(\partial\Omega)$. In addition, if $r\geq r_{\Omega}$, we have that
\[
\int_{\Delta_r(q)}k^2d\sigma\leq\|k\|_2^2\leq\frac{\|k\|_2^2|\partial\Omega|}{\omega^2(\Delta_r(q))}\frac{\omega^2(\Delta_r(q))}{|\Delta_r(q)|}\leq\frac{\|k\|_2^2|\partial\Omega|}{\omega^2(\Delta_{r_{\Omega}}(q))}\frac{\omega^2(\Delta_r(q))}{|\Delta_r(q)|}
\]
and we use that $1=\omega(\partial\Omega)\leq c\omega(\Delta_{r_{\Omega}}(q))$, from the doubling property of $\omega$. A similar inequality holds if $r>s_{\Omega}$, therefore $k\in B_2(\partial\Omega)$.
\end{proof}

As a corollary of the fact $k\in B_2(\partial\Omega)$, we obtain the next theorem on solvability of the Dirichlet problem.

\begin{thm}\label{GeneralDirichlet}
Let $\Omega\subseteq\mathbb R^d$ be a bounded Lipschitz domain, and set $p_d=2$ for $d=3$, and $p_d=d/2$ for $d\geq 4$. Suppose that $A\in M_{\lambda,\mu}^s(\Omega)$ and $b\in\Dr_{p_d}(\Omega)$. Then there exists $\e>0$ such that $D_p$ is uniquely solvable in $\Omega$ for all $p\in(2-\e,\infty)$, with the constant being a good constant that also depends on $p$ and $\|b\|_{\Dr_{p_d}}$.
\end{thm}
\begin{proof}
From proposition \ref{GeneralB2Lemma}, the harmonic measure kernel $k=\frac{d\omega^0}{d\sigma}$ belongs to $B_2(\partial\Omega)$, with its norm being a good constant that also depends on $\|b\|_{\Dr_{p_d}}$. Proposition \ref{BpProperty} then shows that $k\in B_q(\partial\Omega)$ for any $q\in(1,2+\delta)$, where $\delta$ is a good constant. Hence, theorem \ref{DirichletToWeightsRelation} shows that the Dirichlet problem $D_p$ for the equation $Lu=0$ in $\Omega$ is solvable, for any $p\in(2-\e,\infty)$, where $2-\e$ is the conjugate exponent to $2+\delta$, with bounds being good constants that also depend on $p$ and $\|b\|_{\Dr_{p_d}}$.
\end{proof}

Note that, in the theorem above, $A$ has to be symmetric. Later on, we will be able to drop this assumption.

\subsection{Existence and uniqueness for data in $L^p$}
When $D_p$ is solvable in $\Omega$, we can show existence and uniqueness of solutions for the Dirichlet problem with data in $L^p(\partial\Omega)$. We will first need a lemma.

\begin{lemma}
Suppose that $D$ is a Lipschitz domain above the graph of a Lipschitz function $\phi:B\to\mathbb R$ with Lipschitz constant $M$, where $B\subseteq\mathbb R^{d-1}$ is a ball; that is,
\[
D=\left\{(x_0,t)\in\mathbb R^d\Big{|}t>\phi(x_0)\right\}.
\]
Consider the set 
\[
D_{\e}=\left\{x\in D\Big{|}\frac{\e}{2}\leq\delta(x)\leq\e\right\}.
\]
Then, there exists $C=C_M>0$ such that, for all $\e>0$, if $x=(x_0,t)\in D_{\e}$, then 
\[
\frac{\e}{2}\leq t-\phi(x_0)\leq C\e.
\]
\end{lemma}
\begin{proof}
Let $x=(x_0,t)\in D_{\e}$. Then,
\[
\frac{\e}{2}\leq\delta(x)\leq\delta(x,(x_0,\phi(x_0))=t-\phi(x_0),
\]	which shows the first inequality. For the second, suppose that the distance $\delta(x)$ is achieved at a point $(y_0,\phi(y_0))$, for some $y_0\in B$. Then $|y_0-x_0|^2+|\phi(y_0)-t|^2=\delta^2(x)\leq\e^2$, which implies that $|y_0-x_0|\leq \e$ and $|\phi(y_0)-t|\leq\e$. Therefore,
\[
t-\phi(x_0)\leq|t-\phi(y_0)|+|\phi(y_0)-\phi(x_0)|\leq\e+M|y_0-x_0|\leq(M+1)\e,
\]
so we can take $C=M+1$.
\end{proof}
We will also need a version of the Cacciopoli inequality over sets that are not necessarily balls.
\begin{lemma}\label{ModifiedCacciopoli}
Let $D$ be a Lipschitz domain above a domain $U\subseteq\mathbb R^{d-1}$, of the graph of a Lipschitz function $\psi$; that is,
\[
D=\left\{(x_0,t)\in\mathbb R^{d-1}\times\mathbb R\Big{|}t>\phi(x_0)\right\},
\]
and consider a cube $Q\subseteq\mathbb R^{d-1}$, with side length $\e$, such that its triple $3Q$ is subset of $U$. Let also $c_0>1$. Then, if $u:D\to\mathbb R$ is a solution of $Lu=0$, there exists $C=C(M)$ such that
\[
\int_Q\int_{\psi(x_0)+\e/2}^{\psi(x_0)+c_0\e}|\nabla u(x_0,t)|^2\,dx_0dt\leq\frac{C}{\e^2}\int_{3Q}\int_{\psi(x_0)+\e/4}^{\psi(x_0)+3c_0\e}|u(x)|^2\,dx_0dt.
\]
\end{lemma}
\begin{proof} Let 
\[U_i=\left\{(x_0,t)\in\mathbb R^n\bm{\Big|} x_0\in iQ, \psi(x_0)+\frac{\e}{i+1}\leq t\leq\psi(x_0)+\frac{ic_0}{\e}\right\},\,i=1,2,3,
\]
and consider a smooth cutoff function $\phi$ which is supported on $B_{\e/100}(0)$, is nonnegative, and has integral $1$. Consider the convolution $\phi_0=\chi_{U_2}*\phi$. Then $\phi_0$ is smooth, it is equal to $1$ in $U_1$, it is supported in $U_3$, and also $|\nabla\phi_0|\leq C/\e$. The proof is now similar to the proof of lemma \ref{Cacciopoli}.
\end{proof}

The next proposition guarantees uniqueness for the Dirichlet problem with data in $L^p$, whenever $D_p$ is solvable.

\begin{prop}\label{DirichletUniqueness}
Let $\Omega$ be a Lipschitz domain, $A\in M_{\lambda}^s(\Omega)$ and $b\in L^{\infty}(\Omega)$. Suppose that $u\in W^{1,2}_{\loc}(\Omega)$ is a solution to the equation $-\dive(A\nabla u)+b\nabla u=0$ in $\Omega$, such that $u\to 0$ on $\partial\Omega$ nontangentially, almost everywhere with respect to the surface measure. Let also $p\in(1,\infty)$, and suppose that  $\|u^*\|_{L^p(\partial\Omega)}<\infty$. If $D_p$ is solvable in $\Omega$, then $u\equiv 0$ in $\Omega$.
\end{prop}
\begin{proof}
Let $y\in\Omega$. For $\e>0$, set
\[
\Omega_{\e}=\{x\in\Omega|\delta(x)\leq\e\},\,\,R_{\e}=\Omega_{2\e}\setminus \Omega_{\e}.
\]
Consider a smooth cutoff $\phi_{\e}$ which is $1$ outside $\Omega_{\e}$, $0$ in $\Omega_{2\e}$, and $|\nabla\phi_{\e}|\leq C/\e$. Suppose that $\e$ is small enough, such that $G_y^t$ is a solution of $L^tG_y^t=0$ in $T_{10\e}(q)$, for any $q\in\partial\Omega$. Then, we obtain that
\[
u(y)=u(y)\phi_{\e}(y)=\int_{\Omega}A^t\nabla G_y^t\nabla(u\phi_{\e})+b\nabla(u\phi_{\e})G_y^t,
\]
hence
\begin{align*}
u(y)&=\int_{\Omega}A^t\nabla G_y^t\nabla\phi_{\e}\cdot u+A^t\nabla G_y^t\nabla u\cdot\phi_{\e}+b\nabla u\cdot\phi_{\e}G_y^t+b\nabla\phi_{\e}\cdot uG_y^t\\
&=\int_{\Omega}A^t\nabla G_y^t\nabla\phi_{\e}\cdot u+A^t\nabla(G_y^t\phi_{\e})\nabla u+b\nabla u\cdot\phi_{\e}G_y^t+b\nabla\phi_{\e}\cdot uG_y^t-A^t\nabla\phi_{\e}\nabla u\cdot G_y^t\,dx\\
&=\int_{R_{\e}}A^t\nabla G_y^t\nabla\phi_{\e}\cdot u\,dx+\int_{R_{\e}}b\nabla\phi_{\e}\cdot uG_y^t\,dx-\int_{R_{\e}}A^t\nabla\phi_{\e}\nabla u\cdot G_y^t\,dx=I_1+I_2+I_3,
\end{align*}
since $u$ is a solution in $\Omega$, and from the support properties of $\phi_{\e}$.
	
Consider now the covering of $\partial\Omega$ by coordinate cylinders $Z_i$, for $i=1,\dots N$. Then, if $\e$ is small enough, we obtain that
\[
R_{\e}\subseteq\bigcup_{i=1}^NZ_i\Rightarrow R_{\e}\subseteq\bigcup_{i=1}^N(Z_i\cap R_{\e}).
\]
Fix $i$, and suppose that $Z_i$ has basis $B_i$, where $B_i\subseteq\mathbb R^{d-1}$ is a ball. Consider a cube $Q$ with side length $l$, such that $B_i\subseteq Q\subseteq \sqrt{d}B_i$, and split $Q$ in $2^{ad}$ dyadic subcubes $Q_j$, with side length $l/a$, such that $l/a\leq \tilde{c}\e$, where $\tilde{c}$ will be chosen later. Suppose also that $\e$ is small enough, such that $\e<s_{\Omega}$, where $s_{\Omega}$ appears in lemma \ref{InnerRadius}, and set $P_j$ to be the part of $\partial\Omega$ that lies above $Q_j$. Then, using lemma \ref{ModifiedCacciopoli}, we can write every $x\in Z_i\cap R_{\e}$ in the form $(x_0,s)$, with $x_0\in Q$ and $s\in(\psi(x_0)+c_1\e,\psi(x_0)+c_2\e)$, where $\psi$ is a Lipschitz function. Therefore,
\begin{align*}
\int_{Z_i\cap R_{\e}}\left|A^t\nabla G_y^t\nabla\phi_{\e}\cdot u\right|\,dx&\leq\frac{C}{\e}\int_Q\int_{\psi(x_0)+c_1\e}^{\psi(x_0)+c_2\e}|\nabla G_y^t(x_0,s)||u(x_0,s)|\chi_{R_{\e}}(x_0,s)\,dsdx_0\\
&\leq\frac{C}{\e}\sum_j\int_{Q_j}\int_{\psi(x_0)+c_1\e}^{\psi(x_0)+c_2\e}|\nabla G_y^t(x_0,s)||u(x_0,s)|\chi_{R_{\e}}(x_0,s)\,dsdx_0.
\end{align*}
We now bound the last integral using the Cauchy-Schwartz inequality. First, note that, using Cacciopoli's inequality (lemma \ref{ModifiedCacciopoli}),
\begin{multline*}
\int_{Q_j}\int_{\psi(x_0)+c_1\e}^{\psi(x_0)+c_2\e}|\nabla G_y^t(x_0,s)|^2\,dsdx_0\leq\frac{C}{\e^2}\int_{3Q_j}\int_{\psi(x_0)+c_0\e}^{\psi(x_0)+c_4\e}|G_y^t(x_0,s)|^2\,dsdx_0\\
\leq\frac{C}{\e}\sup\{G_y^t(x_0,s)|x_0\in 3Q_j, \psi(x_0)+c_0\e\leq s\leq \psi(x_0)+c_3\e\}^2|3Q_j|.
\end{multline*}
So, applying lemmas \ref{CarlesonEstimate} and \ref{GFromAbove}, the quantity above is bounded by
\[
\frac{C}{\e}\omega^y(\Delta_{2\e}(q))^2\e^{4-2d}|Q_j|\leq C\omega^y(\Delta_{2\e}(q))^2\e^{2-d},
\]
for all $q\in P_j$.
	
We now choose $\tilde{c}$ sufficiently small, such that, for any point $y_0\in Q_j$ and for each $(x_0,t)$ with $\delta(x_0,t)\geq\e/2$, $x_0\in Q_j$, and $\psi(x_0)+c_1\e\leq t\leq\psi(x_0)+c_2\e$, we have that $(x_0,t)$ belongs to the cone $\Gamma(y_0,\psi(y_0))$, defined in the introduction; this choice of $\tilde{c}$ will only depend on the apertures of the cones $\Gamma$, and the constants $c_1,c_2$. Hence, for any $y_0\in Q_j$, we see that
\begin{align*}
\int_{Q_j}\int_{\psi(x_0)+c_1\e}^{\psi(x_0)+c_2\e}|u(x_0,s)|^2\chi_{R_{\e}}(x_0,s)\,dsdx_0&\leq\int_{Q_j}\int_{\psi(x_0)+c_1\e}^{\psi(x_0)+c_2\e}|u^*_{\e}(y_0,\psi(y_0))|^2\,dsdx_0\\&\leq C\e^d|u^*_{\e}(y_0,\psi(y_0))|^2,
\end{align*}
where $u^*_{\e}$ is defined as
\[
u^*_{\e}(q)=\sup\{|u(x)|\bm{|}x\in \Gamma(q), |x-q|\leq C\e\},
\]
Therefore, we finally get that
\begin{multline*}
\int_{Q_j}\int_{\psi(x_0)+c_1\e}^{\psi(x_0)+c_2\e}|\nabla G_y^t(x_0,s)||u(x_0,s)|\chi_{R_{\e}}(x_0,s)\,dsdx_0\\
\leq C\e\omega^y(\Delta_{2\e}(y_0,\psi(y_0)))u^*_{\e}(y_0,\psi(y_0))
\end{multline*}
for all $y_0\in Q_j$. Now, we integrate this relation in $Q_j$ and we change variables, to obtain that 
\begin{multline*}
\int_{Q_j}\int_{\psi(x_0)+c_1\e}^{\psi(x_0)+c_2\e}|\nabla G_y^t(x_0,s)||u(x_0,s)|\chi_{R_{\e}}(x_0,s)\,dsdx_0\\
\leq C\e^{2-d}\int_{P_j}\omega^y(\Delta_{2\e}(q))u^*_{\e}(q)\,d\sigma(q).
\end{multline*}
Hence, we obtain that
\[
\int_{Z_i\cap R_{\e}}\left|A^t\nabla G_y\nabla\phi_{\e} u\right|\,dx\leq C\sum_j\e^{1-d}\int_{P_j}\omega^y(\Delta_{2\e}(q))u^*_{\e}(q)\,d\sigma(q).
\]
Now, using the fact that the surface measure of $\Delta_{2\e}(q)$ is comparable to $\e^{d-1}$, we obtain that
\begin{align*}
\int_{Z_i\cap R_{\e}}\left|A^t\nabla G_y\nabla\phi_{\e} u\right|\,dx&\leq C\sum_j\int_{P_j}\frac{\omega^y(\Delta_{2\e}(q))}{\sigma(\Delta_{2\e}(q))}u^*(q)\,d\sigma(q)\\
&\leq C\sum_j\int_{P_j}M_{\omega^y}(q)u^*_{\e}(q)\,d\sigma(q)=C\int_{U_i\cap\partial\Omega}M_{\omega^y}(q)u^*_{\e}(q)\,d\sigma(q)
\end{align*}
where $M_{\omega^y}$ is the maximal function of $\omega^y$. Since $D_p$ is solvable, by theorem \ref{DirichletToWeightsRelation} we obtain that the kernel $\frac{d\omega^y}{d\sigma}$ is an $L^{p'}(\partial\Omega)$ function, therefore, from the bounds on the maximal function, we obtain that 
\[
\int_{Z\cap R_{\e}}\left|A^t\nabla G_y\nabla\phi_{\e} u\right|\,dx\leq C\left(\int_{U\cap\partial\Omega}|u^*_{\e}(q)|^p\,d\sigma(q)\right)^{1/p}.
\]
But $u\to 0$ on $\partial\Omega$, nontangentially, almost everywhere, therefore the dominated convergence theorem shows that 
\[
\int_{Z_i\cap R_{\e}}\left|\nabla G_y\nabla\phi_{\e} u\right|\,dx\xrightarrow[\e\to 0] {}0.
\]
Adding those integrals over the cylinders $Z_i$ (since we have $N$ of them), we obtain that $I_1\to 0$.
	
We now turn to $I_2$, and we write, as above
\[
\int_{Z_i\cap R_{\e}}|b\nabla\phi_{\e}\cdot uG_y^t\,dx|\leq\frac{C}{\e}\sum_j\int_{Q_j}\int_{\psi(x_0)+c_1\e}^{\psi(x_0)+c_2\e}|G_y^t(x_0,s)||u(x_0,s)|\chi_{R_{\e}}(x_0,s)\,dsdx_0.
\]
We then use the Cauchy-Schwartz inequality, and we bound
\[
\int_{Q_j}\int_{\psi(x_0)+c_1\e}^{\psi(x_0)+c_2\e}|G_y^t(x_0,s)|^2\,dsdx_0\leq C\omega^y(\Delta_{2\e}(q))^2\e^{3-d},
\]
for every $y_0\in Q_j$, from lemmas \ref{CarlesonEstimate} and \ref{GFromAbove}. In addition,
\[
\int_{Q_j}\int_{\psi(x_0)+c_1\e}^{\psi(x_0)+c_2\e}|u(x_0,s)|^2\,dsdx_0\leq C\e^d|u^*_{\e}(y_0,\psi(y_0))|^2,
\]
therefore, if we multiply the two estimates above and we integrate over $Q_j$, we obtain that
\[
\int_{Q_j}\int_{\psi(x_0)+c_1\e}^{\psi(x_0)+c_2\e}|G_y^t(x_0,s)||u(x_0,s)|\chi_{R_{\e}}(x_0,s)\,dsdx_0\leq C\e^{\frac{3}{2}-d}\int_{P_j}\omega^y(\Delta_{2\e}(q))u^*_{\e}(q)\,d\sigma(q).
\]
This shows that
\[
\int_{Z_i\cap R_{\e}}|b\nabla\phi_{\e}\cdot uG_y^t\,dx|\leq C\sqrt{\e}\sum_j\e^{1-d}\int_{P_j}\omega^y(\Delta_{2\e}(q))u^*_{\e}(q)\,d\sigma(q),
\]
which goes to $0$, as above; hence $I_2\to 0$.
	
Finally, for $I_3$, 
\[
\int_{Z_i\cap R_{\e}}|A^t\nabla\phi_{\e}\nabla u\cdot G_y^t\,dx|\leq\frac{C}{\e}\sum_j\int_{Q_j}\int_{\psi(x_0)+c_1\e}^{\psi(x_0)+c_2\e}|G_y^t(x_0,s)||\nabla u(x_0,s)|\chi_{R_{\e}}(x_0,s)\,dsdx_0,
\]
and, for all $y_0\in Q_j$,
\[
\int_{Q_j}\int_{\psi(x_0)+c_1\e}^{\psi(x_0)+c_2\e}|G_y^t(x_0,s)|^2\,dsdx_0\leq C\omega^y(\Delta_{2\e}(y_0,\psi(y_0)))^2\e^{3-d},
\]
while
\begin{align*}
\int_{Q_j}\int_{\psi(x_0)+c_1\e}^{\psi(x_0)+c_2\e}|\nabla u(x_0,s)|^2\,dsdx_0&\leq\frac{C}{\e^2}\int_{3Q_j}\int_{\psi(x_0)+c_0\e}^{\psi(x_0)+c_4\e}|u(x_0,s)|^2\,dsdx_0\\
&\leq C\e^{d-2}|u^*_{\e}(y_0,\psi(y_0))|^2.
\end{align*}
Therefore, as above,
\[
\int_{Z_i\cap R_{\e}}|A^t\nabla\phi_{\e}\nabla u\cdot G_y^t\,dx|\leq C\sum_j\int_{P_j}M_{\omega^y}(q)u^*_{\e}(q)\,d\sigma(q),
\]
which goes to $0$ as $\e\to 0$. Therefore $I_3\to 0$ as well, which shows that $u(y)=0$. This completes the proof.
\end{proof}

The next proposition shows existence of solutions with boundary data in $L^p(\partial\Omega)$.

\begin{prop}\label{DirichletSolvability}
Suppose that $\Omega$ is a bounded Lipschitz domain, and let $A\in M_{\lambda}(\Omega)$, $b\in L^{\infty}(\Omega)$. Assume also that that $D_p$ is solvable in $\Omega$, with constant $C$. Then, for every $f\in L^p(\partial\Omega)$, there exists a unique $u\in W^{1,2}_{\loc}(\Omega)$ which satisfies the following:
\begin{enumerate}[i)]
\item $u$ is a weak solution to $Lu=0$ in $\Omega$.
\item $u$ converges to $f$ nontangentially, for almost every point $q\in\partial\Omega$ with respect to the surface measure: that is,
\[
\lim_{\substack{x\in\Gamma(q)\\x\to q}}u(x)=f(q),\,\,\sigma-{\rm a.e.}\,\,q\in\partial\Omega
\]
\item $\|u^*\|_{L^p(\Omega)}<\infty$.
\end{enumerate}
Then, this solution $u$ will satisfy the inequality $\|u^*\|_{L^p(\Omega)}\leq C\|f\|_{L^p(\Omega)}$.
\end{prop}
\begin{proof}
Uniqueness follows from proposition \ref{DirichletUniqueness}.
	
For existence, let $f\in L^p(\partial\Omega)$, and consider a sequence $f_n\in C(\partial\Omega)$ with $f_n\to f$ in $L^p(\partial\Omega)$. Consider also the solutions $u_n$ to the Dirichlet problem with boundary data $f_n$. From solvability of $D_p$, we obtain that
\[
\|u_n^*\|_{L^p(\partial\Omega)}\leq C\|f_n\|_{L^p(\partial\Omega)}.
\]
Consider now the approximation scheme $\Omega_j\uparrow\Omega$, from theorem \ref{ApproximationScheme}, and fix $j\in\mathbb N$. Let also $q\in\partial\Omega$. Since $\Lambda_j(q)\in\Gamma(q)$ from the same theorem and $j$ is fixed, we obtain that $\Lambda_j(q)\in\Gamma(p)$ for all $p\in\Delta_{r_j}(q)$, where $r_j$ is sufficiently small. This will show that
\[
|u_n(\Lambda_j(q))|\leq|u_n^*(p)|\Rightarrow |u_n(\Lambda_j(q))|\leq\fint_{\Delta_{r_j}(q)}|u_n^*|\,d\sigma\leq\frac{\|f_n\|_{L^p(\partial\Omega)}}{r_j}.
\]
It follows that, for any $j\in\mathbb N$, $(u_n)$ is a uniformly bounded sequence of solutions in $\Omega_j$. Therefore, from equicontinuity of solutions (proposition \ref{Equicontinuity}) and also applying a diagonal argument, $(u_n)$ converges pointwise to a solution $u$ in $\Omega$, uniformly in compact subsets of $\Omega$.
	
Let now $q\in\partial\Omega$, and $x\in\Gamma(q)$. Then, for all $n\in\mathbb N$,
\[
|u(x)-u_n(x)|=\liminf_{m\to\infty}|u_m(x)-u_n(x)|\leq\liminf_{m\to\infty}(u_m-u_n)^*(q),
\]
which implies that $(u-u_n)^*\leq\liminf_{m\to\infty}(u_m-u_n)^*$. We now integrate and apply Fatou's lemma, to obtain
\[
\|(u-u_n)^*\|_p\leq\lim_{m\to\infty}\|(u_m-u_n)^*\|_p\leq C\lim_{m\to\infty}\|f_m-f_n\|_p=C\|f-f_n\|_p,
\] 
where we also used that $D_p$ is solvable in $\Omega$. Therefore, $\|(u-u_n)^*\|_p\to 0$ as $n\to\infty$.
	
We now compute, for $x\in\Gamma(q)$ and $n\in\mathbb N$,
\begin{align*}
\limsup_{x\to q}u(x)-f(q)&\leq\limsup_{x\to q}(u(x)-u_n(x))+\limsup_{x\to q}u_n(x)-f(q)\\
&=\limsup_{x\to q}(u(x)-u_n(x))+f_n(q)-f(q)\leq (u-u_n)^*(q)+f_n(q)-f(q).
\end{align*}
Therefore, for any $\e>0$,
\begin{multline*}
\sigma\left(\left\{q\in\Gamma(q)\Big{|}\limsup_{x\to q}u(x)-f(q)>\e\right\}\right)\\
\leq\sigma\left(\left\{q\in\Gamma(q)\Big{|}(u-u_n)^*(q)+f_n(q)-f(q)>\e\right\}\right),
\end{multline*}
for all $n\in\mathbb N$. But, $(u-u_n)^*+f_n-f$ converges to $0$ almost everywhere, so
\[
\limsup_{x\to q}u(x)-f(q)\leq\e
\]
for almost all $q\in\partial\Omega$. This shows that $\limsup_{x\to q}u(x)\leq f(q)$ for almost all $q\in\partial\Omega$. A similar procedure shows that $\liminf_{x\to q}u(x)\geq f(q)$ for almost all $q\in\partial\Omega$, therefore $u$ converges nontangentially to $f$, almost everywhere.
	
Finally,
\[
\|u^*\|_p\leq\|(u-u_n)^*\|_p+\|u_n^*\|_p\leq C\|f-f_n\|_p+C\|f_n\|_p\xrightarrow[n\to\infty]{}C\|f\|_p,
\]
which shows that $u$ is the required solution.
\end{proof}
\section{The Regularity problem for $L$}
In this chapter we turn our attention to the Regularity problem for the equation $Lu=0$ in a Lipschitz domain $\Omega$. Rather than obtaining our results for smooth data on the boundary and using approximations, as in the case of the Dirichlet problem in the previous chapter, we will show solvability for the Regularity problem providing a formula for the solutions which will hold for data in $W^{1,2}(\partial\Omega)$. This will be achieved using the method of \emph{layer potentials}.

\subsection{Formulation and uniqueness}
We begin with the formulation of the Regularity problem. Recall that the space $W^{1,p}(\partial\Omega)$ is defined in definition \ref{W1pPartialOmega}.

\begin{dfn}
Let $\Omega$ be a Lipschitz domain, and $p\in(1,\infty)$. We say that the Regularity problem $R_p$ is solvable in $\Omega$, if there exists $C>0$ such that, for every $f\in W^{1,p}(\partial\Omega)$, there exists a solution $u\in W^{1,2}_{\loc}(\Omega)$ to the Dirichlet problem
\[
\left\{\begin{array}{c l}
Lu=0,&{\rm in}\,\,\Omega,\\
u=f,&{\rm on}\,\,\partial\Omega,
\end{array}\right.
\]
which also satisfies the estimate
\[
\|(\nabla u)^*\|_{L^p(\partial\Omega)}\leq C\|\nabla f\|_{L^p(\partial\Omega)},
\]
and $u=f$ on the boundary is interpreted in the nontangential, almost everywhere sense.
\end{dfn}

Note that only the gradient of $f$ appears on the right hand side. This is to be expected since constants are solutions to the equation.

The next proposition relates the nontangential maximal functions of a function and its gradient. This will be the basic estimate that shows uniqueness for the Regularity problem.

Recall first that, if $u$ is a function in $\Omega$, then for $q\in\partial\Omega$ we define 
\[
u_{\e}^*(q)=\sup\left\{|u(x)|\Big{|}x\in\Gamma(q),\,|x-q|<C\e\right\},
\]

for some constant $C$.

\begin{prop}\label{UByNablaU}
Let $\Omega$ be a Lipschitz domain, and let $\e>0$. Suppose that $u\in C^1_{{\rm loc}}(\Omega)$ and $u$ converges nontangentially, almost everywhere to a function $f\in L^p(\partial\Omega)$ on the boundary, for some $p\in[1,\infty)$. Then, for all $\e>0$, and almost all $q\in\partial\Omega$,
\[
u_{\e}^*\leq C\e(\nabla u)^*+|f|.
\]
\end{prop}
\begin{proof}
Consider first $q\in\partial\Omega$ such that $f(q)$ is finite and $u$ converges nontangentially to $f(q)$ at $q$, and suppose that $x\in\Gamma(q)$ and $|x-q|<C\e$. Since $\Gamma(q)$ is a cone, the line segment $[x,q)$ is a subset of $\Gamma(q)$. Consider now the sequence of points
\[
x_m=q+2^{1-m}(x-q),
\]
then all those points lie in $[x,q)$, and $x_1=x$.

Let now $m\in\mathbb N$. Since $u$ is continuously differentiable in a neighborhood of the line segment from $x_m$ to $x_{m+1}$, there exists $z_m$ in this line segment such that
\[
|u(x_m)-u(x_{m+1})|\leq |\nabla u(z_m)||x_m-x_{m+1}|.
\]
Since $z_m$ lies in the line segment from $x$ to $q$, we obtain that $z_m\in\Gamma(q)$, therefore
\[
|u(x_m)-u(x_{m+1})|\leq (\nabla u)^*(q)|x_m-x_{m+1}|.
\]
In addition, $x_m\to q$, and $x_m\in\Gamma(q)$, therefore $u(x_m)\to f(q)$. This shows that
\begin{align*}
|u(x)-f(q)|&\leq\sum_{m\in\mathbb N}|u(x_m)-u(x_{m+1})|\leq\sum_{m\in\mathbb N}(\nabla u)^*(q)|x_m-x_{m+1}|\\
&=(\nabla u)^*(q)\sum_{m\in\mathbb N}2^{-m}|x-q|=C(\nabla u)^*(q)|x-q|\\
&\leq C\e(\nabla u)^*(q),
\end{align*}
since $|x-q|<C\e$. This inequality shows that
\[
|u(x)|\leq C\e(\nabla u)^*+|f(q)|
\]
for all $x\in\Gamma(q)$, with $|x-q|<C\e$. We now consider the supremum for those $x$ to obtain the desired inequality.
\end{proof}

We now show uniqueness for the Regularity problem; in order to do this, we will follow an argument similar to the proof of proposition \ref{DirichletUniqueness}. A standard way of showing this uniqueness would involve using proposition \ref{DirichletUniqueness} directly, since from proposition \ref{UByNablaU} the norm of $(\nabla u)^*$ is stronger than the norm of $u^*$. However, proposition \ref{DirichletUniqueness} shows uniqueness after assuming that the Dirichlet problem is solvable in $\Omega$, which we have obtained after assuming a condition on $\dive b$, as in theorem \ref{GeneralDirichlet}. Since we want to treat the Regularity problem only assuming that $b$ is bounded, the previous argument will not work, which justifies the need for the next proposition.

\begin{prop}\label{UniquenessForRegularity}
Suppose that $\Omega$ is a bounded Lipschitz domain, $A\in M_{\lambda}(\Omega)$, $b\in L^{\infty}(\Omega)$. Let $u\in W_{\loc}^{1,2}(\Omega)$ be a solution to $Lu=0$ in $\Omega$, with $(\nabla u)^*\in L^1(\partial\Omega)$ and $u\to 0$ nontangentially, almost everywhere. Then, $u\equiv 0$.
\end{prop}
\begin{proof}
Let $\e>0$, such that $\e<s_{\Omega}$, where $s_{\Omega}$ is defined in lemma \ref{InnerRadius}. Set also
\[
\Omega_{\e}=\{x\in\Omega|\delta(x)\leq\e\},\quad R_{\e}=\Omega_{2\e}\setminus \Omega_{\e}.
\]
Consider a smooth cutoff $\phi_{\e}$ which is $1$ outside $\Omega_{\e}$, $0$ in $\Omega_{2\e}$, and $|\nabla\phi_{\e}|\leq C/\e$. We then proceed as in proposition \ref{DirichletUniqueness}, to obtain that
\[
u(y)=\int_{R_{\e}}A\nabla G_y^t\nabla\phi_{\e}\cdot u+\int_{R_{\e}}b\nabla \phi_{\e}\cdot uG_y^t-\int_{R_{\e}}A\nabla u\nabla\phi_{\e}\cdot G_y^t=I_1+I_2+I_3.
\]
Consider now the cylinders $Z_i$ for $i=1,\dots N$, the sets $B_i,Q,Q_j$ and $P_j$, and the function $\psi$ that appear in the proof of theorem \ref{DirichletUniqueness}. We then estimate, for $I_1$,
\begin{align*}
\int_{Z_i\cap R_{\e}}\left|A\nabla G_y^t\nabla\phi_{\e}\cdot u\right|\,dx&\leq\frac{C}{\e}\int_Q\int_{\psi(x_0)+c_1\e}^{\psi(x_0)+c_2\e}|\nabla G_y^t(x_0,s)||u(x_0,s)|\chi_{R_{\e}}(x_0,s)\,dsdx_0\\
&\leq\frac{C}{\e}\sum_j\int_{Q_j}\int_{\psi(x_0)+c_1\e}^{\psi(x_0)+c_2\e}|\nabla G_y^t(x_0,s)||u(x_0,s)|\chi_{R_{\e}}(x_0,s)\,dsdx_0.
\end{align*}
We now apply the Cauchy-Schwartz inequality to bound the integral. First, we obtain that
\begin{align*}
\int_{Q_j}\int_{\psi(x_0)+c_1\e}^{\psi(x_0)+c_2\e}|\nabla G_y^t(x_0,s)|^2\,dsdx_0&\leq\frac{C}{\e^2}\int_{3Q_j}\int_{\psi(x_0)+c_0\e}^{\psi(x_0)+c_3\e}|G_y^t(x_0,s)|^2\,dsdx_0\\
&\leq\frac{C\e^d}{\e^2}\sup_{\delta(x)<c_4\e}|G_y^t(x)|^2\leq C\e^{d+2\alpha-2},
\end{align*}
for some $\alpha\in(0,1)$, where we also used proposition \ref{HolderOnTheBoundary}. Moreover, for any $y_0\in Q_j$, we see that
\begin{align*}
\int_{Q_j}\int_{\psi(x_0)+c_1\e}^{\psi(x_0)+c_2\e}|u(x_0,s)|^2\chi_{R_{\e}}(x_0,s)\,dsdx_0&\leq\int_{Q_j}\int_{\psi(x_0)+c_1\e}^{\psi(x_0)+c_2\e}|u^*_{\e}(y_0,\psi(y_0))|^2\,dsdx_0\\&\leq C\e^d|u^*_{\e}(y_0,\psi(y_0))|^2,
\end{align*}
therefore, after multiplying the last two inequalities, we obtain that
\[
\int_{Q_j}\int_{\psi(x_0)+c_1\e}^{\psi(x_0)+c_2\e}|\nabla G_y^t(x_0,s)||u(x_0,s)|\chi_{R_{\e}}(x_0,s)\,dsdx_0\leq C\e^{d+\alpha-1}|u^*_{\e}(y_0,\psi(y_0))|,
\]
for all $y_0\in Q_j$. Now, we integrate this relation in $Q_j$ and we change variables, to obtain that 
\[
\int_{Q_j}\int_{\psi(x_0)+c_1\e}^{\psi(x_0)+c_2\e}|\nabla G_y^t(x_0,s)||u(x_0,s)|\chi_{R_{\e}}(x_0,s)\,dsdx_0\leq C\e^{\alpha}\int_{P_j}|u^*_{\e}|\,d\sigma.
\]
Hence, we have shown that
\begin{align*}
\int_{Z_i\cap R_{\e}}\left|A\nabla G_y\nabla\phi_{\e} u\right|&\leq C\sum_j\e^{\alpha-1}\int_{P_j}|u_{\e}^*|\,d\sigma\leq C\e^{\alpha-1}\int_{\partial\Omega}|u_{\e}^*|\,d\sigma\\
&\leq C\e^{\alpha}\int_{\partial\Omega}|(\nabla u)^*|\,d\sigma\xrightarrow[\e\to 0] {}0,
\end{align*}
where we also used proposition \ref{UByNablaU}. Adding those integrals over the cylinders $Z_i$ (since we have $N$ of them), we obtain that $I_1\to 0$.

We now turn to $I_2$. From proposition \ref{HolderOnTheBoundary} we obtain that
\begin{align*}
\left|\int_{Z_i\cap R_{\e}}b\nabla \phi_{\e}\cdot uG_y^t\right|&\leq C\|b\|_{\infty}\e^{\alpha-1}\int_{R_{\e}}|u|\\
&\leq C\|b\|_{\infty}\e^{\alpha-1}\sum_j\int_{Q_j}\int_{\psi(x_0)+c_1\e}^{\psi(x_0)+c_2\e}|u(x_0,s)|\,dsdx_0.
\end{align*}
But, as above, for any $y_0\in Q_j$,
\begin{align*}
\int_{Q_j}\int_{\psi(x_0)+c_1\e}^{\psi(x_0)+c_2\e}|u(x_0,s)|\,dsdx_0&\leq \int_{Q_j}\int_{\psi(x_0)+c_1\e}^{\psi(x_0)+c_2\e}|u^*_{\e}(y_0,\phi(y_0))|\,dsdx_0\\
&\leq C\e^d|u^*_{\e}(y_0,\phi(y_0))|,
\end{align*}
and after integrating on $Q_j$ and summing for $j$, we obtain that
\[
\left|\int_{Z_i\cap R_{\e}}b\nabla \phi_{\e}\cdot uG_y^t\right|\leq C\e^{\alpha}\int_{\partial\Omega}|u_{\e}^*|\,d\sigma.
\]
Since $(\nabla u)^*\in L^1(\partial\Omega)$, proposition \ref{UByNablaU} shows that $u^*\in L^1(\partial\Omega)$, therefore the last integral above converges to $0$, as $\e\to 0$. Therefore, $I_2\to 0$ as well.

Finally, we turn to $I_3$. As in the case for $I_2$, we estimate
\[
\left|\int_{Z_i\cap R_{\e}}A\nabla u\nabla\phi_{\e}\cdot G_y^t\right|\leq C\e^{\alpha-1}\int_{R_{\e}}|\nabla u|\leq C\e^{\alpha-1}\sum_j\int_{Q_j}\int_{\psi(x_0)+c_1\e}^{\psi(x_0)+c_2\e}|\nabla u(x_0,s)|\,dsdx_0.
\]
Then, for any $y_0\in Q_j$,
\begin{align*}
\int_{Q_j}\int_{\psi(x_0)+c_1\e}^{\psi(x_0)+c_2\e}|\nabla u(x_0,s)|\,dsdx_0&\leq \int_{Q_j}\int_{\psi(x_0)+c_1\e}^{\psi(x_0)+c_2\e}|(\nabla u)^*(y_0,\phi(y_0))|\,dsdx_0\\
&\leq C\e^d|(\nabla u)^*(y_0,\phi(y_0))|,
\end{align*}
and after integrating on $Q_j$ and summing for $j$, we obtain that
\[
\left|\int_{Z_i\cap R_{\e}}A\nabla u\nabla\phi_{\e}\cdot G_y^t\right|\leq C\e^{\alpha}\int_{\partial\Omega}|(\nabla u)^*|\,d\sigma.
\]
Since $(\nabla u)^*\in L^1(\partial\Omega)$, letting $\e\to 0$ shows that $I_3\to 0$ as well. This finishes the proof.
\end{proof}

\subsection{Singular integrals}
In the following, we will turn our attention to symmetric matrices $A$. 

Given a Lipschitz domain $\Omega$, we will assume that $0\in\Omega$ and $\diam(\Omega)<1/40$, and we will also set $B$ to be the unit ball in $\mathbb R^d$; we will denote this class of domains by $\mathcal{D}$.

Given a matrix $A\in M_{\lambda,\mu}(\Omega)$, we will extend $A$ periodically as in lemma \ref{AExtension}. Moreover, given a function $b\in L^{\infty}(\Omega)$, we will extend it by $0$ outside $\Omega$; if $b\in\Lip(\Omega)$, we will extend it as in lemma \ref{bExtension}. Finally, we set $G$ to be Green's function for the equation
\[
\tilde{L}u=-\dive(A\nabla u)+b\nabla u=0
\]
in $B$.

For $p,q\in\Omega$ with $p\neq q$, we define the kernel
\[
k(p,q)=\nabla_pG(p,q),
\]
where differentiation takes place with respect to the $p$ variable. We also define
\[
T_*f(p)=\sup_{\e>0}\left|\int_{|p-q|>\e}k(p,q)f(q)\,d\sigma(q)\right|,
\]
and, if $e_i$ is the unit vector in the $x_i$ direction, we set
\[
T_if(p)=\lim_{\e\to 0}\int_{|p-q|>\e}\left<k(p,q),e_i\right>f(q)\,d\sigma(q).
\]
The first operator is the maximal truncation operator, while the second is a singular integral operator; the fact that they define bounded operators will be shown in the next propositions.

\begin{prop}\label{MaximalTruncationBound}
Let $\Omega\in\mathcal{D}$, $A\in M_{\lambda,\mu}(\Omega)$ and $b\in L^{\infty}(\Omega)$. Then, the operator $T_*$ is bounded from $L^2(\partial\Omega)$ to $L^2(\partial\Omega)$, and its norm is bounded by a good constant.
\end{prop}
\begin{proof}
Consider the periodic extension $A_p$ of $A$ in $\mathbb R^d$, as in lemma \ref{PeriodicAExtension} (which we still denote by $A$) extend $b$ by $0$ outside $\Omega$, and set $G_0$ be Green's function for the equation $-\dive(A\nabla u)=0$ in $B$. Let also $\Gamma(x,y)$ to be the fundamental solution of the equation $-\dive(A\nabla u)=0$ in \[
\tilde{T}_*f(p)=\sup_{\delta>0}\left|\int_{|p-q|>\delta}\nabla_p\Gamma(p,q)\cdot f(q)\,d\sigma(q)\right|
\]
is bounded from $L^2(\partial\Omega)$ to itself, with the bound being a good constant.
	
We will now interpolate the maximal truncation operator with the analogous operator with kernel $\nabla_pG_0(p,q)$ to obtain boundedness: we write
\begin{align*}
\nabla_pG(p,q)&=\left(\nabla_pG(p,q)-\nabla_pG_0(p,q)\right)+\left(\nabla_pG_0(p,q)-\nabla_p\Gamma(p,q)\right)+\nabla_p\Gamma(p,q)\\
&=k_1(p,q)+k_2(p,q)+\nabla_p\Gamma(p,q).
\end{align*}	
After fixing $\delta>0$, multiplying with $f$ and integrating, we estimate
\begin{equation}\label{eq:Fork1A}
\int_{|p-q|>\delta}|k_1(p,q)f(q)|\,d\sigma(q)\leq C\int_{|p-q|>\delta}|p-q|^{3/2-d}|f(q)|\,d\sigma(q),
\end{equation}
from proposition \ref{ContinuityArgumentForB}. For $k_2$, we fix $q\in\partial\Omega$ and we set
\[
u(x)=G_0(x,q)-\Gamma(x,q),
\]
for $x\in B$. Then, the regularity properties of Green's function show that $u\in W^{1,\frac{d}{2(d-1)}}(B)$, and for every $\phi\in C_c^{\infty}(B)$,
\begin{align*}
\int_BA\nabla u(x)\nabla\phi(x)\,dx&=\int_BA\nabla_xG_0(x,q)\nabla\phi(x)\,dx-\int_BA\nabla_x\Gamma(x,q)\nabla\phi(x)\,dx\\
&=\phi(q)-\phi(q)=0,
\end{align*}
since $G_0$ is Green's function for $-\dive(A\nabla u)=0$ in $B$, and $\Gamma(x,y)$ is the fundamental solution of the equation $-\dive(A\nabla u)=0$ in $\mathbb R^d$. This shows that $u$ is a $W^{1,\frac{d}{2(d-1)}}(B)$ solution to $-\dive(A\nabla u)=0$ in $B$, therefore lemma \ref{LowRegularityEstimate} shows that $u$ is a $W^{1,2}(B/2)$ solution to $-\dive(A\nabla u)=0$ in $B/2$. Hence, estimate 2.2 in \cite{KenigShen} shows that the function
\[
\nabla u(x)=\nabla_xG_0(x,q)-\nabla_x\Gamma(x,q)
\]
is bounded in $B/4$, hence $k_2$ is bounded, with the bound being a good constant. Therefore
\[
\int_{|p-q|>\delta}|k_2(p,q)f(q)|\,d\sigma(q)\leq C\int_{|p-q|>\delta}|f(q)|\,d\sigma(q)
\]
for a good constant $C$.
	
We now add the last estimate with \eqref{eq:Fork1A} and we use the definition of $\tilde{T}_*$, to obtain that, for any $\delta>0$,
\[
\left|\int_{|p-q|>\delta}\nabla_pG(p,q)\cdot f(q)\,d\sigma(q)\right|\leq C\int_{|p-q|>\delta}\left(|p-q|^{3/2-d}+1\right)|f(q)|\,d\sigma(q)+\tilde{T}_*f(p),
\]
hence
\[
T_*f(p)\leq C\int_{|p-q|>\delta}\left(|p-q|^{3/2-d}+1\right)|f(q)|\,d\sigma(q)+\tilde{T}_*f(p).
\]
The fact that the kernel $|p-q|^{3/2-d}+1$ is integrable on $\partial\Omega$, together with boundedness of $\tilde{T}_*$, complete the proof.
\end{proof}

We also treat the operators $T_i$.

\begin{prop}\label{T_iBoundedness}
Let $\Omega\in\mathcal{D}$, $A\in M_{\lambda,\mu}(\Omega)$ and $b\in L^{\infty}(\Omega)$. Then, the operators $T_i$ are bounded from $L^2(\partial\Omega)$ to $L^2(\partial\Omega)$, and their norms are bounded by a good constant.
\end{prop}
\begin{proof}
The proof follows using the almost everywhere existence of the operator $T_A^1(f)$ in theorem 3.1 in \cite{KenigShen}, an argument similar to the proof of proposition \ref{MaximalTruncationBound}, and the boundedness of the operators $T_*$ shown in the same proposition.
\end{proof}

\subsection{Layer potentials}
In this section we will apply the results of the previous section to obtain the basic properties of the single layer potential. We will always assume that $\Omega\in\mathcal{D}$.

\begin{dfn}
If $f\in L^2(\partial\Omega)$, we define the \emph{single layer potential} of $f$ to be
\[
\mathcal{S}_{\pm}f:B\to\mathbb R,\,\,\,\mathcal{S}_{\pm}f(x)=\int_{\partial\Omega}G(x,q)f(q)\,d\sigma(q),
\]
where differentiation takes place with respect to $q$.
\end{dfn}
We will write $\mathcal{S}_+f(x)$ for $x\in\Omega$ and $\mathcal{S}_-f(x)$ for $x\in B\setminus\overline{\Omega}$. Note that, from the pointwise bounds on Green's function and its derivatives, the integrals in those definitions are absolutely convergent for $x\in B\setminus\partial\Omega$, since we are integrating far from $x$.

We now define, for $f\in L^2(\partial\Omega)$,
\[
\mathcal{S}f(p)=\int_{\partial\Omega}G(p,q)f(q)\,d\sigma(q).
\]
Since $G(p,q)\leq C|p-q|^{2-d}$, $G$ is integrable on $\partial\Omega$, therefore $\mathcal{S}$ is a bounded operator from $L^2(\partial\Omega)$ to $L^2(\partial\Omega)$. In fact, as the next lemma shows, $\mathcal{S}$ maps $L^2(\partial\Omega)$ to $W^{1,2}(\partial\Omega)$.

\begin{lemma}\label{SBoundedness}
Let $\Omega\in\mathcal{D}$ and $A\in M_{\lambda,\mu}(B)$, $b\in L^{\infty}(B)$. Then $S$ is bounded from $L^2(\partial\Omega)$ to $W^{1,2}(\partial\Omega)$, with
\[
\nabla_T\mathcal{S}f(p)=\lim_{\e\to 0}\int_{|p-q|>\e}\nabla_T^pG(p,q)f(q)\,d\sigma(q),
\]
in every coordinate cylinder $(Z,\phi)$ on the boundary of $\Omega$. Moreover, the norm of $S$ from $L^2(\partial\Omega)$ to $W^{1,2}(\partial\Omega)$ is a good constant.
\end{lemma}
\begin{proof}
For $f\in L^2(\partial\Omega)$, the pointwise bounds on $G$ show that
\begin{align*}
|\mathcal{S}f(p)|&=\left|\int_{\partial\Omega}G(p,q)f(q)\,d\sigma(q)\right|\leq\int_{\partial\Omega}|p-q|^{2-d}|f(q)|\,d\sigma(q)\\
&\leq\left(\int_{\partial\Omega}|p-q|^{2-d}\,d\sigma(q)\right)^{1/2}\left(\int_{\partial\Omega}|p-q|^{2-d}|f(q)|^2\,d\sigma(q)\right)^{1/2}
\end{align*}
The first integral is absolutely convergent and uniformly bounded with respect to $p$, with bounds depending only on $\Omega$. Therefore,
\[
\int_{\partial\Omega}|\mathcal{S}f|^2\leq C\int_{\partial\Omega}\int_{\partial\Omega}|p-q|^{2-d}|f(q)|^2\,d\sigma(q)d\sigma(p)\leq C\int_{\partial\Omega}|f|^2.
\]
For the second part, note that the singular integral bounds in propositions \ref{MaximalTruncationBound} and \ref{T_iBoundedness} will imply the theorem, once the formula for the tangential derivative is established. For this reason, consider a coordinate cylinder $(Z,\phi)$, where $Z$ has basis $B_Z\subseteq\mathbb R^{d-1}$, and suppose that $f$ is supported in $Z\cap\partial\Omega$. Let $h\in C_c^{\infty}(Z\cap\mathbb R^{d-1})$, and set $G_0$ to be Green's function for the equation $-\dive(A\nabla u)=0$ in $B$. If $\mathcal{S}_0$ is the single layer potential for this equation, we have that
\[
\nabla_T\mathcal{S}_0f(p)=\lim_{\e\to 0}\int_{|p-q|>\e}\nabla_T^pG_0(p,q)f(q)\,d\sigma(q),
\]
using an argument similar to the proof of proposition \ref{MaximalTruncationBound}. Therefore, it suffices to find the formula for the tangential partial derivative of the difference $\mathcal{S}f-\mathcal{S}_0f$.

Set $g_q(x)=G(x,q)-G_0(x,q)$ for $q\in Z\cap\partial\Omega$, and fix $p\in Z\cap\partial\Omega$. For any $q\in Z\cap\partial\Omega$, let $q_0\in B_Z$ be the point on the basis of $B_Z$ that lies right under $q$. Define also
\[
\Phi:B_Z\to Z\cap\partial\Omega,\,\,\Phi(y)=(y,\phi(y)),
\]
then $\Phi$ is Lipschitz, with a Lipschitz constant $C_M$ depending on $M$, the Lipschitz constant of $\partial\Omega$.
Then, for $\e>0$, we integrate by parts in $B_Z\setminus B_{\e}(q_0)$ to obtain that
\begin{multline*}
\int_{y:|y-q_0|>\e}g_q(y,\phi(y))\partial_ih(y)\,dy=\\
-\int_{\partial(B_{\e}(q_0))}g_q(y,\phi(y))h(y)\nu_i(y)\,d\sigma_{n-1}(y)-\int_{y:|y-q_0|>\e}\partial_ig_q(y,\phi(y))h(y)\,dy,
\end{multline*}
where $\partial(B_{\e}(q_0))$ is the $d-2$ dimensional boundary of $B_{\e}(q_0)$. We now apply proposition \ref{ContinuityArgumentForB}, and we note that $|\Phi(y)-q|\geq |y-q_0|$, to obtain that
\begin{align*}
\int_{\partial(B_{\e}(q_0))}\left|g_q(y,\phi(y))h(y)\nu_i(y)\right|\,d\sigma_{d-1}(y)&\leq C\int_{\partial(B_{\e}(q_0))}|\Phi(y)-q|^{5/2-d}|h(y)|\,d\sigma_{d-1}(y)\\
&\leq C_M\int_{\partial(B_{\e}(q_0))}|y-q_0|^{5/2-d}|h(y)|\,d\sigma_{d-1}(y)\\
&=C\e^{5/2-d}\int_{\partial(B_{\e}(q_0))}|h(y)|\,d\sigma_{d-1}(y)\\
&\leq C\e^{5/2-d}\sigma_{d-1}\left(\partial(B_{\e}(q_0))\right)\|h\|_{\infty}\\
&\leq C\e^{5/2-d}\e^{d-2}\|h\|_{\infty},
\end{align*}
which goes to $0$ as $\e\to 0$. This shows that
\begin{equation}\label{eq:I1I2}
\int_{y:|y-q_0|>\e}g_q(y,\phi(y))\partial_ih(y)\,dy+\int_{y:|y-q_0|>\e}\partial_ig_q(y,\phi(y))h(y)\,dy=I_{\e}^1(q)+I_{\e}^2(q)\xrightarrow[\e\to 0]{}0,
\end{equation}
for all $q\in Z\cap\partial\Omega$.

From the pointwise bounds on Green's function we now obtain that
\[
I_{\e}^1(q)\xrightarrow[\e\to 0]{}\int_{\mathbb R^{d-1}}g_q(y,\phi(y))\partial_ih(y)\,dy.
\]
Moreover,
\[
|I_{\e}^1(q)|\leq C\int_{y:|y-q_0|>\e}|\Phi(y)-q|^{5/2-d}|\partial_ih(y)|\,dy\leq C\int_{Z\cap\partial\Omega}|q'-q|^{5/2-d}\,d\sigma(q'),
\]
and the last function is bounded with respect to $q$. Therefore, the dominated convergence theorem shows that
\begin{multline*}
\int_{Z\cap\partial\Omega}\int_{y:|y-q_0|>\e}g_q(y,\phi(y))\partial_ih(y)\cdot f(q)\,dyd\sigma(q)\xrightarrow[\e\to 0]{}\\\int_{Z\cap\partial\Omega}\int_{\mathbb R^{d-1}}g_q(y,\phi(y))\partial_ih(y)\cdot f(q)\,dyd\sigma(q).
\end{multline*}
Applying a similar procedure to $I_{\e}^2(q)$ and using \eqref{eq:I1I2}, we obtain that
\[
\int_{\partial\Omega}\int_{\mathbb R^{d-1}}g_q(\Phi(y))\partial_ih(y)f(q)\,dyd\sigma(q)=-\int_{\partial\Omega}\int_{\mathbb R^{d-1}}\partial_ig_q(\Phi(y))h(y)f(q)\,dyd\sigma(q).
\]
But, from Fubini's theorem, note that the integral on the left is equal to
\[
\int_{\mathbb R^{d-1}}\int_{\partial\Omega}g_q(\Phi(y))f(q)\partial_ih(y)\,d\sigma(q)dy=\int_{\mathbb R^{d-1}}\left(\mathcal{S}f(\Phi(y))-\mathcal{S}_0f(\Phi(y))\right)\partial_ih(y)\,dy,
\]
therefore we obtain that
\[
\int_{\mathbb R^{d-1}}\left(\mathcal{S}f(\Phi(y))-\mathcal{S}_0f(\Phi(y))\right)\partial_ih(y)\,dy=\int_{\mathbb R^{d-1}}\left(\int_{\partial\Omega}\partial_ig_q(\Phi(y))f(q)\,d\sigma(q)\right)h(y)\,dy
\]
From definition \ref{W1pPartialOmega}, this shows that the difference $\mathcal{S}f-\mathcal{S}_0f$ is tangentially differentiable on $\partial\Omega$, with
\[
\nabla_T\left(\mathcal{S}f(p)-\mathcal{S}_0f(p)\right)=\int_{\partial\Omega}\nabla_T^p\left(G(p,q)-G_0(p,q)\right)f(q)\,d\sigma(q).
\]
Finally, we combine with the formula for $\nabla_T\mathcal{S}_0f(p)$, and the proof is complete.
\end{proof}

The single layer potential is the solution to the Dirichlet problem with data $\mathcal{S}f$ on $\partial\Omega$; this is shown in the next proposition.

\begin{prop}\label{SingleLayerInequalities}
Let $\Omega\in\mathcal{D}$, $A\in M_{\lambda,\mu}(B)$ and $b\in L^{\infty}(B)$. If $f\in L^2(\partial\Omega)$, then $\mathcal{S}_+f\in W_{\loc}^{1,2}(\Omega)$ is the solution to the Dirichlet problem $D_2$ for the equation $Lu=0$ in $\Omega$, with boundary values $\mathcal{S}f$ on $\partial\Omega$. Similarly, $\mathcal{S}_-f$ is the solution to $D_2$ in $B\setminus\overline{\Omega}$, and has boundary values $\mathcal{S}f\cdot\chi_{\partial\Omega}$ on $\partial(B\setminus\Omega)$. In addition,
\[
\|(\nabla\mathcal{S}_{\pm}f)^*\|_{L^2(\partial\Omega)}\leq C\|f\|_{L^2(\partial\Omega)},
\]
where $C$ is a good constant.
\end{prop}
\begin{proof}
We will compute the weak derivatives of $\mathcal{S}_+f$: if $\phi\in C_c^{\infty}(\Omega)$, then
\[
\int_{\Omega}\mathcal{S}_+f(x)\partial_i\phi(x)\,dx=\int_{\partial\Omega}\left(\int_{\Omega}G(x,q)\partial_i\phi(x)\,dx\right)f(q)\,d\sigma(q),
\]
from Fubini's theorem, since $\partial_i\phi$ is supported at a positive distance from $\partial\Omega$. So, from differentiability of $G(\cdot,q)$ and Fubini's theorem, we obtain that
\[
\int_{\partial\Omega}\left(\int_{\Omega}G(x,q)\partial_i\phi(x)\,dx\right)f(q)\,d\sigma(q)=-\int_{\partial\Omega}\left(\int_{\Omega}\partial_iG(x,q)\phi(x)\,dx\right)f(q)\,d\sigma(q),
\]
which shows that
\[
\nabla\mathcal{S}_+f(x)=\int_{\partial\Omega}\nabla_xG(x,q)f(q)\,d\sigma(q).
\]
With a similar procedure, we can show that $\nabla\mathcal{S}_+f\in L^2_{{\rm loc}}(\Omega)$. Now, let $\phi\in C_c^{\infty}(\Omega)$. We then have
\begin{align*}
\alpha(\mathcal{S}_+f,\phi)&=\int_{\Omega}A\nabla\mathcal{S}_+f\nabla\phi+b\nabla\mathcal{S}_+f\cdot\phi\\
&=-\int_{\Omega}\int_{\partial\Omega}\left(A(x)\nabla_xG(x,q)\nabla\phi(x)+b(x)\nabla G(x,q)\cdot\phi(x)\right)f(q)\,d\sigma(q)dx\\
&=\int_{\partial\Omega}\left(\int_{\Omega}A(x)\nabla_xG(x,q)\nabla\phi(x)+b(x)\nabla G(x,q)\cdot\phi(x)\,dx\right)f(q)\,d\sigma(q),
\end{align*}
from Fubini's theorem, since the integrand is supported on a strictly positive distance from $\partial\Omega$, hence the integral converges absolutely. But, the inner integral is equal to $\phi(q)=0$, so $\mathcal{S}_+f$ is indeed a solution in $\Omega$.

We now turn to the boundary values of $\mathcal{S}_+f$. Let $A_f\subseteq\partial\Omega$ be the set of $p\in\partial\Omega$ such that $\mathcal{S}f(p)$ is finite. Fix $p\in A_f$, and let $x\in\Gamma(p)$. Let also $n\in\mathbb N$, and suppose that $|x-p|<\frac{1}{2n}$. We then bound $|\mathcal{S}_+f(x)-\mathcal{S}f(p)|$ by
\[
\int_{\Omega\setminus\Delta_{1/n}(p)}|G(x,q)-G(p,q)||f(q)|\,d\sigma(q)+\int_{\Delta_{1/n}(p)}|G(x,q)-G(p,q)||f(q)|\,d\sigma(q)=I_1+I_2.
\]
To bound $I_1$: if $|p-q|\geq 1/n$, then
\[
|x-q|\geq|p-q|-|x-p|\geq\frac{|p-q|}{2}+\frac{1}{2n}-|x-p|\geq\frac{|p-q|}{2},
\]
and, if we use Lipschitz continuity of Green's function (proposition \ref{HolderContinuityOfGreenAsIs}), we obtain that 
\begin{align*}
I_1&\leq C|x-p|\int_{\Omega\setminus\Delta_{1/n}(p)}\left(|x-q|^{1-d}+|p-q|^{1-d}\right)|f(q)|\,d\sigma(q)\\
&\leq C|x-p|\int_{\Omega\setminus\Delta_{1/n}(p)}\left(\left(\frac{|p-q|}{2}\right)^{1-d}+|p-q|^{1-d}\right)|f(q)|\,d\sigma(q)\\
&\leq C|x-p|\int_{\Omega\setminus\Delta_{1/n}(p)}|p-q|^{1-d}|f(q)|\,d\sigma(q)\\
&\leq C|x-p|n^{1-d}\|f\|_2^2.
\end{align*}
For $I_2$, note first that since $x\in\Gamma(q)$, we have that $|x-p|\leq C\delta(x)$, therefore
\begin{equation}\label{eq:Gamma0}
q\in\partial\Omega\Rightarrow |p-q|\leq|x-q|+|x-p|\leq|x-q|+C\delta(x)\leq C|x-q|.
\end{equation}
Using the pointwise bounds on $G$, we then obtain
\begin{align*}
I_2&\leq\int_{\Delta_{1/n}(p)}\left(|G(x,q)|+|G(p,q)|\right)|f(q)|\,d\sigma(q)\\
&\leq C\int_{\Delta_{1/n}(p)}\left(|x-q|^{2-d}+|p-q|^{2-d}\right)|f(q)|\,d\sigma(q)\\
&\leq C\int_{\Delta_{1/n}(p)}|p-q|^{2-d}|f(q)|\,d\sigma(q)\leq C\int_{\partial\Omega}G(p,q)|f(q)\chi_{\Delta_{1/n}(q)}|\,d\sigma(q).
\end{align*}
Combining the estimates for $I_1$ and $I_2$, we finally obtain that, if $p\in A_f$, $|x-p|<\frac{1}{2n}$ and $x\in\Gamma(p)$, then
\begin{equation}\label{eq:SBound}
|\mathcal{S}_+f(x)-\mathcal{S}f(p)|\leq C|x-p|n^{1-d}\|f\|_2^2+C\int_{\partial\Omega}G(p,q)|f(q)\chi_{\Delta_{1/n}(q)}|\,d\sigma(q).
\end{equation}
Since the kernel $G(p,q)$ in integrable on $\partial\Omega$ and $f\chi_{\Delta_{1/n}}\xrightarrow[n\to\infty]{\|\cdot\|_2}0$, we obtain that
\[
\int_{\partial\Omega}G(p,q)|f(q)\chi_{\Delta_{1/n}(q)}|\,d\sigma(q)\xrightarrow[n\to\infty]{\|\cdot\|_2}0,
\] 
therefore, for a subsequence of the $n$, we obtain pointwise convergence to $0$ for all $p\in B_f$, where the set $B_f\subseteq A_f$ has full measure.

Let now $p\in B_f$ and $\e>0$. Then there exists $N\in\mathbb N$ such that
\[
\left|\int_{\partial\Omega}G(p,q)|f(q)\chi_{\Delta_{1/N}(q)}|\,d\sigma(q)\right|<\frac{\e}{2}.
\]
Then, for all $x\in\Gamma(p)$ with $|x-p|<1/2N$ and
\[
C|x-p|N^{1-d}\|f\|_2^2<\frac{\e}{2},
\]
we obtain from estimate \eqref{eq:SBound} that $|\mathcal{S}_+f(x)-\mathcal{S}f(p)|<\e$. Hence, $\mathcal{S}_+f(x)\to\mathcal{S}f(p)$ nontangentially, almost everywhere on $\partial\Omega$.

We now show the boundedness of the nontangential maximal function of the gradient. For this purpose, let $G_0$ be Green's function for the equation $-\dive(A\nabla u)=0$ in $B$, and let $\mathcal{S}_+^0$ be the single layer potential for the same equation in $\Omega$. Let also $p\in\partial\Omega$ and $x\in\Gamma(p)$. We then write
\begin{align*}
|\nabla\mathcal{S}_+f(x)|&\leq|\nabla\mathcal{S}_+f(x)-\nabla\mathcal{S}_+^0f(x)|+|\nabla\mathcal{S}_+^0f(x)|\\
&\leq\int_{\partial\Omega}\left|\nabla_xG(x,q)-\nabla_xG_0(x,q)\right||f(q)|\,d\sigma(q)+|\nabla\mathcal{S}_+^0f(x)|\\
&\leq C\int_{\partial\Omega}|x-q|^{3/2-d}|f(q)|\,d\sigma(q)+|\nabla\mathcal{S}_+^0f(x)|\\
&\leq C\int_{\partial\Omega}|p-q|^{3/2-d}|f(q)|\,d\sigma(q)+|\nabla\mathcal{S}_+^0f(x)|,
\end{align*}
where we also used proposition \ref{ContinuityArgumentForB} and the fact that $x\in\Gamma(p)$. This shows that
\[
\left(\nabla\mathcal{S}_+f\right)^*(p)\leq C\int_{\partial\Omega}|p-q|^{3/2-d}|f(q)|\,d\sigma(q)+\left(\nabla\mathcal{S}_+^0f\right)^*(p),
\]
for all $p\in\partial\Omega$. Similarly to the proof of proposition \ref{MaximalTruncationBound}, the operator $(\nabla\mathcal{S}_+^0f)^*$ is bounded from $L^2(\partial\Omega)$ to $L^2(\partial\Omega)$; hence, after integrating over $\partial\Omega$, we obtain that
\[
\int_{\partial\Omega}\left(\nabla\mathcal{S}_+f\right)^*\,d\sigma\leq C\int_{\partial\Omega}|f|^2\,d\sigma,
\]
as we wanted.

Finally, we show that $(\mathcal{S}_+f)^*, (\mathcal{S}_-f)^*|_{\partial\Omega}\in L^2(\partial\Omega)$: for this purpose, note that proposition \ref{UByNablaU} shows that
\[
(\mathcal{S}_+f)^*\leq C(\nabla\mathcal{S}_+f)^*+|\mathcal{S}f|,
\]
and boundedness follows after integrating on $\partial\Omega$, using the bound on the nontangential maximal function of the gradient, and lemma \ref{SBoundedness}. A similar argument works for $(\mathcal{S}_-f)^*$. Finally, to show that $(\mathcal{S}_-f)^*|_{\partial B}\in L^2(\partial B)$ we use th., which completes the proof.
\end{proof}

\subsection{Behavior on the boundary, and the jump relations}
The goal of this section is to study how the derivative of the single layer potential $\nabla\mathcal{S}_+f(x)$ in $\Omega$ behaves as $x\to p\in\partial\Omega$. We start with the tangential derivative first.

The next proposition shows that, in fact, the tangential derivatives of $\mathcal{S}_+f,\mathcal{S}_-f$ are nontangentially continuous on $\partial\Omega$.

\begin{prop}\label{DerivativeIsNonTangentiallyContinuous}
Let $\Omega\in\mathcal{D}$, $A\in M_{\lambda,\mu}(B)$ and $b\in L^{\infty}(B)$. Then, for any $f\in L^2(\partial\Omega)$,
\[
\nabla\mathcal{S}_+f(x)\cdot T(p)\xrightarrow[x\to p]{}\nabla_T\mathcal{S}f(p),
\]
non-tangentially, almost everywhere on $\partial\Omega$.
\end{prop}
\begin{proof}
Consider the operator $R:L^2(\partial\Omega)\to L^2(\partial\Omega)$, with
\[
Rf(p)=\int_{\Omega}|p-q|^{5/4-d}f(q)\,d\sigma(q),
\]
then $R$ is bounded. Therefore, $|Rf|<\infty$ almost everywhere on $\partial\Omega$.

Let $\delta_0>0$. We begin by writing
\begin{multline}\label{eq:S_+}
\nabla\mathcal{S}_+f(x)\cdot T(p)-\int_{|p-q|>\delta_0}\nabla_TG(p,q)f(q)\,d\sigma(q)=\\
\int_{|p-q|\geq\delta_0}\left(\nabla_xG(x,q)-\nabla_xG(p,q)\right)T(p)\cdot f(q)\,d\sigma(q)+\int_{|p-q|<\delta_0}\nabla_xG(x,q)T(p)\cdot f(q)\,d\sigma(q),
\end{multline}
where $T(p)$ is a tangential vector on $\partial\Omega$ at $p$. Note then that, if $x\in\Gamma(p)$, then
\[
|x-q|\geq\delta(x)\geq C|x-p|,
\]
therefore the last term is bounded by
\begin{align*}
\left|\int_{|p-q|<\delta_0}\nabla_xG(x,q)T(p)\cdot f(q)\,d\sigma(q)\right|&\leq C\int_{|p-q|<\delta_0}|x-q|^{1-d}|f(q)|\,d\sigma(q)\\
&\leq C|x-p|^{1-d}\int_{|p-q|<\delta_0}|f(q)|\,d\sigma(q)\\
&\leq C|x-p|^{1-d}\delta_0^{(d-1)/2}\|f\|_{L^2(\partial\Omega)}.
\end{align*}
The last relation and \eqref{eq:S_+} show that for almost all $p\in\partial\Omega$, and any $x\in\Gamma(p)$,
\begin{equation}\label{eq:BetterS_+}
\nabla\mathcal{S}_+f(x)\cdot T(p)-\nabla_T\mathcal{S}f(p)=
\lim_{\delta_0\to 0}\int_{|p-q|\geq\delta_0}\left(\nabla_xG(x,q)-\nabla_xG(p,q)\right)T(p)\cdot f(q)\,d\sigma(q).
\end{equation}

Similarly, if $\mathcal{S}_+^0$ is the single layer potential for the equation $-\dive(A\nabla u)=0$ in $\Omega$, and $G_0$ is Green's function for the same equation in $B$, we obtain that, for almost every $p\in\partial\Omega$, and every $x\in\Gamma(p)$,
\begin{equation}\label{eq:S_+^0}
\nabla\mathcal{S}_+^0f(x)\cdot T(p)-\nabla_T\mathcal{S}^0f(p)=
\lim_{\delta_0\to 0}\int_{|p-q|\geq\delta_0}\left(\nabla_xG^0(x,q)-\nabla_xG^0(p,q)\right)T(p)\cdot f(q)\,d\sigma(q).
\end{equation}
If we now subtract \eqref{eq:S_+^0} from \eqref{eq:S_+}, we obtain that
\begin{multline}\label{eq:S_+-S_+^0}
\left|\nabla(\mathcal{S}_+f-\mathcal{S}_+^0f)(x)\cdot T(p)-\nabla_T(\mathcal{S}f-\mathcal{S}_0f)(p)\right|\\
\leq\lim_{\delta_0\to 0}\int_{|p-q|\geq\delta_0}\left|\nabla_xG(x,q)-\nabla_xG(p,q)+\nabla_xG^0(x,q)-\nabla_xG^0(p,q)\right||f(q)|\,d\sigma(q)\\
\leq\int_{\partial\Omega}\left|\nabla_xG(x,q)-\nabla_xG(p,q)+\nabla_xG^0(x,q)-\nabla_xG^0(p,q)\right||f(q)|\,d\sigma(q).
\end{multline}
Let now $\delta>0$. Note that, if $x\in\Gamma(p)$, then
\[
|p-q|\leq|x-p|+|x-q|\leq C\delta(x)+|x-q|\leq C|x-q|,
\]
therefore, proposition \ref{HolderContinuityOfGreen} shows that
\begin{multline}\label{eq:ForNablaG}
\left|\int_{|p-q|\geq\delta}\left(\nabla_xG(x,q)-\nabla_xG(p,q)\right)T(p)\cdot f(q)\,d\sigma(q)\right|\leq\\
C|x-p|^{\alpha}\int_{|p-q|\geq\delta}|p-q|^{1-d-\alpha}|f(q)|\,d\sigma(q)\leq C|x-p|^{\alpha}\delta^{1-d-\alpha} \int_{|p-q|\geq\delta}|f(q)|\,d\sigma(q).
\end{multline}
Similarly, we obtain the same estimate with $G^0$ in the place of $G$.

Finally, we apply proposition \ref{ContinuityArgumentForB} to obtain that
\begin{multline}\label{eq:ForTheDifferenceOfNablaG}
\left|\int_{|p-q|<\delta}\left(\nabla_xG(x,q)-\nabla_xG(p,q)-\nabla_xG^0(x,q)+\nabla_xG^0(p,q)\right)T(p)\cdot f(q)\,d\sigma(q)\right|\leq\\
\int_{|p-q|<\delta}\left(|x-q|^{3/2-d}+|p-q|^{3/2-d}\right)|f(q)|\,d\sigma(q)\leq\\
\int_{|p-q|<\delta}|p-q|^{3/2-d}|f(q)|\,d\sigma(q),
\end{multline}
since $x\in\Gamma(p)$. Then, for the last term, we estimate
\[
\int_{|p-q|<\delta}|p-q|^{3/2-d}|f(q)|\,d\sigma(q)\leq\delta^{1/4}\int_{\partial\Omega}|p-q|^{3/2-d}|f(q)|\,d\sigma(q)\leq\delta^{1/4}R|f|(p).
\]
We now plug this estimate in \eqref{eq:ForTheDifferenceOfNablaG}, and we substitute in \eqref{eq:S_+-S_+^0}, together with \eqref{eq:ForNablaG} and its analog for $G^0$: then, we obtain that, for almost all $p\in\partial\Omega$, and almost all $x\in\Gamma(p)$,
\[
\left|\nabla(\mathcal{S}_+f-\mathcal{S}_+^0f)(x)\cdot T(p)-\nabla_T(\mathcal{S}f-\mathcal{S}^0f)(p)\right|\leq
C|x-p|^{\alpha}\delta^{1-d-\alpha}\int_{\partial\Omega}|f|\,d\sigma+\delta^{1/4}R|f|(p).
\]
Let now $A$ be the set of $p\in\partial\Omega$ for which $|Rf(p)|<\infty$, and consider any $p\in A$. Let $\e>0$, then there exists $\delta>0$, such that $\delta^{1/4}R|f|(p)<\e$. Therefore, if $x\in\Gamma(p)$ is such that
\[
C|x-p|^{\alpha}\delta^{1-d-\alpha}\int_{\partial\Omega}|f|\,d\sigma<\e,
\]
we obtain that 
\[
\left|\nabla(\mathcal{S}_+f-\mathcal{S}_+^0f)(x)\cdot T(p)-\nabla_T(\mathcal{S}f-\mathcal{S}^0f)(p)\right|\leq 2\e,
\]
which shows that
\[
\nabla(\mathcal{S}_+f-\mathcal{S}_+^0f)(x)\cdot T(p)\xrightarrow[x\to p]{}\nabla_T(\mathcal{S}f-\mathcal{S}^0f)(p),
\]
nontangentially, almost everywhere. Since now, from theorem 4.4 in \cite{KenigShen} and an argument similar to proposition \ref{MaximalTruncationBound}, $\nabla_T\mathcal{S}_+^0f(x)$ converges to $\nabla_T\mathcal{S}^0f(p)$ nontangentially, almost everywhere, we obtain
\[
\nabla\mathcal{S}_+f(x)\cdot T(p)\xrightarrow[x\to p]{}\nabla_T\mathcal{S}f(p),
\]
nontangentially, almost everywhere, which completes the proof.
\end{proof}

The convergence shown above leads to the following corollary, which will be useful in an approximation argument later.

\begin{cor}\label{LimsupEstimateOnNorm}
Let $\Omega\in\mathcal{D}$, $A\in M_{\lambda,\mu}(B)$ and $b\in L^{\infty}(B)$. Then, for all $f\in L^2(\partial\Omega)$,
\[
\limsup_{j\to\infty}\int_{\partial\Omega_j}|\nabla_{T_j}\mathcal{S}_+f|^2\,d\sigma_j\leq \|\nabla_T\mathcal{S}f\|_{L^2(\partial\Omega)}^2,
\]
where $\Omega_j\uparrow\Omega$ is the approximation scheme in theorem \ref{ApproximationScheme} and $\nabla_{T_j}$ denotes the tangential gradient on $\partial\Omega_j$. The analogous inequality also holds for $S_-f$, using the domains $\Omega_j'\downarrow\Omega$.
\end{cor}
\begin{proof}
Recall the approximation scheme $\Omega_j\uparrow\Omega$, from theorem \ref{ApproximationScheme}. From the same theorem, we have that
\[
T_j\circ\Lambda_j\to T,
\]
where the $T_j$ are locally defined tangent vectors, and convergence takes place in $L^2$ and almost everywhere. Using proposition \ref{DerivativeIsNonTangentiallyContinuous}, the dominated convergence theorem, and the fact that $(\nabla\mathcal{S}_+f)^*\in L^2(\partial\Omega)$ from proposition \ref{SingleLayerInequalities}, we obtain that
\[
\limsup_{j\to\infty}\int_{\partial\Omega}|\nabla_{T_j}\mathcal{S}_+f\circ\Lambda_j|^2\,d\sigma\leq\int_{\partial\Omega}|\nabla_T\mathcal{S}f|^2\,d\sigma.
\]
Since $\tau_j\to 1$ from theorem \ref{ApproximationScheme}, after changing variables we obtain that
\[
\limsup_{j\to\infty}\int_{\partial\Omega_j}|\nabla_{T_j}\mathcal{S}_+f|^2\,d\sigma_j\leq\int_{\partial\Omega}|\nabla_T\mathcal{S}f|^2\,d\sigma,
\]
which completes the proof.
\end{proof}

The next property we show is discontinuity of the normal derivative of the single layer potential across the boundary of a domain $\Omega$. In the calculations that follow, we will need to assume that $b$ is Lipschitz in $\Omega$, therefore we will consider an extension of $b$ in $B$ using lemma \ref{bExtension}.

\begin{lemma}\label{ContinuousInside}
Let $\Omega\in\mathcal{D}$, let $A\in M_{\lambda,\mu}(B)$, $b\in\Lip(B)$ and consider a Lipschitz function $F:\overline{B}\to\mathbb R$ which vanishes on $\partial B$. Then, for $x\in\Omega$,
\[
\int_{\partial\Omega}\partial_{\nu}^qG^t(x,q)\cdot F(q)\,d\sigma(q)=-\int_{B\setminus\Omega}A\nabla G_x\nabla F+b\nabla G_x\cdot F,
\]
while, for $x\in B\setminus\overline{\Omega}$,
\[
\int_{\partial\Omega}\partial_{\nu}^qG^t(x,q)\cdot F(q)\,d\sigma(q)=\int_{\Omega}A\nabla G_x\nabla F+b\nabla G_x\cdot F,
\]
where $\partial_{\nu}^q$ denotes the conormal derivative with respect to $q$ on $\partial\Omega$, associated to $L$, and $G_x(\cdot)=G^t(x,\cdot)$.
\end{lemma}
\begin{proof}
Suppose first that $x\in\Omega$. Then, from proposition \ref{SymmetryWithAdjoint} and theorems 8.8 and 8.12 in \cite{Gilbarg}, $G_x(\cdot)=G(\cdot,x)$ and $G_x$ is a classical solution of $Lu=0$ in $B\setminus\Omega$; that is,
\[
\dive(A\nabla G_x)=b\nabla G_x,
\]
almost everywhere in $B$, away from $x$, and $G_x\in W^{2,2}(B\setminus B_{\e}(x))$, for all small $\e>0$. Therefore, since $F\equiv 0$ on $\partial B$, and from proposition \ref{DerivativeRegularity} the derivative of $G_x$ is continuous away from $x$ and the boundary of $B$,
\begin{align*}
\int_{\partial\Omega}\partial_{\nu}G_x\cdot F\,d\sigma&=-\int_{\partial(B\setminus\Omega)}\partial_{\nu}G_x\cdot F\,d\sigma=-\int_{B\setminus\Omega}\dive(F\cdot A\nabla G_x)\\
&=-\int_{B\setminus\Omega}A\nabla G_x\nabla F+\dive(A\nabla G_x)\cdot F=-\int_{B\setminus\Omega}A\nabla G_x\nabla F+b\nabla G_x\cdot F,
\end{align*}
because $G_x$ is a classical solution of $Lu=0$ $B\setminus\Omega$. Now, if $x\in B\setminus\overline{\Omega}$, then $G_x$ is a classical solution of $Lu=0$ in $\Omega$, therefore
\[
\int_{\partial\Omega}\partial_{\nu}G_x\cdot F\,d\sigma=\int_{\Omega}\dive(F\cdot A\nabla G_x)=\int_{\Omega}A\nabla G_x\nabla F+b\nabla G_x\cdot F,
\]
which concludes the proof.
\end{proof}

\begin{lemma}\label{IntegratesTo1}
Let $B$ be a ball, and let $A\in M_{\lambda,\mu}(B)$, and $b\in\Lip(B)$.  Then, for all $p\in B$,
\[
\lim_{\e\to 0}\int_{\partial B_{\e}(p)}\partial_{\nu}^qG^t(p,q)\,d\sigma_{B_{\e}}(q)=-1,
\]
where $\partial_{\nu}$ is the conormal derivative associated with $L$.
\end{lemma}
\begin{proof}
Let $\e>0$ such that $B_{\e}(p)\subseteq B$, and consider the domain $U_{\e}=B\setminus B_{\e}(p)$. Set also $G_p(\cdot)=G^t(p,\cdot)$. As in the proof of lemma \ref{ContinuousInside}, $G_p$ is a classical solution of the equation
\[
\dive(A\nabla G_p)=b\nabla G_p
\]
away from $p$, therefore
\[
\int_{U_{\e}}\dive(A\nabla G_p)=\int_{U_{\e}}b\nabla G_p\xrightarrow[\e\to 0]{}\int_Bb\nabla G_p,
\]
since $b\nabla G_p$ is integrable in $B$. We then integrate by parts, to obtain
\[
\int_{U_{\e}}\dive(A\nabla G_p)=\int_{\partial U_{\e}}\partial_{\nu}G_p\,d\sigma_{U_{\e}}=\int_{\partial B}\partial_{\nu}G_p\,d\sigma_B-\int_{\partial B_{\e}(p)}\partial_{\nu}G_p\,d\sigma_{B_{\e}(p)},
\]
therefore
\begin{equation}\label{eq:NoI(p)Needed}
\int_{\partial B_{\e}(p)}\partial_{\nu}G_p\,d\sigma_{B_{\e}(p)}\xrightarrow[\e\to 0]{}\int_{\partial B}\partial_{\nu}G_p\,d\sigma_B-\int_Bb\nabla G_p.
\end{equation}
Let now $\phi\in C_c^{\infty}(B)$ be a smooth cutoff which is equal to $1$ in a small ball $B_p$ that is centered at $p$, and vanishes outside a ball $B_p'$ that contains $B_p$. Then $\psi=1-\phi$ is identically $1$ outside $B_p'$ and identically $0$ inside $B_p$, therefore the divergence theorem shows that
\begin{align*}
\int_{\partial B}\partial_{\nu}G_p\,d\sigma_B&=\int_{\partial B}\partial_{\nu}G_p\cdot\psi\,d\sigma_B=\int_{\partial B}\left<A\nabla G_p\cdot\psi,\nu\right>\,d\sigma_B=\int_{B\setminus B_p}\dive(A\nabla G_p\cdot\psi)\\
&=\int_{B\setminus B_p}A\nabla G_p\nabla\psi+\dive(A\nabla G_p)\cdot\psi=\int_{B\setminus B_p}A\nabla G_p\nabla\psi+b\nabla G_p\cdot\psi\\
&=\int_BA\nabla G_p\nabla\psi+b\nabla G_p\cdot\psi=\int_BA\nabla G_p\nabla(1-\phi)+b\nabla G_p\cdot(1-\phi)\\
&=\int_Bb\nabla G_p-\int_BA\nabla G_p\nabla\phi+b\nabla G_p\cdot\phi=\int_Bb\nabla G_p-\phi(p)\\
&=\int_Bb\nabla G_p-1,
\end{align*}
where we also used that $G_p$ is a classical solution of $Lu=0$ away from $p$. We then plug this into \eqref{eq:NoI(p)Needed} to conclude the proof.
\end{proof}

We now define, for a Lipschitz function $f:\partial\Omega\to\mathbb R$ and $p\in\partial\Omega$,
\[
\mathcal{K}f(p)=\lim_{\e\to 0}\int_{|p-q|>\e}\partial_{\nu}^qG^t(p,q)F(q)\,d\sigma(q).
\]
The fact that this limit exists is shown in the next lemma.

\begin{lemma}\label{JumpFor1}
Let $\Omega\in\mathcal{D}$, let $A\in M_{\lambda,\mu}(B)$, $b\in\Lip(B)$, and consider a Lipschitz function $F:\overline{B}\to\mathbb R$ which vanishes on $\partial B$. Then, for almost all $p\in\partial\Omega$,
\begin{align*}
\mathcal{K}f(p)&=\frac{1}{2}F(p)-\int_{B\setminus\Omega}A\nabla G_p\nabla F+b\nabla G_p\cdot F\\
&=-\frac{1}{2}F(p)+\int_{\Omega}A\nabla G_p\nabla F+b\nabla G_p\cdot F.
\end{align*}
\end{lemma}
\begin{proof}
Let $V_{\e}=\Omega\cup B_{\e}(p)$. We also define
\[
\partial^1_{\e}=\Omega^c\cap\partial(B_{\e}(p)),\,\,\partial^2_{\e}=\Omega\cap\partial(B_{\e}(p)),
\]
and we write
\begin{align*}
\int_{\partial\Omega\setminus\Delta_{\e}(p)}\partial_{\nu}G_p\cdot F\,d\sigma&=\int_{\partial V_{\e}}\partial_{\nu}G_p\cdot F\,d\sigma-\int_{\partial^1_{\e}}\partial_{\nu}G_p\cdot F\,d\sigma\\
&=\int_{\partial B}\partial_{\nu}G_p\cdot F\,d\sigma-\int_{\partial(B\setminus V_{\e})}\partial_{\nu}G_p\cdot F\,d\sigma-\int_{\partial^1_{\e}}\partial_{\nu}G_p\cdot F\,d\sigma\\
&=0-I_1-I_2,
\end{align*}
since $F$ vanishes on $\partial B$.

We now treat $I_1$. As in the proof of lemma \ref{ContinuousInside}, using the divergence theorem we compute
\[
I_1=\int_{\partial(B\setminus V_{\e})}\partial_{\nu}G_p\cdot F\,d\sigma=\int_{B\setminus V_{\e}}\dive(F\cdot A\nabla G_p)=\int_{B\setminus V_{\e}}A\nabla G_p\nabla F+b\nabla G_p\cdot F,
\]
Then, since the terms $A\nabla G_p\nabla F,b\nabla G_p$ are integrable in $B\setminus\Omega$, we obtain that
\[
I_1\to\int_{B\setminus\Omega}A\nabla G_p\nabla F+b\nabla G_p\cdot F.
\]
For $I_2$, we write
\[
I_2=\int_{\partial^1_{\e}}\partial_{\nu}G_p\cdot F\,d\sigma=\int_{\partial^1_{\e}}\partial_{\nu}G_p\cdot (F-F(p))\,d\sigma+F(p)\int_{\partial^1_{\e}}\partial_{\nu}G_p\,d\sigma=I_3+I_4.
\]
From Lipschitz continuity of $F$ and the pointwise bounds on the gradient of $G$, we obtain that
\[
|I_3|\leq C\int_{\partial^1_{\e}}|p-p'|^{1-d}|F(p')-F(p)|\,d\sigma(p')\leq C\e^{2-d}\sigma_{d-1}(\partial B_{\e}(p))\xrightarrow[\e\to 0]{}0.
\]
For $I_4$, for almost all $p\in\partial\Omega$ there exists a well defined tangent plane to $\partial\Omega$ at $p$. For those $p$, the symmetric difference between $\partial_{\e}^1$ and $\partial_{\e}^2$ is contained in a strip
\[
A_{\e}(p)=\{y\in B_{\e}(p)\big{|}|y\cdot\nu(p)|\leq C\e^2\},
\]
(as in \cite[p.~125]{Folland}), and if we combine with the pointwise bounds for the gradient of $G$, we obtain that
\[
\int_{\partial_{\e}^1}\partial_{\nu}G_p-\int_{\partial_{\e}^2}\partial_{\nu}G_p\xrightarrow[\e\to 0]{}0.
\]
Using lemma \ref{IntegratesTo1}, we then obtain that
\begin{align*}
\int_{\partial^1_{\e}}\partial_{\nu}G_p\,d\sigma&=\frac{1}{2}\left(\int_{\partial^1_{\e}}\partial_{\nu}G_p\,d\sigma+\int_{\partial^2_{\e}}\partial_{\nu}G_p\,d\sigma\right)+\frac{1}{2}\left(\int_{\partial^1_{\e}}\partial_{\nu}G_p\,d\sigma-\int_{\partial^2_{\e}}\partial_{\nu}G_p\,d\sigma\right)\\
&=\frac{1}{2}\int_{B_{\e}(p)}\partial_{\nu}G_p\,d\sigma+\frac{1}{2}\left(\int_{\partial^1_{\e}}\partial_{\nu}G_p\,d\sigma-\int_{\partial^2_{\e}}\partial_{\nu}G_p\,d\sigma\right)\xrightarrow[\e\to 0]{}-\frac{1}{2},
\end{align*}
therefore $I_2\to-\frac{1}{2}F(p)$. Therefore,
\[
\mathcal{K}F(p)=\lim_{\e\to 0}(-I_1-I_2)=-\frac{I(p)}{2}F(p)-\int_{B\setminus\Omega}A\nabla G_p\nabla F+b\nabla G_p\cdot F.
\]
For the second representation, using the first representation in this lemma, we write
\begin{align*}
F(p)&=\int_BA\nabla G_p\nabla F+b\nabla G_p\cdot F\\
&=\int_{\Omega}A\nabla G_p\nabla F+b\nabla G_p\cdot F+\int_{B\setminus\Omega}A\nabla G_p\nabla F+b\nabla G_p\cdot F\\
&=\int_{\Omega}A\nabla G_p\nabla F+b\nabla G_p\cdot F+\frac{1}{2}F(p)-\mathcal{K}F(p),
\end{align*}
which concludes the proof after rearranging the terms.
\end{proof}

We are now led to the following convergence lemma. 

\begin{lemma}\label{JumpRelations}
Let $\Omega\in\mathcal{D}$, $A\in M_{\lambda,\mu}(B)$ and $b\in \Lip(B)$, and consider two Lipschitz functions $F,H:\overline{B}\to\mathbb R$ with $F,H\equiv 0$ on $\partial B$. Then, for all $j\in\mathbb N$,
\[
\int_{\partial\Omega_j}\partial_{\nu_j}\mathcal{S}_+F(p_j)\cdot H(p_j)\,d\sigma_j(p_j)\xrightarrow[j\to\infty]{}\int_{\partial\Omega}\left(\frac{1}{2}F(q)H(q)+F(q)\cdot \mathcal{K}H(q)\right)\,d\sigma(q),
\]
and also
\[
\int_{\partial\Omega_j'}\partial_{\nu_j'}S_-F(p_j')\cdot H(p_j')\,d\sigma_j'(p_j')\xrightarrow[j\to\infty]{}\int_{\partial\Omega}\left(-\frac{1}{2}F(q)H(q)+F(q)\cdot\mathcal{K}H(q)\right)\,d\sigma(q),
\]
where $\Omega_j,\Omega_j'$ are defined in theorem \ref{ApproximationScheme} and the comments right after it, and $\nu_j$, $\nu_j'$ are the unit outer normals on $\Omega_j$, $\Omega_j'$, respectively.
\end{lemma}
\begin{proof}
Let $I_j$ be the first integral above. From the formula for $\partial_{\nu_j}\mathcal{S}_+F(p_j)$, we first have that
\[
I_j=\int_{\partial\Omega_j}\left(\int_{\partial\Omega}\partial_{\nu_j}^{p_j}G(p_j,q)F(q)\,d\sigma(q)\right)H(p_j)\,d\sigma_j(p_j).
\]
Now, for $j$ fixed, since $|p_j-q|$ is bounded below by some positive number, the integral above is absolutely convergent, so we can apply Fubini's theorem to obtain that
\[
I_j=\int_{\partial\Omega}\left(\int_{\partial\Omega_j}\partial_{\nu_j}^{p_j}G(p_j,q)H(p_j)\,d\sigma_j(p_j)\right)F(q)\,d\sigma(q),
\]
where differentiation takes place with respect to the $p_j$ variable of $G$. We now apply the second representation in lemma \ref{ContinuousInside} for fixed $j$, for the domain $\Omega_j$ and for $G$. Since $q\notin\overline{\Omega_j}$, we obtain that
\[
I_j=\int_{\partial\Omega}\left(\int_{\Omega_j}A\nabla G_q\nabla H+b\nabla G_q\cdot H\right)F(q)\,d\sigma(q),
\]
where $G_q=G(\cdot,q)$. By letting $j\to\infty$, the dominated convergence theorem shows that
\begin{align*}
I_j&\xrightarrow[j\to\infty]{}\int_{\partial\Omega}\left(\int_{\Omega}A\nabla G_q\nabla H+b\nabla G_q\cdot H\right)F(q)\,d\sigma(q)\\
&=\int_{\partial\Omega}\frac{1}{2}FH\,d\sigma(q)+\int_{\partial\Omega}F(q)\left(-\frac{1}{2}H(q)+\int_{\Omega}A\nabla G_q\nabla H+b\nabla G_q\cdot H\right)\,d\sigma(q)
\end{align*}
and, since $q\in\partial\Omega$, the second equality in lemma \ref{JumpFor1} shows that
\[
I_j\to\int_{\partial\Omega}\left(\frac{1}{2}FH+F\cdot\mathcal{K}H\right)\,d\sigma(q).
\]
Set now $I_j'$ to be the second integral. As above, and since now $q\in\Omega_j'$, we obtain from the first representation in lemma \ref{ContinuousInside} that
\[
I_j'=-\int_{\partial\Omega}\left(\int_{B\setminus\Omega_j'}A\nabla G_q\nabla H+b\nabla G_q\cdot H\right)F(q)\,d\sigma(q).
\]
We then apply the dominated convergence theorem to obtain
\begin{align*}
I_j'&\xrightarrow[j\to\infty]{}-\int_{\partial\Omega}F(q)\left(\int_{B\setminus\Omega}A\nabla G_q\nabla H+b\nabla G_q\cdot H\right)\,d\sigma(q)\\
&=\int_{\partial\Omega}-\frac{1}{2}FH\,d\sigma+\int_{\partial\Omega}F(q)\left(\frac{1}{2}H(q)-\int_{B\setminus\Omega}A\nabla G_q\nabla H+b\nabla G_q\cdot H\right)\,d\sigma(q)\\
&=\int_{\partial\Omega}\left(-\frac{1}{2}FH+F\cdot\mathcal{K}
H\right)\,d\sigma(q),
\end{align*}
where we used the first equality in lemma \ref{JumpFor1}.
\end{proof}

As a consequence of the previous lemma, we obtain the next important relation.

\begin{cor}[Jump Relation]\label{JumpRelation}
Let $\Omega\in\mathcal{D}$, $A\in M_{\lambda,\mu}(B)$ and $b\in\Lip(B)$. Then
\[
\int_{\partial\Omega_j}\partial_{\nu_j}\mathcal{S}_+F(p_j)\cdot H(p_j)\,d\sigma_j(p_j)-\int_{\partial\Omega_j'}\partial_{\nu_j'}\mathcal{S}_-F(p_j')\cdot H(p_j')\,d\sigma_j'(p_j')\xrightarrow[j\to\infty]{}\int_{\partial\Omega}FH\,d\sigma,
\]
for all $F,H:\overline{B}\to\mathbb R$ which are Lipschitz continuous and vanish on $\partial B$, where $\nu_j$, $\nu_j'$ are the unit outer normals on $\Omega_j$, $\Omega_j'$, respectively.
\end{cor}
\begin{proof}
To obtain this convergence, we subtract the second line in lemma \ref{JumpRelations} from the first.
\end{proof}

\subsection{Invertibility of the single layer potential}
We will now turn our attention to the global Rellich estimates, which will lead to invertibility of the single layer potential operator. Since we will apply the Rellich estimates, we will need to assume that $A$ is symmetric.

\begin{lemma}\label{SolidBound}
Let $\Omega$ be a Lipschitz domain, $A\in M_{\lambda,\mu}^s(\Omega)$ and $b\in L^{\infty}(\Omega)$. Suppose that $u$ is a $W^{1,2}_{\loc}(\Omega)$ solution to the equation $Lu=-\dive(A\nabla u)+b\nabla u=0$ with $(\nabla u)^*\in L^2(\partial\Omega)$, and $\nabla_Tu$ has non-tangential limits almost everywhere on $\partial\Omega$. Then,
\[
\int_{\Omega}|\nabla u|^2\leq C\int_{\partial\Omega}|\nabla_Tu|^2\,d\sigma,
\]
where $C$ is a good constant.
\end{lemma}
\begin{proof}
Consider the approximation scheme $\Omega_j\uparrow\Omega$ that appears in theorem \ref{ApproximationScheme} and fix $j$. After subtracting a constant $c_j$, we obtain that the average of $u_j=u-c_j$ over $\partial\Omega_j$ is equal to $0$. Consider the operator
\[
L_0=-\dive(A\nabla),
\]
and let $v_j$ be a solution to $L_0v_j=0$ in $\Omega_j$, with $v_j=u_j$ on $\partial\Omega_j$, and set $w_j=u_j-v_j$. We then compute
\[
-\dive(A\nabla w_j)+b\nabla w_j=-\dive(A\nabla u_j)+\dive(A\nabla v_j)+b\nabla u-b\nabla v_j=-b\nabla v_j,
\]
since $u_j$ solves the equation $Lu=0$. Note that $u_j$ is continuously differentiable in the interior of $\Omega$ (from proposition \ref{DerivativeRegularity}); hence $v_j\in W^{1,2}(\Omega_j)$ from proposition \ref{WeakSolvabilityForLipschitz}. Since $v_j=u_j$ on $\partial\Omega_j$, we obtain that $w_j\in W_0^{1,2}(\Omega_j)$, therefore proposition \ref{GoodBoundOnSolutions} shows that
\[
\|\nabla w_j\|_{L^2(\Omega_j)}\leq C\|b\nabla v_j\|_{L^2(\Omega_j)}\leq C\|b\|_{\infty}\|\nabla v_j\|_{L^2(\Omega_j)},
\]
where $C$ is a good constant, hence
\begin{equation}\label{eq:UByV}
\|\nabla u\|_{L^2(\Omega_j)}=\|\nabla u_j\|_{L^2(\Omega_j)}\leq C\|\nabla v_j\|_{L^2(\Omega_j)}+C\|\nabla w_j\|_{L^2(\Omega_j)}\leq C\|\nabla v_j\|_{L^2(\Omega_j)},
\end{equation}
where $C$ is a good constant. Since now $v_j$ solves the equation $L_0v_j=0$ in $\Omega_j$ and $v_j=u_j$ on $\partial\Omega_j$, we compute
\[
\lambda\int_{\Omega_j}|\nabla v_j|^2\leq\int_{\Omega_j}A\nabla v_j\nabla v_j=\int_{\partial\Omega_j}\partial_{\nu_j}v_j\cdot v_j\leq\int_{\partial\Omega_j}|u_j|^2\,d\sigma_j+\int_{\partial\Omega_j}|\partial_{\nu_j}v_j|^2\,d\sigma_j.
\]
Therefore, since $u_j$ has average $0$ over $\partial\Omega_j$, using Poincare's inequality on $\partial\Omega_j$ and plugging in \eqref{eq:UByV} we obtain that
\begin{equation}\label{eq:LambdaNablaV}
\int_{\Omega_j}|\nabla u|^2\leq C\int_{\Omega_j}|\nabla v_j|^2\leq C\int_{\partial\Omega_j}|\nabla_{T_j}u|^2\,d\sigma_j+C\int_{\partial\Omega_j}|\partial_{\nu_j}v_j|^2\,d\sigma_j,
\end{equation}
where $C$ depends on $\lambda$ and the Lipschitz character of $\Omega_j$.

To treat the last term, note that from \cite{KenigShen}, the Rellich property holds for the operator $-\dive(A\nabla)$ in $\Omega_j$ with a good constant $C$, since $A$ is symmetric. Therefore, if we consider the homeomorphisms $\Lambda_j:\partial\Omega\to\partial\Omega_j$ that appear in theorem \ref{ApproximationScheme}, we obtain that
\begin{align*}
\int_{\partial\Omega_j}|\partial_{\nu_j}v_j|^2\,d\sigma_j&\leq C\int_{\partial\Omega_j}|\nabla_{T_j} v_j|^2\,d\sigma_j=C\int_{\partial\Omega_j}|\nabla_{T_j}u_j|^2\,d\sigma_j=C\int_{\partial\Omega_j}|\nabla u\cdot T_j|^2\,d\sigma_j\\
&=C\int_{\partial\Omega}|\nabla u(\Lambda_j(q))\cdot T_j(\Lambda_j(q))|^2\tau_j(q)\,d\sigma(q),
\end{align*}
since $v_j=u_j$ on $\partial\Omega_j$, and we are considering the tangential derivatives on $\partial\Omega_j$. Recall now that $\Lambda_j(q)\in\Gamma(q)$ for large $j$, and for all $q\in\partial\Omega$,
\[
\Lambda_j(q)\xrightarrow[j\to\infty]{}q,\quad T_j(\Lambda_j(q))\xrightarrow[j\to\infty]{}T(q),\quad\tau_j(q)\xrightarrow[j\to\infty]{}1.
\]
Since $(\nabla u)^*\in L^2(\partial\Omega)$ and $\nabla u$ has nontangential limits almost everywhere, the dominated convergence theorem shows that
\begin{equation}\label{eq:TangentialLimit}
\int_{\partial\Omega_j}|\nabla u(\Lambda_j(q))\cdot T_j(\Lambda_j(q))|^2\tau_j(q)\,d\sigma(q)\xrightarrow[j\to\infty]{}\int_{\partial\Omega}|\nabla_Tu|^2\,d\sigma,
\end{equation}
which finally shows that
\[
\limsup_{j\to\infty}\int_{\partial\Omega_j}|\partial_{\nu_j}v_j|^2\,d\sigma_j\leq C\int_{\partial\Omega}|\nabla_Tu|^2\,d\sigma.
\]
Plugging in \eqref{eq:LambdaNablaV} and letting $j\to\infty$, we obtain that
\[
\int_{\Omega}|\nabla u|^2\leq\limsup_{j\to\infty}\int_{\Omega_j}|\nabla u|^2\leq C\limsup_{j\to\infty}\int_{\partial\Omega_j}|\nabla_{T_j}u|^2\,d\sigma_j+C\int_{\partial\Omega}|\nabla_Tu|^2\,d\sigma.
\]
We then use \eqref{eq:TangentialLimit} to complete the proof.
\end{proof}

We are now in position to show the global Rellich estimate.

\begin{prop}[Global Rellich Estimate]\label{GlobalRellich}
Let $\Omega$ be a smooth domain, $A\in M_{\lambda,\mu}^s(\Omega)$ and $b\in L^{\infty}(\Omega)$. Suppose that $u$ is a $C^1(\overline{\Omega})$ solution of $Lu=0$ in $\Omega$. Then,
\[
\int_{\partial\Omega}|\partial_{\nu}u|^2\,d\sigma\leq C\int_{\partial\Omega}|\nabla_Tu|^2\,d\sigma,
\]
where $C$ is a good constant.
\end{prop}
\begin{proof}
Fix $q\in\partial\Omega$ and set $r_0=r_{\Omega}/2$, where $r_{\Omega}$ is defined in section 1.1. Note that the local Rellich estimate (proposition \ref{LocalRellich}) is applicable, from theorem 8.12 in \cite{Gilbarg}; that is,
\[
\int_{\Delta_{r_0}(q)}\left|\partial_{\nu}u\right|^2\,d\sigma\leq C\int_{\Delta_{2r_0}(q)}|\nabla_Tu|^2\,d\sigma+
\frac{C}{r_0}\int_{T_{2r_0}(q)}|\nabla u|^2.
\]
We now integrate for $q\in\partial\Omega$ and we use Fubini's theorem to obtain that
\begin{align*}
\int_{\partial\Omega}\left|\partial_{\nu}u\right|^2\,d\sigma&\leq C\int_{\partial\Omega}|\nabla_Tu|^2\,d\sigma+
\frac{C}{r_0}\int_{\Omega}|\nabla u|^2\\
&\leq C\int_{\partial\Omega}|\nabla_Tu|^2\,d\sigma+
\frac{C}{r_0}\int_{\partial\Omega}|\nabla_Tu|^2\,d\sigma\\
&\leq C\int_{\partial\Omega}|\nabla_Tu|^2\,d\sigma,
\end{align*}
where we used lemma \ref{SolidBound} in the last step.
\end{proof}

The global Rellich estimate leads us to the next bound for the single layer potential operator on the boundary of a Lipschitz domain.

\begin{prop}\label{InverseInequalityForGoodB}
Suppose that $\Omega\in\mathcal{D}$, $A\in M_{\lambda,\mu}^s(B)$ and $b\in\Lip(B)$. Then, for every $f\in L^2(\partial\Omega)$,
\[
\|f\|_{L^2(\partial\Omega)}\leq C\|\nabla_T\mathcal{S}f\|_{L^2(\partial\Omega)},
\]
where $C$ is a good constant.
\end{prop}
\begin{proof}
Suppose first that $f$ is Lipschitz, and consider a Lipschitz extension $F:\overline{B}\to\mathbb R$ of $f$, which vanishes on $\partial B$. Set $u_+=\mathcal{S}_+f$, and $u_-=\mathcal{S}_-f$, then the jump relation (corollary \ref{JumpRelation}) with $H=F$ shows that
\[
\int_{\partial\Omega_j}\partial_{\nu_j}u_+\cdot F\,d\sigma_j-\int_{\partial\Omega_j'}\partial_{\nu_j'}u_-\cdot F\,d\sigma_j'\xrightarrow[j\to\infty]{}\int_{\partial\Omega}F^2\,d\sigma.
\]
Since now $F$ is continuous in $\Omega$, the Cauchy-Schwartz inequality shows that
\begin{align*}
\|F\|_{L^2(\partial\Omega)}^2&\leq\limsup_{j\to\infty}\left(\|\partial_{\nu_j}u_+\|_{L^2(\partial\Omega_j)}\|F\|_{L^2(\partial\Omega_j)}+\|\partial_{\nu_j'}u_-\|_{L^2(\partial\Omega_j')}\|F\|_{L^2(\partial\Omega_j')}\right)\\
&=\|F\|_{L^2(\partial\Omega)}\limsup_{j\to\infty}\left(\|\partial_{\nu_j}u_+\|_{L^2(\partial\Omega_j)}+\|\partial_{\nu_j'}u_-\|_{L^2(\partial\Omega_j')}\right),
\end{align*}
therefore we obtain 
\begin{equation}\label{eq:Sum}
\|F\|_{L^2(\partial\Omega)}^2\leq 2\limsup_{j\to\infty}\left(\|\partial_{\nu_j}u_+\|_{L^2(\partial\Omega_j)}^2+\|\partial_{\nu_j'}u_-\|_{L^2(\partial\Omega_j')}^2\right).
\end{equation}
From proposition \ref{SingleLayerInequalities} and \ref{DerivativeRegularity}, $u_+$ is a $C^1$ solution in $\overline{\Omega_j}$. Note also that, from the global Rellich estimate (proposition \ref{GlobalRellich}), we obtain that
\[
\|\partial_{\nu_j}u_+\|_{L^2(\partial\Omega_j)}^2\leq C\|\nabla_Tu_+\|_{L^2(\partial\Omega_j)}^2,
\]
where $C_j$ is a good constant for $\Omega_j$. We now apply corollary \ref{LimsupEstimateOnNorm} to obtain that
\[
\limsup_{j\to\infty}\|\partial_{\nu_j}u_+\|_{L^2(\partial\Omega_j)}^2\leq C\|\nabla_T\mathcal{S}f\|_{L^2(\partial\Omega)}^2.
\]
A similar process shows that
\[
\limsup_{j\to\infty}\|\partial_{\nu_j'}u_-\|_{L^2(\partial\Omega_j')}^2\leq C\|\nabla_T\mathcal{S}f\|_{L^2(\partial\Omega)}^2.
\]
Adding those inequalities and plugging in \eqref{eq:Sum}, we finally obtain that
\[
\|f\|_{L^2(\partial\Omega)}^2\leq 2C\|\nabla_T\mathcal{S}f\|_{L^2(\partial\Omega)}^2,
\]
which shows the desired inequality for Lipschitz functions $f:\partial\Omega\to\mathbb R$.

To obtain the estimate for $f\in L^2(\partial\Omega)$, we use the fact that $\Lip(\partial\Omega)$ is dense in $L^2(\partial\Omega)$ and $\mathcal{S}:L^2(\partial\Omega)\to W^{1,2}(\partial\Omega)$ is continuous (from lemma \ref{SBoundedness}) to conclude the proof.
\end{proof}

We now pass to bounded drifts.

\begin{prop}\label{InverseInequality}
Suppose that $\Omega\in\mathcal{D}$, $A\in M_{\lambda,\mu}^s(B)$ and $b\in L^{\infty}(B)$. Then, for every $f\in L^2(\partial\Omega)$,
\[
\|f\|_{L^2(\partial\Omega)}\leq C\|\nabla_T\mathcal{S}f\|_{L^2(\partial\Omega)},
\]
where $C$ is a good constant.
\end{prop}
\begin{proof}
Let $(b_n)$ be a mollification of $b$, and let $\mathcal{S}_n$ be the single layer potential operator for the equation
\[
-\dive(A\nabla u)+b_n\nabla u=0
\]
in $\Omega$. Let also $G_n$ be Green's function for the same equation in $B$. Then, for any fixed $n\in\mathbb N$ and $p\in\partial\Omega$,
\begin{align*}
\left|\nabla_T\mathcal{S}_nf(p)-\nabla_T\mathcal{S}f(p)\right|&\leq\int_{\partial\Omega}|\nabla_pG_n(p,q)-\nabla_pG(p,q)||f(q)|\,d\sigma(q)\\
&\leq C\|b_n-b\|_{L^{2d}(B)}\int_{\partial\Omega}|p-q|^{3/2-d}|f(q)|\,d\sigma(q),
\end{align*}
from proposition \ref{ContinuityArgumentForB}. Since the kernel $|p-q|^{3/2-d}$ is integrable on $\partial\Omega$ and
\[
\|b_n-b\|_{L^{2d}(B)}\xrightarrow[n\to\infty]{} 0,
\]
we obtain that
\[
\nabla_T\mathcal{S}_nf\xrightarrow[n\to\infty]{L^2(\partial\Omega)}\nabla_T\mathcal{S}f.
\]
Since $b_n\in\Lip(B)$, we apply proposition \ref{InverseInequalityForGoodB} to obtain that
\[
\|f\|_{L^2(\partial\Omega)}\leq C\|\nabla_T\mathcal{S}_nf\|_{L^2(\partial\Omega)},
\]
where $C$ is a good constant. We then let $n\to\infty$ to conclude the proof.
\end{proof}

The last proposition, together with the continuity method, show invertibility of the single layer potential on the boundary for symmetric matrices $A$.

\begin{thm}\label{InvertibilityOfS}
Let $\Omega$ be a Lipschitz domain $\Omega$, $A\in M_{\lambda,\mu}^s(\Omega)$ and $b\in L^{\infty}(\Omega)$. Then, the operator $\mathcal{S}:L^2(\partial\Omega)\to W^{1,2}(\partial\Omega)$ is invertible, with
\[
\|\mathcal{S}^{-1}f\|_{L^2(\partial\Omega)}\leq C\|\nabla_Tf\|_{L^2(\partial\Omega)},
\]
and $C$ being a good constant.
\end{thm}
\begin{proof}
After a dilation and a translation, we can assume that $\Omega\in\mathcal{D}$.

We will use the continuity method: consider the family of equations
\[
L_tu=-\dive(A\nabla u)+tb\nabla u=0,
\]
for $t\in[0,1]$, and the family of operators $\mathcal{S}_t:L^2(\partial\Omega)\to W^{1,2}(\partial\Omega)$, with
\[
\mathcal{S}_tf(p)=\int_{\partial\Omega}G_t(p,q)f(q)\,d\sigma(q),
\]
where $G_t$ is Green's function for the equation $L_tu=0$ in a large ball containing $\Omega$. We now show that the map $t\mapsto \mathcal{S}_t$, is continuous in the $W^{1,2}(\partial\Omega)$ norm, by showing that, for any $t_1,t_2\in [0,1]$ and any $f\in L^2(\partial\Omega)$,
\[
\|\mathcal{S}_{t_1}f-\mathcal{S}_{t_2}f\|_{W^{1,2}(\partial\Omega)}\leq C|t_1-t_2|\|f\|_{L^2(\partial\Omega)}.
\]
For this purpose, we first use the estimate in proposition \ref{ContinuityArgumentForB}, to obtain that, for $p\in\partial\Omega$,
\begin{align*}
|\mathcal{S}_{t_1}f(p)-\mathcal{S}_{t_2}f(p)|&=\left|\int_{\partial\Omega}(G_{t_1}(p,q)-G_{t_2}(p,q))f(q)\,d\sigma(q)\right|\\
&\leq C|t_1-t_2|\int_{\partial\Omega}|p-q|^{\frac{5}{2}-d}|f(q)|\,d\sigma(q)\\
&\leq C|t_1-t_2|\left(\int_{\partial\Omega}|p-q|^{\frac{5}{2}-d}|f(q)|^2\,d\sigma(q)\right)^\frac{1}{2}\left(\int_{\partial\Omega}|p-q|^{\frac{5}{2}-d}\,d\sigma(q)\right)^\frac{1}{2}\\
&\leq C|t_1-t_2|\left(\int_{\partial\Omega}|p-q|^{\frac{5}{2}-d}|f(q)|^2\,d\sigma(q)\right)^{1/2},
\end{align*}
since $|p-q|^{5/2-d}$ is integrable on $\partial\Omega$. So, after squaring and integrating for $p\in\partial\Omega$, we obtain that
\[
\|\mathcal{S}_{t_1}f-\mathcal{S}_{t_2}f\|_{L^2(\partial\Omega)}\leq C|t_1-t_2|\|f\|_{L^2(\partial\Omega)}.
\]
For the tangential gradient of the difference $\mathcal{S}_{t_1}f-\mathcal{S}_{t_2}f$, we apply the second estimate in proposition \ref{ContinuityArgumentForB} and a procedure similar to above, to obtain that
\[
|\nabla_T\mathcal{S}_{t_1}f(p)-\nabla_T\mathcal{S}_{t_2}f(p)|\leq C|t_1-t_2|\left(\int_{\partial\Omega}|p-q|^{3/2-d}|f(q)|^2\,d\sigma(q)\right)^{1/2},
\]
which shows continuity of $t\mapsto \mathcal{S}_t$. Note now that the operator $\mathcal{S}_0$ in invertible, from remark $6.9$ in \cite{KenigShen}. The continuity method now shows that $\mathcal{S}:L^2(\partial\Omega)\to W^{1,2}(\partial\Omega)$ is invertible. We then use the estimate in proposition \ref{InverseInequality} to bound the norm of the inverse, which completes the proof.
\end{proof}

We are now led to solvability of the $R_2$ Regularity problem in Lipschitz domains. 

\begin{thm}\label{R2Solvability}
Let $\Omega$ be a Lipschitz domain, $A\in M_{\lambda,\mu}^s(\Omega)$, and $b\in L^{\infty}(\Omega)$. Then the Regularity problem $R_2$ is uniquely solvable in $\Omega$, with constants depending on $d,\lambda,\mu,\|b\|_{\infty}$, the Lipschitz character of $\Omega$ and the diameter of $\Omega$. Moreover, the solution admits the representation
\[
u(x)=\mathcal{S}_+(\mathcal{S}^{-1}f)(x)=\int_{\partial\Omega}G(x,q)\mathcal{S}^{-1}f(q)\,d\sigma(q).
\]
\end{thm}
\begin{proof}
Let $f\in W^{1,2}(\Omega)$. From theorem \ref{InvertibilityOfS}, the operator $\mathcal{S}:L^2(\partial\Omega)\to W^{1,2}(\partial\Omega)$ is invertible, therefore we can consider the function $g=\mathcal{S}^{-1}f\in L^2(\partial\Omega)$. Set also $u=\mathcal{S}_+(\mathcal{S}^{-1}f)$. From proposition \ref{SingleLayerInequalities}, $u$ solves the Dirichlet problem in $\Omega$, with data $\mathcal{S}(\mathcal{S}^{-1}f)=f$. Moreover, from the same proposition,
\[
\|(\nabla u)^*\|_{L^2(\Omega)}=\|(\nabla\mathcal{S}_+(\mathcal{S}^{-1}f))^*\|_{L^2(\Omega)}\leq C\|\mathcal{S}^{-1}f\|_{L^2(\Omega)}\leq C\|f\|_{W^{1,2}(\Omega)},
\]
where $C$ is a good constant, from theorem \ref{InvertibilityOfS}. This shows that $\mathcal{S}_+(\mathcal{S}^{-1}f)$ solves the Regularity problem with data $f$ on $\partial\Omega$. Uniqueness now follows from proposition \ref{UniquenessForRegularity}.
\end{proof}

We will be able to drop the symmetry assumption on $A$ later, in theorem \ref{RegularityForNonSymmetric}.

\section{The Dirichlet problem for $L^{\lowercase{t}}$}
This chapter will focus on solvability of the Dirichlet problem for the equation $L^tu=0$. We will rely on the results of the previous chapter, and we will use the adjoint of the single layer potential operator in order to establish existence.

\subsection{Uniqueness}
In order to show uniqueness, we will need a nontangential maximal bound on the derivatives of Green's function. To show this, we first show the next lemma.

\begin{lemma}\label{ToMaximalBoundForG}
Let $\Omega$ be a Lipschitz domain, $A\in M_{\lambda}(\Omega)$ and $b\in L^{\infty}(\Omega)$. Let also $y\in\Omega$, and $q\in\partial\Omega, r>0$ fixed, such that $y\notin T_{5r}(q)$. If $G_y$ denotes Green's function for $L$ in $\Omega$ with pole at $y$ from theorem \ref{GoodGreenFunctionEstimates}, then $G_y\in W^{1,2}(T_{2r}(q))$.
\end{lemma}
\begin{proof}
We will assume that $b$ is Lipschitz; the case of bounded $b$ can be shown by using a mollification argument.
	
Consider the functions $G_n$ that are constructed in lemma \ref{LimitGreenConstruction}, and consider $n\in\mathbb N$ large, such that $G_n$ is a solution of the equation $LG_n=0$ in $T_{4r}(q)$. Since $G_n\in W^{1,2}(T_{4r}(q))$, proposition \ref{HolderOnTheBoundary} shows that $G_n$ is continuous in $W^{1,2}(T_{4r}(q))$.

Note now that, from Carleson's estimate (lemma \ref{CarlesonEstimate})
\[
G_n(x)\leq CG_n(A_{4r}(q))
\]
for all $x\in T_{4r}(q)$. But, (v) in lemma \ref{LimitGreenConstruction} shows that $(G_n)$ is equicontinuous in a small neighborhood of $A_{4r}(q)$, therefore there exists $C>0$ such that, for a subsequence,
\[
G_{k_n}(A_{4r}(q))\leq C
\]
for all $n\in\mathbb N$. This shows that $G_{k_n}$ is uniformly bounded in $T_{4r}(q)$, hence, Cacciopoli's estimate in $T_{4r}(q)$ (lemma \ref{Cacciopoli}) shows that $(G_{k_n})$ is uniformly bounded in $W^{1,2}(T_{2r}(q))$, therefore a subsequence converges to some $g\in W^{1,2}(T_{2r}(q))$ almost everywhere in $T_{2r}(q)$. But, again from (v) in lemma \ref{LimitGreenConstruction}, there exists a subsequence of $G_{k_n}$ that converges to $G_y$ in every compact subset of $\Omega\setminus B_0$. Hence $G_y\in W^{1,2}(T_{2r}(q))$, which completes the proof.
\end{proof}

\begin{lemma}\label{R2BoundOnGreen}
Let $\Omega$ be a Lipschitz domain, $A\in M_{\lambda,\mu}^s(\Omega)$, and $b\in L^{\infty}(\Omega)$. Let also $y\in\Omega$, and $G_y$ be Green's function for the equation $Lu=0$ in $\Omega$, centered at $y$. Let also $B_0\subseteq\Omega$ be a ball which is compactly supported in $\Omega$ and is centered at $y$. Then, for $\e>0$ sufficiently small,
\[
\int_{\partial\Omega}\left|(\nabla G_y)^*_{\e}\right|^2\,d\sigma<\infty,
\]
where $(\nabla G_y)^*_{\e}$ is the nontangential maximal function in $\Omega\setminus B_0$, and where the supremum in the nontangential maximal function is taken $\e$-close to the boundary.
\end{lemma}
\begin{proof}
From theorem \ref{GoodGreenFunctionEstimates}, $G_y$ is a $W_0^{1,\frac{d}{2(d-1)}}(\Omega\setminus B_0/2)$ solution of the equation $Lu=0$ in $\Omega\setminus B_0/2$, were $B_0/2$ is the half ball of $B_0$. Therefore, corollary \ref{GreenContinuity} shows that $G_y$ is continuously differentiable close to the boundary of $B_0$.

Consider now the solution $u$ to the Regularity problem $R_2$ for $L$ in $\Omega\setminus B_0$, with $u|_{\partial\Omega}\equiv 0$ and $u|_{\partial B_0}=G_y$, which exists from theorem \ref{R2Solvability}. Note also that $u\in W^{1,2}(\Omega\setminus B_0)$, and also $G_y\in W^{1,2}(\Omega\setminus B_0)$, which follows from lemma \ref{ToMaximalBoundForG} and the fact that $G_y$ is continuously differentiable in the interior of $\Omega\setminus\{y\}$. Therefore, $u-G_y$ is a $W_0^{1,2}(\Omega\setminus B_0)$ solution to $Lu=0$, hence $u\equiv G_y$ in $\Omega\setminus B_0$. But, for $\e>0$ sufficiently small, since $u$ solves the Regularity problem,
\[
\int_{\partial\Omega}|(\nabla u)^*_{\e}|^2\,d\sigma<\infty,
\]
which completes the proof.
\end{proof}

The next proposition shows uniqueness for the Dirichlet problem for the equation $L^tu=0$.

\begin{prop}\label{UniquenessForDirichletForAdjoint}
Let $\Omega$ be a Lipschitz domain, $A\in M_{\lambda,\mu}(\Omega)$, and $b\in L^{\infty}(\Omega)$. Suppose that $u:\Omega\to\mathbb R$ is a weak solution to the equation $L^tu=0$ in $\Omega$, with $u^*\in L^2(\partial\Omega)$ and $u\to 0$ non-tangentially, almost everywhere on $\partial\Omega$. Then $u\equiv 0$.
\end{prop}
\begin{proof}
The proof is similar to the argument in proposition \ref{DirichletSolvability}. Fix $y$ in $\Omega$, write $G_y(x)$ for $G(x,y)$, and for $\e>0$ recall the definitions of $\Omega_{\e}$, $R_{\e}$ and $\phi_{\e}$ from proposition \ref{DirichletSolvability}. Then we obtain that, for $\e$ small,
\[
u(y)=u(y)\phi_{\e}(y)=\int_{\Omega}A\nabla G_y\nabla(u\phi_{\e})+b\nabla G_y\cdot u\phi_{\e},
\]
which implies that
\[
u(y)=\int_{R_{\e}}A\nabla G_y\nabla\phi_{\e}\cdot u-\int_{R_{\e}}A\nabla\phi_{\e}\nabla u\cdot G_y-b\nabla\phi_{\e}\cdot G_yu=I_1+I_2+I_3.
\]
since $u$ is a solution of $L^tu=0$ in $\Omega$, and from the support properties of $\phi_{\e}$.

Recall now the definitions of $Z_i$, $P_j$, $Q_j$ and $\psi$ from theorem \ref{DirichletSolvability}. We then write
\[
|I_1|\leq\frac{C}{\e}\sum_j\int_{Q_j}\int_{\psi(x_0)+c_1\e}^{\psi(x_0)+c_2\e}|\nabla G_y(x_0,s)||u(x_0,s)|\,dsdx_0.
\]
Then, note that for each one of the summands, if $x_0\in Q_j$ and $s\in(\psi(x_0)+c_1\e,\psi(x_0)+c_2\e)$, then $(x_0,s)\in\Gamma(q)$ for all $q\in P_j$. Therefore,
\[
\int_{Q_j}\int_{\psi(x_0)+c_1\e}^{\psi(x_0)+c_2\e}|\nabla G_y(x_0,s)||u(x_0,s)|\,dsdx_0\leq C\e^d(\nabla G_y)^*(q)u_{\e}^*(q),
\]
and, after integrating on $P_j$, we obtain that
\[
\int_{Q_j}\int_{\psi(x_0)+c_1\e}^{\psi(x_0)+c_2\e}|\nabla G_y(x_0,s)||u(x_0,s)|\,dsdx_0\leq C\e\int_{P_j}(\nabla G_y)_{\e}^*u_{\e}^*\,d\sigma.
\]
Therefore, the sum above is bounded by
\[
\frac{C}{\e}\sum_jC\e\int_{P_j}(\nabla G_y)^*u_{\e}^*\,d\sigma=C\int_{U_i\cap\partial\Omega}(\nabla G_y)_{\e}^*u_{\e}^*\,d\sigma.
\]
For $\e$ sufficiently small, the term
\[
\int_{\partial\Omega}\left|(\nabla G_y)^*_{\e}\right|^2\,d\sigma
\]
is uniformly bounded, from lemma \ref{R2BoundOnGreen}. Therefore, the Cauchy-Schwartz inequality and the dominated convergence theorem show that the last term goes to $0$ as $\e\to 0$. Hence, adding those integrals for $i=1,\dots N$, we obtain that $I_1\to 0$.

For $I_2$, we first estimate
\[
\int_{Z_i\cap R_{\e}}|A\nabla\phi_{\e}\nabla u\cdot G_y|\leq\frac{C}{\e}\sum_j\int_{Q_j}\int_{\psi(x_0)+c_1\e}^{\psi(x_0)+c_2\e}|\nabla u(x_0,s)||G_y(x_0,s)|\,dsdx_0.
\]
For each one of the summands, we apply the Cauchy-Schwartz inequality. For the term with the gradient of $u$, we apply the Cacciopoli inequality, to obtain
\begin{align*}
\int_{Q_j}\int_{\psi(x_0)+c_1\e}^{\psi(x_0)+c_2\e}|\nabla u(x_0,s)|^2\,dx_0ds&\leq\frac{C}{\e^2}\int_{2Q_j}\int_{\psi(x_0)+c_0\e}^{\psi(x_0)+c_3\e}|u(x_0,s)|^2\,dx_0ds\\
&\leq\frac{C}{\e^2}\int_{2Q_j}\int_{\psi(x_0)+c_0\e}^{\psi(x_0)+c_3\e}|u^*_{\e}(q)|^2\,dx_0ds=C\e^{d-2}|u^*_{\e}(q)|^2,
\end{align*}
for each $q\in P_j$. Moreover, for every $q\in P_j$,
\[
\int_{Q_j}\int_{\psi(x_0)+c_1\e}^{\psi(x_0)+c_2\e}|G_y(x_0,s)|^2\,dx_0ds\leq C\e^d|(G_y)_{\e}^*|^2(q)\leq C\e^{d+2}\left|(\nabla G_y)^*_{\e}(q)\right|^2,
\]
where we also used proposition \ref{UByNablaU}. Therefore the Cauchy-Schwartz inequality shows that
\[
\int_{Q_j}\int_{\psi(x_0)+c_1\e}^{\psi(x_0)+c_2\e}|\nabla u(x_0,s)||G_y(x_0,s)|\,dsdx_0\leq C\e^du^*_{\e}(q)(\nabla G_y)^*_{\e}(q),
\]
for all $q\in P_j$, and after integrating on $P_j$, we obtain that
\[
\int_{Q_j}\int_{\psi(x_0)+c_1\e}^{\psi(x_0)+c_2\e}|\nabla u(x_0,s)||G_y(x_0,s)|\,dsdx_0\leq C\e\int_{P_j}u^*_{\e}(\nabla G_y)^*_{\e}\,d\sigma.
\]
Returning to $I_2$, we estimate
\[
\int_{Z_i\cap R_{\e}}|A\nabla u\nabla\phi\cdot G_y|\leq\frac{C}{\e}\sum_jC\e\int_{P_j}u^*_{\e}(\nabla G_y)^*_{\e}\,d\sigma=\int_{U_i\cap Z_i}u^*_{\e}(\nabla G_y)^*_{\e}\,d\sigma,
\]
and the last term goes to $0$ as $\e\to 0$, with an argument similar to the case of $I_1$. Adding for $i=1,\cdots N$ shows that $I_2\to 0$ as well.

Finally, for $I_3$, we write
\[
\int_{Z_i\cap R_{\e}}|b\nabla\phi_{\e}\cdot G_yu|\leq\frac{C}{\e}\sum_j\int_{Q_j}\int_{\psi(x_0)+c_1\e}^{\psi(x_0)+c_2\e}\left|G_y(x_0,s)u(x_0,s)\right|\,dx_0ds,
\]
and for each one of the summands, for all $q\in P_j$,
\[
\int_{Q_j}\int_{\psi(x_0)+c_1\e}^{\psi(x_0)+c_2\e}\left|G_y(x_0,s)u(x_0,s)\right|\,dx_0ds\leq\e^{d-1}(G_y)^*_{\e}(q)u^*_{\e}(q).
\]
Therefore, after integrating over $P_j$ and summing for $j$, we obtain that
\[
\int_{Z_i\cap R_{\e}}|b\nabla\phi\cdot G_yu|\leq C\sum_j\int_{P_j}(G_y)^*_{\e}u^*_{\e}\,d\sigma\leq C\e\int_{U_i\cap Z_i}u^*_{\e}(\nabla G_y)^*_{\e}\,d\sigma,
\]
and the last term goes to $0$ as $\e\to 0$. This finishes the proof.
\end{proof}

\subsection{Singular integrals}
We will now turn our attention to integral operators that will be central to establishing existence for the Dirichlet problem for $L^t$. The setting will be as in the case of the Regularity problem for $L$: we will assume that $\Omega$ is a subset of a ball $B$, and we will extend the coefficients $A$ and $b$ in $\Omega$. We will also set $G^t(y,x)$ to be Green's function for the operator $L^tu=-\dive(A\nabla u)-\dive(bu)$ in $B$.

Note that, from proposition \ref{SymmetryWithAdjoint}, $G^t(y,x)=G(x,y)$, where $G$ is Green's function for the operator $Lu=-\dive(A\nabla u)+b\nabla u$ in $B$.

The first operator we will consider is the \emph{maximal truncation operator}
\[
T^*f(p)=\sup_{\delta>0}\left|\int_{|p-q|>\delta}\nabla_T^qG(q,p)\cdot f(q)\,d\sigma(q)\right|,
\]
for $f\in L^2(\partial\Omega)$ and $p\in\partial\Omega$. An important property of $T^*$ is the fact that it is bounded from $L^2$ to $L^2$, as the next proposition shows.

\begin{prop}\label{MaximalTruncationBoundForAdjoint}
Let $\Omega$ be a Lipschitz domain, $A\in M_{\lambda,\mu}(\Omega)$ and $b\in L^{\infty}(\Omega)$. Then the operator $T^*$ is bounded from $L^2(\partial\Omega)$ to $L^2(\partial\Omega)$, and its norm is bounded by a good constant.
\end{prop}
\begin{proof}
We will mimic the proof of proposition \ref{MaximalTruncationBound}.
	
Without loss of generality, we will assume that $0\in\Omega$ and $\diam(\Omega)<1/40$. Then $B\subseteq B_{1/4}(0)$. Let now $G_0$ be Green's function for the equation $-\dive(A^t\nabla u)=0$ in $B$. Consider also the periodic extension $A_p$ of $A$ in $\mathbb R^d$, as in lemma \ref{PeriodicAExtension}, and set $\Gamma(x,y)$ to be the fundamental solution of the equation $-\dive(A^t\nabla u)=0$ in \[
\tilde{T}^*f(p)=\sup_{\delta>0}\left|\int_{|p-q|>\delta}\nabla^q\Gamma(q,p)\cdot f(q)\,d\sigma(q)\right|
\]
is bounded from $L^2(\partial\Omega)$ to itself, with the bound being a good constant.

We now write
\begin{align*}
\nabla_T^qG(q,p)&=\left(\nabla_T^qG(q,p)-\nabla_T^qG_0(q,p)\right)+\left(\nabla_T^qG_0(q,p)-\nabla_T^q\Gamma(q,p)\right)+\nabla_T^q\Gamma(q,p)\\
&=k_1(q,p)+k_2(q,p)+\nabla_T^q\Gamma(q,p).
\end{align*}

After fixing $\delta>0$, multiplying with $f$ and integrating, we estimate
\begin{equation}\label{eq:Fork1}
\int_{|p-q|>\delta}|k_1(q,p)f(q)|\,d\sigma(q)\leq C\int_{|p-q|>\delta}|p-q|^{3/2-d}|f(q)|\,d\sigma(q),
\end{equation}
from proposition \ref{ContinuityArgumentForB}. For $k_2$, we fix $p\in\partial\Omega$ and we set
\[
u(x)=G_0(x,p)-\Gamma(x,p),
\]
for $x\in B$. Then, the regularity properties of Green's function show that $u\in W^{1,\frac{d}{2(d-1)}}(B)$, and for every $\phi\in C_c^{\infty}(B)$,
\begin{align*}
\int_{B}A^t\nabla u(x)\nabla\phi(x)\,dx&=\int_{B}A^t\nabla_xG_0(x,p)\nabla\phi(x)\,dx-\int_{B}A^t\nabla_x\Gamma(x,p)\nabla\phi(x)\,dx\\
&=\phi(p)-\phi(p)=0,
\end{align*}
since $G_0$ is Green's function for $-\dive(A^tu)=0$ in $B$, and $\Gamma(x,y)$ is the fundamental solution of the equation $-\dive(A^t\nabla u)=0$ in $\mathbb R^d$. This shows that $u$ is a $W^{1,\frac{d}{2(d-1)}}(B)$ solution to $-\dive(A\nabla u)=0$ in $B$, therefore lemma \ref{LowRegularityEstimate} shows that $u$ is a $W^{1,2}(B/2)$ solution to $-\dive(A\nabla u)=0$ in $B/2$. Hence, estimate 2.2 in \cite{KenigShen} shows that the function
\[
\nabla u(x)=\nabla_xG_0(x,p)-\nabla_x\Gamma(x,p)
\]
is bounded in $B/4$, hence $k_2$ is bounded, with the bound being a good constant. Therefore
\[
\int_{|p-q|>\delta}|k_2(q,p)f(q)|\,d\sigma(q)\leq C\int_{|p-q|>\delta}|f(q)|\,d\sigma(q)
\]
for a good constant $C$.

We now add the last estimate with \eqref{eq:Fork1} and we use the definition of $\tilde{T}^*$, to obtain that, for any $\delta>0$,
\[
\left|\int_{|p-q|>\delta}\nabla_T^qG(q,p)\cdot f(q)\,d\sigma(q)\right|\leq C\int_{|p-q|>\delta}\left(|p-q|^{3/2-d}+1\right)|f(q)|\,d\sigma(q)+\tilde{T}^*f(p),
\]
hence
\[
T^*f(p)\leq C\int_{|p-q|>\delta}\left(|p-q|^{3/2-d}+1\right)|f(q)|\,d\sigma(q)+\tilde{T}^*f(p).
\]
The fact that the kernel $|p-q|^{3/2-d}+1$ is integrable on $\partial\Omega$, together with boundedness of $\tilde{T}^*$. complete the proof.
\end{proof}

The second operator we will be interested in will be the operator
\begin{equation}\label{eq:TDefinition}
Th(p)=\lim_{\e\to 0}\int_{|p-q|>\e}\nabla_T^pG(p,q)\cdot\nabla_Th(q)\,d\sigma(q),
\end{equation}
for $h\in W^{1,2}(\partial\Omega)$.

\begin{prop}\label{SingularBoundForAdjoint}
Let $\Omega$ be a Lipschitz domain. Then the limit in the definition of $T$ in \eqref{eq:TDefinition} exists in $L^2(\partial\Omega)$ and almost everywhere, hence $Th\in L^2(\partial\Omega)$.
\end{prop}
\begin{proof}
Since $\nabla_Th\in L^2(\partial\Omega)$ for $h\in W^{1,2}(\partial\Omega)$, the proof is similar to the proof of proposition \ref{T_iBoundedness} using theorem 3.1 in \cite{KenigShen} and boundedness of the operator $T_A^2$ that appears there.
\end{proof}

\subsection{Existence}
In order to show existence for the Dirichlet problem for the equation $L^tu=0$, we will consider the adjoint of the single layer potential for the equation $Lu=0$. Our first observation is that, from  theorem \ref{InvertibilityOfS}, we obtain the following proposition.

\begin{prop}\label{AdjointInvertibility}
Let $\Omega$ be a Lipschitz domain, $A\in M_{\lambda,\mu}^s(\Omega)$ and $b\in L^{\infty}(\Omega)$. Set
\[
\mathcal{S}^*:W^{-1,2}(\partial\Omega)\to L^2(\partial\Omega)
\]
to be the adjoint operator to $\mathcal{S}:L^2(\partial\Omega)\to W^{1,2}(\partial\Omega)$, from theorem \ref{InvertibilityOfS}. Then this operator is invertible, and its norm is bounded by a good constant.
\end{prop}

Note that, at this point, we need to assume that $A$ is symmetric: recall that we have proved theorem \ref{InvertibilityOfS} using the Rellich estimates, which in turn we have showed assuming that $A$ is symmetric.

We turn to finding the formula for the adjoint: for this purpose, let $f\in L^2(\partial\Omega)$ and $H\in W^{-1,2}(\Omega)$. Then, lemma \ref{Riesz} shows that there exists a unique $h\in W^{1,2}(\partial\Omega)$, such that $H=R_2h$: that is,
\[
\left<\mathcal{S}f,H\right>_{W^{1,2}(\partial\Omega)}=\int_{\partial\Omega}\mathcal{S}f\cdot h\,d\sigma+\int_{\partial\Omega}\nabla_T\mathcal{S}f\cdot\nabla_Th\,d\sigma.
\]
For the first integral, we compute
\begin{align*}
\int_{\partial\Omega}\mathcal{S}f\cdot h\,d\sigma&=\int_{\partial\Omega}\left(\int_{\partial\Omega}G(p,q)f(q)\,d\sigma(q)\right)\cdot h(p)\,d\sigma(p)\\
&=\int_{\partial\Omega}\left(\int_{\partial\Omega}G(p,q)h(p)\,d\sigma(p)\right)\cdot f(q)\,d\sigma(q),
\end{align*}
from Fubini's theorem, since the double integral converges absolutely from the pointwise bounds for $G$. For the second integral, proposition \ref{MaximalTruncationBoundForAdjoint} shows that the dominated convergence theorem is applicable, hence
\begin{align*}
\int_{\partial\Omega}\nabla_T\mathcal{S}f\cdot\nabla_Tg\,d\sigma&=\int_{\partial\Omega}\left(\lim_{\e\to 0}\int_{|p-q|>\e}\nabla_T^pG(p,q)f(q)\,d\sigma(q)\right)\cdot\nabla_Th(p)\,d\sigma(p)\\
&=\lim_{\e\to 0}\int_{\partial\Omega}\left(\int_{|p-q|>\e}\nabla_T^pG(p,q)f(q)\,d\sigma(q)\right)\cdot\nabla_Th(p)\,d\sigma(p)\\
&=\lim_{\e\to 0}\int_{\partial\Omega}\left(\int_{|p-q|>\e}\nabla_T^pG(p,q)\cdot\nabla_Th(p)\,d\sigma(p)\right)\cdot f(q)\,d\sigma(q)\\
&=\int_{\partial\Omega}\left(\lim_{\e\to 0}\int_{|p-q|>\e}\nabla_T^pG(p,q)\cdot\nabla_Th(p)\,d\sigma(p)\right)\cdot f(q)\,d\sigma(q),
\end{align*}
where we used Fubini's theorem for the third equality, since for fixed $\e$ the inner integral is absolutely convergent, and the dominated convergence theorem for the last equality. Therefore, we finally obtain that
\[
\mathcal{S}^*H(q)=\int_{\partial\Omega}G(p,q)h(p)\,d\sigma(p)+\lim_{\e\to 0}\int_{|p-q|>\e}\nabla_T^pG(p,q)\cdot\nabla_Th(p)\,d\sigma(p),
\]
therefore using proposition \ref{SymmetryWithAdjoint}, we are led to the following lemma.

\begin{lemma}\label{FirstRepresentationOfS*}
Let $\Omega$ be a Lipschitz domain, $A\in M_{\lambda,\mu}^s(\Omega)$ and $b\in L^{\infty}(\Omega)$. Then, for any $H\in W^{-1,2}(\partial\Omega)$ with $H=R_2h$ in the sense of lemma \ref{Riesz},
\[
\mathcal{S}^*H(p)=\int_{\partial\Omega}G^t(p,q)h(q)\,d\sigma(q)+\lim_{\e\to 0}\int_{|p-q|>\e}\nabla_T^qG^t(p,q)\cdot\nabla_Th(q)\,d\sigma(q).
\]
\end{lemma}

We will also need to compute the formula for $\mathcal{S}^*H$ when $H$ is of the special form $H=E_2h$ for $h\in L^2(\partial\Omega)$, where $E_2:L^2(\partial\Omega)\to W^{-1,2}(\partial\Omega)$ is the embedding that appears right before lemma \ref{L2DensityInRd}.

\begin{lemma}\label{SecondRepresentationOfS*}
Let $\Omega$ be a Lipschitz domain, $A\in M_{\lambda,\mu}^s(\Omega)$ and $b\in L^{\infty}(\Omega)$. Then, if $H\in W^{-1,2}(\partial\Omega)$ is of the form $H=E_2h$ for some $h\in L^2(\partial\Omega)$,
\[
\mathcal{S^*}H(p)=\int_{\partial\Omega}G^t(p,q)h(q)\,d\sigma(q).
\]
\end{lemma}
\begin{proof}
Let $f\in L^2(\partial\Omega)$. The definition of $E_2$ shows that
\begin{align*}
\left<\mathcal{S}f,E_2h\right>_{W^{1,2}(\partial\Omega)}&=\int_{\partial\Omega}\mathcal{S}f\cdot h\,d\sigma=\int_{\partial\Omega}\int_{\partial\Omega}G(p,q)f(q)h(p)\,d\sigma(q)d\sigma(p)\\
&=\int_{\partial\Omega}\left(\int_{\partial\Omega}G(p,q)h(p)\,d\sigma(p)\right)f(q)d\sigma(q)\\
&=\left<\int_{\partial\Omega}G^t(\cdot,p)h(p)\,d\sigma(p),f\right>_{L^2(\partial\Omega)},
\end{align*}
from Fubini's theorem, since the kernel $G(p,q)$ is integrable on $\partial\Omega$. This completes the proof.
\end{proof}

We will now extend the operator $\mathcal{S}^*$ inside $\Omega$. In order to do this, note that for any fixed $x\in\Omega$, the function
\[
y\mapsto G^t(x,y)=G_x(y)
\]
solves the equation $Lu=0$ away from $x$, therefore it is continuously differentiable near $\partial\Omega$ from proposition \ref{DerivativeRegularity}. Hence, restricted on $\partial\Omega$, $G_x\in W^{1,2}(\partial\Omega)$. Therefore, for $F\in W^{-1,2}(\partial\Omega)$, we can define
\[
\mathcal{S}^*_+F(x)=F(G_x).
\]
We are then led to the following expressions.

\begin{lemma}\label{RepresentationOfS*_+}
Let $\Omega$ be a Lipschitz domain, $A\in M_{\lambda,\mu}^s(\Omega)$ and $b\in L^{\infty}(\Omega)$. If $H\in W^{-1,2}(\partial\Omega)$ can be represented as $H=R_2h$ for $h\in W^{1,2}(\partial\Omega)$, in the sense of lemma \ref{Riesz}, then
\[
\mathcal{S}^*_+H(x)=\int_{\partial\Omega}\left(G^t(x,q)h(q)+\nabla_T^qG^t(x,q)\cdot\nabla_Th(q)\right)\,d\sigma(q).
\]
In the special case where $H=E_2f$ for some $f\in L^2(\partial\Omega)$, then
\[
\mathcal{S}^*_+H(x)=\int_{\partial\Omega}G^t(x,q)f(q)\,d\sigma(q).
\]
\end{lemma}
\begin{proof}
The first representation follows from the definition of $\mathcal{S}^*_+$ and a proof analogous to the proof of lemma \ref{FirstRepresentationOfS*}. The second representation follows from an analogue of the proof of lemma \ref{SecondRepresentationOfS*}.
\end{proof}

The basic properties of the operator $\mathcal{S}_+^*$ are now demonstrated in the next proposition.

\begin{prop}\label{SingleLayerAdjoint}
Let $\Omega$ be a Lipschitz domain, $A\in M_{\lambda,\mu}^s(\Omega)$, $b\in L^{\infty}(\Omega)$. For every $F\in W^{-1,2}(\partial\Omega)$, $u=\mathcal{S}^*_+F$ is a $W_{\loc}^{1,2}(\Omega)$ solution of the equation $L^tu=0$ in $\Omega$. Moreover, $u$ converges to $\mathcal{S}^*F$ on $\partial\Omega$, nontangentially, almost everywhere, and also
\[
\|u^*\|_{L^2(\partial\Omega)}\leq C\|F\|_{W^{-1,2}(\partial\Omega)},
\]
where $C$ is a good constant.
\end{prop}
\begin{proof}
We mimic the proof of theorem \ref{SingleLayerInequalities}. The proof that $u\in W^{1,2}_{\loc}(\Omega)$ and that $u$ solves $L^tu=0$ in $\Omega$ is identical, after noting that the functions
\[
x\mapsto G^t(x,q),\quad x\mapsto\nabla^qG^t(x,q)\cdot T(q)
\]
are solutions of $L^tu=0$ in $\Omega$ away from $q$, for any $q\in\partial\Omega$, from lemma \ref{MixedGreenIsASolution}, and where $T(q)$ is any tangential vector to $q$ at $\partial\Omega$.

To show the bound on the non-tangential maximal function, we write $F=R_2f$ for some $f\in W^{1,2}(\partial\Omega)$, in the sense of lemma \ref{Riesz}. Let $p\in\partial\Omega$ and $x\in \Gamma(p)$. Set $r=|x-p|$. We then write, using lemma \ref{RepresentationOfS*_+},
\[
u(x)=\int_{\partial\Omega}G^t(x,q)f(q)\,d\sigma(q)+\int_{\partial\Omega}\nabla_T^qG^t(x,q)\cdot\nabla_Tf(q)\,d\sigma(q)=I_1+I_2.
\]
To bound $I_1$, note that $x\in\Gamma(p)$, therefore $|x-q|\geq C|p-q|$. Hence, the pointwise bounds on $G$ show that
\[
|I_1|\leq\int_{\partial\Omega}|G^t(x,q)f(q)|\,d\sigma(q)\leq C\int_{\partial\Omega}|x-q|^{2-d}|f(q)|\,d\sigma(q)\leq C\int_{\partial\Omega}|p-q|^{2-d}|f(q)|\,d\sigma(q).
\]
To bound $I_2$ we set $r=|x-p|$, and we write
\begin{align*}
|I_2|&\leq \int_{|p-q|\leq r}\left|\nabla_T^qG(q,x)\cdot\nabla_Tf(q)\right|\,d\sigma(q)+\left|\int_{|p-q|>r}\nabla_T^qG(q,x)\cdot\nabla_Tf(q)\,d\sigma(q)\right|\\
&=I_3+|I_4|.
\end{align*}
For $I_3$, note that $|x-q|\geq C|x-p|=Cr$ for any $q\in\partial\Omega$, hence the pointwise bounds on $\nabla G$ show that
\[
I_3\leq C\int_{|p-q|\leq r}|x-q|^{1-d}\left|\nabla_Tf(q)\right|\,d\sigma(q)\leq\frac{C}{r^{d-1}}\int_{\Delta_{cr}(p)}\left|\nabla_Tf(q)\right|\,d\sigma(q)\leq CM(\nabla_Tf)(p),
\]
where $M$ is the Hardy-Littlewood maximal operator on $\partial\Omega$.

For $I_4$, we use H{\"o}lder continuity of the derivative of Green's function in the adjoint variable from proposition \ref{HolderContinuityOfGreenForAdjointOnAdjoint}, and the definition of the maximal truncation operator, and we estimate
\begin{align*}
|I_4|&\leq \left|\int_{|p-q|>r}\left(\nabla_T^qG(q,p)+(\nabla_T^qG(q,x)-\nabla_T^qG(q,p))\right)\cdot\nabla_Tf(q)\,d\sigma(q)\right|\\
&\leq T^*(\nabla_Tf)(p)+\int_{|p-q|>r}\left|\nabla_T^qG(q,x)-\nabla_T^qG(q,p)\right||\nabla_Tf(q)|\,d\sigma(q)\\
&\leq T^*(\nabla_Tf)(p)+C|x-p|^{\alpha}\int_{|p-q|>r}\left(|x-q|^{1-d-\alpha}+|p-q|^{1-d-\alpha}\right)|\nabla_Tf(q)|\,d\sigma(q)\\
&\leq T^*(\nabla_Tf)(p)+C|x-p|^{\alpha}\int_{|p-q|>r}|p-q|^{1-d-\alpha}|\nabla_Tf(q)|\,d\sigma(q)\\
&=T^*(\nabla_Tf)(p)+Cr^{\alpha}I_5,
\end{align*}
since $|x-q|\geq C|p-q|$. Finally, to bound $I_5$, we write
\begin{align*}
I_5&\leq\sum_{k=0}^{\infty}\int_{2^kr<|p-q|\leq 2^{k+1}r}|p-q|^{1-d-\alpha}|\nabla_Tf(q)|\,d\sigma(q)\\
&\leq \sum_{k=0}^{\infty}\int_{2^kr<|p-q|\leq 2^{k+1}r}\left(2^kr\right)^{1-d-\alpha}|\nabla_Tf(q)|\,d\sigma(q)\\
&\leq C\sum_{k=0}^{\infty}\left(2^kr\right)^{1-d-\alpha}\left(2^{k+1}r\right)^{d-1}\fint_{|p-q|\leq2^{k+1}r}|\nabla_Tf(q)|\,d\sigma(q)\\
&\leq C\sum_{k=0}^{\infty}2^{-k\alpha}2^{d-1}r^{-\alpha}M(\nabla_Tf)(p)=\frac{C\cdot 2^{d-1}}{1-2^{-\alpha}}r^{-\alpha}M(\nabla_Tf)(p)\\
&\leq Cr^{-\alpha}M(\nabla_Tf)(p),
\end{align*}
since $-\alpha<0$, which implies that the series converges; moreover, $C$ is a good constant. Combining the bounds for $I_i$, $i=1,\dots 5$, we finally obtain that
\[
|u(x)|\leq C\int_{\partial\Omega}|p-q|^{2-d}|f(q)|\,d\sigma(q)+CM(\nabla_Tf)(p)+T^*(\nabla_Tf)(p),
\]
for any $x\in\Gamma(p)$, hence
\[
u^*(p)\leq C\int_{\partial\Omega}|p-q|^{2-d}|f(q)|\,d\sigma(q)+CM(\nabla_Tf)(p)+T^*(\nabla_Tf)(p).
\]
Since now the kernel $|p-q|^{2-d}$ is integrable on $\partial\Omega$, $M$ is bounded from $L^2(\partial\Omega)$ to itself, and applying proposition \ref{MaximalTruncationBoundForAdjoint}, we obtain the estimate
\[
\|u^*\|_{L^2(\partial\Omega)}\leq C\|f\|_{L^2(\partial\Omega)}+C\|\nabla_Tf\|_{L^2(\partial\Omega)}\leq C\|F\|_{W^{-1,2}(\partial\Omega)},
\]
which completes the estimate on the non-tangential maximal fuction.

Finally, we turn to non-tangential, almost everywhere convergence on the boundary. We claim that, from the bound on the nontangential maximal function we have just shown, it is enough to show that
\begin{equation}\label{eq:InDense}
\mathcal{S}^*_+F(x)\xrightarrow[x\to p]{}\mathcal{S}^*F(p),
\end{equation}
nontangentially, almost everywhere, for $F\in V$, where $V$ is a dense subset of $W^{-1,2}(\partial\Omega)$. Suppose that this is the case; then let $F\in W^{-1,2}(\partial\Omega)$ and set $u(x)=\mathcal{S}^*_+F(x)$. Consider also $F_n\in V$ such that $F_n\to F$ in $W^{-1,2}(\partial\Omega)$, and set $u_n(x)=\mathcal{S}^*_+F_n(x)$. Set also $A_n$ to be the set of $p\in\partial\Omega$ such that 
\[
u_n(x)\xrightarrow[x\to p]{}\mathcal{S}^*F_n(p),
\]
and $A=\bigcup_nA_n$; then $A_n$ has full measure on $\partial\Omega$, hence $A$ has full measure on $\partial\Omega$ as well.

Fix now $n\in\mathbb N$ and $p\in A$. Then, since
\[
u_n(x)\xrightarrow[x\to p]{}S^*F_n(p)
\]
for $x\in\Gamma(p)$, we compute
\begin{align*}
\limsup_{x\to p}u(x)-\mathcal{S}^*F(p)&\leq\limsup_{x\to p}\left(u(x)-u_n(x)\right)+\limsup_{x\to p}u_n(x)-\mathcal{S}^*F(p)\\
&=\limsup_{x\to p}\left(u(x)-u_n(x)\right)+\mathcal{S}^*F_n(p)-\mathcal{S}^*F(p)\\
&\leq \sup_{x\in\Gamma(p)}\left|u(x)-u_n(x)\right|+\mathcal{S}^*(F_n-F)(p)\\
&\leq (u-u_n)^*(p)+\mathcal{S}^*(F_n-F)(p),
\end{align*}
for every $n\in\mathbb N$. Hence, for any $\e>0$,
\[
\left\{p\in A\Big{|}\limsup_{x\to p}u(x)-\mathcal{S}^*F(p)>\e\right\}\subseteq\left\{p\in A\Big{|}(u-u_n)^*(p)+\mathcal{S}^*(F_n-F)(p)>\e\right\},
\]
hence Chebyshev's inequality shows that
\begin{multline*}
\sigma\left(\left\{p\in A\Big{|}\limsup_{x\to p}u(x)-\mathcal{S}^*F(p)>\e\right\}\right)\leq\frac{1}{\e}\int_{\partial\Omega}\left((u-u_n)^*+\mathcal{S}^*(F_n-F)\right)\,d\sigma\\
\leq\frac{C}{\e}\left(\|(u-u_n)^*\|_{L^2(\partial\Omega)}+\|\mathcal{S}^*(F_n-F)\|_{L^2(\partial\Omega)}\right),
\end{multline*}
for any $n\in\mathbb N$. But, from the choice of the $F_n$ and boundedness of $\mathcal{S}^*$,
\[
\|\mathcal{S}^*(F_n-F)\|_{L^2(\partial\Omega)}\xrightarrow[n\to\infty]{}0,
\]
and also, from the bound on the nontangential maximal function shown above, we obtain
\[
\|(u-u_n)^*\|_{L^2(\partial\Omega)}=\|(\mathcal{S}^*_+(F_n-F))^*\|_{L^2(\partial\Omega)}\leq C\|F_n-F\|_{W^{-1,2}(\partial\Omega)}\xrightarrow[n\to\infty]{}0.
\]
This shows that, for any $\e>0$,
\[
\sigma\left(\left\{p\in A\Big{|}\limsup_{x\to p}u(x)-\mathcal{S}^*F(p)>\e\right\}\right)=0,
\]
hence
\[
\limsup_{x\to p}u(x)\leq\mathcal{S}^*F(p)
\]
for almost every $p\in A$. A similar process shows that
\[
\liminf_{x\to p}u(x)\geq\mathcal{S}^*F(p)
\]
for almost every $p\in A$, therefore $u$ converges to $\mathcal{S}^*F$ nontangentially, almost everywhere.

To conclude the proof, it remains to show that \eqref{eq:InDense} holds for $F$ in a dense subset $V$ of $W^{-1,2}(\partial\Omega)$. This subset will be the set $E_2(L^2(\partial\Omega))$, which is dense in $W^{-1,2}(\partial\Omega)$ from lemma \ref{DensityOfL2}, and where $E_2:L^2(\partial\Omega)\to W^{-1,2}(\partial\Omega)$ is the canonical embedding. We now let $f\in L^2(\partial\Omega)$, and we note that, from lemma \ref{RepresentationOfS*_+},
\[
\mathcal{S}^*_+(E_2f)(x)=\int_{\partial\Omega}G(q,x)f(q)\,d\sigma(q),
\]
and, from lemma \ref{SecondRepresentationOfS*},
\[
\mathcal{S}^*(E_2f)(p)=\int_{\partial\Omega}G(q,p)f(q)\,d\sigma(q).
\]
In order now to show that
\[
\int_{\partial\Omega}G(q,x)f(q)\,d\sigma(q)\xrightarrow[x\to p]{}\int_{\partial\Omega}G(q,p)f(q)\,d\sigma(q)
\]
nontangentially, almost everywhere, we follow the proof of the analogous result in proposition \ref{SingleLayerInequalities}: instead of Lipschitz continuity of Green's function from proposition \ref{HolderContinuityOfGreenAsIs}, we use H{\"o}lder continuity of Green's function in the adjoint variable, which holds from proposition \ref{HolderInside}, since $G^t$ solves the adjoint equation away from the pole. This completes the proof.
\end{proof}

We can now show existence for $D_2$, for the equation $L^tu=0$, when $A$ is symmetric.

\begin{prop}\label{GeneralDirichletForAdjoint}
Let $\Omega$ be a Lipschitz domain, $A\in M_{\lambda,\mu}^s(\Omega)$ and $b\in L^{\infty}(\Omega)$. Then the Dirichlet problem $D_2$ for the equation $L^tu=0$ is uniquely solvable in $\Omega$, with constants depending only on $d,\lambda,\mu,\|b\|_{\infty}$ and the Lipschitz character of $\Omega$. Moreover, the solution admits the representation
\[
u(x)=\mathcal{S}^*_+\left((\mathcal{S}^*)^{-1}f\right)(x).
\]
\end{prop}
\begin{proof}
Uniqueness follows from proposition \ref{UniquenessForDirichletForAdjoint}. For existence, let $f\in L^2(\partial\Omega)$. From proposition \ref{AdjointInvertibility}, the operator
\[
\mathcal{S}^*:W^{-1,2}(\partial\Omega)\to L^2(\partial\Omega)
\]
is invertible, therefore we can consider the element $F=(\mathcal{S}^*)^{-1}f\in W^{-1,2}(\partial\Omega)$. Set now $u=\mathcal{S}^*_+F$. Then proposition \ref{SingleLayerAdjoint} shows that $u$ is a solution to $D_2$, with boundary values $\mathcal{S^*}F=f$, and also
\[
\|u^*\|_{L^2(\partial\Omega)}\leq C\|F\|_{W^{-1,2}(\partial\Omega)}=\|(\mathcal{S}^*)^{-1}f\|_{W^{-1,2}(\partial\Omega)}\leq C\|f\|_{L^2(\partial\Omega)},
\]
where we also used proposition \ref{AdjointInvertibility}. This completes the proof.
\end{proof}

We will be able to drop the symmetry assumption on $A$ later, in theorem \ref{DirichletForNonSymmetricForAdjoint}.
\section{The Regularity problem for $L^{\lowercase{t}}$}
We now turn to solvability of the Regularity problem for the equation $L^tu=0$. We will mainly follow the method of chapter 8, but since constants are not necessarily solutions to the equation $L^tu=0$, a couple of modifications need to be made.

\subsection{Formulation and uniqueness}
We begin with the formulation of $R_p$.

\begin{dfn}
Let $\Omega$ be a Lipschitz domain, and $p\in(1,\infty)$. We say that the Regularity problem $R_p$ for the equation $L^tu=0$ in $\Omega$ is solvable, if there exists $C>0$ such that, for every $f\in W^{1,p}(\Omega)$, there exists a solution $u\in W^{1,2}_{\loc}(\Omega)$ to the Dirichlet problem
\[
\left\{\begin{array}{c l}
L^tu=0,&{\rm in}\,\,\Omega\\
u=f,&{\rm on}\,\,\partial\Omega,
\end{array}\right.
\]
such that
\[
\|\nabla u\|_{L^p(\partial\Omega)}\leq C\|f\|_{W^{1,p}(\partial\Omega)},
\]
and $u=f$ on the boundary is interpreted in the nontangential, almost everywhere sense.
\end{dfn}

We now turn to uniqueness for the Regularity problem for $L^tu=0$.

\begin{thm}\label{UniquenessForRegularityForAdjoint}
Suppose that $\Omega$ is a bounded Lipschitz domain, $A\in M_{\lambda,\mu}(\Omega)$ and $b\in L^{\infty}(\Omega)$. Let $u\in W^{1,2}_{\loc}(\Omega)$ be a solution to $L^tu=0$ in $\Omega$, with $(\nabla u)^*\in L^2(\partial\Omega)$ and $u\to 0$ nontangentially, almost everywhere. Then, $u\equiv 0$.
\end{thm}
\begin{proof}
We begin by computing, as in proposition \ref{DirichletUniqueness},
\[
u(y)=u(y)\phi_{\e}(y)=\int_{\Omega}A\nabla G_y\nabla(u\phi_{\e})+b\nabla G_y\cdot u\phi_{\e},
\]
where $G_y(x)=G(x,y)$ is Green's function for the equation $Lu=0$ in $\Omega$, therefore
\begin{align*}
u(y)&=\int_{\Omega}A\nabla G_y\nabla\phi_{\e}\cdot u+A\nabla G_y\nabla u\cdot\phi_{\e}+b\nabla(G_y\phi_{\e})\cdot u-b\nabla\phi_{\e}\cdot uG_y\\
&=\int_{\Omega}A\nabla G_y\nabla\phi_{\e}\cdot u+A\nabla(G_y\phi_{\e})\nabla u+b\nabla(G_y\phi_{\e})\cdot u-b\nabla\phi_{\e}\cdot uG_y-A\nabla\phi_{\e}\nabla u\cdot G_y\,dx\\
&=\int_{R_{\e}}A\nabla G_y\nabla\phi_{\e}\cdot u\,dx-\int_{R_{\e}}b\nabla\phi_{\e}\cdot uG_y\,dx-\int_{R_{\e}}A\nabla\phi_{\e}\nabla u\cdot G_y\,dx=I_1+I_2+I_3,
\end{align*}
since $u$ is a solution of $L^tu=0$, in $\Omega$, and from the support properties of $\phi_{\e}$. We then bound the last terms exactly as in the proof of proposition \ref{UniquenessForRegularity} to show that they go to 0 as $\e\to 0$.
\end{proof}

\subsection{Layer potentials}
We will now assume that $\Omega\in\mathcal{D}$, as in chapter $8$; that is, $0\in\Omega$ and $\diam(\Omega)<1/40$. We will also set $B$ to be the unit ball in $\mathbb R^d$.

Given $A\in M_{\lambda,\mu}(\Omega)$, we will extend $A$ periodically as in lemma \ref{AExtension}, and similarly for $b$, depending whether $b\in\Lip(\Omega)$ or $b\in L^{\infty}(\Omega)$. We will then set $G^t$ to be Green's function for the equation $L^tu=-\dive(A\nabla u)-\dive(bu)=0$ in $B$.

For $f\in L^2(\partial\Omega)$, let $\mathcal{S}_+^t$ be the operator
\[
\mathcal{S}^t_+f(x)=\int_{\partial\Omega}G^t(x,q)f(q)\,d\sigma(q),
\]
for $x\in\Omega$. If $x\notin\overline{\Omega}$ we denote this operator by $\mathcal{S}^t_-$, and for $x\in\partial\Omega$, we denote it by $\mathcal{S}^t$; this will be called the single layer potential operator for the adjoint equation $L^tu=0$ in $\Omega$.

In order to consider differentiability properties of $\mathcal{S}^t$ on $\partial\Omega$, we will assume that $b$ is H{\"o}lder continuous. This, together with proposition \ref{GreenDerivativeBoundsForAdjoint} will show that we obtain bounds on the derivative of Green's function in the adjoint variable that are similar to the bounds for the derivative of Green's function. Hence, proceeding as in lemma \ref{SBoundedness}, we can show the next proposition.

\begin{prop}
Let $\Omega\in\mathcal{D}$, $A\in M_{\lambda,\mu}(B)$ and $b\in C^{\alpha}(B)$. The operator $\mathcal{S}^t$ maps $L^2(\partial\Omega)$ to $W^{1,2}(\partial\Omega)$, with
\[
\nabla_T\mathcal{S}^tf(p)=\lim_{\e\to 0}\int_{|p-q|>\e}G^t(p,q)f(q)\,d\sigma(q),
\]
and its norm is bounded above by a good constant that also depends on $\|b\|_{C^{\alpha}(\Omega)}$.
\end{prop}

We can also show the next proposition.

\begin{prop}\label{SingleLayerInequalitiesForAdjoint}
Let $\Omega\in\mathcal{D}$, $A\in M_{\lambda,\mu}(B)$ and $b\in C^{\alpha}(B)$. If $f\in L^2(\partial\Omega)$, then $\mathcal{S}^t_+f\in W^{1,2}_{\loc}(\Omega)$ is a solution to the Dirichlet problem $D_2$ for the equation $L^tu=0$ in $\Omega$, with boundary values $\mathcal{S}^tf$ on $\partial\Omega$. Similarly, $\mathcal{S}^t_-f$ is the solution to $D_2$ in $B\setminus\overline{\Omega}$, and has boundary values $\mathcal{S}^tf\cdot\chi_{\partial\Omega}$ on $\partial(B\setminus\Omega)$. In addition,
\[
\|(\nabla\mathcal{S}^t_{\pm}f)^*\|_{L^2(\partial\Omega)}\leq C\|f\|_{L^2(\partial\Omega)},
\]
where $C$ is a good constant.
\end{prop}
\begin{proof}
The proof is identical to the proof of proposition \ref{SingleLayerInequalities}, for the fact that $\mathcal{S}^t_+f$ and $\mathcal{S}^t_-f$ are solutions. For the boundary values, instead of Lipschitz continuity of Green's function in proposition \ref{HolderContinuityOfGreenAsIs}, we use the analogous result for the adjoint of Green's function, from proposition \ref{HolderContinuityOfGreenAsIsForAdjoint}. Moreover, for the boundedness of the nontangential maximal function, we use proposition \ref{ContinuityArgumentForBAdjoint} instead of proposition \ref{ContinuityArgumentForB}, which completes the proof.
\end{proof}

We now proceed to studying the behavior of the single layer potential on the boundary, and the jump relations.

\begin{prop}\label{LimsupEstimateOnNormForAdjoint}
Let $\Omega\in\mathcal{D}$, $A\in M_{\lambda,\mu}(B)$ and $b\in C^{\alpha}(B)$. Then, for any $f\in L^2(\partial\Omega)$,
\[
\nabla\mathcal{S}^t_+f(x)\cdot T(p)\xrightarrow[x\to p]{}\nabla_T\mathcal{S}^tf(p),
\]
non-tangentially, almost everywhere on $\partial\Omega$. Therefore, if $\Omega_j\uparrow\Omega$ is the approximation scheme in theorem \ref{ApproximationScheme}, we obtain that
\[
\limsup_{j\to\infty}\int_{\partial\Omega_j}\left(|\mathcal{S}_+^tf|^2+|\nabla_{T_j}\mathcal{S}_+^tf|^2\right)\,d\sigma_j\leq \|\mathcal{S}^tf\|_{W^{1,2}(\partial\Omega)}^2,
\]
and similarly for $\mathcal{S}_-^tf$.
\end{prop}
\begin{proof}
The proof is identical to the proof of proposition \ref{DerivativeIsNonTangentiallyContinuous}, using the analogous estimates for the adjoint of Green's function. For the second part, we proceed as in corollary \ref{LimsupEstimateOnNorm}.
\end{proof}

We now pass to the discontinuity of the normal derivative of the single layer potential across the boundary of $\Omega$.

\begin{lemma}\label{ContinuousInsideForAdjoint}
Let $\Omega\in\mathcal{D}$, let $A\in M_{\lambda,\mu}(B)$, $b\in\Lip(\Omega)$ and consider a Lipschitz function $F:\overline{B}\to\mathbb R$ with $F\equiv 0$ on $\partial B$. Then, for $x\in\Omega$,
\[
\int_{\partial\Omega}\partial_{\nu}^qG(x,q)\cdot F(q)\,d\sigma(q)=-\int_{B\setminus\Omega}A^t\nabla G_x^t\nabla F+b\nabla F\cdot G_x^t,
\]
while, for $x\in B\setminus\overline{\Omega}$,
\[
\int_{\partial\Omega}\partial_{\nu}^qG(x,q)\cdot F(q)\,d\sigma(q)=\int_{\Omega}A^t\nabla G_x^t\nabla F+b\nabla F\cdot G_x^t,
\]
where $\partial_{\nu}^q$ denotes the conormal derivative with respect to $q$ on $\partial\Omega$, associated to $L^t$, and $G_x^t(\cdot)=G(x,\cdot)$.
\end{lemma}
\begin{proof}
Suppose first that $x\in\Omega$. Then, from proposition \ref{SymmetryWithAdjoint} and theorems 8.8 and 8.12 in \cite{Gilbarg}, $G_x^t(\cdot)=G^t(\cdot,x)$ and $G_x^t$ is a classical solution of $L^tu=0$ in $B\setminus\Omega$; that is,
\[
\dive(A^t\nabla G_x)=-\dive(bG_x),
\]
almost everywhere in $B$, away from $x$. Therefore, since $F\equiv 0$ on $\partial B$,
\begin{align*}
\int_{\partial\Omega}\partial_{\nu}G_x^t\cdot F\,d\sigma&=-\int_{\partial(B\setminus\Omega)}\partial_{\nu}G_x^t\cdot F\,d\sigma=-\int_{B\setminus\Omega}\dive(F\cdot A^t\nabla G_x^t)\\
&=-\int_{B\setminus\Omega}A^t\nabla G_x^t\nabla F+\dive(A^t\nabla G_x^t)\cdot F\\
&=-\int_{B\setminus\Omega}A^t\nabla G_x^t\nabla F+b\nabla F\cdot G_x^t,
\end{align*}
because $G_x^t$ is a classical solution of $L^tu=0$ $B\setminus\Omega$. Now, if $x\in B\setminus\overline{\Omega}$, then $G_x^t$ is a classical solution of $L^tu=0$ in $\Omega$, therefore
\[
\int_{\partial\Omega}\partial_{\nu}G_x^t\cdot F\,d\sigma=\int_{\Omega}\dive(F\cdot A^t\nabla G_x^t)=\int_{\Omega}A^t\nabla G_x^t\nabla F+b\nabla F\cdot G_x^t,
\]
which concludes the proof.
\end{proof}

\begin{lemma}\label{IntegratesTo1ForAdjoint}
Let $B$ be a ball, and let $A\in M_{\lambda,\mu}(B)$, and $b\in\Lip(B)$. Then, for all $p\in B$,
\[
\lim_{\e\to 0}\int_{\partial B_{\e}(p)}\partial_{\nu}^qG(p,q)\,d\sigma_{B_{\e}}(q)=-1,
\]
where $\partial_{\nu}$ is the conormal derivative associated with $L^t$.
\end{lemma}
\begin{proof}
Let $\e>0$, and consider the domain $U_{\e}=B\setminus B_{\e}(p)$. Set also $G_p^t(\cdot)=G(p,\cdot)$. As in lemma \ref{ContinuousInsideForAdjoint}, $G_p^t$ is a classical solution of the equation
\[
L^tu=-\dive(A^t\nabla u)-\dive(bu)=0
\]
away from $p$, therefore
\[
\int_{U_{\e}}\dive(A\nabla G_p^t)=\int_{U_{\e}}\dive(b\nabla G_p^t).
\]
For the last term, the divergence theorem shows that
\[
\int_{U_{\e}}\dive(b\nabla G_p^t)=\int_{\partial U_{\e}}G_p^t\left<b,\nu\right>\,d\sigma=-\int_{\partial B_{\e}(p)}G_p^t\left<b,\nu\right>\,d\sigma_{\e},
\]
since $G_p^t$ vanishes in $\partial B$. But, from the pointwise estimates on $G$,
\[
\left|\int_{\partial B_{\e}(p)}G_p^t\left<b,\nu\right>\,d\sigma_{\e}\right|\leq\|b\|_{\infty}\int_{\partial B_{\e}(p)}|p-q|^{2-d}\,d\sigma_{\e}(q)\leq C\|b\|_{\infty}\e^{2-d}\sigma_{\e}(\partial B_{\e}(p))\xrightarrow[\e\to 0]{}0,
\]
therefore
\[
\int_{U_{\e}}\dive(A\nabla G_p^t)\xrightarrow[\e\to 0]{}0.
\]
We now integrate by parts, to obtain
\[
\int_{U_{\e}}\dive(A\nabla G_p^t)=\int_{\partial U_{\e}}\partial_{\nu}G_p^t\,d\sigma_{U_{\e}}=\int_{\partial B}\partial_{\nu}G_p^t\,d\sigma_B-\int_{\partial B_{\e}(p)}\partial_{\nu}G_p^t\,d\sigma_{B_{\e}(p)},
\]
therefore
\[
\int_{\partial B_{\e}(p)}\partial_{\nu}G_p^t\,d\sigma_{B_{\e}(p)}\xrightarrow[\e\to 0]{}\int_{\partial B}\partial_{\nu}G_p^t\,d\sigma_B.
\]
Note now that, from proposition \ref{HarmonicRepresentation},
\[
\partial_{\nu}^qG_p^t(q)=\partial_{\nu}^qG(p,q)=-\frac{d\omega_B^p(q)}{d\sigma_B(q)},
\]
which is the harmonic measure kernel on $\partial B$. Since the harmonic measure is a probability measure, we finally obtain that
\[
\int_{\partial B_{\e}(p)}\partial_{\nu}G_p^t\,d\sigma_{B_{\e}(p)}\xrightarrow[\e\to 0]{}\int_{\partial B}\partial_{\nu}G_p^t\,d\sigma_B=-\int_{\partial B}d\omega^p=-1,
\]
and this completes the proof.
\end{proof}

We now define, for a Lipschitz function $f:\partial\Omega\to\mathbb R$ and $p\in\partial\Omega$,
\[
\mathcal{K}^tf(p)=\lim_{\e\to 0}\int_{|p-q|>\e}\partial_{\nu}^qG(p,q)F(q)\,d\sigma(q).
\]
The fact that this limit exists is shown in the next lemma.

\begin{lemma}\label{JumpFor1ForAdjoint}
Let $\Omega\in\mathcal{D}$, let $A\in M_{\lambda,\mu}(B)$, $b\in\Lip(B)$, and consider a Lipschitz function $F:\overline{B}\to\mathbb R$ with $F\equiv 0$ on $\partial B$. Then, for almost all $p\in\partial\Omega$,
\begin{align*}
\mathcal{K}f(p)&=\frac{1}{2}F(p)-\int_{B\setminus\Omega}A^t\nabla G_p^t\nabla F+b\nabla F\cdot G_p^t\\
&=-\frac{1}{2}F(p)+\int_{\Omega}A^t\nabla G_p^t\nabla F+b\nabla F\cdot G_p^t.
\end{align*}
\end{lemma}
\begin{proof}
Set $G_p^t(q)=G(p,q)$, and let $V_{\e}=\Omega\cup B_{\e}(p)$. We also define
\[
\partial^1_{\e}=\Omega^c\cap\partial(B_{\e}(p)),\,\,\partial^2_{\e}=\Omega\cap\partial(B_{\e}(p)),
\]
and we write
\begin{align*}
\int_{\partial\Omega\setminus\Delta_{\e}(p)}\partial_{\nu}G_p^t\cdot F\,d\sigma&=\int_{\partial V_{\e}}\partial_{\nu}G_p^t\cdot F\,d\sigma-\int_{\partial^1_{\e}}\partial_{\nu}G_p^t\cdot F\,d\sigma\\
&=\int_{\partial B}\partial_{\nu}G_p^t\cdot F\,d\sigma-\int_{\partial(B\setminus V_{\e})}\partial_{\nu}G_p^t\cdot F\,d\sigma-\int_{\partial^1_{\e}}\partial_{\nu}G_p^t\cdot F\,d\sigma\\
&=0-I_1-I_2,
\end{align*}
since $F$ vanishes on $\partial B$.

We now treat $I_1$. We first write
\[
I_1=\int_{\partial(B\setminus V_{\e})}\partial_{\nu}G_p^t\,d\sigma=\int_{B\setminus V_{\e}}\dive(F\cdot A^t\nabla G_p^t)=\int_{B\setminus V_{\e}}A^t\nabla G_p^t\nabla F+b\nabla F\cdot G_p^t,
\]
since $G_p$ is a solution of $L^tu=0$ away from $p$. Then, since the term $b\nabla F\cdot G_p^t$ is integrable, we obtain that
\[
I_1\to\int_{B\setminus\Omega}A^t\nabla G_p^t\nabla F+b\nabla F\cdot G_p^t.
\]

For $I_2$, we write
\[
I_2=\int_{\partial^1_{\e}}\partial_{\nu}G_p^t\cdot F\,d\sigma=\int_{\partial^1_{\e}}\partial_{\nu}G_p^t\cdot (F-F(p))\,d\sigma+F(p)\int_{\partial^1_{\e}}\partial_{\nu}G_p^t\,d\sigma=I_3+I_4.
\]
From Lipschitz continuity of $F$ and the pointwise bounds on the gradient of $G^t$, we obtain that
\[
|I_3|\leq C\int_{\partial^1_{\e}}|p-p'|^{1-d}|F(p')-F(p)|\,d\sigma(p')\leq C\e^{2-d}\sigma_{d-1}(\partial B_{\e}(p))\xrightarrow[\e\to 0]{}0.
\]
For $I_4$, for almost all $p\in\partial\Omega$ there exists a well defined tangent plane to $\partial\Omega$ at $p$. For those $p$, the symmetric difference between $\partial_{\e}^1$ and $\partial_{\e}^2$ is contained in a strip
\[
A_{\e}(p)=\{y\in B_{\e}(p)\big{|}|y\cdot\nu(p)|\leq C\e^2\},
\]
as in lemma \ref{JumpFor1}, and if we combine with the pointwise bounds for the gradient of $G^t$, we obtain that
\[
\int_{\partial_{\e}^1}\partial_{\nu}G_p-\int_{\partial_{\e}^2}\partial_{\nu}G_p\xrightarrow[\e\to 0]{}0.
\]
Using lemma \ref{IntegratesTo1ForAdjoint}, we then obtain that
\begin{align*}
\int_{\partial^1_{\e}}\partial_{\nu}G_p^t\,d\sigma&=\frac{1}{2}\left(\int_{\partial^1_{\e}}\partial_{\nu}G_p^t\,d\sigma+\int_{\partial^2_{\e}}\partial_{\nu}G_p^t\,d\sigma\right)+\frac{1}{2}\left(\int_{\partial^1_{\e}}\partial_{\nu}G_p^t\,d\sigma-\int_{\partial^1_{\e}}\partial_{\nu}G_p^t\,d\sigma\right)\\
&=\frac{1}{2}\int_{B_{\e}(p)}\partial_{\nu_A}G_p^t\,d\sigma+\frac{1}{2}\left(\int_{\partial^1_{\e}}\partial_{\nu_A}G_p^t\,d\sigma-\int_{\partial^1_{\e}}\partial_{\nu}G_p^t\,d\sigma\right)\xrightarrow[\e\to 0]{}-\frac{1}{2},
\end{align*}
therefore $I_2\to-\frac{1}{2}F(p)$. Therefore,
\[
\mathcal{K}^tF(p)=\lim_{\e\to 0}(-I_1-I_2)=\frac{1}{2}F(p)-\int_{B\setminus\Omega}A^t\nabla G_p^t\nabla F+b\nabla F\cdot G_p^t.
\]
For the second representation, using the fact that $L^tG_p^t=\delta_p$ and the first representation in this lemma, we write
\begin{align*}
F(p)&=\int_BA^t\nabla G_p^t\nabla F+b\nabla F\cdot G_p^t\\
&=\int_{\Omega}A^t\nabla G_p^t\nabla F+b\nabla F\cdot G_p^t+\int_{B\setminus\Omega}A^t\nabla G_p^t\nabla F+b\nabla F\cdot G_p^t\\
&=\int_{\Omega}A^t\nabla G_p^t\nabla F+b\nabla F\cdot G_p^t+\frac{1}{2}F(p)-\mathcal{K}^tF(p),
\end{align*}
which concludes the proof after rearranging the terms.
\end{proof}

We are now led to the following convergence lemma.

\begin{lemma}\label{JumpRelationsForAdjoint}
Let $\Omega\in\mathcal{D}$, let $A\in M_{\lambda,\mu}(B)$ and $b\in\Lip(B)$, and consider two Lipschitz functions $F,H:\overline{B}\to\mathbb R$ with $F,H\equiv 0$ on $\partial B$. Then, for all $j\in\mathbb N$,
\[
\int_{\partial\Omega_j}\partial_{\nu_j}\mathcal{S}_+^tF(p_j)\cdot H(p_j)\,d\sigma_j(p_j)\xrightarrow[j\to\infty]{}\int_{\partial\Omega}\left(\frac{1}{2}F(q)H(q)+F(q)\cdot \mathcal{K}^tH(q)\right)\,d\sigma(q),
\]
and also
\[
\int_{\partial\Omega_j'}\partial_{\nu_j'}\mathcal{S}_-^tF(p_j')\cdot F(p_j')\,d\sigma_j'(p_j')\xrightarrow[j\to\infty]{}\int_{\partial\Omega}\left(\frac{1}{2}F(q)H(q)+F(q)\cdot\mathcal{K}^tH(q)\right)\,d\sigma(q),
\]
where $\nu_j$, $\nu_j'$ are the unit outer normals on $\Omega_j$, $\Omega_j'$, respectively.
\end{lemma}
\begin{proof}
The proof is identical to the proof of lemma \ref{JumpRelations}, using lemmas \ref{ContinuousInsideForAdjoint} and \ref{JumpFor1ForAdjoint} instead of lemmas \ref{ContinuousInside} and \ref{JumpFor1}.

More specifically, let $I_j$ be the first integral above. From the formula for $\partial_{\nu_j}\mathcal{S}_+^tF(p_j)$, we first have that
\[
I_j=\int_{\partial\Omega_j}\left(\int_{\partial\Omega}\partial_{\nu_j}^{p_j}G^t(p_j,q)F(q)\,d\sigma(q)\right)H(p_j)\,d\sigma_j(p_j).
\]
Now, for $j$ fixed, since $|p_j-q|$ is bounded below by some positive number, the last integral is absolutely convergent, so we can apply Fubini's theorem to obtain that
\[
I_j=\int_{\partial\Omega}\left(\int_{\partial\Omega_j}\partial_{\nu_j}^{p_j}G^t(p_j,q)H(p_j)\,d\sigma_j(p_j)\right)F(q)\,d\sigma(q),
\]
where differentiation takes place with respect to the second variable of $G$. We now apply the second representation in lemma \ref{ContinuousInsideForAdjoint} for fixed $j$, for the domain $\Omega_j$ and for $G^t$. Since $q\notin\overline{\Omega_j}$, we obtain that
\[
I_j=\int_{\partial\Omega}\left(\int_{\Omega_j}A^t\nabla G_q^t\nabla H+b\nabla H\cdot G^t_q\right)F(q)\,d\sigma(q).
\]
By letting $j\to\infty$, the dominated convergence theorem shows that
\begin{align*}
I_j&\xrightarrow[j\to\infty]{}\int_{\partial\Omega}\left(\int_{\Omega}A^t\nabla G_q^t\nabla H+b\nabla H\cdot G_q^t\right)F(q)\,d\sigma(q)\\
&=\int_{\partial\Omega}\frac{1}{2}FH\,d\sigma(q)+\int_{\partial\Omega}F(q)\left(-\frac{1}{2}H(q)+\int_{\Omega}A^t\nabla G^t_q\nabla H+b\nabla H\cdot G_q^t\right)\,d\sigma(q)
\end{align*}
and, since $q\in\partial\Omega$, the second equality in lemma \ref{JumpFor1ForAdjoint} shows that
\[
I_j\to\int_{\partial\Omega}\left(\frac{1}{2}FH+F\cdot\mathcal{K}^tH\right)\,d\sigma(q).
\]
Set now $I_j'$ to be the second integral. As above, and since now $q\in\Omega_j'$, we obtain from the first representation in lemma \ref{ContinuousInsideForAdjoint} that
\[
I_j'=-\int_{\partial\Omega}\left(\int_{B\setminus\Omega_j'}A^t\nabla G_q^t\nabla H+b\nabla H\cdot G_q^t\right)F(q)\,d\sigma(q).
\]
We then apply the dominated convergence theorem to obtain
\begin{align*}
I_j'&\xrightarrow[j\to\infty]{}-\int_{\partial\Omega}F(q)\left(\int_{B\setminus\Omega}A^t\nabla G_q^t\nabla H+b\nabla H\cdot G_q^t\right)\,d\sigma(q)\\
&=\int_{\partial\Omega}-\frac{1}{2}FH\,d\sigma+\int_{\partial\Omega}F(q)\left(\frac{1}{2}H(q)-\int_{B\setminus\Omega}A^t\nabla G_q^t\nabla H+b\nabla H\cdot G_q^t\right)\,d\sigma(q)\\
&=\int_{\partial\Omega}\left(-\frac{1}{2}FH+F\cdot\mathcal{K}^t
H\right)\,d\sigma(q),
\end{align*}
where we used the first equality in lemma \ref{JumpFor1ForAdjoint}.
\end{proof}

As a consequence of the previous lemma, we obtain the jump relation.

\begin{cor}[Jump Relation]\label{JumpRelationForAdjoint}
Let $\Omega\in\mathcal{D}$, $A\in M_{\lambda,\mu}(B)$ and $b\in\Lip(B)$. Then
\[
\int_{\partial\Omega_j}\partial_{\nu_j}\mathcal{S}_+^tF(p_j)\cdot H(p_j)\,d\sigma_j(p_j)-\int_{\partial\Omega_j'}\partial_{\nu_j'}\mathcal{S}_-^tF(p_j')\cdot H(p_j')\,d\sigma_j'(p_j')\xrightarrow[j\to\infty]{}\int_{\partial\Omega}FH\,d\sigma,
\]
for all $F,H:\overline{B}\to\mathbb R$ which are Lipschitz continuous and vanish on $\partial B$, where $\nu_j$, $\nu_j'$ are the unit outer normals on $\Omega_j$, $\Omega_j'$, respectively.
\end{cor}
\begin{proof}
To obtain this convergence, we subtract the second line in lemma \ref{JumpRelationsForAdjoint} from the first.
\end{proof}

\subsection{Invertibility of $\mathcal{S}^t$}
As in the case for the single layer potential for the equation $Lu=0$, we will turn our attention to the global Rellich estimates that will lead to invertibility of $\mathcal{S}^t:L^2(\partial\Omega)\to W^{1,2}(\partial\Omega)$. For the adjoint operator, though, the situation is more complicated, since we need to show control on the size of the divergence of $b$, together with control of the term $u$ in $\Omega$. For this purpose, we begin with the next lemma.

\begin{lemma}\label{ExtensionBound}
Let $\Omega$ be a Lipschitz domain, and $q\in\partial\Omega$, $r\in(0,r_{\Omega})$. Then, for any $u\in W^{1,2}(\Omega)$,
\[
\|u\|_{L^{2^*}(T_r(q))}^2\leq Cr^{-2}\|u\|_{L^2(T_r(q))}^2+C\|\nabla u\|_{L^2(T_r(q))}^2,
\]
where $C$ depends only on the Lipschitz character of $\Omega$.
\end{lemma}
\begin{proof}
Set $T_0=r^{-1}T_r(q)$, and let $v(x)=u(rx)$, for $x\in T_0$. Consider also Stein's extension operator (\cite{SteinSingular}, chapter VI, section 3)
\[
\mathcal{E}:W^{1,2}(T_0)\to W^{1,2}(\mathbb R^d),
\]
which extends $W^{1,2}(T_0)$ functions to $W^{1,2}(\mathbb R^d)$ functions. Then the norm of $\mathcal{E}$ depends only on the Lipschitz character of $T_0$. Since now $\mathcal{E}v\in W^{1,2}(\mathbb R^d)$, there exists a sequence $(v_n)$ in $C_c^{\infty}(\mathbb R^d)$ with $v_n\to v$ in $W^{1,2}(\mathbb R^d)$. Then, Sobolev's inequality shows that
\begin{align*}
\|v\|_{L^{2^*}(T_0)}&\leq \|\mathcal{E}u\|_{L^{2^*}(\mathbb R^d)}\leq C_d\|\nabla\mathcal{E}v\|_{L^2(\mathbb R^d)}\leq C_d\|\mathcal{E}v\|_{W^{1,2}(\mathbb R^d)}\\
&\leq C_d\|\mathcal{E}\|\|v\|_{W^{1,2}(T_0)}\leq C\|v\|_{L^2(T_0)}+C\|\nabla v\|_{L^2(T_0)}.
\end{align*}
We now compute, for any $p\geq 1$,
\[
\int_{T_0}|v(y)|^p\,dy=\int_{T_0}|u(ry)|^p\,dy=\int_{T_r(q)}|u(x)|^pr^{-d}\,dx,
\]
and also
\[
\int_{T_0}|\nabla v(y)|^p\,dx=\int_{T_0}r^p|\nabla u(ry)|^p\,dy=\int_{T_r(q)}|u(x)|^pr^{p-d}\,dx,
\]
which shows that
\[
r^{-d/2^*}\|u\|_{L^{2^*}(T_r(q))}\leq Cr^{-d/2}\|u\|_{L^2(T_r(q))}+Cr^{1-d/2}\|\nabla u\|_{L^2(T_r(q))}.
\]
Therefore, we finally obtain that
\begin{align*}
\|u\|_{L^{2^*}(T_r(q))}&\leq Cr^{d/2^*-d/2}\|u\|_{L^2(T_r(q))}+Cr^{d/2^*+1-d/2}\|\nabla u\|_{L^2(T_r(q))}\\
&=Cr^{-1}\|u\|_{L^2(T_r(q))}+C\|\nabla u\|_{L^2(T_r(q))},
\end{align*}
which completes the proof.
\end{proof}

For the next lemma, we recall the definition of $u_r^*$, from right before proposition \ref{UByNablaU}.

\begin{lemma}\label{SolidByBoundary}
Let $\Omega$ be a Lipschitz domain, $q\in\partial\Omega$ and $r\in(0,r_{\Omega})$. Then, for any function $u$ defined in $\Omega$,
\[
\int_{T_r(q)}|u|^2\leq Cr\int_{\Delta_r(q)}|u_r^*|^2\,d\sigma.
\]
\end{lemma}
\begin{proof}
Let $B_r\subseteq\mathbb R^{d-1}$ be the basis of the cylinder portion $T_r(q)$, and suppose that $\Omega$ is given as the graph of the Lipschitz function $\phi$ above $B_r$. Since the height of $T_r(q)$ is comparable to $r$, we obtain that
\begin{align*}
\int_{T_r(q)}|u|^2&\leq\int_{B_r}\int_{\phi(q)}^{\phi(q)+Cr}|u(x_0,t)|^2\,dtdx_0\leq\int_{B_r}\int_{\phi(q)}^{\phi(q)+Cr}|u_r^*(x_0)|^2\,dtdx_0\\
&\leq Cr\int_{B_r}|u_r^*(x_0)|^2\,dx_0\leq Cr\int_{\Delta_r(q)}|u_r^*|^2\,d\sigma,
\end{align*}
where we perform a change a variables from $B_r$ to $\Delta_r(q)$ for the last equality.
\end{proof}

\begin{lemma}\label{2*By2}
Let $\Omega$ be a Lipschitz domain, and consider a function $u\in C^1(\overline{\Omega})$. Let $q\in\partial\Omega$ and $r\in(0,r_{\Omega})$. Then,
\[
\|u\|_{L^{2^*}(T_r)}^2\leq Cr^{-1}\int_{\Delta_r(q)}|u_r^*|^2\,d\sigma+Cr\int_{\Delta_r(q)}|(\nabla u)^*|^2\,d\sigma,
\]
where $C$ is a good constant.
\end{lemma}
\begin{proof}
We first apply lemma \ref{ExtensionBound} to $T_r(q)$, to obtain
\[
\left(\int_{T_r(q)}|u^*|^2\right)^{2/2^*}\leq Cr^{-2}\int_{T_r(q)}|u|^2+C\int_{T_r(q)}|\nabla u|^2,
\]
where $C$ is a good constant. We then apply lemma \ref{SolidByBoundary} to bound the last two integrals, which completes the proof.
\end{proof}

The next proposition shows how to control the size of $\nabla u$ in $\Omega$.

\begin{prop}\label{SolidBoundForAdjoint}
Let $\Omega\in\mathcal{D}$, $A\in M_{\lambda,\mu}^s(B)$ and $b\in\Lip(B)$. Let also $u$ be a solution to $L^tu=0$ in $\Omega$, with $(\nabla u)^*\in L^2(\partial\Omega)$, and $u$ and $\nabla_Tu$ having nontangential limits almost everywhere on $\partial\Omega$. Then,
\[
\int_{\Omega}|\nabla u|^2\leq C\int_{\partial\Omega}|u|^2+C\int_{\partial\Omega}|\nabla_Tu|^2\,d\sigma,
\]
where $C$ is a good constant.
\end{prop}
\begin{proof}
We mimic the proof of lemma \ref{SolidBound}: consider the approximation scheme $\Omega_j\uparrow\Omega$  as in theorem \ref{ApproximationScheme}, and fix $j$. Let $L_0$ be the operator $L_0=-\dive(A^t\nabla)$, and let $v_j$ be the solution to $L_0v_j=0$ in $\Omega_j$, with $v_j=u$ on $\partial\Omega_j$. Then, if $w_j=u-v_j$, we compute
\[
-\dive(A^t\nabla w_j)-\dive(bw_j)=-\dive(A^t\nabla u)+\dive(A^t\nabla v_j)-\dive(bu_j)+\dive(bv_j)=\dive(bv_j).
\]
Since now $w_j\in W_0^{1,2}(\Omega_j)$, proposition \ref{GoodBoundOnSolutionsForAdjoint} shows that
\[
\|\nabla w_j\|_{L^2(\Omega_j)}\leq C\|\dive(bv_j)\|_{W^{-1,2}(\Omega_j)}=C\|v_j\|_{L^2(\Omega_j)},
\]
therefore
\begin{equation}\label{eq:NablaUByNablaV}
\|\nabla u\|_{L^2(\Omega_j)}\leq C\|\nabla w_j\|_{L^2(\Omega_j)}+C\|\nabla v_j\|_{L^2(\Omega_j)}\leq C\|v_j\|_{L^2(\Omega_j)}+C\|\nabla v_j\|_{L^2(\Omega_j)}.
\end{equation}
Let $\mathcal{S}_{j,+}^t$ be the single layer potential operator for the operator $L_0$ in $\Omega_j$, which is given by integration with respect to the fundamental solution $\Gamma$ for $L_0$, as in \cite{KenigShen}. From its definition, $v_j$ solves the $R_2$ Regularity problem for $L_0$ in $\Omega_j$, with boundary data $u$; therefore, theorems 6.3 and 5.3 in \cite{KenigShen} show that
\[
v_j(x)=\mathcal{S}_{j,+}(\mathcal{S}_j^{-1}u(x))=\int_{\partial\Omega_j}\Gamma(x,q)\mathcal{S}_j^{-1}u(q)\,d\sigma_j(q).
\]
Set $a_j$ to be the average of $v_j$ in $\Omega_j$. Using estimate 2.5 in \cite{KenigShen}, we compute
\begin{align*}
|a_j|&=\left|\fint_{\Omega_j}v_j\right|=\frac{1}{|\Omega_j|}\int_{\Omega_j}\left|\int_{\partial\Omega_j}\Gamma(x,q)\mathcal{S}_j^{-1}u(q)\,d\sigma_j(q)dx\right|\\
&\leq\frac{C}{|\Omega_j|}\int_{\partial\Omega_j}\left(\int_{\Omega_j}|x-q|^{2-d}\,dx\right)|\mathcal{S}_j^{-1}u(q)|\,d\sigma_j(q)\\
&\leq\frac{C}{|\Omega_j|}\int_{\partial\Omega_j}|\mathcal{S}_j^{-1}u(q)|\,d\sigma_j(q).
\end{align*}
Then, from remark 5.8 in \cite{KenigShen}, we obtain that, for a good constant $C$,
\begin{align*}
|a_j|^2&\leq\frac{C\sigma_j(\partial\Omega_j)}{|\Omega_j|^2}\int_{\partial\Omega_j}|\mathcal{S}_j^{-1}u|^2\,d\sigma_j\leq C\|\mathcal{S}_j(\mathcal{S}_j^{-1}u)\|_{W^{1,2}(\partial\Omega_j)}^2\\
&=C\|u\|_{W^{1,2}(\partial\Omega_j)}^2=C\int_{\partial\Omega_j}|u|^2\,d\sigma_j+\int_{\partial\Omega_j}|\nabla_Tu|^2\,d\sigma_j.
\end{align*}
Therefore, Poincare's inequality in $\Omega_j$ shows that
\begin{align*}
\int_{\Omega_j}|v_j|^2&\leq C\int_{\Omega_j}|v_j-a_j|^2+C\int_{\Omega_j}|a_j|^2\\
&\leq C\int_{\Omega_j}|\nabla v_j|^2+C|a_j|^2\\
&\leq C\int_{\Omega_j}|\nabla v_j|^2+C \int_{\partial\Omega_j}|u|^2\,d\sigma_j+\int_{\partial\Omega_j}|\nabla_Tu|^2\,d\sigma_j,
\end{align*}
hence, plugging the last estimate in \eqref{eq:NablaUByNablaV}, we obtain that
\begin{equation}\label{eq:NablaUByNablaV2}
\|\nabla u\|_{L^2(\partial\Omega_j)}\leq C\int_{\Omega_j}|\nabla v_j|^2+C \int_{\partial\Omega_j}|u|^2\,d\sigma_j+\int_{\partial\Omega_j}|\nabla_Tu|^2\,d\sigma_j.
\end{equation}
We now treat the first term on the right hand side exactly as in lemma \ref{SolidBound}: we compute
\[
\lambda\int_{\Omega_j}|\nabla v_j|^2\leq\int_{\partial\Omega_j}
\partial_{\nu_j}v_j\cdot v_j\,d\sigma_j\leq C\int_{\partial\Omega_j}|u|^2+C\int_{\partial\Omega_j}|\partial_{\nu_j}v_j|^2.
\]
But, since the Rellich property holds for $v_j$ in $\Omega_j$ with a good constant $C$, after letting $j\to\infty$ we find that
\begin{align*}
\limsup_{j\to\infty}\int_{\Omega_j}|\nabla v_j|^2&\leq C\limsup_{j\to\infty}\int_{\partial\Omega_j}|u|^2\,d\sigma_j+C\limsup_{j\to\infty}\int_{\partial\Omega_j}|\partial_{\nu_j}v_j|^2\,d\sigma_j\\
&\leq C\limsup_{j\to\infty}\int_{\partial\Omega_j}|u|^2\,d\sigma_j+C\limsup_{j\to\infty}\int_{\partial\Omega_j}|\nabla_{T_j}u|^2\,d\sigma_j\\
&\leq C\int_{\partial\Omega}|u|^2\,d\sigma+C\int_{\partial\Omega}|\nabla_Tu|^2\,d\sigma.
\end{align*}
Therefore, letting $j\to\infty$ in \eqref{eq:NablaUByNablaV2}, we obtain the required estimate.
\end{proof}

We are now in position to show the global Rellich estimate for solutions to the adjoint equation.

\begin{prop}[Global Rellich Estimate]\label{GlobalRellichForAdjoint}
Let $\Omega$ be a smooth domain, $A\in M_{\lambda,\mu}^s(\Omega)$ and $b\in\Lip(\Omega)$. Let also $u$ be a $C^1(\overline{\Omega})$ solution of $L^tu=0$ in $\Omega$. Then, for any $r\in(0,r_{\Omega})$,
\[
\int_{\partial\Omega}|\partial_{\nu}u|^2\,d\sigma\leq Cr\int_{\partial\Omega}|(\nabla u)^*|^2\,d\sigma+\frac{C}{r^d}\int_{\partial\Omega}|u|^2+\frac{C}{r^d}\int_{\partial\Omega}|\nabla_Tu|^2\,d\sigma,
\]
where $C$ is a good constant that also depends on $\|\dive b\|_{L^d(\Omega)}$.
\end{prop}
\begin{proof}
Let $r<r_{\Omega}$ and consider $q\in\partial\Omega$. From theorem $8.12$ in \cite{Gilbarg}, $u\in W^{2,2}(\Omega)$, therefore after applying proposition \ref{LocalRellichForAdjoint} we obtain that
\begin{align*}
\int_{\Delta_r(q)}|\partial_{\nu}u|^2\,d\sigma&\leq C\int_{\Delta_{2r}(q)}|\nabla_Tu|^2\,d\sigma+C\int_{T_{2r}(q)}|\dive b||u\nabla u|+\frac{C}{r}\int_{T_{2r}(q)}|\nabla u|^2\\
&=C\int_{\Delta_{2r}(q)}|\nabla_Tu|^2\,d\sigma+CI_1+\frac{C}{r}\int_{\partial\Omega}|u|^2+\frac{C}{r}\int_{\partial\Omega}|\nabla_Tu|^2\,d\sigma,
\end{align*}
where $C$ is a good constant, and where we also used proposition \ref{SolidBoundForAdjoint}.

To bound $I_1$, note that, from H{\"o}lder's inequality 
\begin{align*}
I_1&\leq\left(\int_{T_{2r}(q)}|\dive b|^d\right)^{1/d}\|u\|_{L^{2^*}(T_{2r}(q))}\|\nabla u\|_{L^2(T_{2r}(q))}\\
&\leq\|\dive b\|_{L^d}\left(\|u\|_{L^{2^*}(T_{2r}(q))}^2+\|\nabla u\|_{L^2(T_{2r}(q))}^2\right)\\
&\leq\|\dive b\|_{L^d}\left(\|u\|_{L^{2^*}(T_{2r}(q))}^2+\|\nabla u\|_{L^2(\Omega)}^2\right)=\|\dive b\|_{L^d}(I_3+I_4).
\end{align*}
We estimate $I_3$ using lemma \ref{2*By2}, and $I_4$ by using proposition \ref{SolidBoundForAdjoint}. Then, for a good constant $C$,
\begin{align*}
I_3+I_4\leq\frac{C}{r}\int_{\Delta_{4r}(q)}|u_r^*|^2\,d\sigma+Cr\int_{\Delta_{4r}(q)}|(\nabla u)^*|^2\,d\sigma+C\int_{\partial\Omega}|u|^2+C\int_{\partial\Omega}|\nabla_Tu|^2\,d\sigma.
\end{align*}
But, from proposition \ref{UByNablaU},
\begin{align*}
\frac{C}{r}\int_{\Delta_{4r}(q)}|u_r^*|^2\,d\sigma&\leq \frac{C}{r}\int_{\Delta_{4r}(q)}\left(Cr^2|(\nabla u)^*|^2+C|u|^2\right)\,d\sigma\\
&=Cr\int_{\Delta_{4r}(q)}|(\nabla u)^*|^2\,d\sigma+\frac{C}{r}\int_{\Delta_{4r}(q)}|u|^2\,d\sigma,
\end{align*}
which shows that
\[
I_1\leq Cr\int_{\Delta_{4r}(q)}|(\nabla u)^*|^2\,d\sigma+\left(\frac{C}{r}+C\right)\int_{\partial\Omega}|u|^2\,d\sigma+C\int_{\partial\Omega}|\nabla_Tu|^2\,d\sigma,
\]
where $C$ is a good constant that also depends on the $d$ norm of $\dive b$. Therefore, plugging into the first estimate in this proof, we obtain that
\[
\int_{\Delta_r(q)}|\partial_{\nu}u|^2\,d\sigma\leq Cr\int_{\Delta_{4r}(q)}|(\nabla u)^*|^2\,d\sigma+\frac{C}{r}\int_{\partial\Omega}|u|^2+\frac{C}{r}\int_{\partial\Omega}|\nabla_Tu|^2\,d\sigma.
\]
To finish the proof, we then integrate for $q\in\partial\Omega$.
\end{proof}

We now turn to the analog of proposition \ref{InverseInequalityForGoodB}, which will lead to invertibility of $\mathcal{S}^t$ on $\partial\Omega$.

\begin{prop}\label{InverseInequalityForAdjointForGoodB}
Suppose that $\Omega\in\mathcal{D}$, $A\in M_{\lambda,\mu}^s(B)$ and $b\in\Lip(\mathbb R^d)$. Then, for every $f\in L^2(\partial\Omega)$,
\[
\|f\|_{L^2(\partial\Omega)}\leq C\|\mathcal{S}^tf\|_{W^{1,2}(\partial\Omega)},
\]
where $C$ is a good constant that also depends on $\|b\|_{\Dr_{d,\alpha}}$.
\end{prop}
\begin{proof}
As in the proof of proposition \ref{InverseInequalityForGoodB}, suppose first that $f$ is Lipschitz, and consider a Lipschitz extension $F:\overline{B}\to\mathbb R$ of $f$, which vanishes on $\partial B$. Set $u_+=\mathcal{S}_+^tf$, and $u_-=\mathcal{S}_-^tf$, then the jump relation (corollary \ref{JumpRelationForAdjoint}) with $H=F$ shows that
\[
\int_{\partial\Omega_j}\partial_{\nu_j}u_+\cdot F\,d\sigma_j-\int_{\partial\Omega_j'}\partial_{\nu_j'}u_-\cdot F\,d\sigma_j'\xrightarrow[j\to\infty]{}\int_{\partial\Omega}F^2\,d\sigma.
\]
Since now $F$ is continuous in $\Omega$, the Cauchy-Schwartz inequality shows that
\begin{align*}
\|F\|_{L^2(\partial\Omega)}^2&\leq\limsup_{j\to\infty}\left(\|\partial_{\nu_j}u_+\|_{L^2(\partial\Omega_j)}\|F\|_{L^2(\partial\Omega_j)}+\|\partial_{\nu_j'}u_-\|_{L^2(\partial\Omega_j)}\|F\|_{L^2(\partial\Omega_j')}\right)\\
&=\|F\|_{L^2(\partial\Omega)}\limsup_{j\to\infty}\left(\|\partial_{\nu_j}u_+\|_{L^2(\partial\Omega_j)}+\|\partial_{\nu_j'}u_-\|_{L^2(\partial\Omega_j)}^2,\right)
\end{align*}
therefore we obtain 
\begin{equation}\label{eq:AdjointSum}
\|F\|_{L^2(\partial\Omega)}^2\leq 2\limsup_{j\to\infty}\left(\|\partial_{\nu_j}u_+\|_{L^2(\partial\Omega_j)}^2+\|\partial_{\nu_j'}u_-\|_{L^2(\partial\Omega_j)}^2\right).
\end{equation}
Note now that $u_+$ is a $C^1$ solution in $\overline{\Omega_j}$, from propositions \ref{SingleLayerInequalitiesForAdjoint} and \ref{DerivativeRegularityForAdjoint}. Moreover, from the global Rellich estimate (proposition \ref{GlobalRellichForAdjoint}), we obtain that, for all $r<r_{\Omega}$,
\[
\|\partial_{\nu_j}u_+\|_{L^2(\partial\Omega_j)}^2\leq Cr\|(\nabla u_+)^*\|_{L^2(\partial\Omega_j)}^2+\frac{C}{r^d}\|u_+\|_{L^2(\partial\Omega_j)}^2+\frac{C}{r^d}\|\nabla_Tu_+\|_{L^2(\partial\Omega_j)}^2,\]
where $C_j$ is a good constant for $\Omega_j$. We now let $j\to\infty$, and applying proposition \ref{LimsupEstimateOnNormForAdjoint}, we obtain that
\begin{align*}
\limsup_{j\to\infty}\|\partial_{\nu_j}u_+\|_{L^2(\partial\Omega_j)}^2&\leq Cr\|(\nabla u_+)^*\|_{L^2(\partial\Omega)}^2+\frac{C}{r^d}\|\mathcal{S}f\|_{L^2(\partial\Omega)}^2+\frac{C}{r^d}\|\nabla_T\mathcal{S}f\|_{L^2(\partial\Omega)}^2\\
&=Cr\|(\nabla \mathcal{S}_+f)^*\|_{L^2(\partial\Omega)}^2+\frac{C}{r^d}\|\mathcal{S}f\|_{L^2(\partial\Omega)}^2+\frac{C}{r^d}\|\nabla_T\mathcal{S}f\|_{L^2(\partial\Omega)}^2\\
&\leq Cr\|f\|_{L^2(\partial\Omega)}^2+\frac{C}{r^d}\|\mathcal{S}f\|_{L^2(\partial\Omega)}^2+\frac{C}{r^d}\|\nabla_T\mathcal{S}f\|_{L^2(\partial\Omega)}^2,
\end{align*}
where we also used the maximal function bound from proposition \ref{SingleLayerInequalitiesForAdjoint}. A similar process shows that
\[
\limsup_{j\to\infty}\|\partial_{\nu_j'}u_-\|_{L^2(\partial\Omega_j')}^2\leq Cr\|f\|_{L^2(\partial\Omega)}^2+\frac{C}{r^d}\|\mathcal{S}f\|_{L^2(\partial\Omega)}^2+\frac{C}{r^d}\|\nabla_T\mathcal{S}f\|_{L^2(\partial\Omega)}^2.
\]
Adding those inequalities and plugging in \eqref{eq:AdjointSum}, we finally obtain that
\[
\|f\|_{L^2(\partial\Omega)}^2\leq Cr\|f\|_{L^2(\partial\Omega)}^2+\frac{C}{r^d}\|\mathcal{S}f\|_{L^2(\partial\Omega)}^2+\frac{C}{r^d}\|\nabla_T\mathcal{S}f\|_{L^2(\partial\Omega)}^2,
\]
where $C$ is a good constant that also depends on $\|\dive b\|_d$. Choosing $r<r_{\Omega}$ with $Cr<1/2$ shows that
\[
\|f\|_{L^2(\partial\Omega)}^2\leq C\|\mathcal{S}^tf\|_{L^2(\partial\Omega)}^2+C\|\nabla_T\mathcal{S}^tf\|_{L^2(\partial\Omega)}^2=C\|\mathcal{S}^tf\|_{W^{1,2}(\partial\Omega)},
\]
where $C$ is a good constant that also depends on $\|\dive b\|_d$. This shows the desired inequality for Lipschitz functions $f:\partial\Omega\to\mathbb R$.

To obtain the estimate for $f\in L^2(\partial\Omega)$, we use the fact that ${\rm Lip}(\partial\Omega)$ is dense in $L^2(\partial\Omega)$ and $\mathcal{S}^t:L^2(\partial\Omega)\to W^{1,2}(\partial\Omega)$ is continuous.
\end{proof}

We now pass to non differentiable drifts, as in proposition \ref{InverseInequality}.

\begin{prop}\label{InverseInequalityForAdjoint}
Suppose that $\Omega\in\mathcal{D}$, $A\in M_{\lambda,\mu}^s(B)$ and $b\in \Dr_{d,\alpha}(\mathbb R^d)$. Then, for every $f\in L^2(\partial\Omega)$,
\[
\|f\|_{L^2(\partial\Omega)}\leq C\|\mathcal{S}^tf\|_{W^{1,2}(\partial\Omega)},
\]
where $C$ is a good constant that also depends on $\|b\|_{\Dr_{d,\alpha}}$.
\end{prop}
\begin{proof}
Suppose first that $f$ is Lipschitz on $\partial\Omega$, and extend it to a Lipschitz function $F$ in $\overline{\Omega}$. Consider a mollification $b_n=b*\phi_n$ of $b$, where
\[
\phi_n(x)=n^d\phi(nx),
\]
and $\phi$ is positive, it is supported in $B_1$ and has integral $1$. We then compute
\[
|b_n(x)|\leq\int_{B_{1/n}}|b(x-z)||\phi_n(z)|\,dz\leq\|b\|_{\infty}\int_{B_{1/n}}\phi_n(z)\,dz=\|b\|_{\infty},
\]
and also
\begin{align*}
|b_n(x)-b_n(y)|&\leq\int_{B_{1/n}}|b(x-z)-b(y-z)||\phi_n(z)|\,dz\leq\|b\|_{C^{0,\alpha}}|x-y|^{\alpha}\int_{B_{1/n}}\phi_n(z)\,dz\\
&=\|b\|_{C^{0,\alpha}}|x-y|^{\alpha}
\end{align*}
which shows that $\|b_n\|_{C^{\alpha}(\Omega)}\leq\|b\|_{C^{\alpha}(\Omega)}$. In addition, for any $\psi\in C_c^{\infty}(\mathbb R^d)$,
\begin{align*}
\left|\int_{\mathbb R^d}b_n\nabla\psi\right|&=\left|\int_{\mathbb R^d}\int_{B_{1/n}}b(x-z)\phi_n(z)\nabla\psi(x)\,dzdx\right|\\
&\leq\int_{B_{1/n}}\phi_n(z)\left|\int_{\mathbb R^d}b(x-z)\nabla\psi(x)\,dx\right|\,dz\\
&=\int_{B_{1/n}}\phi_n(z)\left|\int_{\mathbb R^d}b(x)\nabla\psi(x+z)\,dx\right|\,dz\\
&\leq\int_{B_{1/n}}\phi_n(z)\left|\int_{\mathbb R^d}b(x)\nabla\psi(x+z)\,dx\right|\,dz\\
&\leq\int_{B_{1/n}}\phi_n(z)\|\dive b\|_d\|\nabla\psi\|_{L^{\frac{d}{d-1}}(\mathbb R^d)}\,dz\\
&=\|\dive b\|_d\|\nabla\psi\|_{L^{\frac{d}{d-1}}(\mathbb R^d)}.
\end{align*}
therefore $b_n\in\Dr_{d,\alpha}(\mathbb R^d)$, with
\[
\|b_n\|_{\Dr_{d,\alpha}(\mathbb R^d)}\leq\|b\|_{\Dr_{d,\alpha}(\mathbb R^d)}.
\]
Let now $G_n^t$ be Green's function for the operator
\[
L_n^t=-\dive(A\nabla u)-\dive(b_nu)
\]
in $B$, and set $\mathcal{S}_n^t$ to be the single layer potential operator on $\partial\Omega$ for the same operator; that is,
\[
\mathcal{S}_n^tf(p)=\int_{\partial\Omega}G_n^t(p,q)f(q)\,d\sigma(q),
\]
for $p\in\partial\Omega$. Set also $\beta=\alpha/2$. Then, we apply the estimates in proposition \ref{ContinuityArgumentForBAdjoint} for $\beta$ instead of $\alpha$ to obtain, as in proposition \ref{InverseInequality},
\[
\|\mathcal{S}_n^tf-\mathcal{S}^tf\|_{W^{1,2}(\partial\Omega)}\leq C\|b_n-b\|_{C^{\beta}(\Omega)}\|f\|_{L^2(\partial\Omega)},
\]
where $C$ is a good constant. But, the embedding $C^{\beta}(\Omega)\hookrightarrow C^{\alpha}(\Omega)$ is compact, and also $(b_n)$ is bounded in $C^{\alpha}(\Omega)$ and $b_n\to b$ in $C(\overline{\Omega})$, therefore, for a subsequence,
\[
\|b_{k_n}-b\|_{C^{\beta}(\Omega)}\xrightarrow[n\to\infty]{} 0,
\]
which shows that
\[
\|\mathcal{S}_{k_n}^tf-\mathcal{S}^tf\|_{W^{1,2}(\partial\Omega)}\xrightarrow[n\to\infty]{} 0.
\]
Since now $b_n\in C^1(\overline{\Omega})$, we can apply proposition \ref{InverseInequalityForAdjointForGoodB} to obtain that
\[
\|f\|_{L^2(\partial\Omega)}\leq C\|\mathcal{S}_{k_n}^tf\|_{W^{1,2}(\partial\Omega)},
\]
where $C$ is a good constant that does not depend on $n$. Letting $n\to\infty$ then shows the result for Lipschitz functions $f$. To pass to all $f\in L^2(\partial\Omega)$ we use the fact that ${\rm Lip}(\partial\Omega)$ is dense in $L^2(\partial\Omega)$ and $\mathcal{S}^t:L^2(\partial\Omega)\to W^{1,2}(\partial\Omega)$ is continuous.
\end{proof}

We are then led to the next theorem.

\begin{thm}\label{InvertibilityOfSAdjoint}
Let $\Omega$ be a Lipschitz domain, $A\in M_{\lambda,\mu}^s(\Omega)$ and $b\in \Dr_{d,\alpha}(\mathbb R^d)$. Then, the operator $\mathcal{S}^t:L^2(\partial\Omega)\to W^{1,2}(\partial\Omega)$ is invertible, with
\[
\|(\mathcal{S}^t)^{-1}f\|_{L^2(\partial\Omega)}\leq C\|f\|_{W^{1,2}(\partial\Omega)},
\]
and $C$ being a good constant that also depends on $\|b\|_{\Dr_{d,\alpha}}$.
\end{thm}
\begin{proof}
The proof is identical to the proof of theorem \ref{InvertibilityOfS}, but instead of the gradient estimates that appear in proposition \ref{ContinuityArgumentForB}, we use the estimate in proposition \ref{ContinuityArgumentForBAdjoint}. We then conclude using proposition \ref{InverseInequalityForAdjoint}.
\end{proof}

Using invertibility of $\mathcal{S}^t$, we can then obtain solvability of $R_2$ for the adjoint equation.

\begin{thm}\label{R2SolvabilityForAdjoint}
Let $\Omega$ be a Lipschitz domain, $A\in M_{\lambda,\mu}^s(\Omega)$, and $b\in \Dr_{d,\alpha}(\mathbb R^d)$. Then the Regularity problem $R_2$ is uniquely solvable in $\Omega$, with constants depending on $d,\lambda,\mu$, $\|b\|_{\Dr_{d,\alpha}(\Omega)}$, the Lipschitz character of $\Omega$ and the diameter of $\Omega$. Moreover, the solution admits the representation
\[
u(x)=\mathcal{S}_+^t((\mathcal{S}^t)^{-1}f)(x)=\int_{\partial\Omega}G^t(x,q)\left((\mathcal{S}^t)^{-1}f\right)(q)\,d\sigma(q).
\]
\end{thm}
\begin{proof}
The proof is identical to the proof of theorem \ref{R2Solvability}, after using theorems \ref{InvertibilityOfSAdjoint} and \ref{UniquenessForRegularityForAdjoint}.
\end{proof}

We will be able to drop the symmetry assumption on $A$ later, but we will then have to assume that the derivatives of $A$ are H{\"o}lder continuous (theorem \ref{RegularityForNonSymmetricForAdjoint}).
\section{The case of non-symmetric coefficients}
We are now in position to drop the symmetry assumption on $A$, and show solvability for the Dirichlet and the Regularity problem for $A\in M_{\lambda,\mu}(\Omega)$, for the operators $L$ and $L^t$. For this purpose, we will reduce the general case to the cases treated before, as in \cite{PipherDrifts}.

We first prove the next lemma, to explain how we will transform our equations so that the matrix $A$ becomes symmetric. The crucial observation is that we obtain an equation with a drift which depends on the derivatives of $A$, but it is divergence free.

\begin{lemma}\label{NonSymmetricEquations}
Let $A\in M_{\lambda,\mu}(\Omega)$, and let $b\in L^{\infty}(\Omega)$. Define $\tilde{b}$ by the relation
\[
\tilde{b}_i=\frac{1}{2}\sum_{j=1}^d\partial_j(a_{ij}-a_{ji}),\quad i=1,\dots d.
\]
Then $\dive \tilde{b}=0$, $\|\tilde{b}\|_{\infty}\leq C\mu$, and $u$ is a solution of the equation $-\dive(A^t\nabla u)+b\nabla u=0$ in $\Omega$ if and only if $u$ solves the equation
\[
-\dive\left(A_s\nabla u\right)+(b+\tilde{b})\nabla u=0
\]
in $\Omega$, where $A_s=\frac{1}{2}(A+A^t)$ is symmetric. Moreover, $v$ is a solution of the equation $-\dive(A\nabla v)-\dive(bv)=0$ in $\Omega$ if and only if $v$ solves the equation
\[
-\dive\left(A_s\nabla v\right)-\dive\left((b+\tilde{b})v\right)=0.
\]
\end{lemma}
\begin{proof}
First, in the sense of distributions, we compute
\[
2\dive \tilde{b}=\sum_{i=1}^d\partial_i\left(\sum_{j=1}^d\partial_j(a_{ij}-a_{ji})\right)=\sum_{i,j}\partial_{ij}a_{ij}-\sum_{i,j}\partial_{ij}a_{ji}=\sum_{i,j}\partial_{ji}a_{ji}-\sum_{i,j}\partial_{ij}a_{ij}=0,
\]
since $\partial_{ij}=\partial_{ji}$, which shows that $\tilde{b}$ has divergence $0$. Since $\tilde{b}$ is written as the sum of the derivatives of $A$, we also obtain that $\|\tilde{b}\|_{\infty}\leq C\mu$.

Let now $u\in W^{1,2}_{{\rm loc}}(\Omega)$, and $\phi\in C_c^{\infty}(\Omega)$. Since $\tilde{b}$ has divergence $0$, we integrate by parts to compute
\begin{align*}
\int_{\Omega}A^t\nabla u\nabla\phi+2\tilde{b}\nabla u\cdot\phi&=\int_{\Omega}A^t\nabla u\nabla\phi-2\tilde{b}\nabla\phi\cdot u\\
&=\int_{\Omega}\sum_{i,j}a_{ji}\partial_ju\cdot\partial_i\phi-\sum_{i,j}\partial_j(a_{ij}-a_{ji})\partial_i\phi\cdot u\\
&=\int_{\Omega}\sum_{i,j}a_{ji}\partial_ju\cdot\partial_i\phi+\sum_{i,j}(a_{ij}-a_{ji})(\partial_{ji}\phi\cdot u+\partial_i\phi\cdot\partial_ju)\\
&=\int_{\Omega}\sum_{i,j}a_{ij}\partial_ju\cdot\partial_i\phi+\sum_{i,j}(a_{ij}-a_{ji})\partial_{ji}\phi\cdot u=\int_{\Omega}A\nabla u\nabla\phi,
\end{align*}
since $\partial_{ij}=\partial_{ji}$. Therefore,
\[
\int_{\Omega}(A+A^t)\nabla u\nabla\phi+2(b+\tilde{b})\nabla u\cdot\phi=\int_{\Omega}2A\nabla u\nabla\phi+2b\nabla u\cdot\phi,
\]
which shows the first claim. For the second claim, a similar calculation shows that
\[
\int_{\Omega}A\nabla u\nabla\phi+2\tilde{b}\nabla\phi\cdot u=\int_{\Omega}A^t\nabla u\nabla\phi,
\]
therefore
\[
\int_{\Omega}(A+A^t)\nabla u\nabla\phi+2(b+\tilde{b})\nabla\phi\cdot u=\int_{\Omega}2A^t\nabla u\nabla\phi+2b\nabla\phi\cdot u,
\]
which completes the proof.
\end{proof}

\begin{thm}\label{DirichletForNonSymmetric}
Let $\Omega$ be a Lipschitz domain, and let $A\in M_{\lambda,\mu}(\Omega)$ and $b\in\Dr_{p_d}(\Omega)$. Then, there exists $\e>0$ such that, for $p\in(2-\e,\infty)$, the Dirichlet problem $D_p$ for the equation
\[
-\dive(A\nabla u)+b\nabla u=0
\]
is uniquely solvable in $\Omega$, with constants depending only on $d,p,\lambda,\mu,\|b\|_{\Dr_{p_d}},{\rm diam}(\Omega)$ and the Lipschitz character of $\Omega$. Here, $p_d=2$ for $d=3$, and $p_d=d/2$ for $d\geq 4$.
\end{thm}
\begin{proof}
Consider the $\e$ that appears in theorem \ref{GeneralDirichlet}, and let $p\in(2-\e,\infty)$. Let $f\in L^p(\partial\Omega)$, and consider the symmetric matrix $A_s$ and $\tilde{b}$ that appear in lemma \ref{NonSymmetricEquations}. Since $\tilde{b}$ is divergence free, we obtain that
\[
\|b+\tilde{b}\|_{\Dr_{p_d}}\leq \|b\|_{\Dr_{p_d}}+\|\tilde{b}\|_{\infty}\leq \|b\|_{\Dr_{p_d}}+C\mu.
\]
Therefore, from theorem \ref{GeneralDirichlet}, there exists a unique solution to the equation
\[
\left\{\begin{array}{c l}
-\dive\left(A_s\nabla u\right)+(b+\tilde{b})\nabla u=0, &{\rm in}\,\,\Omega\\
u=f,& {\rm on}\,\,\partial\Omega\\
\|u^*\|_{L^p(\partial\Omega)}\leq C\|f\|_{L^p(\partial\Omega)}
\end{array}
\right..
\]
Combining with lemma \ref{NonSymmetricEquations}, we obtain that $u$ is the unique solution of the Dirichlet problem $D_p$ for the equation $-\dive(A\nabla u)+b\nabla u=0$ with $u=f$ on the boundary, which completes the proof.
\end{proof}

\begin{thm}\label{RegularityForNonSymmetric}
Let $\Omega$ be a Lipschitz domain, and let $A\in M_{\lambda,\mu}(\Omega)$ and $b\in L^{\infty}(\Omega)$. Then, the Regularity problem $R_2$ for the equation 
\[
-\dive(A\nabla u)+b\nabla u=0
\]
is uniquely solvable in $\Omega$, with constants depending only on $d,\lambda,\mu,\|b\|_{\infty},{\rm diam}(\Omega)$ and the Lipschitz character of $\Omega$.
\end{thm}
\begin{proof}
The proof is similar to the proof of theorem \ref{DirichletForNonSymmetric}, using theorem \ref{R2Solvability}.
\end{proof}

\begin{thm}\label{DirichletForNonSymmetricForAdjoint}
Let $\Omega$ be a Lipschitz domain, and let $A\in M_{\lambda,\mu}(\Omega)$ and $b\in L^{\infty}(\Omega)$. Then, the Dirichlet problem $D_2$ for the equation
\[
-\dive(A^t\nabla u)-\dive(bu)=0
\]
is uniquely solvable in $\Omega$, with constants depending only on $d,\lambda,\mu,\|b\|_{\infty},{\rm diam}(\Omega)$ and the Lipschitz character of $\Omega$.
\end{thm}
\begin{proof}
The proof is similar to the proof of theorem \ref{DirichletForNonSymmetric}, using theorem \ref{GeneralDirichletForAdjoint}.
\end{proof}

\begin{thm}\label{RegularityForNonSymmetricForAdjoint}
Let $\Omega$ be a Lipschitz domain, and let $A\in M_{\lambda,\mu}(\Omega)$ with $A\in C^{1,\alpha}(\Omega)$, and $b\in \Dr_{d,\alpha}(\mathbb R^d)$. Then, the Regularity problem $R_2$ for the equation
\[
-\dive(A^t\nabla u)-\dive(bu)=0
\]
is uniquely solvable in $\Omega$, with constants depending only on $d,\lambda,\|A\|_{C^{1,\alpha}},\|b\|_{\Dr_{d,\alpha}},{\rm diam}(\Omega)$ and the Lipschitz character of $\Omega$.
\end{thm}
\begin{proof}
The proof is similar to the proof of theorem \ref{DirichletForNonSymmetric}, using theorem \ref{R2SolvabilityForAdjoint}, since
\[
\|b+\tilde{b}\|_{\Dr_{d,\alpha}}\leq \|b\|_{\Dr_{d,\alpha}}+\|\tilde{b}\|_{C^{\alpha}},
\]
and the last term is bounded above by a constant times the $C^{1,\alpha}$ norm of $A$.
\end{proof}

\bibliographystyle{amsalpha}
\bibliography{Bibliography}

\newcommand{\etalchar}[1]{$^{#1}$}
\providecommand{\bysame}{\leavevmode\hbox to3em{\hrulefill}\thinspace}
\providecommand{\MR}{\relax\ifhmode\unskip\space\fi MR }
\providecommand{\MRhref}[2]{%
  \href{http://www.ams.org/mathscinet-getitem?mr=#1}{#2}
}
\providecommand{\href}[2]{#2}
\begin{thebibliography}{HKMP15}

\bibitem[AAA{\etalchar{+}}11]{HofmannAnalyticity}
M.~Angeles Alfonseca, Pascal Auscher, Andreas Axelsson, Steve Hofmann, and
  Seick Kim, \emph{Analyticity of layer potentials and {$L^2$} solvability of
  boundary value problems for divergence form elliptic equations with complex
  {$L^\infty$} coefficients}, Adv. Math. \textbf{226} (2011), no.~5,
  4533--4606.

\bibitem[DHM16]{MayborodaGreen}
Blair Davey, Jonathan Hill, and Svitlana Mayboroda, \emph{Fundamental matrices
  and green matrices for non-homogeneous elliptic systems}, arXiv preprint
  arXiv:1610.08064 (2016).

\bibitem[DK87]{DahlbergKenig}
Bj\"orn E.~J. Dahlberg and Carlos~E. Kenig, \emph{Hardy spaces and the
  {N}eumann problem in {$L^p$} for {L}aplace's equation in {L}ipschitz
  domains}, Ann. of Math. (2) \textbf{125} (1987), no.~3, 437--465.

\bibitem[DPP07]{DindosPetermichl}
Martin Dindos, Stefanie Petermichl, and Jill Pipher, \emph{The {$L^p$}
  {D}irichlet problem for second order elliptic operators and a {$p$}-adapted
  square function}, J. Funct. Anal. \textbf{249} (2007), no.~2, 372--392.

\bibitem[Eva10]{Evans}
Lawrence~C. Evans, \emph{Partial differential equations}, second ed., Graduate
  Studies in Mathematics, vol.~19, American Mathematical Society, Providence,
  RI, 2010.

\bibitem[Fol95]{Folland}
Gerald~B. Folland, \emph{Introduction to partial differential equations},
  second ed., Princeton University Press, Princeton, NJ, 1995.

\bibitem[Geh73]{Gehring}
F.~W. Gehring, \emph{The {$L^{p}$}-integrability of the partial derivatives of
  a quasiconformal mapping}, Acta Math. \textbf{130} (1973), 265--277.

\bibitem[GT01]{Gilbarg}
David Gilbarg and Neil~S. Trudinger, \emph{Elliptic partial differential
  equations of second order}, Classics in Mathematics, Springer-Verlag, Berlin,
  2001, Reprint of the 1998 edition.

\bibitem[GW82]{Gruter}
Michael Gr{\"u}ter and Kjell-Ove Widman, \emph{The {G}reen function for
  uniformly elliptic equations}, Manuscripta Math. \textbf{37} (1982), no.~3,
  303--342.

\bibitem[HKMP15]{HofmannDuality}
Steve Hofmann, Carlos Kenig, Svitlana Mayboroda, and Jill Pipher, \emph{The
  regularity problem for second order elliptic operators with complex-valued
  bounded measurable coefficients}, Math. Ann. \textbf{361} (2015), no.~3-4,
  863--907.

\bibitem[HL01]{HofmannLewis}
Steve Hofmann and John~L. Lewis, \emph{The {D}irichlet problem for parabolic
  operators with singular drift terms}, Mem. Amer. Math. Soc. \textbf{151}
  (2001), no.~719, viii+113.

\bibitem[IR05]{RiahiGreen}
Abdoul Ifra and Lotfi Riahi, \emph{Estimates of {G}reen functions and harmonic
  measures for elliptic operators with singular drift terms}, Publ. Mat.
  \textbf{49} (2005), no.~1, 159--177.

\bibitem[JK82]{JerisonKenigBoundary}
David~S. Jerison and Carlos~E. Kenig, \emph{Boundary behavior of harmonic
  functions in nontangentially accessible domains}, Adv. in Math. \textbf{46}
  (1982), no.~1, 80--147.

\bibitem[Ken94]{KenigCBMS}
Carlos~E. Kenig, \emph{Harmonic analysis techniques for second order elliptic
  boundary value problems}, CBMS Regional Conference Series in Mathematics,
  vol.~83, Published for the Conference Board of the Mathematical Sciences,
  Washington, DC; by the American Mathematical Society, Providence, RI, 1994.

\bibitem[KP01]{PipherDrifts}
Carlos~E. Kenig and Jill Pipher, \emph{The {D}irichlet problem for elliptic
  equations with drift terms}, Publ. Mat. \textbf{45} (2001), no.~1, 199--217.

\bibitem[KS11]{KenigShen}
Carlos~E. Kenig and Zhongwei Shen, \emph{Layer potential methods for elliptic
  homogenization problems}, Comm. Pure Appl. Math. \textbf{64} (2011), no.~1,
  1--44.

\bibitem[LSW63]{Littman}
W.~Littman, G.~Stampacchia, and H.~F. Weinberger, \emph{Regular points for
  elliptic equations with discontinuous coefficients}, Ann. Scuola Norm. Sup.
  Pisa (3) \textbf{17} (1963), 43--77.

\bibitem[MMT01]{MitreaTaylorBook}
Dorina Mitrea, Marius Mitrea, and Michael Taylor, \emph{Layer potentials, the
  {H}odge {L}aplacian, and global boundary problems in nonsmooth {R}iemannian
  manifolds}, Mem. Amer. Math. Soc. \textbf{150} (2001), no.~713, x+120.

\bibitem[MT99]{MitreaTaylor}
Marius Mitrea and Michael Taylor, \emph{Boundary layer methods for {L}ipschitz
  domains in {R}iemannian manifolds}, J. Funct. Anal. \textbf{163} (1999),
  no.~2, 181--251.

\bibitem[Ngu16]{NguyenPaper}
Nguyen~T. Nguyen, \emph{The {D}irichlet and regularity problems for some second
  order linear elliptic systems on bounded {L}ipschitz domains}, Potential
  Anal. \textbf{45} (2016), no.~1, 167--186.

\bibitem[Ste70]{SteinSingular}
Elias~M. Stein, \emph{Singular integrals and differentiability properties of
  functions}, Princeton Mathematical Series, No. 30, Princeton University
  Press, Princeton, N.J., 1970.

\bibitem[Ste93]{SteinHarmonic}
\bysame, \emph{Harmonic analysis: real-variable methods, orthogonality, and
  oscillatory integrals}, Princeton Mathematical Series, vol.~43, Princeton
  University Press, Princeton, NJ, 1993.

\bibitem[Ver84]{Verchota}
Gregory Verchota, \emph{Layer potentials and regularity for the {D}irichlet
  problem for {L}aplace's equation in {L}ipschitz domains}, J. Funct. Anal.
  \textbf{59} (1984), no.~3, 572--611.

\end{thebibliography}


\end{document}